\documentclass[12pt]{article}
\usepackage{amsmath, amsfonts, amssymb,amsthm, amscd}

\usepackage{graphicx}
\usepackage{epsf}
\usepackage{epsfig}
\usepackage{rlepsf}
\usepackage{psfrag}
\usepackage{color}
\usepackage{hyperref}
\newtheorem{thm}{Theorem}[section]    % Standard theorem environment
\newtheorem{lem}[thm]{Lemma}          % Lemma environment with numbering
\newtheorem{prop}[thm]{Proposition}          % Lemma environment with numbering
\newtheorem{cor}[thm]{Corollary}  
\theoremstyle{definition}
\newtheorem{defn}[thm]{Definition}    % Definition environment with
\newtheorem{rem}[thm]{Remark}
\theoremstyle{definition}
\newtheorem{examp}[thm]{Example}    % Definition environment with
\newenvironment{Myitemize}{%
\begin{itemize}}{\end{itemize}}

\theoremstyle{definition}
\newtheorem*{quest}{Question}

\theoremstyle{definition}
\newtheorem*{induction}{${\bf ( S_{k})}$}
\theoremstyle{definition}

\newcommand{\bp}{\begin{proof}}
\newcommand{\ep}{\end{proof}}
\newcommand{\avk}{\text{av}_{a_k}}

\newcommand{\nr}{\pmb{|}\negthinspace\pmb{|}}
\newcommand{\norm}[1]{\left\Vert#1\right\Vert}
\newcommand{\abs}[1]{\left|#1\right|}

\newcommand{\carem}{\ov{\partial}_0}
\newcommand{\care}{\ov{\partial}}
\newcommand{\ssc}{\text{sc}}
\newcommand{\loc}{{\text{loc}}}
\newcommand{\av}{\text{av}_a}
\newcommand{\co}{{\mathcal O}}

\renewcommand{\epsilon}{\varepsilon}

\newcommand{\what}{\widehat}
\newcommand{\wh}{\widehat}

\newcommand{\wt}{\widetilde}

\newcommand{\ov}{\overline}

\newcommand{\Z}{{\mathbb Z}}

\newcommand{\id}{\operatorname{id}}

\newcommand{\R}{{\mathbb R}}

\providecommand{\ker}[1]{$\text{ker}\ {#1}$}

\newcommand{\N}{{\mathbb N}}

\newcommand{\C}{{\mathbb C}}
\newcommand{\mcd}{{\mathcal D}}

\def\abs#1{\mathopen|#1\mathclose|}

\def\norm#1{\mathopen\|#1\mathclose\|}

{\catcode`"=12 \gdef\hex{"}}

\mathchardef\laplace=\hex0001

\mathchardef\nabla=\hex0272

\makeatletter

\def\@@dalembert#1#2{\setbox0\hbox{$#1\mathrm I$}

  \vrule height\ht0 depth\z@ width.04\ht0

  \rlap{\vrule height\ht0 depth-.96\ht0 width.8\ht0}

  \vrule height.1\ht0 depth\z@ width.8\ht0

  \vrule height\ht0 depth\z@ width.1\ht0 }

\def\dalembert{\mathbin{\mathpalette\@@dalembert{}}\,}

\makeatother

\begin{document}
\title{Sc-Smoothness, Retractions and
New Models for Smooth Spaces}
\author{ H. Hofer\footnote{Funding provided by NSF grant  DMS-0603957.}\\
Institute for Advanced Study\\ USA\\ \small{\texttt{hofer@ias.edu}}
\and K. Wysocki \footnote{Funding provided by NSF grants DMS-0606588, DMS-0906280,  and the Oswald Veblen Fund.  }\\Penn State University\\
USA\\ \small{\texttt{wysocki@math.psu.com}}
\and E. Zehnder\footnote{Funding provided by the Fund for Mathematics and the Oswald Veblen Fund.}\\ETH-Zurich\\Switzerland\\ \small{\texttt{zehnder@math.ethz.ch}}}

 \maketitle

\begin{center}
\Large{Dedicated to Louis Nirenberg on the Occasion of his 85th Birthday}
\end{center}
\tableofcontents
\section{Sc-Smoothness and M-Polyfolds}\label{chapter1}

In  paper \cite{HWZ2}, the authors  have described a generalization of  differential geometry based on the notion of splicings. The associated Fredholm theory  in polyfolds, presented in  \cite{HWZ2,HWZ3,HWZ3.5}, is a crucial ingredient  in the  functional analytic foundation of  the Symplectic  Field Theory (SFT). The theory also applies to the Floer theory as well as  to the Gromov-Witten theory and quite generally should have  applications in nonlinear analysis, in particular in  studies of  families of elliptic pde's on varying domains,
which can even change their topology.

A basic ingredient for the generalization of differential geometry is a new notion of differentiability in infinite dimensions, called sc-smoothness. The goal of this paper is
to describe these ideas and, in particular,  to provide some of the ``hard''  analysis results which enter the polyfold constructions in symplectic field theory (SFT).
The advantage of the polyfold Fredholm theory can be summarized  as follows.

\begin{Myitemize}
\item Many spaces, though they do not carry a classical smooth structure,  can be equipped with a  weak version of a smooth structure. The local models for the spaces,  be they finite- or infinite-dimensional, can  even have locally varying dimensions. 
\item Since the notion of the smooth structure is so weak,  there are many charts so that  many spaces carry a manifold structure  in the new smoothness category.
\item Finite-dimensional subsets  in good position in these generalized manifold inherit  an induced differentiable structure in the familiar  sense.
\item There is a notion of a bundle. Smooth sections
of such  bundles,  which, under a suitable coordinate change,  can be brought  into a sufficiently nice form,  are Fredholm sections. A Fredholm section  looks nice (near a point)  only
in a very particular coordinate system and not necessarily in the smoothly compatible other ones. Since we have plenty of  coordinate  systems,   many sections turn out to be Fredholm.
\item The zero sets of Fredholm sections lie in the smooth parts of the big ambient space, so that they look smooth  in all coordinate descriptions (systems). The invariance of  the properties of solution sets under arbitrary coordinate changes is, of course, a crucial input for having a viable theory.
\item  There is an intrinsic perturbation theory,  and moreover, a version of Sard-Smale's theorem  holds true. In applications, for example to a geometric problem, one might try to make the problem generic by perturbing auxiliary geometric data. As the Gromov-Witten and SFT-examples show,  this  is,  in general,  not possible and one needs to find a sufficiently large abstract universe,   which offers  enough freedom to construct  generic perturbations. The abstract polyfold Fredholm theory provides such a framework. 
\item  Important for the applications is  a version of this new Fredholm theory for an even more general class of spaces,  called polyfolds. In this case the generic solution spaces can be thought of locally as a finite union of (classical) manifolds divided out by a finite group action. Moreover,  the points in these spaces carry rational weights. Still the  integration of differential forms can be defined for such spaces and Stokes' theorem is valid. This is used in order to  define invariants. The Gromov-Witten invariants provide an example.
\end{Myitemize}

The current paper develops the analytical foundations for some of the applications  of the theory described above. It also provides  examples illustrating the ideas.

The organization of the paper is as follows. 

The introductory chapter describes the new notions of smoothness for spaces and mappings leading, in particular, to novel local models  of spaces,  which generalize  manifolds and which  are called M-polyfolds. The general Fredholm theory in this analytical setting is outlined and an outlook to some applications is given, the proofs of which are postponed to  chapter  \ref{longcylinders}.

The second chapter is of technical nature and is devoted to detailed proofs of the new smoothness results which are crucial for many applications. 

The third chapter illustrates the concepts by constructing M-polyfold structures on a set of  mappings between conformal cylinders which break apart as the modulus tends to infinity. A strong bundle over this M-polyfold is constructed which admits the Cauchy-Riemann operator as an sc-smooth Fredholm section. Its zero-set consists of the holomorphic isomorphisms between cylinders of various sizes. Since the solution set carries a smooth structure, this  has interesting functional analytic consequences for the behavior of families of holomorphic mappings. 
\mbox{}\\[1ex]

{\bf Acknowledgement:} We  would like to thank  J. Fish for useful comments and suggestions.

\subsection{Sc-Structures on Banach Spaces}\label{Sc_Banach-ssection}
Sc-structures on Banach spaces  generalize  the  smooth structure from  finite dimensions to infinite dimensions. We first  recall  the definition of an sc-structure on a Banach space $E$ from \cite{HWZ2}.  In the following $\N_0$ stands  for $\N\cup \{0\}$.

\begin{defn}
 An sc-structure on a Banach space $E$  is a nested sequence of Banach spaces $(E_m)_{m\in \N_0}$,
 $$E=:E_0\supset E_1\supset \ldots \supset E_{\infty}:=\bigcap_{m \in \N_0}E_m, $$ 
 so that the following two conditions  are satisfied:
\begin{itemize}
\item[(1)] The inclusion maps $E_{m+1}\to  E_m$ are compact operators.
 \item[(2)] The vector space $E_\infty$  is dense in every $E_m$.
  \end{itemize}
\end{defn}

Points in $E_{\infty}$ are called smooth points. What  just has been defined is a compact discrete scale of Banach spaces  which is a standard object in interpolation theory for which we refer to \cite{Tr}.  Our  interpretation as a generalization of a smooth structure on $E$ seems to be new. The only sc-structure on a finite-dimensional vector space $E$ is given by the constant structure $E_m=E$ for all $m$. If $E$ is an  infinite-dimensional Banach space, the constant structure is not  an sc-structure because it fails  property (1).

\begin{defn}
 A linear map $T:E\to  F$ between  the two sc-Banach spaces $E$ and $F$ is called an $\ssc$-operator if  $T(E_m)\subset F_m$ and  if $T:E_m\to  F_m$ is continuous for every $m\in \N_0$.
\end{defn}
We shall need the notion of a partial quadrant $C$ in an $\ssc$-Banach space $E$.
\begin{defn} 
A closed subset of an sc-Banach space $E$ is called a partial quadrant if there are  an sc-Banach space  $W$, a nonnegative integer $k$,  and a linear $\ssc$-isomorphism $T:E\to  \R^k\oplus W$  so that $T(C)=[0,\infty)^k\oplus W$.
\end{defn}
Given a partial quadrant  $C$ in an sc-Banach space $E$, we define   the degeneration index
$$
d_C:C\to  \N_0
$$
as follows. We choose a linear $\ssc$-isomorphism $T:E\to  \R^k\oplus W$ satisfying $T(C)=[0,\infty)^k\oplus W$. Hence for $x\in C$, we have  
$$T(x)=(r_1, \ldots, r_k, w), \quad x\in C, $$
where   $(r_1, \ldots, r_k)\in [0,\infty )^k$ and $w\in W$. Then we  define  the integer $d_C(x)$ by
$$
d_C(x)=\sharp\{i\in \{1,\ldots, k\}\vert \, r_i=0\}.$$
It is not difficult to see that this definition is independent of  the choice of an sc-linear isomorphism $T$.

Let $U$ be  a relatively open subset of a  partial quadrant $C$  in an sc-Banach space $E$. Then the  sc-structure on $E$ induces  the sc-structure on $U$  defined by the sequence $U_m=U\cap E_m$ equipped with the topology of $E_m$ and  called  the induced sc-structure on $U$.  The points of $U_{\infty}=U\cap E_{\infty}$ are called  smooth points of $U$.  We adopt the convention that  $U^k$  denotes  the set $U_k$  equipped  with the sc-structure $(U_{k})_{m}:=U_{k+m}$ for all $m\in \N_0$.  If $U$ and $V$ are open subsets equipped with the induced sc-structures, we write  $U\oplus V$ for the product $U\times V$ equipped with the sc-structure $(U_m\times V_m)_{m\in \N_0}$.
\begin{defn}
If  $U$ is a relatively open subset of a partial quadrant $C$ in an sc-Banach space $E$,  then its 
tangent $TU$ is  defined by
$$
TU=U^1\oplus E.
$$
\end{defn}
\begin{examp}\label{sc-example1}
A good example which illustrates  the concepts,  and is also  relevant for SFT,  is  as follows.  We choose a strictly increasing sequence $(\delta_m)_{m\in \N_0}$  of real numbers starting with $\delta_0=0$.  We consider the Banach spaces  $E=L^2(\R\times S^1)$ and  $E_m=H^{m,\delta_m}(\R\times S^1)$ where   the space  $H^{m,\delta_m}(\R\times S^1)$ consists of those  elements in $E$  having 
 weak  partial derivatives up to order $m$ which if  weighted by $e^{\delta_m |s|}$ belong to $E$.  Using Sobolev's compact embedding theorem for bounded domains and the assumption that the sequence $(\delta_m)$ is strictly increasing, one sees that the sequence $(E_m)_{m\in \N_{0}}$ defines  
an  sc-structure on  $E$.  We take as the partial quadrant $C$ the whole space $E$ and let 
$B_E$ be the  open unit ball  centered at $0$  in $E$. Then the tangent of $B_E$  is given by
$$
TB_E =(B_E)^1\oplus E=\{(u,h)\vert \,  u\in H^{1,\delta_1},\, \abs{u}_{L^2}<1,\ h\in L^2\}.
$$
The  sc-structure  on  $TB_E$ is defined  by
$$
{(TB_E)}_m=\left\{(u,h)\ |\ u\in H^{m+1,\delta_{m+1}},\, \abs{u}_{L^2}<1, \, h\in H^{m,\delta_m} \right\}.
$$
\end{examp}

The notion of a continuous map $f:U\to  V$ between two relatively open subsets of partial quadrants in sc-Banach spaces is as follows.  
\begin{defn}
A map $f:U\to  V$ is said to be $\ssc^0$  if  $f(U_m)\subset V_m$ for all $m\in \N_{0}$ and  if the induced maps  $f:U_m\to  V_m$ are  continuous.
\end{defn}
\begin{examp}\label{shift-example1}
An important example used later on  is the shift-map.  We consider the Hilbert space   $E=L^2(\R\times S^1)$ equipped with  the sc-structure $(E_m)_{m\in \N_0}$ introduced  in Example \ref{sc-example1}. 
Then we define the map
$$
\Phi:\R^2\oplus L^2(\R\times S^1)\to  L^2(\R\times S^1), \quad (R,\vartheta ,u)\mapsto  (R,\vartheta )\ast u
$$
where 
$$((R,\vartheta )\ast u)(s,t)=u(s+R,t+\vartheta) .$$

The shift-map $\Phi$  is $\ssc^0$  as proved in Proposition \ref{sc}. It is clearly not differentiable in the classical sense.   However, in Proposition \ref{sc-sm} we shall prove that the map $\Phi$ is  sc-smooth for the new notion of smoothness which we shall introduce  next.
The shift map will be an important  ingredient in later constructions and its sc-smoothness will be crucial.

\end{examp}

\subsection{Sc-Smooth Maps and M-Polyfolds}
Having defined an appropriate notion of continuity we define what it means that  the map is of class $\ssc^1$. This is the notion corresponding to a map being $C^1$ in our  sc--framework.
\begin{defn}\label{sc-def}
Let $U$ and $V$ be relatively open subsets of partial quadrants $C$ and $D$ in sc-Banach spaces $E$ and $F$, respectively.  An  $\ssc^0$-map $f:U\to  V$ is said to be $\ssc^1$ if  for every $x\in U_1$ there exists a bounded  linear operator $Df(x):E_0\to  F_0$ so that the following holds:
\begin{itemize}
\item[(1)]  If  $h\in E_1$ and  $x+h\in C$,  then
$$
\lim_{\abs{h}_1\to  0} \frac{1}{\abs{h}_1}\cdot \abs{f(x+h)-f(x)-Df(x)h}_0 =0.
$$
\item[(2)] The map $Tf:TU\to  TV$ , called the tangent map of $f$,  and defined by 
$$
(x, h)\mapsto (f(x),Df(x)h),$$ 
is of class $\ssc^0$.
\end{itemize}
 \end{defn}
In general,  the map $U_1\to  L(E_0,F_0)$, $x\to  Df(x)$ will not(!) be continuous if the space of bounded linear operators is equipped with the operator norm. However,  if we equip it with the compact open topology it will be continuous. The $\ssc^1$-maps between finite dimensional Banach spaces are the familiar $C^1$-maps.

Proceeding inductively,  we define  what it means for  the map $f$  to be  $\ssc^k$ or $\ssc^\infty$. Namely, an $\ssc^0$--map $f$ is said to be an $\ssc^2$--map if  it is $\ssc^1$ and if  its tangent map $Tf:TU\to TV$ is  $\ssc^1$. By Definition \ref{sc-def},  the  tangent map of $Tf$,
$$T^2f:=T(Tf):T^2(U)=T(TU)\to T^2(V)=T(TV),$$
is of class $\ssc^0$. 
If the tangent map $T^2f$  is $\ssc^1$, then $f$ is said to be $\ssc^3$,  and so on. The map $f$ is $\ssc^{\infty}$, if it is $\ssc^k$ for all $k$.

Useful in our applications are the next two propositions  which relate the  sc-smoothness with the familiar notion of smoothness. 

\begin{prop}[Upper Bound]\label{up-prop}
 Let $E$ and $F$ be sc-Banach spaces and  let $U$ be a relatively  open subset of a partial quadrant $C$ in $ E$. Assume that $f:U\to  F$  is an $\ssc^0$-map so that for every $0\leq l \leq k$ and every $m\geq 0$ the induced map 
$$
f:U_{m+l}\to  F_m
$$
is  of class $C^{l+1}$. Then  $f$ is $\ssc^{k+1}$.
\end{prop}

\begin{prop}[Lower Bound] \label{lo-prop} Let $E$ and $F$ be sc-Banach spaces and  let $U$ be a relatively  open subset of a partial quadrant $C$ in $ E$.  If the map 
$f:U\to  F$ is $\ssc^k$, then the induced map
$$
f:U_{m+l}\to  F_m
$$
is of class $C^l$  for every $0\leq l \leq k$  and every  $m\geq 0$.
\end{prop}
The proofs of the two propositions will be carried out  in section \ref{section2.1}.
In view of the following chain rule,  the sc-smoothness  is a viable concept.

\begin{thm}[Chain Rule]\label{chain-thm1}
Assume that $U$,  $V$,  and $W$ are relatively open subsets of partial quadrants in sc-Banach spaces and  let  $f:U\to  V$ and $g:V\to  W$ be  $\ssc^1$. Then the composition $g\circ f:U\to  W$ is $\ssc^1$ and 
$$T(g\circ f) =(Tg)\circ (Tf).$$
\end{thm}
The proof can be found in   \cite{HWZ2}. We would like to point out that  the proof relies  on the assumption that the inclusion operators between spaces in the  nested sequence of Banach spaces are compact.

The next  definition introduces the notions of an sc-smooth retraction and an sc-smooth retract. This will be the starting point for a differential geometry based on new local models.
\begin{defn}
Let $U$ be a relatively open subset of a partial quadrant $C$ in an  $\ssc$-Banach space $E$. An  $\ssc$-smooth map $r:U\to  U$ is called an  $\ssc^\infty$-retraction provided it satisfies
$$r\circ r=r.$$
A subset $O$ of a partial quadrant $C$ is called an sc-smooth or $\ssc^\infty$-retract (relative to  $C$) if  there exists a relatively open subset $U\subset C$ and an $\ssc$-smooth retraction $r:U\to  U$ so that $$O=r(U).$$
\end{defn}

If $r:U\to  U$ is an $\ssc^\infty$-retraction,  then its  tangent map  $Tr:TU\to  TU$ is also an  $\ssc^\infty$-retraction. 
This follows from the chain rule.  Next  comes the crucial definition of the new local models of smooth spaces.
\begin{defn}\label{m-pl}
A local M-polyfold model is a triple $(O,C,E)$ in which  $E$ is an sc-Banach space, $C$ is  a partial quadrant of $E$, and $O$ is a  subset of $C$  having the following properties:
\begin{itemize}
\item[(1)] There is an sc-smooth retraction $r:U\to  U$ defined on a relative open subset $U$ of $C$ so that 
$$O=r(U).$$
\item[(2)] For every  smooth point $x\in O_\infty$,  the kernel of the map  $(\id-Dr(x))$ possesses  an sc-complement which is  contained in $C$.
\item[(3)] For every $x\in O$,  there exists a sequence of smooth points $(x_k)\subset O_\infty$ converging to $x$ in $O$ and satisfying $d_C(x_k)=d_C(x)$.
\end{itemize}
\end{defn}
The choice of $r$ in the above definition is irrelevant as long as it is an sc-smooth retraction onto $O$ defined on a relatively open subset $U$ of $C$.

A special  M-polyfold model has the form $(O,E,E)$.
Such triples  can be viewed as the local models for sc-smooth space $S$  without boundary whereas the more general triples are models for spaces with boundaries with corners. In the case without boundary the conditions (2) and (3) of Definition \ref{m-pl} are automatically satisfied.

In our applications the local sc-models $(O,C,E)$ quite often arise in the following way. We assume that we  are given a partial quadrant $D$ in an sc-Banach space $W$ and a relatively open subset $V$ of $D$. Moreover, we assume that for every $v\in V$ we have a bounded  linear 
projection 
$$\pi_v:F\to F$$
into  another sc-Banach space $F$.  In general,  the projection $\pi_v$ is not an sc-operator. %However, 
We require that the  map  
$$
V\oplus F\to  F, \quad (v,f)\mapsto  \pi_v(f)
$$
is  sc-smooth. Then we look at  the sc-Banach space $E=W\oplus F$, the partial quadrant $C=D\oplus F$,  and the relatively open subset $U=V\oplus F$ of $E$. Finally,  we define the map 
$r:U\to U$  by 
$$r(v,f)=(v,\pi_v(f)).$$ 
Then  the map  $r$ is an sc--smooth retraction and the set 
$$K=\{(v, f)\vert \, \pi_v (f)=f\}$$
is an  sc--smooth retract.  We call this particular retraction,  due to its partially linear character, a splicing. For   more details on splicings we refer to \cite{HWZ2}.

\begin{lem}
Let  $(O, C, E)$ be  an sc-smooth  local model and assume  that $r:U\to U$ and $r':U'\to U'$ are two sc-smooth retractions defined on relatively open subsets $U$ and $U'$ of $C$  and satisfying $r(U)=r'(U')=O$. Then 
$$Tr (TU)=Tr'(TU').$$
\end{lem}
\begin{proof}
If $y\in U$, then there exists $y'\in U'$ so that $r(y)=r'(y')$.  Consequently, $r'\circ r(y)=r'\circ r'(y')=r'(y')=r(y)$, and hence $r'\circ r=r$.  Similarly, one sees that $r\circ r'=r'$.
If  $(x, h)\in Tr (TU)$,  then $(x, h)=Tr(y, k)$ for a pair  $(y, k)\in TU$.  
Moreover,  $x\in O_1\subset U'_1$ so that $(x, h)\in TU'$.  From $r'\circ r=r$ it follows using the chain rule that 
$$Tr'(x, h)=Tr' \circ Tr(y,k)=T(r'\circ r)(y, k)=Tr(y, k)=(x, h)$$
implying $Tr (TU)\subset Tr'(TU')$.  Similarly one shows that $Tr' (TU')\subset Tr(TU)$ and the proof of the proposition is complete.
\end{proof}

The lemma allows us  to  define the tangent of a local M-polyfold model $(O,C,E)$ as follows.
\begin{defn}
The tangent of a local M-polyfold model $(O,C,E)$,  denoted by $T(O,C,E)$,  is defined as a triple 
$$
T(O,C,E)=(TO,TC,TE),
$$
in which $TC=C^1\oplus E$ is the tangent of the partial quadrant $C$ and  $TO:=Tr(TU)$, where  $r:U\to U$ is any sc-smooth retraction onto $O$.
\end{defn}
As we already pointed out,  the tangent map $Tr:TU\to TU$  of the retraction is an sc-smooth retraction. It is  defined on the relatively open subset $TU$ of $TC$ and  $TO=Tr (TU)$ is an $\ssc^{\infty}$-retract. Thus, the tangent $T(O, C, E)$ of a local  M-polyfold model is also a a local  M-polyfold model.

It is clear what it means that the map $f:(O,C,E)\to  (O',C',E')$ between two local M-polyfold models 
is $\ssc^0$.
In order to define $\ssc^k$--maps  between local models we need the following  lemma.
\begin{lem}\label{easy-lem}
Let   $f:(O,C,E)\to  (O',C',E')$  be a map between two local M-polyfold models 
and let $r:U\to  U$ and $s:V\to  V$ be  $\ssc$-smooth retractions onto $O$. Then the map 
$f\circ r:U\to  E'$ is $\ssc^1$ if and only if  the same holds true for the map $f\circ s:V\to  E'$. Moreover,  
the map $$T(f\circ r)\vert Tr(TU):TO\to TO'$$
does not depend on the choice of an  sc-smooth retraction $r$ as long as $r$ is an sc-smooth retraction onto $O$.
\end{lem}
\begin{proof}
Assume that $f\circ r:U\to  E'$ is $\ssc^1$. Since  $s:V\to  U\cap V$ is $\ssc^\infty$, the chain rule implies that the composition $f\circ r\circ s:V\to  F$ is $\ssc^1$. 
Using  the identity  $f\circ r\circ s=f\circ s$, we conclude that $f\circ s$ is $\ssc^1$. Interchanging the role of $r$ and $s$,  the first part of the lemma is proved.
If $(x, h)\in TO$, then $(x, h)=Ts (x, h)$ and using the identity $f\circ r\circ s=f\circ s$ and the chain rule, we conclude 
\begin{equation*}
T(f\circ r)(x,h)=T(f\circ r)(Ts)(x,h)=T(f\circ r\circ s)(x,h)=T(f\circ s)(x,h)
\end{equation*}
Now take any  sc--smooth retraction $q:W\to  W$ defined on a  relatively open subset $W$ of the the partial quadrant $C'$ in $E'$  satisfying $q(W)=O'$. Then $q\circ f= f$ so that $q\circ f\circ r=f\circ r$. Application of  the chain rule  yields the identity  
$$
T(f\circ r)(x,h)=T(q\circ f\circ r)(x,h)=Tq\circ T(f\circ r)(x,h)
$$
for all $(x, h)\in Tr(TU).$
Consequently, $T(f\circ r)\vert Tr(TU):TO\to TO'$ and this map is  independent of the choice of an sc-smooth retraction onto $O$.
\end{proof}

In view of the lemma, we define the map  $f:(O,C,E)\to (O',C',E')$ between local models to be  of class $\ssc^1$ if  the composition $f\circ r:U\to E'$  is of class $\ssc^1$. 
If this is the case, we define the tangent map $Tf$ as
$$
Tf =T(f\circ r)\vert Tr(TU),
$$
where  $r:U\to U$ is any sc-smooth retraction onto $O$.  Similarly,  $f:(O,C,E)\to (O',C',E')$ is of class $\ssc^k$ provided that  the composition $f\circ r:U\to E'$  is of class $\ssc^k$ where $r:U\to U$ is any sc-smooth retraction defined on relatively open subset $U$ of $C$  satisfying $O=r(U)$.

In the following we simply write  $O$ instead of $(O,C,E)$ for  the local M-polyfold model, however, we always keep in mind that there are  more data in the background.

With the above definition of $\ssc^1$--maps between local M-polyfold models, the next theorem is 
an immediate consequence of the chain rule stated in Theorem \ref{chain-thm1}.
\begin{thm}[General Chain Rule]
Assume that $f:O\to  O'$ and $g:O'\to  O''$ are $\ssc^1$-maps between local M-polyfold models. Then the composition $g\circ f:O\to  O''$ is an $\ssc^1$-map
and 
$$T(g\circ f)=Tg\circ Tf.$$
\end{thm}

The degeneracy index $d_C:C\to \N_0$ introduced in Section \ref{Sc_Banach-ssection} generalizes to local models as follows. 
\begin{defn}
The degeneracy index
$d:O\to  \N_0$  of the local M-polyfold model  $(O,C,E)$ is defined by
$$
d(x):=d_C(x), \quad x\in O.
$$
\end{defn}

The next result shows that sc-diffeomorphisms recognize  the difference between
a straight boundary and a corner.  Of course, this is true also for the usual notion of smoothness  but not for homeomorphisms.

\begin{thm}[Boundary Recognition]\label{breco-thm}
Consider  local M-polyfold models $(O,C,E)$ and $(O',C',E')$ and let $f:O\to  O'$ be an  sc-diffeomorphism. Then
$$
d_C(x)=d_{C'}(f(x))
$$
at each point $x\in O$.
\end{thm}
\begin{proof}
We slightly modify the argument in \cite{HWZ2}. 
First, assuming that the theorem  holds at smooth points $x\in O$, we show that it also holds at points on level $0$. Indeed, take a  point $x\in O$.
By the condition (3) of Definition \ref{m-pl}, we find a sequence of smooth points $(x_k)\subset O$ converging to $x$ in $O$ and satisfying $d_C(x_k)=d_C(x)$. That is, 
$d(x_k)=d(x).$  By assumption, $d(f(x_k))=d(x_k)=d(x)$. Since $f(x_k)\to  f(x)$,  it follows immediately that $d(f(x))\geq d(x)$. The same argument applied to $f^{-1}$  shows that we must have equality
at every point on the  level $0$ in  $O$.

Now we prove the equality for smooth points. Without loss of generality we may assume that $E=\R^n\oplus W$ and $C=[0,\infty)^n\oplus W$ and, similarly,  $E'=\R^{n'}\oplus W'$ and $C'=[0,\infty)^{n'}\oplus W'$. Take a smooth point $x=(a,w)\in  O$ and let $r:U\to U$ be an sc-smooth retract defined on the  relatively open subset $U$ of $C$ satisfying  $O=r(U)$. Abbreviate by  $N$ the kernel of $\id -Dr(x)$ at the point $x$. The kernel $N$ is a closed subspace of $\R^n\oplus E$ possessing, by the condition  (2) 
of Definition \ref{m-pl}, an sc-complement $M$ which is contained in $\{0\}\oplus W$.
Then $N$ is the sc-direct sum of two closed sc-subspaces, namely, 
$$
N=(N\cap (\{0\}\oplus W))\oplus \{(q,p)\in N\ |\ (0,p)\in M\}=:N_1\oplus N_2.
$$
Indeed, if $(p, q)\in N_1\cap N_2$, then  $p=0$ implying that $(0, q)\in N\cap M=\{(0, 0)\}$, i.e., $q=0$.  If $(p, q)\in N$, then there is a unique decomposition 
$(p, 0)=(A, B)+(0, C)=(A, B+C)$ with $(A, B)\in N$ and $(0, C)\in M$  implying that $p=A$ and $B=-C$ and hence $(0, B)=(0, -C)\in M$. So,
$$(p, q)=(0, q-B)+(A, B) $$
with  $(0, q-B)=(A,q)-(A, B)=(p, q)-(A, B)\in N\cap (\{0\}\oplus W)=N_1$ and $(A, B)\in N_2$.

For our smooth point $x=(a,w)$, we denote by  $I$ be the set of indices $1\leq i\leq n$ at which $a_i=0$. In particular, $d(x)=\sharp I$. 
We denote  by $\Sigma$  the subspace of $N$ consisting of vectors $(p,q)\in N$ with $p_i=0$ for $i\in I$. From the decomposition $N=N_1\oplus N_2$,  we see that $\Sigma$ has codimension $\sharp I$ in $N$.
The same arguments  apply at the point $x'=(a',w')=f(x)$. Abbreviate by  $I'$  the set of indices $1\leq j\leq n'$  at which  $a_j'=0$. With the 
 kernel $N'$  of the map $\id -Dr'(x')$ at  $x'$, we let  $\Sigma'$ be  the codimension $\sharp I'$ subspace of $N'$ consisting of all vectors $(p',q')$ satisfying $p_{j}'=0$ for $j\in I'$.  We observe  that the subspaces $N$ and $N'$ are precisely the tangent spaces $T_xO$ and $T_{x'}O'$.

Take any smooth vector $(p,q)\in \Sigma$. Then $x+\tau(p,q)\in U\cap C$ for   $0<\tau <\varepsilon$ and sufficiently small $\varepsilon>0$, so that  $r(x+\tau(p,q))\in O$ for $0<\tau<\varepsilon$. Hence
$$
f\circ r(x+\tau(p,q))\in O'\quad \text{for all $0<\tau<\varepsilon$.}
$$
By assumption,  the map $f\circ r$ is sc-smooth. This implies that the map
$$
[0,\varepsilon)\to  E_m', \quad \tau\mapsto f\circ r(x+\tau(q,p))
$$
is smooth for  every $m\geq 0$.  Now for every $1\leq j\leq n'$ we introduce the  bounded  linear functional   $\lambda_j:\R^{n'}\oplus W'\to \R$ defined by 
$\lambda_j (b, y)=b_j$.  Then we have 
 $$
0\leq \lambda_{j }(f\circ r(x+\tau(q,p))),\quad 0<\tau <\varepsilon.
$$
Fix  $j\in I'$. Then, since $f\circ r(x+\tau(q,p))=f(x)+\tau Df(x)(p,q)+o_m(\tau)$ where $\frac{o_m((\tau)}{\tau}\to 0$ as $\tau\to 0$ and $\lambda_j (f(x))=0$, we conclude  that 
\begin{equation*}
0\leq \frac{1}{\tau} \lambda_{j }(f(x)+\tau Df(x)(p,q)+o_m(\tau))= \lambda_{j }\left( Df(x)(p,q)+\frac{o_m(\tau)}{\tau}\right)
\end{equation*}
which after letting $\tau\to 0$ gives
$$0\leq \lambda_{j }\left( Df(x)(p,q) \right).$$
If $(p, q)\in \Sigma$, then $(-p, -q)\in \Sigma$ and the  above inequality with $(-p,-q)$ replacing $(p, q)$ implies that  $\lambda_{j }\left( Df(x)(p,q)\right)=0$. 
Consequently,
$$\lambda_{j }\left( Df(x)(p,q)\right) =0$$
for all $j\in I'$ and all  vectors $(p, q)\in \Sigma$. At this point we have proved that $Tf(x)\Sigma \subset \Sigma'$.
Recalling that $N$ and $N'$ are the  tangent spaces $T_xO$ and $T_{x'}O'$ and that $Tf:T_xO\to T_{x'}O'$ is an  sc--linear isomorphism, we see that 
the subspace $Tf(x)\Sigma$ has codimension $\sharp I$ in $N'$. It follows that $\sharp I'\leq \sharp I$. Repeating the same argument for the sc-diffeomorphism $f^{-1}:O'\to O$, we obtain 
the opposite inequality. Hence we conclude that $f$ preserves indeed  the degeneracy index for smooth points and the proof of Theorem \ref{breco-thm} is complete.
\end{proof}

At this point it is clear that one can take potentially any recipe from differential geometry and construct new objects. We begin with the recipe for a manifold. Let $X$ be a metrizable space. A chart for $X$ is a triple $(\varphi,U,(O,C,E))$ in which  $(O,C,E)$ is a local  M-polyfold model, $U$ is an open subset of $X$,  and $\varphi:U\to  O$ is  a homeomorphism. Two such charts 
$$
(\varphi,U,(O,C,E))\quad \text{and}\quad (\varphi',U',(O',C',E'))
$$
are called  sc-smoothly compatible if the composition  
$$\varphi'\circ\varphi^{-1}:O\to O'$$ 
is an sc-smooth map between local M-polyfold models. 
\begin{defn}
An $\ssc$-smooth atlas for $X$ consists of a collection of charts $(\varphi,U,(O,C,E))$ such that  the associated open sets  $U$ cover $X$ and the transition maps are $\ssc$-smooth.
Two atlases are equivalent if the union is also an sc--smooth atlas.
The space $X$ equipped with an equivalence class of sc--smooth atlases is called an  M-polyfold.
\end{defn}
Observe that an  M-polyfold $X$ inherits a filtration 
$$X=X_0\supset X_1\supset \ldots \supset X_{\infty}=\bigcap_{m\in \N_0}X_i$$
The same is true for  subsets of $X$. The tangent $TX$ is defined  by the usual recipe used in the infinite-dimensional situation.  Consider tuples
$$
t=(x,\varphi,U,(O,C,E),h)
$$ 
in which $(\varphi,U,(O,C,E))$ is a chart, $x\in U_1$, and  
$(\varphi(x),h)\in TO\subset E^1\oplus E$.
Two such tuples, 
$$
(x,\varphi,U,(O,C,E),h)\quad  \text{and}\quad (x',\varphi',U',(O',C',E'),h'),
$$
are called  equivalent 
if  $x=x'$ and
$$
T(\varphi'\circ\varphi^{-1})(\varphi(x),h)=(\varphi'(x),h').
$$
An equivalence class $[t]$ of a tuple $t$ is called a tangent vector at $x$. The collection of all tangent vectors at $x$ is denoted by $T_xX$ and the tangent $TX$ is defined by
$$
TX=\bigcup_{x\in X_1} \{x\}\times T_xX.
$$
One easily verifies that $TX$ is an  M-polyfold in a natural way with specific charts $T\varphi$ given by
$$
T\varphi:TU\to  TO, \quad T\varphi([x,\varphi,U,(O,C,E),h])=(\varphi(x),h).
$$
Here $TU$ is the union of all $T_xX$ with $x\in U_1$, 
$$
TU=\bigcup_{x\in U_1} \{x\}\times T_xX.
$$

\begin{rem} For SFT we need a more general class of spaces called polyfolds.
They are essentially a Morita equivalence class of ep-groupoids, where the latter is a generalization of the notion of an \'etale proper Lie groupoid to the sc-world. In a nutshell, this is a category in which  the class of objects as well as the collection of morphisms are sets, which in addition carry M-polyfold structures. Further all category operations are sc-smooth and between any two objects there are only finitely many morphisms.  A polyfold is a generalization of the modern notion  of orbifold as presented in \cite{Mj}.  We note that the notion of proper has to be reformulated for  our  generalization as  is explained in \cite{HWZ3.5}.
\end{rem}
Next, we  illustrate the previous  concepts by an example. We shall construct a connected subset
of a  Hilbert space which is an sc--smooth retract and which has one- and two-dimensional parts. (By using the fact that the direct sum of two separable Hilbert spaces is isomorphic to itself one can use the following ideas  to construct connected sc--smooth subsets which have pieces of many different finite dimensions.)

\begin{examp}\label{funny-example}
We take a strictly increasing sequence of weights $(\delta_m) _{m\in \N_0}$ starting at $\delta_0=0$ and equip the Hilbert space  $E=L^2(\R)$ with the sc-structure given by $E_m=H^{m,\delta_m}(\R)$ for all $m\in \N_0$. Next we choose  a smooth compactly supported function $\beta:\R\to  [0,\infty)$ having $L^2$-norm equal to $1$, 
$$
\int_{\R} \beta(s)^2 ds=1.
$$
Then we  define a family of sc-operators $\pi_t:E\to E$ as follows.
For $t\leq 0$,  we put  $\pi_t=0$, and for $t>0$,  we define 
$$
\pi_t(f) =\langle f,\beta(\cdot+e^{\frac{1}{t}})\rangle_{L^2}\cdot \beta(\cdot+e^{\frac{1}{t}}), \quad f\in E.
$$
In  other words,  for $t>0$ we take the $L^2$-orthogonal projection
onto the $1$-dimensional subspace span by the  function $\beta$ with argument shifted by $e^{\frac{1}{t}}$. 
Define
$$
r:\R\oplus E\to  \R\oplus E,\quad r(t,f)=(t,\pi_t(f)).
$$
Clearly, $r\circ r=r$. We shall prove below that the map $r$ is sc-smooth. Consequently, $r$ is an sc-smooth retraction and the image $r(\R\oplus E)$ of $r$, which is the subset 
$$
\{(t,0)\ |\ t\leq 0\}\cup\{(t,s\cdot\beta(\cdot+e^{\frac{1}{t}}))\ |\ t>0,\ \ s\in \R\}
$$
of  $\R\times E$,   is an sc-smooth retract.
Observe that the above retract is connected and consists of $1$- and $2$-dimensional parts. 
\begin{figure}[htbp]
\mbox{}\\[1ex]
\centerline{\relabelbox
\epsfxsize 3.2truein \epsfbox{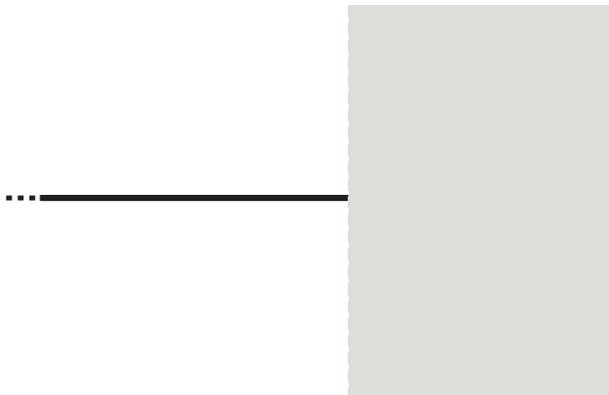}
\endrelabelbox}
\caption{The retract  in the example is homeomorphic
to  the set $(\R^-\times \{0\} )\cup (\R^+\times \R)^0$. }
\end{figure}

Out of this example one can define more retractions which have parts
of any finite dimension.  The above example is enough to show that the subspace $S$
of the plane obtained by taking the open unit disk and attaching a closed interval to
it by mapping the end points to the unit circle in fact admits an sc-smooth atlas, i.e.
is a generalization of a manifold in the sc-smooth world, see Figure 2. 

\begin{figure}[htbp]
\mbox{}\\[2ex]
\centerline{\relabelbox
\epsfxsize 4.5truein \epsfbox{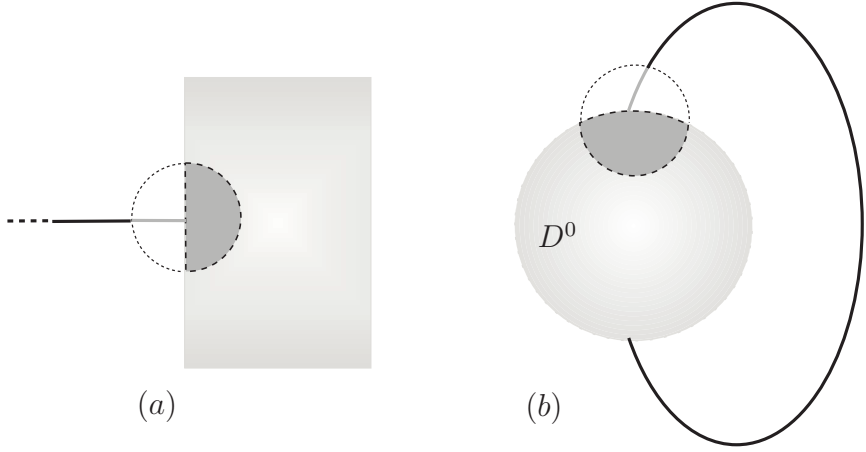}
\relabel {di}{$D^0$}
\relabel {a}{$(a)$}
\relabel {b}{$(b)$}
\endrelabelbox}
\caption{This figure shows a topological space obtained from the open unit
disk by adding a closed arc. This space carries a smooth M-polyfold structure.}\label{Fig2}
\end{figure}

Observe that
$S$ with this M-polyfold structure satisfies $S_m=S$, i.e. the induced filtration is constant. This implies that $S$ has a tangent space at every point. As a consequence of later constructions one could even construct a family of  sc--projections 
$$
\rho_t:E\to  E
$$
where $t\in \R$ and where $E$ is the sc-Banach space from above, such that
 $\rho_t=0$ for $t\leq 0$ and  $\rho_t$  has an infinite-dimensional image for $t>0$,  and such that  the map  $(t,f)\mapsto  (t,\rho_t(f))$ is sc-smooth. Hence the local jumps of dimension can be quite stunning.
Of course, we can even combine the two previous examples in various ways.

We shall prove that the retraction $r:V\oplus E\to V\oplus E$ is sc-smooth.
We recall  that  taking a strictly increasing sequence $(\delta_m)_{m\geq 0}$ of real numbers starting with $\delta_0=0$, the space $E=L^2(\R)$  is equipped with the  sc-structure $(E_m)_{m\in \N_0}$ where $E_m=H^{m,\delta_m}(\R)$.  We choose a smooth function 
$\beta:{\mathbb R}\rightarrow [0,\infty)$ having support contained in  the compact  interval $[-A, A]$  and satisfying $\abs{\beta}_{L^2}=1$.  Define the map $\Phi:{\mathbb R}\oplus E\rightarrow E$  by
$$
\Phi (t, u)=\begin{cases}\langle u,\beta(\cdot+e^\frac{1}{t})\rangle _{L^2}\cdot \beta(\cdot+e^\frac{1}{t})&\quad \text{$t>0$}\\
0&\quad \text{$t\leq 0$}.
\end{cases}
$$
\begin{lem}\label{proof-funy-example}
The map $\Phi$ is of class $\ssc^{\infty}$.
\end{lem}

\begin{proof} It is clear that the restriction $\Phi:(\R\setminus \{0\})\oplus E_m\to E_m$ is smooth.  Abbreviate $F(t, s)=\beta (s+e^{\frac{1}{t}})$ for $s\in \R$ and $t>0$. 
Observe that for the $k$-th derivative of  $F$ with respect to $t$  we have the  estimate,
\begin{equation}\label{sF1}
\abs{F^{(k)}(t, s)} \leq p(t),
\end{equation}
where $p(t)=P\left( e^{\frac{1}{t} }, \frac{1}{t} \right) $ is a polynomial  of degree $k$ in the variable 
$e^{\frac{1}{t}}$ and of degree $2k$ in the variable $\frac{1}{t}.$ It depends on $\beta$ and has nonnegative coefficients. In addition, the function $s\mapsto F^{(k)}(t, s)$  has  its support contained in  the interval $I_t:=[-A-e^{\frac{1}{t}}, A-e^{\frac{1}{t}}]$. 
Note  that $I_t\subset (-\infty ,0)$ if  $t>0$ is  sufficiently small. 

Then if  $u\in E_m$ and if  $t>0$ is  small we can estimate
\begin{equation}\label{sF0}
\begin{split}
\langle u,F( t, \cdot )\rangle_{L^2}^2&=\int_{\R}\abs{u(s)}^2 F(t, s)^2\ ds\\
&=\int_{I_t}\abs{u(s)}^2 F(t, s)^2 e^{2\delta_m s}e^{-2\delta_m s}\ ds\\
&\leq Ce^{-2\delta_m e^{\frac{1}{t}} } \int_{I_t}\abs{u (s)}^2e^{-2\delta_m s}\ ds
\end{split}
\end{equation}
where we have used that 
$$\max_{s\in I_t}\abs{F(s, t)}^2e^{2\delta_ms}\leq Ce^{-\delta_m e^{\frac{1}{t}}}.$$
Similarly one finds
\begin{equation}\label{sF2}
\begin{split}
&\langle u,F^{(j)}( t, \cdot )\rangle_{L^2}^2=\int_{\R}\abs{u}^2 F^{(j)}(t, s)^2\ ds\\
&=\int_{I_t}\abs{u}^2 F^{(j)}(t, s)^2 e^{2\delta_m s}e^{-2\delta_m s}\ ds\\
&\leq  Cp\left(t\right)^2e^{-2\delta_m e^{\frac{1}{t}} } \int_{I_t}\abs{u}^2e^{-2\delta_m s}\ ds.
\end{split}
\end{equation}
Using  the estimate \eqref{sF0} we shall first show that the map $\Phi:\R\oplus E_m\to E_m$ is continuous at a  point $(0, u_0)$.  Take $h\in E_m$ satisfying $\abs{h}_m<\varepsilon$ and set $u:=u_0+h$. By definition,  $\Phi (0, u_0)=0$ and so, 
\begin{equation*}
\begin{split}
&\abs{\Phi (t, u)-\Phi (0, u_0)}^2_m=\langle u, F(t, \cdot )\rangle_{L^2}\sum_{i\leq m}\int_{I_t}\abs{\partial_s^iF(t, s)}^2e^{-2\delta_m s}\ ds\\
&\leq Ce^{-2\delta_{m}e^{\frac{1}{t}}}\left(  \int_{I_t}\abs{u}^2e^{-2\delta_m s}\ ds \right)  \cdot 
\left(  \int_{I_t}e^{-2\delta_m s}\ ds  \right)\\
&\leq  C\int_{I_t}\abs{u}^2e^{-2\delta_m s}\ ds\leq C\left( \int_{I_t}\abs{u_0}^2e^{-2\delta_m s}\ ds+\int_{I_t}\abs{h}^2e^{-2\delta_m s}\ ds\right)\\
&\leq  C \int_{I_t}\abs{u_0}^2e^{-2\delta_m s}\ ds+ C\cdot \varepsilon.
\end{split}
\end{equation*}
Since $\int_{I_t}\abs{u_0(s)}^2e^{-2\delta_m s}\ ds\to 0$ as $t\to 0^+$, one concludes that $\Phi (t, u)\to 0$ as $(t, u)\to (0, u_0)$ in $\R\oplus E_m$. So far we have proved that the map $\Phi:\R\oplus E\to E$ is 
$\ssc^0$.  In order to prove  that the map $\Phi$ is $\ssc^{\infty}$ we proceed by induction. Our inductive  statements are  as follows. 
\begin{induction} The map $\Phi:\R\oplus E\to E$ is $\ssc^k$ and 
$T^k\Phi (t_1, u_1, \ldots ,t_{2^k}, u_{2^k})=0$ if $t_1\leq 0$. 
Moreover, if $\pi:T^kE\to E_m$  is a projection onto the factor $E_m$ of $T^kE$, then the composition $\pi\circ T^kE$ at the point  $(t_1, u_1, \ldots ,t_{2^k}, u_{2^k})$ where  $t_1>0$ is a linear combination
of maps $\Gamma$ of the following type,
\begin{gather*}
\Gamma: \R^{k+1}\oplus E_n\to E_m\\
(t_1, t_2, \ldots ,t_{k+1}, u)\mapsto \langle u, F^{(i)}(t_1, \cdot )\rangle_{L^2}F^{(j)}(t_1, \cdot )\cdot t_2\cdot \ldots  \cdot t_{i+j}
\end{gather*}
where $n\geq m$ and $i+j\leq k$, and this holds for every $m\geq 0$. 
\end{induction}

We start with $k=1$. The candidate for the linearization $D\Phi (t_1, u_1): \R\oplus E_m\to E_m$ at a  point $(t_1, u_1)\in \R\oplus E_{m+1}$ is given by
\begin{equation*}
D\Phi  (t_1 u_1)=0, \quad\text{if $ t_1\leq 0$}
\end{equation*}
and if  $t_1>0$,  then the candidate is equal to 
\begin{equation}\label{sF4}
\begin{split}
&D\Phi (t_1, u_1)(t_2, u_2)=\langle u_2, F(t_1, \cdot )\rangle_{L^2}F(t_1, \cdot )\\
&\phantom{==}+
\langle u_1, F'(t_1, \cdot )\rangle_{L^2}F(t_1, \cdot )t_2+
\langle u_1, F(t_1, \cdot )\rangle_{L^2}F'(t_1, \cdot )t_2.
\end{split}
\end{equation}
Clearly, $D\Phi (t_1, u_1):\R\oplus E_m\to E_m$ is a bounded linear map.  Moreover, for $t_1\neq 0$, the linear map $D\Phi (t_1,  u_1): \R\oplus E_{m+1}\to E_m$ is the  derivative of the map  $\Phi:\R\oplus E_{m+1}\to E_m$ at the point $(t_1, u_1)$. We shall show that  the  derivative of $\Phi:\R\oplus E_{m+1}\to E_m$ at the point $(0, u_1)$ is the zero map.  To do this,  we estimate  $\abs{\Phi (t, u)}_m$ where  $u\in E_n$ and $n>m$ using \eqref{sF0}, 
\begin{equation}\label{sF3}
\begin{split}
\abs{\Phi (t, u)}^2_m&=\langle u, F(t, \cdot )\rangle_{L^2}\sum_{i\leq m}\int_{I_t}\abs{\partial_s^iF(t, s)}^2e^{-2\delta_m s}\ ds\\
&\leq Ce^{-2\delta_{n}e^{\frac{1}{t}}}\left(  \int_{I_t}\abs{u}^2e^{-2\delta_n s}\ ds \right)  \cdot 
\left(  \int_{I_t}e^{-2\delta_m s}\ ds  \right)\\
&\leq  Ce^{-2(\delta_{n}-\delta_m)e^{\frac{1}{t}}} \int_{I_t}\abs{u}^2e^{-2\delta_n s}\ ds\leq 
Ce^{-2(\delta_{n}-\delta_m)e^{\frac{1}{t}}}\abs{u}^2_n.
\end{split}
\end{equation}
For $n=m+1$ we conclude, using  that $\delta_{m+1}>\delta_m$,  
\begin{equation*}
\begin{split}
\frac{1}{\abs{\delta t}+\abs{\delta u_1}_{m+1} } \abs{\Phi (\delta t, u_1+\delta u_1) }_m 
& \leq  C\cdot \frac{ e^{-(\delta_{m+1}-\delta_m )e^{\frac{1}{\delta t} }}}{\abs{\delta t} } \cdot 
\abs{u_1+\delta u_1 }_{m+1}\to 0
\end{split}
\end{equation*}
as $\abs{\delta t}+\abs{\delta u_1}_{m+1}\to 0$ proving  that $D\Phi (0, u)=0$.   Moreover, the estimate  \eqref{sF2} implies as $t_1\to 0$ that 
$$D\Phi (t_1, u_1)(t_2,u_2)\to 0\quad \text{in $E_m$}$$
where  $ (t_1, u_1)\in \R\oplus E_{m+1}$ and  $(t_2,u_2) \in \R\oplus E_m$. Consequently, 
the tangent map 
\begin{gather*}
T\Phi:T(\R\oplus E)\to T E\\
T\Phi(t_1, u_1, t_2, u_2)=(\Phi (t_1, u_1), D\Phi (t_1, u_1)(t_2, u_2))
\end{gather*}
is $\ssc^0$,  and hence the map $\Phi:\R\oplus E\to E$ is of class $\ssc^1$. 
Moreover, in view of \eqref{sF4},  one sees that for $t_1>0$ the composition $\pi\circ T\Phi$ is of the form required  in ${\bf S_1}$. We have proved  that  the statement  ${\bf S_1}$ holds.

Assuming  that we have proved the statement ${\bf S_k}$, we shall show that 
also the statement ${\bf S_{k+1}}$ holds true. To see this it suffices to consider the map 
$\Gamma:\R^{k+1}\oplus E_n\to E_m$ defined by 
$$\Gamma (t_1, \ldots ,t_{i+j}, u)=0$$
if $t_1\leq 0$,  and if $t_1>0$ it is defined as
$$\Gamma (t_1, \ldots ,t_{i+j}, u)=\langle u, F^{(i)}(t_1, \cdot )\rangle_{L^2}F^{(j)}(t_1, \cdot )\cdot t_2\cdot \ldots \cdot t_{i+j}$$
where $n\geq m$ and $i+j\leq k$. We have to show that $\Gamma$ is of class $\ssc^1$. This  map is clearly smooth for $t_1\neq 0$.  Take a  point  $(t_1, t_2, \ldots ,t_{k+1}, u)\in 
\R^{k+1}\oplus E_{n+1}$. If $t_1<0$, then $D\Gamma (t_1, t_2, \ldots ,t_{k+1}, u)=0$ and we claim that this  holds  also at $t_1=0$. 
Indeed,  we consider the quotient 
$$\frac{1}{\abs{\delta t_1} +\ldots +\abs{\delta t_{i+j} }+\abs{u}_{n+1} }\abs{\Gamma (\delta t_1, t_2+\delta t_2, \ldots , t_{i+j}+\delta t_{i+j},  u+\delta u )}_m.
$$
If $\delta t_1\leq 0$, then $\Gamma$ is equal to $0$ and if $\delta t_1>0$,  one obtains for $\Gamma$ an estimate similar to the one in \eqref{sF3}. Hence the quotient is bounded by 
\begin{equation*}
\begin{split}
\frac{ \abs{\Gamma (\delta t_1, t_2+\delta t_2, \ldots , t_{k+1}+\delta t_{k+1},  u+\delta u) }_m}{\abs{\delta t_1} } \leq 
C\frac{ e^{-2( \delta _{n+1}-\delta_m ) e^{ \frac{1}{\delta t_1 } } } }{\abs{\delta t_1} }
\cdot \abs{ u+\delta u }_{n+1}\to 0
\end{split}
\end{equation*}
as $\delta t_1\to 0^+$.  We have proved that  $D\Gamma (0, t_2, \ldots ,t_{i+j}, u)=0$.  If $t_1>0$, then the linearization $D\Gamma (t_1, t_2, \ldots ,t_{i+j}, u):\R^{k+1}\oplus E_{n}\to E_m$ is equal to 
\begin{equation*}
\begin{split}
&D\Gamma(t_1, \ldots, t_{i+j}, u)(\delta t_1, \ldots ,\delta t_{i+j}, \delta u)\\
&\phantom{=}=
\langle \delta u, F^{(i)}(t_1, \cdot )\rangle_{L^2}F^{(j)}(t_1,\cdot )t_2\cdot \ldots \cdot t_{i+j}\\
&\phantom{==}+
\langle  u, F^{(i+1)}(t_1, \cdot )\rangle_{L^2}F^{(j)}(t_1,\cdot )\cdot \delta t_1\cdot t_2\cdot \ldots \cdot t_{i+j}\\
&\phantom{==}+\langle  u, F^{(i)}(t_1, \cdot )\rangle_{L^2}F^{(j+1)}(t_1,\cdot )\cdot \delta t_1\cdot t_2\cdot \ldots \cdot t_{i+j}\\
&\phantom{==}+\sum_{l=2}^{i+j}\langle  u, F^{(j)}(t_1, \cdot )\rangle_{L^2}F^{(j)}(t_1,\cdot )\cdot 
 t_2\cdot \ldots \cdot \delta t_l\cdot \ldots  \cdot t_{i+j}.
 \end{split}
 \end{equation*}
 Now one easily  verifies that the remaining statements of ${\bf S_{k+1}}$ are satisfied. Hence the map  $\Phi$ is $\ssc^k$ for every $k$ and consequently $\ssc^{\infty}$.
 \end{proof}

\end{examp}
\begin{rem}
Using the standard definition of Frech\'et differentiability a subset of a Banach space
is a (classically) smooth retract of an open neighborhood if and only if it is a submanifold.
So relaxing the notion of smoothness to sc-smoothness, which still is fine enough to detect boundaries with corners,  we obtain new smooth spaces as the above example illustrates.
\end{rem}

In order to carry out basic constructions known from differential geometry
one quite often needs partitions of unity.  It is known that Hilbert spaces always admit smooth partitions of unity, whereas there are Banach spaces which do not.
However, as it turns out, there are many sc-Banach spaces for which we have
sc-smooth partitions of unity. This will become clearer after the next section, particularly from Proposition \ref{ABC-x} and Corollary \ref{ABC-y}.

We finish this section with  a question.
\begin{quest} Assume that $X$ is a connected  second countable paracompact M-polyfold without boundary built on sc-smooth retracts in separable Hilbert spaces equipped with sc-structures of Hilbert spaces. Is it true that there exists a
local sc-model $(O,H,H)$ in which  $H$ is a separable Hilbert space equipped with sc-structures of Hilbert spaces, so that $X$ is sc-diffeomorphic to $O$? In other words,  can  $X$ be  covered by only one chart?
\end{quest}
\subsection{Sc-Smoothness Arising in Applications}\label{arisingx}

We describe several constructions which will be used  in the constructions of polyfolds (a generalization of M-polyfolds) in the Gromov-Witten theory and  the Symplectic Field Theory (SFT). It is not our intention to prove the most general results. We just choose  the examples in such a way that they cover the cases needed for the construction in SFT.

We denote by $D$  the closed unit disk in $\C$ and  equip the Hilbert space  $E=H^3(D,\R^N)$ of maps from $D$ into $\R^N$   with the sc-structure defined by  the sequence $E_m=H^{3+m}(D,\R^N)$ for all $m\in \N_0$. 
We choose   a map $u$ belonging to  $E_\infty$,   satisfying $u(0)=0$ and, in addition,  has a derivative  $Du(0):\C\to \R^N$ which is injective.   By $H$ we abbreviate an algebraic complement of the image  of the derivative $Du(0)$.  Thus, $u$ intersects  $H$ transversally at $0$.  By the implicit function theorem,  we find an open neighborhood $U$ of $u$ in $C^1(D,\R^N)$ and $0<\varepsilon<1$ so that for every $v\in U$, there exists a unique point $z_v$ in the $\varepsilon$-disk around $0$ in $\C$  whose image under $v$,  $v(z_v)$,  lies in $H$ and such that  the image $Dv(z_v)$ is transversal to $H$.  Moreover, the map 
$$
U\to  D, \quad v\mapsto   z_v
$$
is $C^1$. If we view this map as defined on  $U\cap C^k$,  then  it is actually of class  $C^k$. \begin{thm}\label{constr-thm1}
If $U\cap E$ is equipped with the induced sc-structure, then the map 
$$
U\cap E\to  \C,\quad v\mapsto z_v
$$
is $\ssc$-smooth.
\end{thm}
\begin{proof}
The map $U\cap E_m\to \C$ defined by  $v\mapsto  z_v$ is of class $C^{m+1}$ in view of  the Sobolev embedding $H^{m+3}(D,\R^N)\to  C^{m+1}(D,\R^N)$. 
By  Proposition \ref{up-prop} below, the map is sc-smooth.
\end{proof}

In order to describe a typical application of Theorem \ref{constr-thm1} we consider  a map $u:S^2\to  M$  of class $H^3$ defined on the Riemann sphere $(S^2, i)$ into  a smooth manifold $M$ and assume  that
$H_1$, $H_2$ and $H_3$ are three submanifolds of codimension $2$  intersecting $u$  transversally at the three different points $z_1$, $z_2$ and $z_3$ in $S^2$. In view of the previous result, under the map  $v$ which  is $H^3$-close to $u$,  the points
$z_i$ move to the points $z_i(v)$ at which the map $v$ intersects submanifolds $H_i$ transversally. There exists a unique M\"obius transformation on $(S^2,i)$ mapping $z_i$ to $z_i(v)$. We denote this 
M\"obius transformation by $\phi_v$ and consider  the composition 
$$
\Phi: E\to E, \quad v\mapsto \Phi(v):=v\circ \phi_v
$$
defined for maps $v$ which are $H^3$-close to $u$. 
Note that $\Phi(v)$ satisfies the transversal constraints at the original points $z_i$. We shall prove that the map $\Phi$ is sc-smooth. This  is a consequence of the  result we  describe next.  

We consider  a family   $v\mapsto \phi_v$ of maps $\phi_v:D\to \C$ parametrized by $v$ belonging to some open neighborhood of the origin  in $\R^n$. Moreover, we assume that the map 
$$
V\times D\to  {\mathbb C}, \quad (v,x)\mapsto  \phi_v(x)
$$
is smooth, satisfies $\phi_0(0)=0$  and 
$$
\phi_v(x)\in D\quad \text{for all $x\in D$ and $v\in V$},$$
so that the map 
$$
\Phi:V\oplus E\to  E, \quad (v,u)\mapsto  u\circ \phi_v
$$
is well-defined. The  map $\Phi$ is not even differentiable in the classical sense, due to the loss of one derivative, but it is smooth in the sense of sc-smoothness, as the next result shows.

\begin{thm}\label{thm-125}
The composition 
$\Phi:V\oplus E\to  E$ defined by $(v,u)\mapsto  u\circ \phi_v$ is an sc-smooth map.
\end{thm}
Theorem \ref{thm-125} will be proved in Section \ref{actionsmooth-section} below. 

The concept of an sc-smooth retract is motivated by our gluing and anti-gluing constructions described next. The retract is a model for a smooth structure on a 2-dimensional family of Sobolev spaces of functions defined on finite cylinders which become longer and longer and eventually break  apart into two half-cylinders. It also incorporates a twisting with respect to the angular variable. Such a  situation arises in the  study of function spaces on Riemann surfaces if  the surfaces develop nodes.

We take a strictly increasing sequence $(\delta_m)_{m\in \N_0}$ starting with $\delta_0>0$ and denote by  $E$ the space  consisting of pairs  $(u^+,u^-)$ of maps 
$$
u^\pm:\R^\pm\times S^1\to  \R^N
$$
having the following properties. There is  a common asymptotic limit $c\in \R^N$  so that the maps 
 $$r^\pm:=u^\pm-c$$
have partial derivatives up to order $3$ which  if  weighted by $e^{\delta_0|s|}$ belong to $L^2(\R^\pm\times S^1,\R^N)$. We equip the Hilbert space $E$ with the sc--structure $(E_m)_{m\in \N_0}$  where $E_m$  consists of those  pairs  $(u^+, u^-)$ in $E$ for which  $(r^+, r^-)$ are of Sobolev  class $(3+m, \delta_m)$.  The $E_m$-norm of the pair $(u^+, u^-)$ is defined as
$$
\abs{(\eta^+,\eta^-)}_{E_m}^2 =\abs{c}^2+\abs{r^+}^2_{H^{3+m, \delta_m}}+\abs{r^-}^2_{H^{3+m, \delta_m}},
$$
where 
\begin{equation*}
\abs{r^\pm}_{3+m, \delta_m}^2:=\sum_{\abs{\alpha}\leq 3+m}\int_{\R^\pm\times S^1}\abs{D^{\alpha}r^\pm(s, t)}^2e^{2\delta_m \abs{s}} \ ds dt.
\end{equation*} 
 Our aim is the construction of an M-polyfold model $(O,{\mathbb C}\oplus E,{\mathbb C}\oplus E)$ where the set $O$ is obtained by an sc-smooth retraction $r:B_{\frac{1}{2}}\oplus E\to  B_{\frac{1}{2}}\oplus E$ of the special  form
$$
r(a,u^+,u^-)=(a,\pi_a(u^+,u^-)),
$$
where $a\mapsto \pi_a$ is a  family of linear and bounded  projections  parametrized by $a\in B_\frac{1}{2}$. The set $B_\frac{1}{2}$ is the open disk of radius $\frac{1}{2}$ in ${\mathbb C}$.

To proceed we need a  gluing profile $\varphi$.  By definition,   this is a diffeomorphism $(0,1]\to  [0,\infty)$ with suitable growth properties. Convenient for our purposes is  the exponential gluing profile defined by 
$$
\varphi(r)=e^{\frac{1}{r}}-e, \qquad r\in (0,1].
$$

With the complex numbers $\abs{a}<\frac{1}{2}$ we associate the abstract cylinders $C_a$ as follows. 
If $a=0$, the cylinder $C_0$ is the disjoint union 
$$C_0=\R^-\times S^1\bigsqcup \R^+\times S^1.$$
We also use $Z_0$ to denote $C_0$. If $a\neq 0$ and $a=\abs{a}e^{-2\pi i \vartheta}$ is its polar form, we set 
$$R=\varphi (\abs{a}).$$
We call $a$ the  gluing parameter and $R$ obtained this  way the  gluing length. To define the cylinder $C_a$, we  take the disjoint union of  the  half-cylinders $\R^+\times S^1$ and $\R^-\times S^1$ and identify  points $(s,t)\in [0,R]\times S^1$ with  points $(s',t')\in [-R,0]\times S^1$  if the following relations hold, 
$$
s=s'+R\quad \text{and}\quad t=t'+\vartheta.
$$
On $C_a$ we have two sets of conformal coordinates, namely, 
$[s,t]\mapsto   (s,t)$ obtained  by extending  the coordinates coming from  $\R^+\times S^1$ and $[s',t']'\mapsto   (s',t')$ obtained  by extending the coordinates coming from $\R^-\times S^1$. The infinite cylinder contains  the finite subcylinder $Z_a$ consisting of  the equivalence classes $[s,t]$ with $s\in [0,R]$.

\begin{figure}[htbp]
\mbox{}\\[2ex]
\centerline{\relabelbox \epsfxsize 4truein \epsfbox{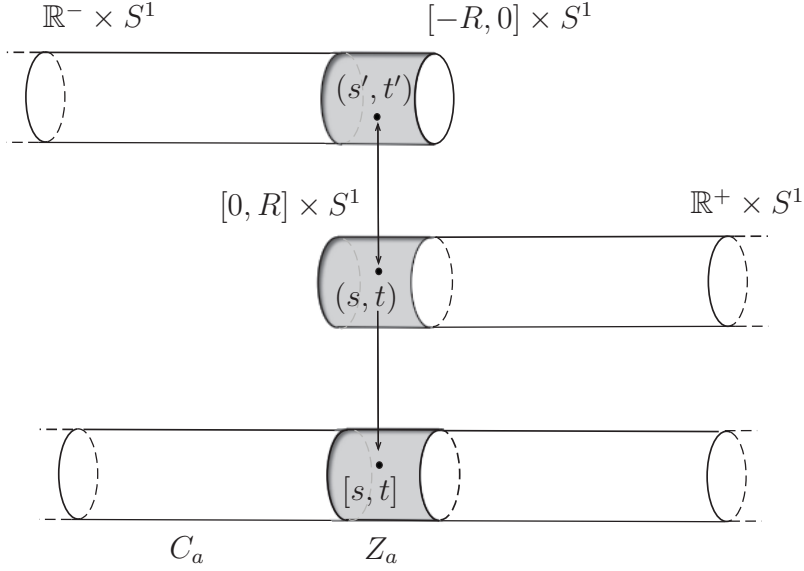}
\relabel {a1}{$(s', t')$}
\relabel {a2}{$(s,t)$}
\relabel {a3}{$[s,t]$}
\relabel {p1}{$\R^-\times S^1$}
\relabel {p2}{$\R^+\times S^1$}
\relabel {r1}{$[-R, 0]\times S^1$}
\relabel {r2}{$[0,R]\times S^1$}
\relabel {p3}{$Z_a$}
\relabel {c}{$C_a$}
\endrelabelbox}
\caption{Glued finite and  infinite cylinders $Z_a$ and $C_a$}\label{Fig3}
\end{figure}

Next, for $a\in B_\frac{1}{2}$,  we define the sc-Hilbert space $G^a$. First,  we introduce the space  $H_c^{3,\delta_0}(C_a)$.  To do this, we choose a smooth function $\zeta:\R\to [0,1]$ satisfying $\zeta(s)=1$ for $s\leq -1$, $\zeta'(s)<0$ for $s\in (-1,1)$,  and $\zeta(s)+\zeta(-s)=0$.  The space $H^{3,\delta_0}_c(C_a)$ consists of those maps $u:C_a\to \R^N$ which belong to $H^3_{\text{loc}}(C_a)$ and there is a constant $c\in \R^N$ such that 
$u+\zeta_a \cdot c\in H^{3,\delta_0}(C_a)$, where 
$$
\zeta_a(s):=\zeta \left(s-\frac{R}{2}\right).
$$
We note that if $u\in H^{3,\delta_0}_c(C_a)$ and $c$ is an asymptotic constant of $u$  at $\infty$, then the asymptotic constant at $-\infty$ is equal to $-c$.  
With the nested sequence $( H^{3+m,\delta_m}_c(C_a))_{m\in \N_0}$ of Hilbert spaces,  the Hilbert space $H_c^{3,\delta_0}(C_a)$ becomes an sc-Hilbert space.  The norms of these Hilbert spaces will be introduced in Section \ref{esttotalgluing} We also equip the Hilbert space $H^3(Z_a)$ with the sc-structure given by $(H^{3+m}(Z_a))_{m\in \N_0}$.  

Now,  if $a=0$, we set 
$$G^0=E\oplus \{0\},$$
and  if  $a\neq 0$, we  define  
$$
G^a = H^3(Z_a)\oplus H^{3,\delta_0}_c(C_a).
$$
The sc-structure of $G^a$ is  given by the sequence $H^{3+m}(Z_a)\oplus H_c^{3+m, \delta_m}(C_a)$ for all $m\in \N_0$.

For every $a\in B_\frac{1}{2}$, we define the so called total gluing map

$$
\boxdot_a=(\oplus_a,\ominus_a):E\to  G^a
$$
as follows.  We fix  a smooth map $\beta:\R\to  [0,1]$ satisfying
\begin{equation}\label{cutoff}
\begin{aligned}
\bullet &\:\:\beta (-s)+\beta (s)=1\quad \text{for all $s\in \R$}\\
\bullet &\:\:\beta (s)=1\quad \text{for all $s\leq -1$}\\
\bullet &\:\:\beta'(s)<0\quad \text{for all $s\in (-1, 1)$}
\end{aligned}
\end{equation}
Moreover, if  $0<\abs{a}<\frac{1}{2}$ we  abbreviate by  $\beta_a$ or $\beta_R$  the translated function 
$$\beta_a(s)=\beta_R (s)=\beta \biggl(s-\frac{R}{2}\biggr), \quad s\in \R, $$
where $R=\varphi (a).$

\begin{figure}[htbp]
\mbox{}\\[2ex]
\centerline{\relabelbox
\epsfxsize 4truein \epsfbox{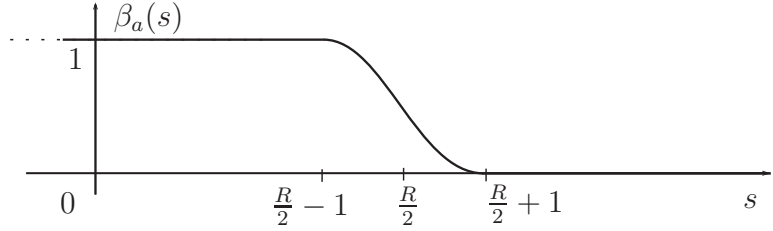}
\relabel {o}{$1$}
\relabel {oo}{$0$}
\relabel {o1}{$\frac{R}{2}+1$}
\relabel {o2}{$\frac{R}{2}-1$}
\relabel {o0}{$\frac{R}{2}$}
\relabel {t}{$\beta_a(s)$}
\relabel {s}{$s$}
\endrelabelbox}
\caption{Graph of the function $\beta_a$ where $R=\varphi (\abs{a})$.}
\end{figure}

Now,  if  $a=0$,  we define 
$$
\oplus_0(u^+,u^-)=(u^+,u^-)\quad \text{and}\quad \ominus_0(u^+,u^-)=0\in\{0\}.
$$
If $a\neq 0$ and  $a=\abs{a}e^{-2\pi i \vartheta}$ is  its polar form,  we set  $R=\varphi(|a|)$ and glue the two maps $u^\pm:\R^\pm \times S^1\to \R^N$ which are defined on positive respectively negative cylinder to a map defined on the glued cylinder $Z_a$ by means of the following convex sum  
$$
\oplus_a(u^+,u^-)([s,t])=\beta_a (s)u^+(s,t)+
(1-\beta_a(s) )u^-(s-R,t-\vartheta).
$$
for $[s,t]\in Z_a$ where $0\leq s\leq R$. Further,  we introduce the map $\ominus_a(u^+, u^-)$ on the glued infinite cylinder $C_a$ by the formula  for $[s, t]\in C_a$,  
\begin{equation*}
\begin{split}
\ominus_a(u^+,u^-)([s,t])&=-(1-\beta_a(s)) \left[ u^+(s,t)-\text{av}_a(u^+,u^-)\right]\\
&\phantom{= \ }+\beta_a(s))\left[ u^-(s-R,t-\vartheta)-\text{av}_a(u^+,u^-)\right]
\end{split}
\end{equation*}
where $s\in \R$ and where the average is defined by 
$$
\av (u^+,u^-):=\frac{1}{2}\int_{S^1}\left[u^+\left(\frac{R}{2},t\right)
+u^-\left(-\frac{R}{2},t\right)\right]dt
$$
and $R=\varphi (\abs{a})$.  We have used that $(1-\beta_a(s))=0$ for $s\leq \frac{R}{2}-1$ and $\beta_a(s)=0$ if $s\geq \frac{R}{2}+1.$

The map $\oplus_a(u^+,u^-)$ is called the glued map for  the gluing parameter $a$ and the map $\ominus_a(u^+,u^-)$ is called the anti-glued map.

\begin{thm}\label{propn-1.27}
For every $a\in B_{\frac{1}{2}}$ the  total gluing map 
$$
\boxdot_a=(\oplus_a, \ominus_a):E\to G^a
$$
is a linear sc-isomorphism. In particular, 
$$E=(\ker \oplus_a)\oplus (\ker \ominus_a)$$
and for $a\neq 0$, the map 
$$
\oplus_a:\ker \ominus_a\to H^3(Z_a)
$$ 
is a linear  sc-isomorphism. If $a=0$,  the map $\oplus_0$ is the identity map.
\end{thm}

We postpone the proof until after the next theorem.

The gluing  and anti-gluing maps determine the linear projection $\pi_a:E\to E$  onto the kernel 
of the anti-glued map $\ominus_a:E\to H_c^{3, \delta_0}(C_a)$ along   the kernel  of the glued map $\oplus_a:E\to H^3(Z_a)$ by means of the following formulae  for $(h^+, h^-)\in E$,
\begin{align*}
\oplus_a\circ (1-\pi_a)(h^+, h^-)&=0\\
\ominus_a\circ  \pi_a (h^+, h^-)&=0.
\end{align*}
Given the pair $(h^+, h^-)\in E$, we set 
$$\pi_a(h^+, h^-)=(\eta^+, \eta^-)$$
so that the pair $(\eta^+, \eta^-)\in E$ is determined by the two equations
\begin{align*}
\oplus_a (\eta^+, \eta^-)&=\oplus_a (h^+, h^-)\\
\ominus_a (\eta^+, \eta^-)&=0.
\end{align*}
Explicitly, abbreviating $\beta_a=\beta_a(s)$, 
\begin{equation*}
\begin{split}
\beta_a\cdot \eta^+(s, t)&+(1-\beta_a )\cdot  \eta^-(s-R, t-\vartheta)\\
&=\beta_a \cdot  h^+(s, t)+(1-\beta_a)\cdot  h^-(s-R, t-\vartheta)\\
\end{split}
\end{equation*}
if $0\leq s\leq R$,  and 
\begin{equation*}
-(1-\beta_a)\cdot \left[ \eta^+(s, t)-\text{av}_R(\eta^+, \eta^-)\right]+\beta_a \cdot \left[ \eta^-(s-R, t-\vartheta )- \text{av}_R(\eta^+, \eta^-)\right]=0
\end{equation*}
for all $s\in \R$. 
Observing that $\beta_a (\frac{R}{2})=\frac{1}{2}$ and integrating the first equation at $s=\frac{R}{2}$ over the circle $S^1$, we obtain for the  averages 
$$\av (\eta^+, \eta^-)=\av (h^+, h^-)$$
so that the equations  for $(\eta^+, \eta^-)$  become
$$
\begin{bmatrix}
\phantom{-}\beta_a&1-\beta_a\\
-(1-\beta_a)&\beta_a
\end{bmatrix}
\begin{bmatrix}
\eta^+\\\eta^-
\end{bmatrix}
=
\begin{bmatrix}
\beta_a h^++(1-\beta_a)h^-\\(2\beta_a-1)\av (h^+, h^-)
\end{bmatrix}
$$
where we have abbreviated $\beta_a=\beta_a(s)$, $h^+=h^+(s, t)$ and $h^-=h^-(s-R, t-\vartheta).$
Introducing the nowhere vanishing function $\gamma$ by
$$
\gamma (s):=\beta (s)^2+(1-\beta (s))^2,
$$
 the determinant of the matrix on the left hand side is equal to $\gamma_a$ which is defined by
 $$
 \gamma_a(s):=\gamma\left(s-\frac{R}{2}\right).
 $$
Hence,  multiplying both sides by the inverse of this matrix, we arrive at the formula
$$
\begin{bmatrix}
\eta^+(s, t)\\ \eta^-(s-R, t-\vartheta)
\end{bmatrix}
=\frac{1}{\gamma_a}
\begin{bmatrix}
\beta_a&-(1-\beta_a)\\
1-\beta_a&\phantom{-}\beta_a
\end{bmatrix}
\begin{bmatrix}
\beta_a h^++(1-\beta_a)h^-\\(2\beta_a-1)\av (h^+, h^-)
\end{bmatrix}
$$
where we  abbreviated $\gamma_a =\gamma_a (s)$. If we denote by $A$ the common asymptotic limit of $(h^+, h^-)$ so that 
$$(h^+, h^-)=(A+r^+, A+r^-)$$
we arrive at the following explicit representation of the map $\pi_a$. 

If $(h^+, h^-)=(A+r^+, A+r^-)\in E$ and $(\eta^+, \eta^-)=\pi_a (h^+, h^-)$, then 
\begin{equation*}
\begin{split}
\eta^+(s, t)&=A+\left(1-\frac{\beta_a (s)}{\gamma_a (s)}\right)\av (r^+, r^-)\\
&\phantom{=}+\frac{\beta_a (s)}{\gamma_a (s)}\left[ \beta_a (s)\cdot r^+(s, t)+(1-\beta_a (s))\cdot r^-(s-R, t-\vartheta)\right]
\end{split}
\end{equation*}
for all $s\geq 0$.  There is a similar formula for $\eta^-(s',t')$.
  
  One reads off that $(\eta^+, \eta^-)\in E$  so that the asymptotic limits of $\eta^+$ and $\eta^-$ coincide. Indeed, the asymptotic limits for $(\eta^+, \eta^-)=\pi_a(h^+, h^-)$ are the following,
\begin{equation*}
\begin{split}
\lim_{s\to \infty}\eta^+(s, t)&=\lim_{s\to -\infty}\eta^-(s, t)\\
&=\av (h^+, h^-)\\
&=A+\av (r^+, r^-),
\end{split}
\end{equation*}
if $h^\pm =A+r^\pm$, with the common limit $A$ of $h^\pm$. In addition, 
$$\av (\eta^+, \eta^-)=\av (h^+, h^-),$$
so that  the projection map $\pi_a$ leaves the averages invariant. A computation shows that 
$$\pi_a\circ \pi_a (h^+, h^-)=\pi_a (h^+, h^-)$$
so that the linear map $\pi_a$ is indeed a projection.

\begin{thm}\label{first-splice}
If  $\pi_a:E\to E$ denotes  the above  linear projection
of $E$ onto $\ker \ominus_a $ along $\ker  \oplus_a$, then  the  map
$$
r:B_\frac{1}{2}\oplus E\to  B_\frac{1}{2}\oplus E, \quad (a,u^+,u^-)\mapsto (a,\pi_a(u^+,u^-))
$$
is an sc-smooth retraction.
\end{thm}
Theorem \ref{first-splice} will be proved in Section \ref{helmut}. At this point we shall prove Theorem  \ref{propn-1.27}. 
\begin{proof}[Proof of Theorem \ref{propn-1.27}]
If $a=0$, there is nothing to prove. We assume that $a\neq 0\in B_{\frac{1}{2}}$, take $(\eta^+, \eta^-)\in E$, and set $\oplus_a(\eta^+, \eta^-)=u$ and $\ominus_a(\eta^+, \eta^-)=v$. Explicitly, 
\begin{equation*}
\begin{split}
\beta_a (s)\cdot \eta^+(s, t)+(1-\beta_a (s))\cdot \eta^-(s-R, t-\vartheta)=u(s, t)
\end{split}
\end{equation*}
for $0\leq s\leq R$, and 
\begin{equation*}
\begin{split}
-(1-\beta_a (s))&\left[ \eta^+(s, t)-\av (\eta^+, \eta^-)\right]\\
& +\beta_a (s)\left[ \eta^-(s-R, t-\vartheta)-\av (\eta^+, \eta^-)\right]=v(s, t)
\end{split}
\end{equation*}
for all $s\in \R$. Since $\beta_a (s)=0$ if $s\geq \frac{R}{2}+1$ and $\beta_a(s)=1$ if $s\leq \frac{R}{2}-1$ , we conclude $\lim_{s\to \infty}v(s, t)=-\lim_{s\to \infty}\eta^+(s, t)+\av (\eta^+, \eta^-)$ and $\lim_{s\to -\infty}v(s, t)=\lim_{s\to -\infty}\eta^-(s, t)-\av (\eta^+, \eta^-)$. Consequently, $(u, v)\in G^a$.

Conversely, 
given $(u, v)\in G^a$ we look for a solution $(\eta^+, \eta^-)$ of the two equations $\oplus_a(\eta^+, \eta^-)=u$ and $\ominus_a(\eta^+, \eta^-)=v$. Integrating the first equations at $s=\frac{R}{2}$ over $S^1$ and observing that $\beta_a (\frac{R}{2})=\beta(0)=\frac{1}{2}$, we obtain 
$$\av (\eta^+, \eta^-)=\int_{S^1}u\left(\frac{R}{2}, t\right)\ dt=:[u]$$
and proceeding as above we arrive at the presentation
\begin{equation*}
\begin{split}
\begin{bmatrix}
\eta^+(s, t)\\
\eta^-(s-R, t-\vartheta)
\end{bmatrix}&=\frac{1}{\gamma_a}
\begin{bmatrix}
\beta_a&-(1-\beta_a)\\
1-\beta_a&\phantom{-}\beta_a
\end{bmatrix}
\begin{bmatrix}
u\\
(2\beta_a-1)[u]+v
\end{bmatrix}\\
&=
\begin{bmatrix}
\frac{1}{\gamma_a}\left( \beta_a u-(1-\beta_a)v\right)-\frac{1}{\gamma_a}(1- \beta_a)(2\beta_a -1)[u] \\
\frac{1}{\gamma_a}\left( (1-\beta_a ) u+\beta_a v\right)+\frac{1}{\gamma_a}\beta_a(2\beta_a -1)[u] 
\end{bmatrix}
\end{split}
\end{equation*}
where we have abbreviated $\beta_a=\beta_a(s)$, $\gamma_a=\gamma_a (s)$, $u=u(s, t)$ and 
$v=v(s,t)$. For the asymptotic limits  we read off that 
 $\lim_{s\to \infty}\eta^+(s, t)=[u]-\lim_{s\to \infty}v(s, t)$ and   $\lim_{s\to -\infty}\eta^-(s, t)=[u]+\lim_{s\to -\infty}v(s, t)$.  Since the asymptotic limits of $v$ have opposite signs, we conclude that 
$\lim_{s\to \infty}\eta^+(s, t)= \lim_{s\to -\infty}\eta^-(s, t)$ and hence 
$(\eta^+, \eta^-)\in E$ as desired.

\end{proof}

If $\co =r(B_\frac{1}{2}\oplus E)$ is the retract of the sc-smooth retraction guaranteed  by Theorem \ref{first-splice},  the  map
$$
\co \to  (\{0\}\times E)\bigcup \biggl(\bigcup_{a\in B_\frac{1}{2}\setminus\{a\}} \{0\}\times H^3(Z_a)\biggr)
$$
defined by
$$
(a,u^+,u^-)\mapsto  (a,\oplus_a(u^+,u^-))
$$
is, by construction, a bijection between $\text{ker}\ \ominus_a$ and $E$ if $a=0$,  respectively  $H^3(Z_a)$ if $a\neq 0$. We equip the target  space with the topology making this  map a homeomorphism. Then the inverse of this map is a chart
on  the local model $\co$.  This way we obtain  a construction describing a smooth structure for a suitable 2-dimensional family of function spaces on cylinders which have increasing modulus. We will study 
this smooth structure in detail in chapter \ref{longcylinders}.

There is another version of gluing which will be used in the proofs in Section \ref{longcylinders}. Taking the strictly increasing sequence  $(\delta_m)_{m\in \N_0}$,  we denote by $F$ the  sc--Hilbert space  consisting  of pairs $(u^+,u^-)$
of maps 
$$u^\pm:\R^\pm\times S^1\to  \R^N$$
whose partial derivatives up to order $2$  weighted by $e^{\delta_0|s|}$ 
belong to $L^2(\R^\pm\times S^1)$. We equip  $F$  with the sc-structure  defined by the sequence $F_m=H^{2+m,\delta_m}(\R^+\times S^1)\oplus H^{2+m,\delta_m}(\R^-\times S^1)$. For $(u^+, u^-)\in F$,  the glued map   $\wh{\oplus}_a(u^+, u^-)$ is defined  by the same formula as $\oplus_a(u^+, u^-)$. However, the anti-glued map $\wh{\ominus}_a$   takes the  simpler form
$$
\wh{\ominus}_a(u^+,u^-)([s,t])=-(1-\beta_a (s))\cdot  u^+(s,t)+
\beta_a(s)\cdot u^-(s-R,t-\vartheta)
$$
for  $[s,t]\in C_a$. The image of  the anti-glued map $\wh{\ominus}_a$ is   $H^{2,\delta_0}(C_a)$. Proceeding as before we can define a projection $\wh{\pi}_a:F\to F$ onto the kernel $\ker \wh{\ominus}_a$ of  the anti-gluing map $\wh{\ominus}_a$ along the kernel $\ker  \wh{\oplus}_a$ of the gluing map $\wh{\oplus}_a$.

We abbreviate by $\wh{G}^a$ the following spaces. If $a=0$, we set 
$$\wh{G}^0=F\oplus \{0\},$$
and  if  $a\neq 0$, we  define  
$$
\wh{G}^a = H^2(Z_a)\oplus H^{2,\delta_0}(C_a).
$$
The sc-structure of $G^a$ is  given by the sequence $H^{2+m}(Z_a)\oplus H^{2+m, \delta_m}(C_a)$ for all $m\in \N_0$. 
The following theorem holds true.

\begin{thm}\label{second-splicing}
\mbox{}
\begin{itemize}
\item[(1)]  If we equip the Hilbert space $ H^2(Z_a)\oplus H^{2, \delta_0}(C_a)$ with the  sc--structure $ H^{2+m}(Z_a)\oplus H^{2+m, \delta_m}(C_a)$, then 
the total hat gluing map 
$$\wh{\boxdot}_a=(\wh{\oplus}_a, \wh{\ominus}_a): F\to \wh{G}^a
$$
is an sc--smooth isomorphism for every $a\in B_{\frac{1}{2}}$.In particular, $E=(\ker \wh{\oplus}_a)\oplus  (\ker \wh{\ominus}_a).$
\item[(2)] 
The map
$$
\wh{r}:B_\frac{1}{2}\oplus F\to  B_\frac{1}{2}\oplus F, \quad (a, u^+,u^- )\mapsto (a,\wh{\pi}_a(u^+,u^-))
$$
is an sc-smooth retraction.
\end{itemize}
\end{thm}
The proof is the same as that of Theorem \ref{propn-1.27} and Theorem \ref{first-splice}.
\begin{rem}
The same result is true if we consider the projection
on the space $H^{3,\delta_0}(\R^+\times S^1)\oplus H^{3,\delta_0}(\R^-\times S^1)$ rather than on $F$.
We choose  $F$ here since this particular retraction $\wh{r}$ occurs in the
construction of bundles in SFT accompanying the constructions of base spaces which involve the previous retraction $r$.
\end{rem}
Our constructions in Gromov-Witten theory and SFT  makes use of the next result described in Theorem \ref{tran-pol}  below which is closely related to the previous constructions.
Using Theorem \ref{first-splice}, we have equipped the set
$$
{\mathcal O}:=\left(\{0\}\times E\right)\bigcup \left(\bigcup_{a\in B_\frac{1}{2}\setminus \{0\}} \left(\{a\}\times H^3(Z_a,\R^N)\right)\right)
$$
with the structure of an  M-polyfold. In fact,  it is covered by a single  chart which also defines the topology. We are  interested in sc-smooth maps ${\mathcal O}\to  {\mathcal O}$ which will arise in the construction of the polyfolds of SFT.

In order to describe these maps we  start with the half-cylinders $\R^\pm\times S^1$ and assume
that we are given two smooth families $v\mapsto  j^\pm(v)$ of complex structures parameterized by $v$ which belongs to some open neighborhood $V$ of $0$ in some finite-dimensional vector space.
We  require that  the complex structures $j^\pm(v)$ away from the boundaries agree with  the standard conformal structure. Suppose  there exist  two smooth families
$v\mapsto  p^\pm(v)$  of marked points on  the boundaries $\partial (\R^\pm\times S^1)$. Given a sufficiently small gluing parameter $a$, 
we can construct  the glued cylinder $Z_a$ equipped with the complex structure $j(a,v)$ induced  from $j^\pm(v)$ and smooth families of induced  marked points $p^\pm(a,v)$. We have defined the family
$$
(a,v)\mapsto  (Z_a,j(a,v),p^+(a,v),p^-(a,v))
$$
 of complex cylinders with marked points. We introduce a second family  of finite cylinders as follows. We fix the special marked points $(0,0)$ on the boundaries of  the two standard cylinders. Then,  given a gluing parameter $b$,  we obtain  the finite cylinder cylinder  $Z_b$  equipped with the standard complex structure $i$ and the two marked points $p_b^\pm$. This way we obtain a  second  family 
$$
b\mapsto (Z_b,i,p^+_b,p^-_b)
$$
of complex finite cylinders.
It is a standard fact  from the uniformization theorem that the cylinder 
$(Z_a,j(a,v))$ is biholomorphic to the cylinder $([0,R]\times S^1, i)$  for a uniquely  determined  $R$. This biholomorphic map is unique up to rotation in the image (and reflection). If we require that the marked point $p^+(a,v)$ is mapped onto the point  $p^+_{b}=(0,0)$, we find a uniquely determined complex number $b=b(a,v)$ such that $R=\varphi(|b|)$ and $p^-(a,v)$ is mapped to the marked point $p^-_b$. 

Thus given the pair $(a,v)$,  there is precisely one gluing parameter $b:=b(a,v)$ for which  there exists a biholomorphic map
$$
\Phi_{(a,v)}:(Z_a,j(a,v),p^+(a,v),p^-(a,v))\to  (Z_b,i,p^+_b,p^-_b).
$$
If we use instead of the exponential gluing profile the logarithmic gluing profile $-\frac{1}{2\pi}\ln(r)$
it is well known  that the map $(a,v)\mapsto b(a,v)$ is  smooth. However, the same is true for the exponential gluing profile. 
This can be deduced from the result
about the logarithmic gluing profile by means of a calculus exercise involving results from Section \ref{calc-lemma}. We leave this approach to the reader.
Here we  shall derive  this fact as a corollary from the following more general result.  We consider the map which associates with the element 
$(v, a, w)\in V\oplus {\mathcal O}$ for $(v, a)$ small the element $(b(a,v),w')\in {\mathcal O}$ in which $w'$ is defined by
$$
w'= w\circ \Phi_{(a, v)}^{-1}.
$$
\begin{thm}\label{tran-pol}
Let $(\delta_m)_{m\geq 0}$ be a strictly increasing sequence satisfying 
$$0<\delta _m<2\pi\quad \text{ for all $m\geq 0$.}$$
Then the map 
 $$ 
 V\oplus {\mathcal O}\to  {\mathcal O}, \quad  (v,a,w)\mapsto  (b(a,v),w\circ \Phi_{(a, v)}^{-1})
 $$ 
 defined for $(a,v)$ small is an sc-smooth map for the M-polyfold structure on ${\mathcal O}$.
 \end{thm}
The restriction that the sequence $(\delta_m)$ lies in the open interval $(0,2\pi)$
 is related to the behavior of  the map $\Phi_{(a,v)}$ as $a\to  0$. In order to prove Theorem \ref{tran-pol} it is necessary  to understand the smoothness properties of the map
$$
(a,v)\mapsto  \Phi_{(a,v)}.
$$
This will be part of the illustration in the Section \ref{illustration}.

\begin{rem}\label{rem1-nnew}
There is no loss of generality  in assuming that 
$p^\pm(v)=(0,0)$ for all $v\in V$. This can be achieved  by taking different complex structures $j^\pm(v)$.
Indeed,  we can choose a smooth family of diffeomorphisms which are supported near the boundary and map the points $p^\pm(v)$ to the point $(0,0)$.
Then we conjugate the original $j^\pm(v)$ with this family to obtain the new one. By the previous discussion,  this family of diffeomorphisms  acts sc-smoothly. So,  Theorem \ref{tran-pol} will follow once it is proved for the special case  just described. We shall assume in the following that this reduction has been carried out.
\end{rem}

\subsection{Fredholm Theory}\label{Fredholm-section}
In this section we  outline the Fredholm theory in M-polyfolds. We shall only describe  the case of M-polyfolds without boundaries and  refer the reader to \cite{HWZ3,HWZ3.5} for the  general case.

We have already introduced the notion of an  M-polyfold and now introduce the notion of a strong bundle in the case that the underlying base space does not have a boundary. Let  $E$ and $F$ be sc-Banach spaces and let $U$ be  an open subset of $E$. We define the nonsymmetric product $U\triangleleft F$ as follows. As a set the product $U\triangleleft F$ is  equal to $U\times F$. However,  it is equipped with the double filtration 
$$
(U\triangleleft F)_{m,k} = U_m\oplus F_{k}
$$
 defined for  all $m\in \N_0$ and all integers $0\leq k\leq m+1$.  Given $U\triangleleft F$ we have, forgetting part of the structure,
the  underlying direct sum $U\oplus F$. We view $U\triangleleft F\to U$ as a bundle with base space $U$  and fiber $F$, where the double filtration has the interpretation 
that above a point $x\in U$ of  regularity $m$ it makes sense to talk about the fiber regularity  of a point $(x, h)$ up to order $k$ provided $k\leq m+1$. 
The tangent space $T_\triangleleft(U\triangleleft F)$  is defined by
$$
T_\triangleleft(U\triangleleft F)=(TU)\triangleleft (TF).
$$
Observe that there is a difference in the order of the factors
between  $T_\triangleleft(U\triangleleft F)$ and  $T(U\oplus F)$. Indeed, 
$$
T_\triangleleft(U\triangleleft F)=(U_1\oplus E)\triangleleft ( F_1\oplus F)
$$
while (for the underlying $U\oplus F$)
$$T(U\oplus F)=U_1\oplus F_1\oplus  E\oplus F.$$
A map $f:U\triangleleft F\rightarrow V\triangleleft G$ between nonsymmetric products is  an  $\ssc_\triangleleft^0$-map if for all $m\in \N_0$ and all $0\leq k\leq m+1$, 
$$
f(U_m\oplus F_k)\subset V_m\oplus G_{k}
$$
and if  the induced maps 
$$
f:U_m\oplus F_k\to V_m\oplus G_{k}
$$
are  continuous. In addition, we require  that the map $f$ is of  the form
$$
f(u,h)=(f_0(u),\phi(u,h))
$$
where $\phi(u,h)$ is  linear in $h$.

We observe that the map
$f$ induces  $\ssc^0$-maps
\begin{equation}\label{bot}
f:U\oplus F^i\rightarrow V\oplus G^i
\end{equation}
for $i=0,1$.  The map $f$ is $\ssc^1_\triangleleft$ if 
the maps in \eqref{bot}  for $i=0,1$ are both of class $\ssc^1$.  In this case the tangent maps  $Tf:T(U\oplus F^i)\rightarrow T(V\oplus G^i)$ are  defined as usual by
$$
Tf(x,h,y,k)=(f_0(x),\phi(x,h),Df_0(x)y,D\phi(x,h)(y,k)).
$$
After reordering of factors  of the domain and the  target spaces, we obtain the  $\ssc^0$- map
$$
T_\triangleleft f: TU\triangleleft TF\rightarrow TV\triangleleft TG
$$
defined by
$$
(x,h,y,k)\mapsto  (f_0(x),Df_0(x)y,\phi(x,h),D\phi(x,h)(y,k)).
$$
This  reordering is consistent with the chain rule and one verifies that if $f$ and $g$ are $\ssc^1_\triangleleft$-maps which can be composed,  then also the composition $g\circ f$ is of class $\ssc^1_\triangleleft$ and satisfies the chain rule, 
$$
T_\triangleleft(g\circ f)= (T_\triangleleft g)\circ (T_\triangleleft f).
$$
Iteratively one defines  the maps of class $\ssc^k_\triangleleft$ for $k=1, 2, \ldots $ and  $\ssc_\triangleleft$-smooth maps.

\begin{defn}
An  $\ssc_\triangleleft$-smooth retraction is an  $\ssc_\triangleleft$-smooth map 
$$R:U\triangleleft F\rightarrow U\triangleleft F$$ 
satisfying $R\circ R=R$. The image  $R(U\triangleleft F)$ is called an sc-smooth strong bundle retract.
\end{defn}
It is implicitly required that the above retraction $R:U\triangleleft F\rightarrow U\triangleleft F$ is of the from $R(x, h)=(r(x), \rho (x, h))$ where $r:U\to U$ is an sc-smooth retraction and $\rho (x, h)$ is linear in $h$. We denote by $K=R(U\triangleleft F)$ the  sc-smooth strong bundle retract and by $O=r(U)$ the associated sc-smooth retract. Then there is a canonical projection map 
$$
p:K\rightarrow O.
$$
The  set $K$ inherits the  double filtration and $p$ maps points of regularity $(m, k)$ to points in $U$ of regularity $m$. The canonical projection $p:K\to O$  is our local model for spaces which we shall call strong M-polyfold bundles. Given $K$,  we define the sets $K(i)$ for $i=0,1$ as follows, 
$$
K(i)=\{(u,h)\in U\oplus F^i\ \vert \  R(u,h)=(u,h)\}.
$$
Clearly,  the set $K(i)$ is an  sc-smooth retract  and the projection $p:K(i)\rightarrow O$ is sc-smooth for $i=0,1$. 

Now we are in a position to define the notion of a strong bundle. We consider a surjective continuous map $p:W\to X$ between two metrizable spaces, so that for every $x\in X$ the space $p^{-1}(x)=:W_x$ comes with the structure of a Banach space. A strong bundle chart is the tuple  $(\Phi, p^{-1}(U), K, U\triangleleft F)$ where $U\subset E$ is an open  subset of an sc-Banach space, $K= R(U\triangleleft F)$  an sc-smooth strong bundle retract covering the smooth retraction $r:U\to O$. Moreover,  $\Phi:p^{-1}(U)\to K$ is a homeomorphism covering a homeomorphism $\varphi:U\to  O$,   which between  between every fiber is a bounded linear operator of Banach spaces.
Two such charts are  $\ssc_\triangleleft$-smoothly equivalent  if the associated transition maps are  $\ssc_\triangleleft$-smooth diffeomorphisms. We can introduce  the notion of a strong bundle atlas and the notion of an equivalence between two such atlases. The continuous surjection  $p:W\to X$, if equipped with an equivalence class of strong bundle atlases, is called a strong bundle.

Given two such local bundles $K\rightarrow O$ and $K'\rightarrow O'$ we can define the notion of a strong bundle map between them. Then,  following the scheme how we defined M-polyfolds, we can define strong bundle $W\rightarrow X$ over an M-polyfold $X$. 

Next assume that $p:W\rightarrow X$ is  a strong M-polyfold bundle over the M-polyfold $X$. We can distinguish two types of sections of $p$. An sc-smooth section of $p$
is a map $f:X\rightarrow W$ with $p\circ f=\id$ satisfying $f(x)\in W_{m,m}=:W(0)_m$  for $x\in X_m$ so that $f:X\rightarrow W(0)$ is sc-smooth. An $\ssc^+$-section of $p$  is a section which satisfies $f(x)\in W_{m,m+1}=:W(1)_m$ for $x\in X_m$  and  the induced map
$f:X\rightarrow W(1)$ is sc-smooth. We shall denote these two classes of sections by $\Gamma(p)$ and $\Gamma^+(p)$, respectively.

Of  special  interest are the so-called (polyfold-) Fredholm sections. Their definition is  much more general than that of classical Fredholm sections. The first property which we require is that such sections should have the  regularizing property. This property models the outcome of elliptic regularity theory.
\begin{defn}
Let $p:W\rightarrow X$ be  a strong M-polyfold bundle over the M-polyfold $X$.  A section $f\in\Gamma(p)$ is said to be regularizing provided that the following holds. If $x\in X_m$ and $f(x)\in W_{m,m+1}$, then $x\in X_{m+1}$.
\end{defn}
We observe that if $f\in \Gamma(p)$ is regularizing and $s\in\Gamma^+(p)$,  then $f+s$ is  also regularizing.

Consider a strong local bundle $K\rightarrow O$. 
Here $K=R(U\triangleleft F)$ is  the sc-smooth strong bundle retract  associated to the $\ssc_\triangleleft$-smooth retraction $R:U\triangleleft F\rightarrow U\triangleleft F$  which covers the sc-smooth retraction $r:U\to U$ defined on an open neighborhood $U$ of $0$ in the  sc-Banach space $E$. Moreover, $O=r(U)$.  In addition, we assume that  $0\in O= r(U)$. Then $R(u,h)=(r(u),\phi(u)h)$ where  $\phi (u):F\to F$ is  linear. 

We are interested in germs of sections $(f,0)$ of the strong local bundle $K\rightarrow O$ which are defined near $0$.  
We  view $f$ as a germ ${\mathcal O}(O,0)\rightarrow F$ 
identifying  the local section with its principal part.
\begin{defn}\label{filled-def}
We say that the section germ $(f,0)$ has a filled version if there exists an sc-smooth  section germ
$(g,0)$ of the bundle $U\triangleleft F\rightarrow U$, again viewed
 as a germ
 $$
 {\mathcal O}(U,0)\rightarrow F, 
 $$
which extends $f$ and has the following properties:
 \begin{itemize}
\item[(1)] $g(x)=f(x)$ for $x\in O$ close to $0$.
\item[(2)] If $g(y)=\phi(r(y))g(y)$ for  a point $y$ in $U$  near $0$, then $y\in O$.
\item[(3)] The linearisation of the map
$$
y\mapsto [\id-\phi(r(y))]g(y)
$$
at  the point $0$,  restricted to $\ker{Dr(0)}$, defines a topological isomorphism
$$
\ker(Dr(0))\rightarrow \ker(\phi(0)).
$$
\end{itemize}
\end{defn}

In order to describe the significance of the three conditions in the above definition we assume that $y\in U$ is a solution of the filled section $g$ so that $g(y)=0$. Then it follows from (2) that $y\in O$ and from (1) that $f(y)=0$. 

We see that the original section $f$ and the filled section $g$ have the same solution set.

The requirement (3) plays a role if we compare the linearization $Df(0)$  with  the linearization $Dg(0)$. It  follows from the definition of a retract that 
$\phi (r(y))\circ \phi (r(y))=\phi (r(y)).$ Hence, since $y=0\in O$ we have $r(0)=0$ and $\phi(0)\circ \phi(0)=\phi (0)$ so that $\phi (0)$ is a linear sc-projection in $F$ and we obtain the sc-splitting 
$$
F=\phi (0)F\oplus (\id -\phi (0))F.
$$
Similarly, it follows from $r(r(y))=r(y)$ for $y\in U$ that $Dr (0)\circ Dr (0)=Dr(0)$ so that $Dr (0)$ is a linear sc-projection in $E$ which gives rise to the sc-splitting 
$$\alpha \oplus \beta \in  E=Dr (0)E\oplus (\id -Dr (0))E.$$
We keep in mind that $Dr (0)\alpha =\alpha$ and $Dr (0)\beta =0$. The tangent space $T_0O$ is equal to $Dr (0)E$ and 
$$\phi (0)\circ Dg (0)\vert Dr (0)E=Df (0):T_0O\to \phi (0)F.$$
 From the identity $\phi (r(y))g(r(y))=\phi (y)g(y)$ for all $y\in O$ and   the identity  $(\id -\phi (r(y))g(r(y))=0$ for all $y\in E$ we obtain by linearization at $y=0$ that $\phi (0)Dg(0)\beta =0$ and $(\id -\phi (0))Dg (0)\alpha =0$.  Hence the matrix representation of $D g (0):E\to F$ with respect to the splittings looks as follows,
$$
Dg(0)\begin{bmatrix}\alpha \\ \beta
\end{bmatrix}=
\begin{bmatrix}Df (0)&0\\0&(\id -\phi (0))\circ Dg (0)\end{bmatrix}\cdot 
\begin{bmatrix}\alpha\\ \beta
\end{bmatrix}.
$$
In view of property (3), the linear map $\beta \mapsto (\id -\phi (0))Dg (0)$ from the kernel of $Dr (0)$ onto the kernel of $\phi (0)$ is an isomorphism. Therefore, we conclude that 
$$\text{kernel}\ Dg (0)=(\text{kernel}\ Df (0))\oplus \{0\}.$$
Moreover, $Df (0):T_0O\to \phi (0)F$ is a Fredholm operator if and only if $Dg (0):E\to E$  is a Fredholm operator and in this case they have the same indices. Clearly, the linearization $Df(0)$ is surjective  if and only if $Dg(0)$ is surjective. 

\begin{rem}
We see that, given a solution $x$ of $f(x)=0$, the local study of the solution set $f(y)=0$ for $y\in O$ near $x$ of the section $f$, is equivalent to the local study of the solution set $g(y)=0$ for $y\in U$ near $x$ of the filled section $g$.
\end{rem}

\begin{rem} We assume that the section  $(f,0)$ has a filled version $(g,0)$ and that $s$ is an $\ssc^+$-section of $K\rightarrow O$. 
If $t$ is the  $\ssc^+$-section of $U\triangleleft F\rightarrow F$  defined by $t(y):=s(r(y))$, then $(g+t,0)$ is a filled version of $(f+s,0)$. 
Indeed, for $x\in O$ we compute $(g+t)(x)=g(x)+s(r(x))=f(x)+s(x)$, which is property (1). From  $(g+t)(y)=\phi(r(y))(g+t(y))$ we deduce that
$$
g(y)=\phi(r(y))g(y) + \phi(r(y))s(r(y))-s(r(y))=\phi(r(y))g(y),
$$
implying that $y\in O$ so that property (2)  holds. Finally, 
$$
[\id-\phi(r(y))](g(y)+t(y))=[\id-\phi(r(y))](g(y)+s(r(y)))=[\id-\phi(r(y))]g(y),
$$
so that the linearisation of the left-hand side at $0$ restricted to $Dr(0)$ satisfies the property (3) in view of  the assumptions on  $g$.
Hence, if $(f,0)$ has a filled version, so does  the section $(f+s,0)$ for every $\ssc^+$-section $s$ of the strong local bundle $K\rightarrow O$.
\end{rem}

Next,  we introduce a class of so-called basic germs denoted by
$\mathfrak{C}_{basic}$.
\begin{defn}
An element in $\mathfrak{C}_{basic}$ is an sc-smooth germ
$$
f:{\mathcal O}({\mathbb R}^n\oplus W,0)\rightarrow ({\mathbb R}^N\oplus W,0), 
$$
where $W$ is an sc-Banach space, so that  if $P:{\mathbb R}^N\oplus W\rightarrow W$  denotes  the projection, then the  germ $P\circ f:{\mathcal O}({\mathbb R}^n\oplus W,0)\rightarrow (W, 0)$ has the form
$$
P\circ f(a,w)=w-B(a,w)
$$
for $(a,w)$ close to $(0,0)\in {\mathbb R}^n\oplus W_0$. Moreover, for every $\varepsilon>0$ and $m\geq 0$,  we have the estimate
$$
\abs{B(a,w)-B(a,w')}_m\leq \varepsilon\cdot\abs{w-w'}_m
$$
for all $(a,w)$, $(a,w')$ close to $(0,0)$ on level $m$.
\end{defn}

We are in the position to define the notion of a Fredholm germ.
\begin{defn}
Let $p:W\rightarrow X$ be a strong bundle, $x\in X_\infty$, and
$f$ a germ of an sc-smooth section of $p$ around $x$. We call $(f,x)$ a Fredholm germ provided there exists a germ of $\ssc^+$-section $s$ of $p$ near $x$ satisfying 
$s(x)=f(x)$ and such that  in suitable strong bundle coordinates mapping $x$ to $0$, the  push-forward  germ $g=\Phi_\ast(f-s)$ around $0$ has a filled version
$\ov{g}$ so that the germ $(\ov{g},0)$ is equivalent  to a germ belonging to  $\mathfrak{C}_{basic}$.
\end{defn}

Let us observe that tautologically if $(f,x)$ is a Fredholm germ
and $s_0$ a germ of $\ssc^+$-section around $x$, then $(f+s_0,x)$ is a Fredholm germ as well. Indeed,  if $(f-s,0)$ in suitable coordinates has a filled version, then $((f-s_0)-(s-s_0),0)$ has as well.

Finally,  we can introduce the Fredholm sections of  strong M-polyfold bundles.
\begin{defn}
Let $p:W\rightarrow X$ be a strong M-polyfold bundle and $f\in\Gamma(p)$ an sc-smooth section. The section $f$ is called  polyfold Fredholm section provided  it has the following properties:
\begin{itemize}
\item[(1)] $f$ is regularizing.
\item[(2)] At  every smooth point $x\in X$,  the germ $(f,x)$ is a Fredholm germ.
\end{itemize}
\end{defn}
 If $(f,x)$ is a Fredholm germ and $f(x)=0$,  then the  linearisation $f'(x):T_xX\rightarrow W_x$ is a linear sc-Fredholm operator. The proof can be found in \cite{HWZ3}.
If,  in addition,  the linearization $f'(x):T_xX\rightarrow W_x$ is surjective, then our implicit function theorem gives a natural smooth structure on the solution set of $f(y)=0$ near $x$ as the following theorem from \cite{HWZ3} shows.
\begin{thm}
Assume that $p:W\rightarrow X$ is a strong M-polyfold bundle and  let $f$ be a  Fredholm section of the bundle $p$. If the point $x\in X$ solves $f(x)=0$ and if the linearization  at this point $f'(x):T_xX\rightarrow W_x$ is surjective,  there exists an open neighborhood $U$ of $x$ so that the solution set  
$$
S_U:=\{y\in U\vert \ f(y)=0\}
$$
has in a natural way a smooth manifold structure induced from $X$. In addition, $S_U\subset X_\infty$.
\end{thm}

\subsection{An Illustration of the Concepts}\label{illustration}
In Section \ref{longcylinders} we shall illustrate  the polyfold concept by setting up a proof
of Theorem \ref{tran-pol} as a polyfold Fredholm problem. It illustrates the  analytical and conceptual difficulties in the study of  maps on conformal cylinders which break apart as the modulus tends to infinity.
It also illustrates the notion of a strong bundle as well as that of a Fredholm section.

We denote by $\Gamma$ the collection of  pairs
$(a,b)$ of complex numbers  satisfying  $ab\neq 0$ and $|a|,|b|<\varepsilon$. The size of $\varepsilon$ will be determined later. We denote by $X$ the set consisting of all tuples $(a,b,w)$ in which 
$(a,b)\in\Gamma$ and the map $w:Z_a\to  Z_b$ is a $C^1$-diffeomorphism of Sobolev class $H^3$ between the two cylinders and satisfying $w(p^\pm_a)=p^\pm_b$.  The points $p_a^\pm$ and $p_b^\pm$ are the points corresponding to the boundary points $(0,0)\in \partial(\R^+\times S^1)$ before `plumbing' the half-cylinders. Define the  filtration $(X_m)_{m\in \N_0}$ on $X$ by declaring  that $(a,b,w)$ belongs to $E_m$ if  $w$ belongs to  $H^{3+m}(Z_a,Z_b)$.

\begin{prop}\label{prop1.32}
The space $X$ carries in a natural way a second countable paracompact topology. For this topology the space is connected. Moreover, $X$  carries
in a natural way the structure of an M-polyfold built  on local models which are open subsets of an sc-Hilbert  space where the level $m$ corresponds to the  Sobolev regularity $m+3$.
\end{prop}

In other words the space $X$ is locally homeomorphic to open subsets of an  sc-Hilbert space so that the transition maps are sc-smooth.

In the next step we shall ``complete'' the space $X$ to the  space $\ov{X}$ by  adding  elements corresponding to the parameter
value $a=b=0$. 
This new space $\ov{X}$ will have as local models sc-smooth retracts.
Moreover,  $\ov{X}$ will be connected and will contain  $X$ as an open dense subset.

We abbreviate  by ${\mathcal D}={\mathcal D}^{3,\delta_0}$ the collection of pairs $(u^+,u^-)$  of $C^1$-diffeomorphisms,
$$u^\pm:\R^\pm\times S^1\to  \R^\pm\times S^1,$$
having the following properties. The maps  $u^\pm$ belong  to the space $H^{3}_{\loc}$ and there are  constants $(d^\pm,\vartheta^\pm)\in \R\times S^1$ and  maps $r^{\pm}\in H^{3,\delta_0}(\R^\pm\times S^1,\R^2)$
so that 
$$
u^{\pm}(s,t)=(s+d^\pm,t+\vartheta^\pm)+r^\pm(s,t).
$$
 Moreover,  we require that 
$$u^\pm(0,0)=(0, 0).$$
The  sc-structure on $\mcd$ is defined by declaring that the $m$-level $\mcd_m$  consists of elements of regularity $(m+3,\delta_m)$.  Then we define the set $\ov{X}$ as the disjoint union 
$$
\ov{X}= X\coprod \left(\{(0,0)\}\times {\mathcal D}\right),
$$
in which $(0,0)$ is the pair $(a=0, b=0)$. 
\begin{thm}\label{thm-comp}\mbox{}\\
\begin{itemize}
\item[(1)]  Fix $\delta_0\in (0,2\pi)$. Then the space $\ov{X}$ possesses
a natural paracompact second countable topology. In this topology the set $X$ is an open subset  of $\ov{X}$ and the induced topology on $X$ coincides with the previously defined topology on $X$.  Moreover,  $\ov{X}$ is connected.
\item[(2)] Fix a strictly increasing sequence $(\delta_m)_{m\in \N_0}$ of real numbers staring at $\delta_0$ and satisfying $0<\delta_m<2\pi$ and fix the exponential gluing profile $\varphi$ given by $\varphi (r)=e^{\frac{1}{r}}-e$. 
Then  there exists a natural M-polyfold
structure on $\ov{X}$ which induces on $X$ the previously defined M-polyfold structure.
\end{itemize}
\end{thm}
In our discussion of this theorem later on we shall  describe some metric aspects related to this  topology which gives a precise meaning of the convergence of elements $(a,b,w)$ to $(0,0,(u^+,u^-))$. 
The construction of the polyfold structure on $\ov{X}$  involves a third kind of splicing which will be explained later.  Variations  of this kind of splicings  will be used  in SFT.

For the following we assume that the sequence $(\delta_m)$ and the exponential gluing profile $\varphi$ are fixed so that $\ov{X}$ has a M-polyfold structure.  We consider  a smooth family
$$
v\mapsto  j^\pm(v)
$$
of complex structures on the half-infinite cylinders $\R^\pm\times S^1$  parametrized by $v$ belonging to an open neighborhood of $0$ in some finite-dimensional vector space $V$. Moreover, we assume that 
$j^\pm(v)=i$ outside of a compact neighborhood of the boundaries. It induces  the complex structure $j(a,v)$ on
the finite cylinder $Z_a$  if  $\abs{a}$ is small enough. Clearly,  $V\times \ov{X}$ is an M-polyfold.

With points  $(v,a,b,w)\in V\times \ov{X}$ satisfying  $a\neq 0$, 
we can associate  maps $z\mapsto  \phi(z)$ defined on the cylinder $Z_a$, whose images 
$$
\phi (z):(T_zZ_a,j(v))\to  (T_{w(z)}Z_b,i)
$$
are  complex anti-linear and belong to  the Sobolev space $H^2$.
If $a=0$ (and consequently $b=0$),  then $Z_0$ is the  disjoint union 
$$\R^+\times S^1\coprod  \R^-\times S^1$$ 
and here we consider two maps 
$z\to  \phi^{\pm}(z)$ defined on $\R^{\pm}\times S^1$ whose  complex anti-linear images
$$\phi^\pm (z):(T_z(\R^{\pm}\times S^1), j(v))\to  (T_{u^\pm (z)}(\R^{\pm}\times S^1), i)$$
belong  to $H^{2}_\loc$ on $\R^{\pm}\times S^1$ where $u^\pm:\R^\pm \times S^1\to \R^\pm \times S^1$ are the diffeomorphisms of the half cylinders introduced above. Moreover, the maps 
$$
(s,t)\mapsto   \phi^{\pm}(s,t)\frac{\partial}{\partial s}, \quad z=(s, t),
$$
belong  to $H^{2,\delta_0}(\R^{\pm}\times S^1,\R^2)$ if the tangent space $T_{w(z)}Z_b$ is identified with $\R^2$. The collection $E$ of all $(v,a,b,w,\phi)$ in which  $\phi$ is as just described,  possesses  a projection map
$$
E\to V\times \ov{X},\quad (v,a,b,w,\phi)\mapsto  (v,a,b,w),
$$
whose fiber has in a natural way the structure of a Hilbert space.
We can define a double filtration $E_{m,k}$ of $E$ where $0\leq k\leq m+1$. An element $(v,a,b,w,\phi)$ belongs to $E_{m,k}$
provided $(v,a,b,w)\in V\times  \ov{X}_m$ and $\phi$ is of class $(k+2,\delta_k)$.

\begin{thm}\label{thm1.42-nnew}

Having fixed the exponential gluing profile $\varphi$ and the increasing sequence $(\delta_m)_{m\in \N_0}$ of real numbers satisfying $0<\delta_m<2\pi$,   the space $E$ admits in a natural way the structure of a strong bundle over $V\times \ov{X}$.
\end{thm}

Finally, we define the section  $\ov{\partial }$ of the  
bundle $E\to V\times \ov{X}$ by its principal  part 
$$
\ov{\partial}(v,a,b,w):=\frac{1}{2}\left(Tw+i\circ (Tw)\circ j
\left(a,v\right)\right),
$$
and observe  that  
$$\ov{\partial}(v,a,b,w)=0$$
if and only if the map $w:(Z_a,j(a,v),p^+_a,p^-_a)\to  (Z_b,i,p^+_b,p^-_b)$
is biholomorphic. 

We shall prove in section \ref{nsection3.4} that the Cauchy-Riemann section $\care$ is an sc-smooth Fredholm section of the strong M-polyfold bundle $E\to V\oplus \ov{X}$. By the unformization theorem, there exists for 
 every point $(v, a)=(v_0,0)$  a unique  pair $(u_0^+,u_0^-)$ of biholomorphic mappings 
$$
u_0^\pm:(\R^\pm\times S^1,j^\pm(v_0))\to  (\R^\pm\times S^1,i)
$$
preserving the boundary points  $(0,0)\in \partial (\R^\pm \times S^1)$ so that the special point $(v_0, 0, 0, u^+_0,u^-_0)\in V\oplus \ov{X}$ is a solution of 
$\care (v_0, 0, 0,u^+_0, u^-_0)=0$. 

As a consequence of the implicit function theorem for polyfold Fredholm sections in  \cite{HWZ3}, we shall establish near the reference solution biholomorphic mappings  between finite cylinders  for $(v,a)$ close to $(v_0, 0)$ and  $a\neq 0$. More precisely,  we shall  prove the following result.
\begin{thm}\label{fred:thm0}
The Cauchy-Riemann section $\ov{\partial}$  of the strong bundle  $E\to V\times  \ov{X}$ is an  sc-smooth polyfold Fredholm section. Its linearization at the reference solution 
$(v_0, 0, 0, u^+_0, u^+_0)\in V\oplus \ov{X}$ is surjective and there exists  a  uniquely determined sc-smooth map 
\begin{gather*}
\Phi: B_{\varepsilon}(a_0)  \oplus V\to  V\oplus \ov{X}\\
(a, v)\mapsto  (v,a,b(a, v),w(a,v))
\end{gather*}
satisfying $(b( 0, v_0), w(0, v_0))=(0, u^+_0, u^+_0)$,  and solving the Cauchy-Riemann equation
$$
\ov{\partial}\Phi(v,a)=0.
$$
Moreover, these are the only solutions near the reference solution.
\end{thm}

Theorem \ref{fred:thm0}  describes,  in particular,  the smoothness
properties of the family of
biholomorphic maps between conformal cylinders which break apart.
Unraveling the M-polyfold structure on $\ov{X}$, 
we shall obtain the following result, where the set ${\mathcal D}^{m,\varepsilon}$ is the space of diffeomorphisms $u^\pm:\R^\pm \times S^1\to \R^\pm \times S^1$ of the half-cylinders of the form 
$$u^\pm (s, t)=(s, t)+(d^\pm, \vartheta^\pm)+r^\pm(s, t)$$
where $r^\pm $ have weak derivatives up to order $m$ which weighted by $e^{\varepsilon \abs{s}}$ belong to $L^2(\R^2\times S^1,\R^2)$.

\begin{thm}\label{strong-x}
Let $(v_0,0)\in V\oplus \C$ be fixed and  let 
$$(v,a)\mapsto  b(v,a)\quad \text{and} \quad  (v,a)\mapsto  w(v,a)$$ be 
 the the germs of maps  guaranteed  by  Theorem \ref{fred:thm0} and  defined on some small open neighborhood  $O(v_0,0)$ and where 
 $$
w(v,a):(Z_a,j(a,v),p^+_a,p^-_a)\to  (Z_{b(v,a)},i,p^+_{b(v,a)},p^-_{b(v,a)}).
$$
is the associated  biholomorphic map between the cylinders if $a\neq 0$.

Then given a constant $\Delta>0$,  the following holds on a  possibly smaller  neighborhood $O(v_0,0)$. There exists a map 
$$(a,v)\mapsto  \wt{w}(v,a)\in \bigcap_{m\geq 3, 0<\varepsilon<2\pi }{\mathcal D}^{m, \varepsilon}$$
having the following properties.  For every $m\geq 3$ and every  $0<\varepsilon<2\pi$ the map 
$$(v,a)\mapsto  \wt{w}(v,a)\in {\mathcal D}^{m, \varepsilon}$$
is smooth and satisfies 
$$
\wt{w}(v,a)(s,t)=w(v,a)(s,t)
$$
 for all $(s,t)\in [0,\frac{1}{2}\varphi(|a|) +\Delta]\times S^1$. In addition,  
$$
\wt{w}(v,a)(s,t)=(s+d_{(v,a)},t+\vartheta_{(v,a)}).
$$
for $s\geq s(a)$, that is, for large $s$ where large depends on the gluing parameter $a$.
\end{thm}

\section{Exploring  Sc-Smoothness}\label{sec2}
In chapter \ref{chapter1}  we have presented  background material and some of the basic results used in the construction of the polyfold structures in the SFT. Their  proofs  in chapter 
\ref{longcylinders} rely on  the technical results which we shall  prove now in chapter  \ref{sec2}.  

\subsection{Basic Results about Abstract Sc-Smoothness}\label{section2.1}
In this subsection we relate sc-smoothness with the more familiar notion of being $C^\infty$ or $C^k$. In particular,  we shall prove Proposition \ref{up-prop} and Proposition \ref{lo-prop}.
The first result relates the $\ssc^1$-notion with that of being $C^1$ and provides an alternative definition of a map between sc-Banach spaces to be of class $\ssc^1$. 
\begin{prop}\label{x1}
Let $E$ and $F$ be $\ssc$-smooth Banach spaces and let $U$ be a relatively open subset of a partial  quadrant $C$ in $E$. 
Then an  $\ssc^0$-map $f:U\to F$ is of class $\text{sc}^1$ if and only if  the following conditions hold true:
\begin{itemize}\label{sc-100}
\item[(1)] For every $m\geq 1$,  the induced map
$$
f:U_m\to  F_{m-1}
$$
is of class $C^1$. In particular,  the derivative  
$$
df: U_m\to  L(E_m,F_{m-1}),\quad x\mapsto df(x)
$$
is a  continuous map
\item[(2)] For every  $m\geq 1$ and every $x\in U_m$,  the bounded  linear operator
$df(x): E_m\to F_{m-1}$ has an extension to a bounded  linear operator $Df(x):E_{m-1}\to F_{m-1}$. In addition, the map
\begin{equation*}
U_m\oplus  E_{m-1}\to  F_{m-1}, \quad  (x,h)\mapsto  Df(x)h
\end{equation*}
is continuous.
\end{itemize}
\end{prop}

\begin{proof}
It is clear that the conditions (1) and (2) imply that  the map $f$ is $\ssc^1$. The other direction is more involved and uses  the compactness of the inclusions $E_{m+1}\to E_m$  in a crucial way.

Assume that $f:U\to F$ is of class $\ssc^1$.   Then the induced map  $f:U_1\to F$ is differentiable at every point $x$ with the derivative $df(x)=Df(x)\vert E_1\in L(E_1, F)$.  Hence  the extension of $df(x):E_1\to F$ to a continuous linear map $E\to F$ is the postulated map $Df(x)$.
We claim that the derivative $x\mapsto df(x)$ from $U_1$ into $L(E_1, F)$ is continuous. Arguing indirectly,  we find an $\varepsilon>0$ and sequences $x_n\to x$ in $U_1$ and $h_n$ of unit norm in $E_1$ satisfying
\begin{equation}\label{cone}
\abs{df(x_n)h_n-df(x)h_n}_0\geq \varepsilon.
\end{equation}
Taking a subsequence we may assume, in view of the compactness of the embedding $E_1\to E_0$,  that $h_n\to h$ in $E_0$. Hence, by the continuity property, $df(x_n)h_n=Df(x_n)h_n\to Df(x)h$ in $F_0$. Consequently,
$$df(x_n)h_n-df(x)h_n=Df(x_n)h_n-Df(x)h_n\to Df(x)h-Df(x)h=0$$ in $F_0$, in contradiction to \eqref{cone}.

Next we prove that $f:U_{m+1}\to F_m$ is differentiable at $x\in U_{m+1}$ with derivative
$$df(x)=Df(x)\vert E_{m+1}\in L(E_{m+1}, F_m)$$ so that the required extension of $df(x)$ is $Df(x)\in L(E_m, F_m)$.  The map  $f:U_1\to F_0$ is of class $C^1$ and $df(x)=Df(x)|E_1$.  Since, by continuity property (2), the map $(x, h)\mapsto Df(x)h$ from
$U_{m+1}\oplus E_m\to F_m$ is continuous, we can estimate for $x\in U_{m+1}$ and $h\in E_{m+1}$,
\begin{equation*}
\begin{split}
&\frac{1}{\abs{h}}_{m+1}\cdot \abs{f(x+h)-f(x)-Df(x)h}_m\\
&\phantom{==}=\frac{1}{\abs{h}}_{m+1}\cdot \left| \int_0^1\bigl[ Df(x+\tau h)\cdot h-Df(x)\cdot h\bigr] d\tau \right|_m\\
&\phantom{==}\leq  \int_0^1\left|  Df(x+\tau h)\cdot \frac{h}{\abs{h}}_{m+1}-Df(x)\cdot \frac{h}{\abs{h}}_{m+1} \right|_m\, d\tau .
\end{split}
\end{equation*}
Take a sequence $h\to 0$ in $E_{m+1}$. By the compactness of the embedding $E_{m+1}\to E_{m}$,  we may assume that $\frac{h}{\abs{h}}_{m+1}\to h_0$ in $E_m$. By the continuity property we now conclude that the integrand converges uniformly in $\tau$ to $\abs{Df(x)h_0-Df(x)h_0}_m=0$ as $h\to 0$ in $E_{m+1}$. This shows that $f:U_{m+1}\to F_m$ is indeed differentiable at $x$ with derivative $df(x)$ being the bounded linear operator
$$
df(x)=Df(x)|E_{m+1}\in L(E_{m+1}, F_m).
$$
The continuity of $x\mapsto df(x)\in L(E_{m+1}, F_m)$ follows by the argument already used above, so that $f:U_{m+1}\to F_m$ is of class $C^1$. This finishes the proof of the proposition.
\end{proof}
As a   consequence of Proposition \ref{x1} we obtain the  following proposition.
\begin{prop}\label{ias}
If $f:U\to  V$ is an  $\ssc^k$-map,  then
the induced map $f:U^1\to  V^1$ is also $\ssc^k$.
\end{prop}
\begin{proof}
We prove the assertion by induction with respect to $k$. Assume that $k=1$ and  that $f:U\to V$ is $\ssc^1$. Equivalently,  $f$ satisfies  parts (1) and (2) of 
of Proposition \ref{x1} for all $m\geq 1$. This implies that, after 
replacing $E$ and $F$ by $E^1$ and $F^1$ 
and $U$ by $U^1$,  the map $f:U^1\to  F^1$ also  satisfies
the points (1) and (2) of Proposition \ref{x1}, Applying  Proposition \ref{x1} again, we conclude that $f:U^1\to  F^1$ is $\ssc^1$.

Now assume that the  assertion holds for all $\ssc^k$--maps and  let $f:U\to V$ be  an $\ssc^{k+1}$-map. This means that  the tangent map $Tf:TU\to  TV$ is $\ssc^k$.  Then, by induction hypothesis, the tangent map  $Tf:(TU)^1\to  (TV)^1$ is an $\ssc^k$--map. Since  $T(U^1)=(TU)^1$ and $T(V^1)=(TV)^1$, we have proved that
$$
Tf:T(U^1)\to  T(V^1)
$$
is  an  $\ssc^k$-map. But this precisely means that $f:U^1\to  V^1$ is an $\ssc^{k+1}$--map. The proof of the proposition is complete.  
\end{proof}

Next we study the relationship between the notions of being $C^k$ and $\ssc^k$.

\begin{prop}\label{lower}
Let $U$ and $V$ be relatively open subsets of partial quadrants in sc-Banach spaces $E$ and $F$, respectively. 
If $f:U\to  V$ is $\ssc^k$, then for every $m\geq 0$ the map $f:U_{m+k}\to  V_m$ is of class $C^k$.  Moreover,  $f:U_{m+l}\to  V_m$ is of class $C^l$ for every $0\leq l\leq k$.
\end{prop}
\begin{proof}
The last statement is a consequence of the former since an $\ssc^k$-map is also $\ssc^l$-map for $0\leq l\leq k$.  Now we prove the main statement. Note that it suffices to 
show that $f:U_k\to  F_0$ is of class $C^k$. Indeed,  by Proposition \ref{ias}, the map $f:U^m\to  V^m$ is of class $\ssc^k$ and so we can repeat  the  argument  of Proposition  \ref{ias} for  the map $f:U^m\to  V^m$  replacing the map  $f:U\to  V$.

We prove the result by induction with respect to $k$. If  $k=0$ the statement is  trivially  true   and if  $k=1$ it  is just the condition (1) of  Proposition \ref{x1}. 
Now we assume that  result holds for all   $\ssc^k$--maps and let $f:U\to V$  be an $\ssc^{k+1}$--map.  
In particular, $f$ is of class $\ssc^k$ which,  by induction hypothesis, implies that $f:U_k\to  F_0$ is of class $C^k$. Also, the tangent map  $Tf:TU\to  TF$ is of class $\ssc^k$ and therefore $Tf:(TU)_k=U_{k+1}\oplus E_k\to  TF$ is of class $C^k$ as well. Denoting by $\pi:TF=F^1\oplus F\to F$  the projection onto the second factor, we consider 
the composition  
$$
\Phi:=\pi\circ Tf: U_{k+1}\oplus E_k\to F, \quad (x, h)\mapsto Df(x)h.
$$
We know that the map  $\Phi$ is of class $C^k$.  Taking $k$ derivatives but only with respect to $x$,  we obtain
a continuous map
$$
U_{k+1}\oplus E_k\to  L^k(E_{k+1},\ldots , E_{k+1};F), \quad (x,h)\mapsto (D^{k}_x\Phi)(x,h).
$$
Observing  that this map is linear in $h\in E_{k}$ and is continuous, we obtain the  map 
$$\Gamma:U_{k+1}\to  L^{k+1}(E_{k+1}, \ldots ,E_{k+1}; F)$$ defined  by 
$$
\Gamma (x):E_{k+1}\oplus\ldots \oplus E_{k+1}\to  F, \quad (h_1,\ldots ,h_k,h)\mapsto (D^k_x\Phi)(x,h)(h_1,\ldots ,h_k).
$$
We claim  that $\Gamma$  is continuous. Indeed, arguing indirectly we find  a point $x$ in $U_{k+1}$, a number $\varepsilon>0$, and sequence of points $(x_n, h_{1,n},\ldots , h_{k, n}, h_n)\in U_{k+1}\oplus E_{k+1}\oplus \ldots \oplus E _{k+1}$  so that 
$x_n\to x$ in $U_{k+1}$, all $h_{m,n}$  and $h_n$  have  length $1$ in the norm $\left| \cdot \right|_{k+1}$,  and
\begin{equation}\label{con-eq}
\abs{(\Gamma(x_n)-\Gamma(x))(h_{1,n},\ldots h_{k,n},h_n)}_0\geq \varepsilon>0.
\end{equation}
Since the inclusion $E_{k+1}\to E_k$ is compact, 
after perhaps taking a subsequence,  we may assume that $h_n\to  h$ in $E_k$.
Then  $(x_n,h_n)\to  (x,h)$ in $U_{k+1}\oplus E_k$ and, by the continuity,  
$$
\abs{D^k_x\Phi)(x_n,h_n)- (D^k_x\Phi)(x,h)}_{L^k(E_{k+1},\ldots ,E_{k+1};F)}\to 0.
$$
This, however,  contradicts  \eqref{con-eq}  since 
\begin{equation*}
\begin{split}
&\abs{(\Gamma (x_n)-\Gamma (x))(h_{1,n}, \ldots , h_{k,n}, h_{n})}_0\\
&=\abs{((D^k_x\Phi)(x_n, h_n)-(D^k_x\Phi)(x, h_n))(h_{1,n}, \ldots , h_{k,n})}_0\\
&\leq 
\abs{((D^k_x\Phi)(x_n, h_n)-(D^k_x\Phi)(x, h_n))}_{L^k(E_{k+1},\ldots ,E_{k+1}; F)}.
\end{split}
\end{equation*}

We shall now  prove that $f:U_{k+1}\to  F$ is of class $C^{k+1}$ by showing that the limit of 
$$
\frac{1}{\abs{\delta x}_{k+1}}[ D^kf (x+\delta x)-D^kf(x)-\Gamma (x)(\cdot , \delta x)]
$$
in $L(E_{k+1},\ldots ,E_{k+1}; F)$ is equal to $0$.  
For $x\in U_{k+1}$, $\delta x\in E_{k+1}$ small, and $t\in [0,1]$, we consider the $C^k$-map 
$$(t,  \delta x)\mapsto Df(x+t\delta x)\delta x.$$
Integrating with respect to $t$,  obtain  the $C^k$-map
$$
(x,\delta x)\mapsto  f(x+\delta x)-f(x).
$$
Differentiating  this map $k$ times with respect to $x$, we find for $h_1,\ldots ,h_k\in E_{k+1}$ that  
\begin{equation*}
\begin{split}
&D^kf(x+\delta x)(h_1,\ldots, h_k)-D^kf(x)(h_1,\ldots, h_k)\\
&\phantom{===}= D^k_x(f(x+\delta x)-f(x))(h_1,\ldots, h_k)\\
&\phantom{===}= D^k_x\left(\int_0^1 (Df(x+t\delta x)\delta x)dt\right)   (h_1,\ldots, h_k)\\
&\phantom{===}=\int_0^1 D_x^k((Df(x+t\delta x)\delta x)    (h_1,\ldots, h_k) dt\\
&\phantom{===}=  \int_0^1 \Gamma(x+t\delta x)  (h_1,\ldots, h_k, \delta x) dt.
\end{split}
\end{equation*}
Hence
\begin{equation*}
\begin{split}
&\frac{1}{\abs{\delta x}}_{k+1}\cdot [(D^kf(x+\delta x)-D^kf(x)))(h_1,\ldots, h_k)-\Gamma(x) (h_1,\ldots, h_k, \delta x) ]\\
&\phantom{====}= \int_0^1\bigl[ \bigl( \Gamma(x+t\delta x) - \Gamma(x) \bigr) (h_1,\ldots, h_k, \frac{\delta x }{\abs{\delta x}}_{k+1}) \bigr]\ dt.
\end{split}
\end{equation*}
Letting $\delta x\to 0$ in $E_{k+1}$ and using continuity of the map  $\Gamma: U_{k+1}\to  L(E_{k+1},\ldots ,E_{k+1};F)$,   we conclude that 
 the left-hand side above converges to $0$ in $L(E_{k+1},\ldots ,E_{k+1}; F)$. Consequently, $f:U_{k+1}\to  F$ is of class $C^{k+1}$ and the proof of Proposition \ref{lower}  is complete. 
\end{proof}

The next result is very useful in proving that a given map between sc-Banach spaces is sc-smooth.

\begin{prop}\label{ABC-x}
Let $E$ and $F$ be sc-Banach spaces and let $U$ be a relatively open subsets of partial quadrants in $E$. 
Assume that for every $m\geq 0$ and $0\leq l\leq k$ the map $f:U\rightarrow V$  induces  a map, 
$$
f:U_{m+l}\rightarrow F_m,
$$
which is of  class $C^{l+1}$. Then $f$ is $\ssc^{k+1}.$
\end{prop}

\begin{proof}
The proof is  by induction with respect to $k$. In order that the induction runs smoothly we have to prove slightly more. Here is  our inductive  assumption:
\begin{induction}
 If a map $f:U\to  F$ induces  maps  $f:U_{m+l}\to  F_m$ of class $C^{l +1}$ for all $m\in N_0$ and $0\leq  l\leq k$,  then $f$ is of class $\ssc^{k+1}$. Moreover,  if $\pi:T^{k+1}F\to  F^l$ denotes the projection onto a factor  $F^l$ of $T^{k+1}F$,  then the composition $\pi\circ T^{k+1} f :T^{k+1}U\to F^l$ is a linear combination of  maps of the following  type,
 
\begin{gather*}
\Gamma: U^{k+1}\oplus E^{n_1}\oplus\ldots \oplus E^{n_j}\to  F^l\\
 (x_1,x_{k_1},\ldots, x_{k_j})\mapsto   d^j(x_1)(x_{k_1},\ldots , x_{k_j}),
\end{gather*}
where $x_1\in U_{k+1}$, $x_{k_i}\in E_{n_i}$ and where  the nonnegative indices $j, k, l$ and $n_i$ satisfy 
$j+l-1\leq n_i\leq k$.
\end{induction}

We start the proof  with $k=0$. By assumption, $f:U_0\to F_0$ is of class $C^1$. For $x_1\in U_1$, we define 
the map $Df(x_1):E_0\to F_0$ by $Df(x_1)x_2=df (x_1)x_2$ where $df (x_1)$ is the linearization of 
$f:U_0\to F_0$ at the point $x_1$. Then $Df(x_1)\in {\mathcal L}(E_0, F_0)$ and as a map from $E_1\to F_0$, the map $Df (x_1)$ is  the derivative of $f:U_1\to F_0$ at the point $x_1$.  The tangent map 
$Tf :TE\to TF$ is, by definition, given by 
$$Tf (x_1, x_2)=(f(x_1), df (x_1)x_2).$$
Since the tangent map is continuous, we have proved that $f$ is $\ssc^1$.  If $\pi$ is a projection onto any factor of $TF$, then the composition 
$\pi\circ Tf$ is either the map $E^1\to F^1$ given by $x_1\mapsto f(x_1)$ or  the map $\pi \circ Tf:E^1\oplus E\to F$ defined by 
$\pi \circ Tf(x_1, x_2)=df(x_1)x_2$. Both maps are of the required form and the indices satisfy the required inequalities.  This finishes the proof of the statement  ${\bf (S_0)}$.

Next we assume that we have established ${\bf (S_k)}$  and  let $f:U\to F$ be a map such that 
$f:U_{m+l}\to F_m$ is of class $C^{l+1}$ for all $m\in \N_0$ and $0\leq l\leq k+1$. 
In particular,  since the map $f$ satisfies  the statement ${\bf (S_k)}$, it is of class 
$\ssc^{k+1}$ and, in addition, the compositions $\pi\circ T^{k+1}f:T^{k+1}E\to F^l$ are linear combinations of maps  of types  described in  ${\bf (S_k)}$.  We have to show that 
$f$ is $\ssc^{k+2}$ and to do this  we show that  the tangent map $T^{k+1}f$ is $\ssc^1$. It suffices to show that any of the terms making up the composition $\pi\circ T^{k+1}f$ is $\ssc^1$. 
Hence we consider the map $\Gamma: U^{k+1}\oplus E^{n_1}\oplus\ldots \oplus E^{n_j}\to  F^l$
given by
\begin{equation}\label{candidate-eq0}
\Gamma (x_1,x_{k_1},\ldots, x_{k_j})=d^j(x_1)(x_{k_1},\ldots , x_{k_j})
\end{equation}
with  $x_1\in U_{k+1}$, $x_{k_i}\in E_{n_i}$ and the indices satisfying $ j+l-1\leq n_i\leq k. $
Given $ (x_1,x_{k_1},\ldots, x_{k_j})\in U^{k+2}\oplus E^{n_1+1}\oplus\ldots \oplus E^{n_j+1}$, the candidate for the linearization 
$$D\Gamma (x_1,x_{k_1},\ldots, x_{k_j}): E^{k+1}\oplus E^{n_1}\oplus\ldots \oplus E^{n_j}\to  F^l$$ is the map 
defined by 
\begin{equation}\label{candidate-eq1}
\begin{split}
d^{j+1}f(x_1)(\delta x_1, x_{k_1},\ldots, x_{k_j})+\sum_{i=1}^jd^{j}f(x_1)(x_{k_1},\ldots, \delta x_{k_i}, \ldots , x_{k_j}).
\end{split}
\end{equation}
The map in   \eqref{candidate-eq1} is well-defined. Indeed, by assumption $j+l-1\leq n_1\leq k$, so that the map $f:U^{j+l-1}\to F^l$ is of class $C^j$. Since 
$E^{k}\subset E^{n_i}\subset E^{j+l-1}$, it follows that the map 
$$E^{n_i}\mapsto F^l, \quad \delta x_{k_i}\mapsto d^j f(x_1)(x_{k_1}, \ldots ,\delta x_{k_i}, \ldots x_{k_j})$$
is a bounded linear map.  Similarly, 
since $j+l-1\leq n_i\leq k$,   we  have $j+l\leq n_i+1\leq k+1$ and that 
$U^{k+2}\oplus E^{k+1}\oplus E^{n_1+1}\oplus\ldots \oplus E^{n_j+1}\subset U^{j+l}\oplus 
E^{j+l}\oplus \ldots \oplus E^{j+l}$. Since by our inductive assumption  the map $f:U^{j+l}\to F^l$ is of class $C^{j+1}$, it follows that  for given  $(x_{k_1},\ldots, x_{k_j})\in  E^{n_1+1}\oplus\ldots \oplus E^{n_j+1}$, the map $U^{k+1}\to F^l$ defined by  $x_1\mapsto d^j f(x_1)(x_{k_1},\ldots, x_{k_j})$ is of class $C^1$ which implies that the first term in \eqref{candidate-eq1} defines a bounded linear map 
$$E^{k+1}\mapsto F^l, \quad \delta x_1\mapsto d^{j+1}f(x_1)(\delta x_1, x_{k_1},\ldots, x_{k_j}).$$
Hence the map $D\Gamma (x_1,x_{k_1},\ldots, x_{k_j}): E^{k+1}\oplus E^{n_1}\oplus\ldots \oplus E^{n_j}\to  F^l$ defines a bounded linear operator and,  as a  map from $E^{k+2}\oplus E^{n_1+1}\oplus\ldots \oplus E^{n_j+1}\to  F^l$,  it  is the derivative of 
$\Gamma: U^{k+2}\oplus E^{n_1+1}\oplus\ldots \oplus E^{n_j+1}\to  F^l$. Moreover, the evaluation map 
\begin{gather*}
D\Gamma:U^{k+2}\oplus E^{n_1+1}\oplus\ldots \oplus E^{n_j+1}\oplus E^{k+1}\oplus E^{n_1}\oplus\ldots \oplus E^{n_j}\to  F^l\\
(x_1,x_{k_1},\ldots, x_{k_j}, \delta x_1,  \delta x_{k_1}, \ldots \delta x_{k_j})\mapsto D\Gamma 
(x_1,x_{k_1},\ldots, x_{k_j}, \delta x_1,  \delta x_{k_1}, \ldots \delta x_{k_j})
\end{gather*}
 is continuous. 
So, the map $\Gamma$ is  $\ssc^1$  and hence the tangent  map 
$T^{k+1}f:T^{k+1}E\to T^{k+1}F$ is of class $\ssc^1$. We have proved that  $f:U\to F$ is  of class $\ssc^{k+2}$. 
The tangent map $T^{k+2}f:T^{k+2}E\to T^{k+2}F$ is of the form
\begin{equation*}
\begin{split}
&T^{k+2}f(x_1, \ldots , x_{2^{k+2}})\\
&\phantom{=}=
(T^{k+1}f(x_1, \ldots , x_{2^{k+1}}), DT^{k+1}f(x_1, \ldots , x_{2^{k+1}})
(x_{2^{k+1}+1}, \ldots , x_{2^{k+2}}).
\end{split}
\end{equation*}
We  consider the composition $\pi\circ T^{k+2}f$.
If $\pi$ is a projection onto the first $2^k$ summands we see the terms guaranteed by the induction hypothesis ${\bf (S_{k})}$ but now  with the index  of the spaces raised  by one. If $\pi$ is a projection onto  any of the last $2^k$ factors,  the previous discussion shows that they are linear combinations of the  maps occurring  in \eqref{candidate-eq1}.  Denote  the new indices by $j', k', l'$ and $n_i'$. Then in case of the map $d^{j+1}f(x_1)(\delta x_1, x_{k_1},\ldots, x_{k_j})$ in \eqref{candidate-eq1}. we have $k'=k+1, j'=j+1, l'=l$ and $n_1'=k+1$ and $n_{i+1}'=n_{i}+1$ for $i=1,\ldots , j$. For the map $d^{j}f(x_1)(x_{k_1},\ldots, \delta x_{k_i}, \ldots , x_{k_j})$, we have 
$k'=k+1, j'=j$, $ l'=l$ and $n_i'=n_i$ and $n'_s=n_s+1$ for $s\neq i$.  One checks that the new indices satisfy the required inequalities. Hence ${\bf (S_{k+1})}$  holds and the proof of Proposition \ref{ABC-x}  is complete. 
\end{proof}

There is a useful  corollary to Proposition \ref{ABC-x}.
\begin{cor}\label{ABC-y}
Let $U\subset C\subset E$ be a relatively open subset of a partial quadrant in the sc-Banach space $E$. Assume that $f:U\to  \R^N$ is a map so that
for some $k$ and all $0\leq l \leq k$ the map $f:U_l \to  \R^N$ belongs to  $C^{l +1}$. Then $f$ is $\ssc^{k+1}$.
\end{cor}

\subsection{Actions by Smooth Maps, Proof of Theorem \ref{thm-125}}\label{actionsmooth-section}
This section is devoted to  the proof of Theorems \ref{thm-125}. We recall the set-up. Denoting by $D$ the closed unit-disk in $\C$ and by $V$ an open neighborhood of  the origin $0$ in $\R^n$, we consider a smooth map 
$$V\times D\to D, \quad (v, x)\mapsto \phi_v (x)$$
satisfying $\phi_0(0)=0$. 
We equip the Hilbert space $E=H^3(D, \R^N)$ with the  sc-structure $E_m=H^{3+m}(D, \R^N)$ and  introduce  the map 
$$\Phi:V\oplus E\to  E,\quad (v,u)\mapsto   u\circ\phi_v.$$
Since $\phi_v:D\to D$ is smooth,    the map $\Phi$ preserves  the  levels.  In the following we shall  refer to the map $\Phi$  as to an ``action by smooth maps''. 
\begin{thm}\label{thm-125p}
The above map $\Phi:V\oplus E\to E$ is sc-smooth.
\end{thm}
\begin{proof} We proceed by induction.  Here is our  inductive hypothesis:

\begin{induction}  The map 
$$\Phi:V\oplus E\to  E,\quad (v,u)\mapsto  u\circ \phi_v$$ is $\ssc^k$ and for every projection $\pi:T^kF\to  E^j$ onto a factor  of $T^kE$, the composition $\pi\circ T^k\Phi$ is a finite linear combination of maps  of the form
\begin{equation*}
V\oplus E^m\oplus(\R^n)^p\to  E^j, \quad (v,h,a_1, \ldots ,a_p)\mapsto  
\Phi(v,D^{\alpha}h)\cdot f(v, a_1, \ldots ,a_p)
\end{equation*}
where $f:V\oplus  (\R^n)^p\times D\to  \R$ is a smooth function which is  linear in every  variable $a_i$.  Moreover,  $\abs{\alpha}\leq m-j$ and $p\leq k$. 
\end{induction}

In the case $k=0$, there is exactly one projection $\pi:T^0E=E\to E$, namely the identity map. So, the composition $\pi\circ \Phi=\Phi$ has the required form 
with  $m=j=0$, $\alpha=(0,0)$, and $f\equiv 1$, where the empty product of $a_i$'s  is defined as being $1$. To prove that $\Phi$ is $\ssc^0$,  we fix a point $(v_0, u_0)\in V\oplus  E_m$, and, for given  $r>0$,  we choose a smooth map $w_0:D\to \R^N $ satisfying $\abs{u_0-w_0}_m\leq r$.
Then we estimate 
\begin{equation*}
\begin{split}
\abs{\Phi(v,u)-\Phi(v_0,u_0)}_m&\leq \abs{\Phi(v,u)-\Phi(v,w_0)}_m +\abs{\Phi(v,w_0)-\Phi(v_0,w_0)}_m\\
&\phantom{\leq}+\abs{\Phi(v_0,w_0)-\Phi(v_0,u_0)}_m\\
&=: I+II+III.
\end{split}
\end{equation*}
For $v$ close enough to $v_0$,  the map $E_m\to E_m$ defined by  $h\mapsto h\circ\phi_{v}$ is a bounded linear operator with uniformly bounded  norm by a constant $C$ which only depends on an arbitrarily fixed small neighborhood of $v_0$. Consequently, we obtain the following estimates for   the terms  $I$ and $III$, 
$$
I\leq C\cdot\abs{u-w_0}_m\leq C\cdot \abs{u-u_0}_m+C\cdot \abs{u_0-w_0}_m\leq C\cdot\abs{u-u_0} +C\cdot r
$$
and
$$
III\leq C\cdot \abs{u_0-w_0}_m\leq C\cdot r.
$$
Since $w_0$ and $(v, x)\mapsto \phi_v (x)$ are smooth, it follows immediately that
$D^{\alpha}(w_0\circ \phi_v)\to  D^{\alpha}(w\circ \phi_{v_0})$  as $v\to v_0$. In particular,  $w_0\circ \phi_v\to  w\circ \phi_{v_0}$ in $E_m$ which implies that  
$$II\to 0, \quad \text{as $v\to v_0$. } $$
The number $r$ can be chosen to be as small as we wish,  so that our estimates  show that $\Phi$ is continuous on every  level $m$, 
that is, the map $\Phi$ is of class $\ssc^0$. So, ${\bf (S_0)}$ holds. 

To simplify the further steps, it turns out to be useful to first  prove ${\bf (S_{1})}$. 
We write $s+it$ for the coordinates on $\C$ and introduce the notation  $\phi_v=(A_v, B_v)$ where the maps 
$V\times D\to \R$, $(v, x)\mapsto A_v(x)$ and $(v, x)\mapsto B_v(x)$ are smooth. The derivatives   of $A_v$ and $B_v$ with respect  to the variable $v$ are denoted by  $DA_v$ and $DB_v$, respectively.

The candidate for the linearization $D\Phi (v,u) : V\oplus E\to E$  at the point $(v, u)\in V\oplus E^1$ is  given by 
\begin{equation}\label{guess}
D\Phi(v,u) (a,h)=\Phi(v,h)+\Phi(v,u_s) \cdot DA_v \cdot a+\Phi(v,u_t)\cdot DB_v\cdot a \
\end{equation}
where  $(a, h)\in V\oplus E$. 

Recalling that $\Phi$ is  $\ssc^0$ and observing that  the maps  $E^1\to  E$ defined by $u\mapsto u_s, u_t$ are sc-operators and  the functions $V\times D\to \R$ defined by $(v, x) \mapsto DA_v(x)\cdot a , DB_v(x)\cdot a$ are smooth,  we see that the map 
\begin{equation}\label{guess1}
T(V\oplus E)\to  TF, \quad (v,u,a,h)\mapsto  (\Phi(v,u),D\Phi(v,u)(a,h)),
\end{equation}
 is $\ssc^0$ where  $D\Phi(v,u)(a,h)$ is defined by \eqref{guess}.

It remains  to show that the right-hand side of \eqref{guess} defines the linearization of $\Phi$. 
With  $(v,u,a,h)\in T(V\oplus E)$, we have 
\begin{equation*}
\begin{split}
& \Phi(v+a,u+h)-\Phi(v,u)-D\Phi(v,u)(a,h)\\
&\phantom{===}=\Phi(v+a,h)-\Phi(v,h) \\
&\phantom{====}+\int_0^1 [\Phi(v+\tau a,u_s)DA_{v+\tau a}\cdot a-\Phi(v,u_s)DA_{v}\cdot a]\ d\tau\\
&\phantom{====}+\int_0^1 [\Phi(v+\tau a,u_t)DB_{v+\tau a}\cdot a-\Phi(v,u_t)DB_{v}\cdot a]\ d\tau\\
&\phantom{===}=I+II+III.
\end{split}
\end{equation*}
We have used the formula 
$$\Phi (v+a, u)-\Phi (a, u)=\int_0^1\frac{d}{d\tau }\Phi (v+\tau \cdot a, u)\ d\tau.$$
We consider the term $I$. Since $\Phi$ is linear with respect to the second variable, we have 
\begin{equation*}
\frac{1}{\abs{a}+\abs{h}_1}\cdot \abs{\Phi(v+a,h)-\Phi(v,h) }_0=\frac{\abs{h}_1}{\abs{a}+\abs{h}_1}\cdot \left| \Phi\left( v+a,\frac{h}{\abs{h}}_1\right)-\Phi\left(v,\frac{h}{\abs{h}}_1\right) \right|_0
\end{equation*}
for $h\neq 0$. The inclusion  $E_1\to  E_0$  is compact and hence we may assume that 
$\frac{h}{\abs{h}_1}\to h_0$ in $E_0$. Since $\Phi$ is $\ssc^0$,  we conclude that 
$$\frac{1}{\abs{a}+\abs{h}_1}\cdot \abs{\Phi(v+a,h)-\Phi(v,h) }_0\to 0$$
as $\abs{a}+\abs{h}_1\to 0$.  Next we consider the second term $II$. We have,   for $a\neq 0$, 
\begin{equation*}
\begin{split}
&\frac{1}{\abs{a}+\abs{h}_1}\left| \int_0^1 [\Phi(v+\tau a,u_s)DA_{v+\tau a}\cdot a-\Phi(v,u_s)DA_{v}\cdot a]\ d\tau\right|_0\\
&\leq \frac{\abs{a}}{\abs{a}+\abs{h}_1}\int_0^1 \left| \Phi(v+\tau a,u_s)DA_{v+\tau a}\cdot \frac{a}{\abs{a}}-\Phi(v,u_s)DA_{v}\cdot  \frac{a}{\abs{a}}\right|_0 \ d\tau
\end{split}
\end{equation*}
Since $\Phi$ is $\ssc^0$ and  $(v, x)\mapsto DA_v (x)$ is smooth,  we conclude that 
the above expression converges to $0$ as $\abs{a}+\abs{h}_1\to 0$. The same holds for the term $III$. 
Thus, 
$$\frac{1}{\abs{a}+\abs{h}_1}\cdot \abs{ \Phi(v+a,u+h)-\Phi(v,u)-D\Phi(v,u)(a,h)  }_0\to 0  $$
as $\abs{a}+\abs{h}_1\to 0$ so that the right-hand side of \eqref{guess} indeed defines  the linearization of $\Phi$ in the sense of Definition \ref{sc-def}. 
Moreover, the tangent map  $T\Phi :T(V\oplus E)\to TE$ given by \eqref{guess1} is $\ssc^0$. Summing up, the map $\Phi$ is $\ssc^1$.  From \eqref{guess},  if follows that the compositions of the tangent map $T\Phi$ with projections $\pi$ onto factors of $TE$  are linear combinations of  maps  of the required form. This completes the proof of  ${\bf (S_{1})}$.

Next we assume that the assertion ${\bf (S_{k})}$ has been proved and claim that  ${\bf (S_{k+1})}$ also  holds.
It suffices to show that the compositions of the  iterated tangent map $T^{k}\Phi:T^k(V\oplus E)\to T^kE$ with the projections $\pi:T^kE\to E^j$ onto the  factors of $T^kE$ 
are $\ssc^1$ and their linearizations have the required form.
By induction hypothesis, 
$\pi\circ T^{k}\Phi$ is the linear combination of maps  having the particular forms, 
and it suffices to show that our claim holds for each of these  maps. Accordingly, we consider the map 
\begin{equation*}
\Psi:V\oplus E^m\oplus (\R^n)^{p}\to  E^j,\quad (v,h,a)\mapsto
\Phi(v,D^{\alpha}h)\cdot f(v,a, \cdot )
\end{equation*}
where we have abbreviated  $a=(a_1, \ldots , a_p)$. The function $f:V\times (\R^n)^p\times D \to \R$ is  smooth and linear in each variable $a_i$. Moreover, $\abs{\alpha}\leq m-j$ and $p\leq k$.
Observe the the map $\Psi$ is the composition of the following maps. The map $E^m\to E^{m-\abs{\alpha}}$ defined by $h\mapsto D^{\alpha}h$ is an sc-operator and hence sc-smooth, it is composed with  the map 
$$\Phi:V\oplus E^m \to  E^j, \quad (v,u)\mapsto \Phi(v,u)$$
which we already know is of class $\ssc^1$. By the chain rule, this composition is at least of class $\ssc^1$.  
So, multiplication of this composition  by a smooth function $V\oplus (\R^n)^p\to \R$ defined by $(v, a)\mapsto f(v, a, \cdot )$ gives  an $\ssc^1$--map.  Having established that $\Psi$ is 
$\ssc^1$, it remains to show that 
the compositions $\pi\circ T\Psi$  of the tangent map $T\Psi:T(V\oplus E^m\oplus (\R^n)^{p})\to T(F^j)$  with the projections onto factors of $T(F^j)$ are linear combinations of maps of the required form.
The tangent map is given by 
$$T\Psi(v, h, a, \delta v, \delta h, \delta a)=(\Psi (v, h, a), D\Psi (v, h, a)( \delta v, \delta h, \delta a))$$
where $(v, h, a)\in V\oplus E^{m+1}\oplus (\R^n)^{p}$ and $(\delta v, \delta h, \delta a)\in \R^n\oplus E^m\oplus (\R^n)^{p}.$ If $\pi:T(E^j)=E^{j+1}\oplus E^j\to E^{j+1}$ is the projection onto  the first factor, then $\pi\circ T\Psi =\Psi$ and this map, in view of inductive hypothesis,  has the  form as  required in ${\bf (S_k)}$  but with the indices $m$ and $j$ raised by $1$. So, we consider the projection onto the second factor and the map $\pi\circ T\Psi=D\Psi$. Using the chain rule and the linearization of $\Phi$ given by \eqref{guess}, 
the linearization $D\Psi$ is a linear combination of the following four types of maps:
\begin{itemize}
\item[(1)] $V\oplus E^{m}\oplus (\R^n)^p \to F$,
$$(v, \delta h, a)\mapsto  \Phi(v,D^{\alpha} (\delta h) )\cdot f(v,a).$$
\item[(2)] $V\oplus  E^{m+1}\oplus (\R^n)^{p+1}\to F$,
$$(v, h,  (\delta v, a))\mapsto  \Phi(v,D^{\alpha +(1,0)} h)\cdot (DA_v\cdot \delta v) f(v,a)$$
and
$$(v, h,  (\delta v, a))\mapsto  \Phi(v,D^{\alpha +(0,1)} h)\cdot (DB_v\cdot \delta v) f(v,a).$$
\item[(3)] $V\oplus  E^{m+1}\oplus (\R^n)^{p+1}\to F$,
$$(v, h, a)\mapsto  \Phi(v,D^{\alpha} h )\cdot D_vf(v,a)\cdot \delta v.$$
\item[(4)] $V\oplus  E^{m+1}\oplus (\R^n)^{p+1}\to F$,
$$(v, h, (a, \delta a_i))\mapsto  \Phi(v,D^{\alpha} h )\cdot  f(v,(a_1,\ldots ,\delta a_i, \ldots , a_p))$$
for every $1\leq i\leq p$.
\end{itemize}
These types are all of the desired form. Having verified the statement ${\bf (S_{k+1})}$, the proof of Theorem \ref{thm-125} is complete.
\end{proof}
We mention a related  result  which has application  in the constructions of SFT. 
We assume that $V$ is an open neighborhood of $0$ in $\R^n$ and let the following data be given:
\begin{itemize}
\item[(1)] Smooth maps $c:V\to \R$ and $d:V\to S^1$.
\item[(2)] A smooth map 
$$V\times (\R^+\times S^1)\to  \R^2, \quad (v,(s,t))\mapsto  r_v(s,t)$$
where the function
$$r_v:V\to H^{m, \varepsilon}(\R^+\times S^1)$$
is smooth  for every $m\geq 3$ and every  $\varepsilon\in (0,2\pi)$.
\item[(3)] If  $v\in V$ and $(s,t)\in \R^+\times S^1$,  then 
$$(s+c(v),t+d(v))+r_v(s,t)\in \R^+\times S^1.$$
\end{itemize}
For every $v\in V$, define the map $\phi_v:\R^+\times S^1\to \R^+\times S^1$ by 
$$\phi_v (s,t)=(s+c(v),t+d(v))+r_v(s,t).$$
Given  a strictly increasing sequence $(\delta_m)_{m\in \N_0}$ of real numbers satisfying  
$0<\delta_0<\delta_m<2\pi$ , we equip the Banach  space 
$$
E=H^{3,\delta_0}(\R^+\times S^1,\R^N)
$$
 with the sc-structure defined by $E=H^{3+m,\delta_m}(\R^+\times S^1,\R^N).$
We also  define  the Banach space $\wh{E}$ by
$$
\wh{E}:=\R^N+H^{3,\delta_0}(\R^+\times S^1, \R^N ).
$$
A map $ u: \R^+\times S^1\to \R^N $ belongs to $\wh{E}$ if it belongs to   $H^{3}_{\text{ loc } } $ and if  there exists a constant $c\in \R^n$   satisfying  
$u- c\in H^{3,\delta_0}(  \R^+\times S^1, \R^N)$ so that  $u$ can be written as $u=c+(u-c)\in \R^N+H^{3,\delta_0}(\R^+\times S^1, \R^N ).$

\begin{thm}\label{thm-126}
Let  $\phi_v$ be as described above. Then the composition 
$$
\Psi:V\oplus E\to  E, \quad (v,u)\mapsto  u\circ \phi_v
$$
is well-defined and $\ssc$-smooth. The same result is true if we replace  $E$  by $\wh{E}$.
\end{thm}

The proof is in its structure quite similar to the previous proof. It is clear that the result follows for $\wh{E}$ once it is proved for $E$.
One first shows  that $\Psi$ is $\ssc^0$. Fixing  an $m$ one should  recall that compactly supported smooth maps are dense in $E_m$.  For  a compactly supported map $w_0$   the convergence $\Psi(v,w_0)\to  \Psi(v_0,w_0)$ in $E_m$ as $v\to  v_0$ is obvious. Then one easily verifies that for $v$ in a suitable open neighborhood of $v_0$ there is a uniform bound  of the operator norm of $\Psi(v,\cdot)$. From this point on we  can argue as in a previous proof to obtain continuity. Next one proves that $\Psi$ is $\ssc^1$ and proceeds by  induction. We leave the details to the reader.

\subsection{A Basic Analytical Proposition}\label{Basic-SC-Smoothness}
We continue with our study of sc-smoothnees.  We denote by $\varphi$ the exponential gluing profile 
$$\varphi (r)=e^{\frac{1}{r}}-e,  \quad \text{$r>0$. }$$
With the  nonzero complex number $a$ (gluing parameter)  we associate the gluing angle 
 $\vartheta\in S^1$ and the gluing length $R$ via  the formulae
$$a=\abs{a}e^{-2\pi i \vartheta}\quad \text{and}\quad R=\varphi (\abs{a}).$$
Note that $R\to \infty$ as $\abs{a}\to 0$.

We denote by $L$ the Hilbert sc-space $L^2(\R\times S^1,\R^N)$ equipped  the sc-structure $(L_m)_{m\in \N_0}$ defined by 
$L_m=H^{m,\delta_m}(\R\times S^1,\R^N)$, where $(\delta_m)$ is a strictly increasing sequence starting with  $\delta_0=0$. Let us also introduce
the sc-Hilbert spaces $F=H^{2,\delta_0}({\mathbb R}\times S^1,{\mathbb R}^N)$ with the sc-structure whose  level $m$ corresponds to the Sobolev regularity
$(m+2,\delta_m)$. Here  $\delta_0>0$ and $(\delta_m)$ is a strictly increasing sequence starting with $\delta_0$. Finally we introduce  the sc-Hilbert space $E=H^{3,\delta_0}({\mathbb R}\times S^1,{\mathbb R}^N)$ whose  level $m$ corresponds to  the regularity $(m+3,\delta_m)$ and  the sequence $(\delta_m)$ is as in the $F$-case.

With these data fixed we prove the following proposition. The  proposition has many applications. In particular, it will be used in Section \ref{subsub} in  order to prove that  the transition maps   between local M-polyfolds are sc-smooth.
\begin{prop}\label{qed}
The following four  maps
$$
\Gamma_i:B_\frac{1}{2}\oplus G\to  G,\ \ i=1,\ldots, 4, 
$$
where $G=L$, $G=F$ or $G=E$,  are sc-smooth.
\begin{itemize}
\item[(1)]  Let $f_1:\R\to  \R$ be a smooth  function which is constant outside of a compact interval so that $f_1(+\infty)=0$. Define
$$
\Gamma_1(a,h)(s, t)=f_1\left(s -\frac{R}{2}\right)h (s, t)
$$
if $a\neq 0$ and $\Gamma_1(0,h)=f(-\infty)h$ if $a=0$. \\
\item[(2)]  Let $f_2:\R\to  \R$ be a compactly supported  smooth function. Define  
$$
\Gamma_2(a,h)(s, t)=f_2\left(s -\frac{R}{2}\right)h(s-R,t-\vartheta)
$$
if  $a\neq 0$ and $\Gamma_2(0,h)=0$ if $a=0$.
\item[(3)]  Let $f_3:\R\to  \R$ be a smooth  function which is constant outside of a compact interval and satisfying $f_3(\infty)=0$. Define 
$$
\Gamma_3(a,h)(s', t')=f_3\left(-s' -\frac{R}{2}\right)h(s',t')
$$
if $a\neq 0$ and $\Gamma_3(0,h)=f_3(-\infty)h$
if $a=0$.
\item[(4)]  Let $f_4:\R\to  \R$ be a smooth function of compact support and define  
$$
\Gamma_4(a,h)(s', t')=f_4\left(-s' -\frac{R}{2}\right)h(s'+R,t'+\vartheta)
$$
if $a\neq 0$ and $\Gamma_4(0,h)=0$  if  $a=0$.
\end{itemize}
\end{prop}
Let us first note that we only have to prove the proposition in the case $G=L$, since the other cases are  obtained by taking  the sequence
$(\delta_m)$ for the $L$-case and  raising the index by $2$ in the case $G=F$ and by $3$ in the case $G=E$. The key point in the proof is the following. The gluing length $R$ as well as the gluing angle $\vartheta$ are functions of the gluing parameter $a$. As long as $a\neq 0$ these functions are smooth. However, as $a\to  0$ their derivatives blow-up.To achieve sc-smoothness as stated in Proposition \ref{qed},  it is important that the other terms occurring in the formulae have a sufficient decay behavior. There the assumption on exponential decays, as well as the filtration by levels comes in. 
 We  will only consider the maps $\Gamma_1$ and $\Gamma_2$, the proofs for the maps $\Gamma_3$ and $\Gamma_4$ are quite similar and left to the reader.
The proof will require several steps and takes the rest of this section. 
In the first step  we prove the $\ssc^0$-property.  
\begin{lem}
The maps $\Gamma_1$ and $\Gamma_2$ are $\ssc^0$.
 \end{lem}
 \begin{proof}
 Since, in view of Proposition \ref{sc},  the shift operator is $\ssc^0$,  the only difficulty can arise at $a=0$. We begin with  the map $\Gamma_1$. We may assume without loss of generality that $f_1(-\infty)=1$ so that   $\Gamma_1(0, h)=h$ for every $h\in L$.
 Fix a level $m$ and observe that 
 \begin{equation}\label{maps-eq}
 \abs{\Gamma_1(a,h)}_m\leq C'\cdot
 \left[\max_{0\leq k\leq m} \sup_{\R} \abs{f^{(k)}}\right]\cdot \abs{h}_m=C \abs{h}_m
\end{equation}
with the  constant $C=C'\cdot \left[\max_{0\leq k\leq m} \sup_{\R} \abs{f^{(k)}}\right]$ independent of $a$ and $h$.

The smooth compactly supported maps are dense in $L_m$ for every $m$.  If  $u_0$ a smooth compactly supported function  and $\abs{a}$ is sufficiently small, then 
$\Gamma_1(a,u_0)=u_0$.  Given $h_0\in L_m$ and $\varepsilon>0$, we choose a smooth compactly supported  function $u_0$  satisfying  $\abs{ u_0-h_0}_m\leq \varepsilon$. 
Then,  recalling  that $\Gamma_1(0, h_0)=h_0$ and using \eqref{maps-eq}, we have,  with $\abs{a}$ sufficiently small,  the following estimate,
 \begin{equation*}
\begin{split}
&\abs{\Gamma_1(a, h)-\Gamma_1(0, h_0)}_m=\abs{\Gamma_1(a, h)- h_0}_m\\
&\phantom{==}=\abs{\Gamma_1(a, h)-\Gamma_1(a,h_0)+\Gamma_1(a,h_0)-\Gamma_1(a,u_0)+u_0-h_0}_m\\
&\phantom{==}\leq \abs{\Gamma_1(a, h)-\Gamma_1(a, h_0}_m+\abs{\Gamma_1(a, h_0)-\Gamma_1(a, u_0)}_m+\abs{u_0-h_0}_m\\
&\phantom{==}\leq C\abs{h-h_0}_m+ (C+1)\abs{u_0-h_0}_m.
\end{split}
\end{equation*}
So, if $\abs{h-h_0}_m<\varepsilon$, then 
$$\abs{\Gamma_1(a, h)-\Gamma_1(0, h_0)}_m< (2C+1)\varepsilon$$
which proves  continuity of  $\Gamma_1$  at $(0, h_0)$ on level $m$.

The $\ssc^0$-property  of the map  $\Gamma_2$ is  more involved. Again,  the difficulty  arises at $a=0$. We fix a level $m$ and  first show that the norm of $\Gamma_2(a, h)$ is uniformly bounded with respect to $a$ close to $0$. By assumption, the support of $f_2$ is contained in the interval $[-A,A]$.  If $a$ is sufficiently small, then 
$[-A-\frac{R}{2},A-\frac{R}{2}]\subset(-\infty,0]$ and $[-A+\frac{R}{2},A+\frac{R}{2}]\subset[0,\infty)$, respectively.
Assuming $h\in L_m$, we estimate the norm $\abs{\Gamma_2(a, h)}_m$. The square of the  norm $\abs{\Gamma_2(a, h)}_m$  is equal to the sum of the integrals 
\begin{equation}\label{int-eq}
I_{\alpha}=\int_{\Sigma_R}\left| D^{\alpha}\left( f_2\left( s-\frac{R}{2}\right) h(s-R, t-\vartheta)\right)\right|^2 e^{2\delta_m \abs{s}}\ ds dt
\end{equation}
with $\abs{\alpha}\leq m$ and where we have abbreviated  $\Sigma_R=[ -A+\frac{R}{2}, A+\frac{R}{2}]\times S^1$. Denoting by $C$ a generic constant  independent of $a$ and $h$,   we estimate the integral $I_{\alpha}$  in \eqref{int-eq}  as follows
\begin{equation*}
\begin{split}
I_{\alpha}&\leq  C\int_{\Sigma_R}\abs{D^{\alpha}h(s-R, t-\vartheta)}^2e^{2\delta_m s }\ ds dt\\
&\leq  C\int_{\Sigma_{-R}}\abs{D^{\alpha}h(s, t)}^2e^{2\delta_m (s+R)}\ ds dt\\
&= C\int_{\Sigma_{-R}}\abs{D^{\alpha}h(s, t)}^2e^{-2\delta_m s}e^{2\delta_m(2s+R)}\ ds dt\\
&\leq  e^{4\delta_m A} C\int_{\Sigma_{-R}}\abs{D^{\alpha}h(s, t)}^2e^{-2\delta_m s}\ ds dt\\
&\leq  e^{4\delta_m A} C\cdot \abs{h}^2_m.
\end{split}
\end{equation*}
Hence
\begin{equation}\label{maps-eq1}
\abs{\Gamma_2(a, h)}_m\leq e^{2\delta_m A}C\cdot \abs{h}_m
\end{equation}
where $C$ is a constant independent of $a$ and $h$.
Now if  $u_0$ is a smooth compactly supported map and   $a$ is sufficiently small, then $\Gamma_2(a,u_0)=0$.  Given $h_0$ and $\varepsilon>0$, we choose a  smooth compactly supported map $u_0$ so that $\abs{h_0-u_0}_m<\varepsilon$.  Using the estimate \eqref{maps-eq1},  we  compute for $\abs{a}$ small and $h\in L_m$ satisfying 
$\abs{h-h_0}_m<\varepsilon$, 
\begin{equation*}
\begin{split}
\abs{\Gamma_2 (a, h)-\Gamma_2 (0, h_0)}_m&=\abs{\Gamma_2 (a, h)}_m\\
&=\abs{\Gamma_2 (a, h)-\Gamma_2 (a, h_0)+\Gamma_2 (a, h_0)-\Gamma_2(a,u_0)}_m\\
&\leq \abs{\Gamma_2 (a, h)-\Gamma_2 (a, h_0)}_m+\abs{\Gamma_2 (a, h_0)-\Gamma_2(a,u_0)}_m\\
&\leq Ce^{\delta_mA}\bigl( \abs{h-h_0}_m+\abs{h_0-u_0}_m \bigr) <2Ce^{\delta_mA} \varepsilon, 
\end{split}
\end{equation*}
showing that $\Gamma_2$ is continuous at $(0, h_0)$ on level $m$. This completes the proof of the lemma.
 \end{proof}
 Next we derive decay estimates.   The constants $d_{m+k, m}$ in the lemma are  defined,   for every pair of nonnegative integers $(m, k)$,  as  the differences
$$d_{m+k,m}:=\frac{1}{2}(\delta_{m+k}-\delta_m).$$
\begin{lem}\label{lkk}
For every pairs $(m, k)$ of nonnegative integers, there exists a constant $C=C(m,k)>0$ independent of $h$ and $a$ so that 
\begin{align*}
\abs{ h-\Gamma_1(a,h)}_m&\leq C\cdot e^{-d_{m+k,m}\cdot R}\cdot\abs{h}_{m+k}\\
\abs{\Gamma_2(a,h)}_m&\leq C\cdot e^{-d_{m+k,m}\cdot R}\cdot\abs{h}_{m+k}
\end{align*}
for all $h\in L_{m+k}$. Here $R=\varphi (\abs{a})$. 
\end{lem}
\begin{proof}
We begin with the map $\Gamma_1$. Again we assume  that $f_1(-\infty)=1$.  The function $f=1-f_1$ satisfies $f(s)=0$ for $s\ll 0$ and $f(+\infty)=1$ and  we  study  the map $\Gamma(a,h)=h-\Gamma_1(a,h)$ which is defined by
 $$
 \Gamma(a,h)(s,t)=f\left(s-\frac{R}{2}\right)h(s,t).
 $$
 The support of $f(\cdot -\frac{R}{2})$ is contained in $[-A+\frac{R}{2}, \infty )$  and hence contained in $[0,\infty )$ if $a$ is sufficiently small. For such an $a$ and for $h\in L_{m+k}$,  the square of the norm 
 $\abs{\Gamma (a, h)}_m$ is the sum of the integrals 
 $$I_{\alpha}=\int_{\Sigma_R}\left|D^{\alpha}\left( f\left(s-\frac{R}{2}\right)h(s, t)\right)\right|^2e^{2\delta_m s}\ ds dt$$
with $\abs{\alpha}\leq m$. Here  we have abbreviated $\Sigma_R=[-A+\frac{R}{2}, \infty )\times S^1$.
Then,   with $C$  denoting a generic constant  independent of $a$ and $h$,   we  estimate
 \begin{equation*}
 \begin{split}
 I_{\alpha}&\leq C\sum_{\abs{\beta}\leq \abs{\alpha}} \int_{\Sigma_R}\abs{D^{\beta}h(s, t))}^2e^{2\delta_m s}\ ds dt\\
 &=C\sum_{\abs{\beta}\leq  \abs{\alpha}}\int_{\Sigma_R}\abs{D^{\beta}h(s, t))}^2e^{2\delta_{m+k}s-4d_{m+k,m} s}\ ds dt\\
 &\leq Ce^{-2d_{m+k. m} R}\sum_{\abs{\beta}\leq  \abs{\alpha}}\int_{\Sigma_R}\abs{D^{\beta}h(s, t))}^2e^{2\delta_{m+k}s}\ ds dt\\
 &\leq Ce^{-2d_{m+k. m} R}\cdot \abs{h}_{m+k}^2
\end{split}
\end{equation*}
Since $\abs{\Gamma (a, h)}_m^2=\sum_{\abs{\alpha}\leq m}I_{\alpha}$,  we obtain the  required estimate,
$$\abs{\Gamma(a, h)}_{m}\leq Ce^{-d_{m+k. m} R}\cdot \abs{h}_{m+k}$$
for  $a$ sufficiently small and $h\in L_{m+k}$ with  some constant $C$ independent of $a$ and $h$. This is exactly the required estimate.

We turn to the map $\Gamma_2$. The support of $f_2$ is contained in $[-A,A]$ for some $A>0$. 
Hence, the support of $f_2(\cdot -\frac{R}{2})$ is contained in the interval $[-A+\frac{R}{2}, A+\frac{R}{2}]$.  Moreover, if $\abs{a}$ is sufficiently small, then  
$[-A-\frac{R}{2},A-\frac{R}{2}]\subset(-\infty,0]$ and $[-A+\frac{R}{2},A+\frac{R}{2}]\subset[0,\infty)$. 
We estimate the square of the norm $\abs{\Gamma_2(a, h}_m$  for  $h\in L_{m+k}$ and  sufficiently small  $\abs{a}$. The square of the  norm  $\abs{\Gamma_2(a, h}_m$  is equal to the sum of the integral expressions
$$I_{\alpha}=\int_{\Sigma_R}\left| D^{\alpha}\left( f_2\left( s-\frac{R}{2}\right) h(s-R, t-\vartheta)\right)\right|^2e^{2\delta_{m}s}\ ds dt $$
with $\abs{\alpha}\leq m$ and where $\Sigma_R$ denotes the finite cylinder $[-A+\frac{R}{2},A+\frac{R}{2}]\times S^1$. With $C$ denoting a generic constant not depending on $a$ and $h$, the   integral $I_{\alpha}$ can be estimated as follows,
\begin{equation*}
\begin{split}
I_{\alpha}&\leq C\int_{\Sigma_R}\abs{D^{\alpha}h(s-R, t-\vartheta)}^2e^{2\delta_{m}s}\ ds dt \\
&=C\int_{\Sigma_{-R}}\abs{D^{\alpha} h(s, t)}^2e^{2\delta_{m}(s+R)}\ ds dt \\
&=C\int_{\Sigma_{-R}}\abs{D^{\alpha} h(s, t)}^2e^{-2\delta_{m+k}s+ 2(\delta_m+\delta_{m+k}) s+2\delta_mR }\ ds dt \\
&\leq C\int_{\Sigma_{-R}}\abs{D^{\alpha} h(s, t)}^2e^{-2\delta_{m+k}s} e^{2(\delta_m+\delta_{m+k}) (A-\frac{R}{2})+2\delta_m R}\ ds dt .\\
&\leq Ce^{-2d_{m+k, m}R}\int_{\Sigma_{-R}}\abs{D^{\alpha} h(s, t)}^2e^{-2\delta_{m+k}s} \ ds dt \\
&\leq Ce^{-2d_{m+k, m}R}\abs{h}_{m+k}^2.
\end{split}
\end{equation*}
Thus, 
$$\abs{\Gamma_2(a, h)}_m\leq Ce^{-d_{m+k, m}R}\abs{h}_{m+k}$$
for all $a$ sufficiently small and  $h\in L_{m+k}$. The constant $C$ is independent of $a$ and $h$.  The proof of Lemma  \ref{lkk} is complete.
\end{proof}
From the sc-smoothness of the shift-map proved in Proposition \ref{sc-sm} we conclude 
 the following lemma.
\begin{lem}
The maps $\R\oplus L\to  L$,  defined by
\begin{align*}
(R,u )&\mapsto  f_1(\cdot-\frac{R}{2})u \\
(R,v)&\mapsto  f_2(\cdot -\frac{R}{2})v(\cdot -R,\cdot -\vartheta), 
\end{align*}
are sc-smooth.
\end{lem}

We will need estimates for the derivative of the functions $a\mapsto R(a)=e^{\frac{1}{|a|}}-e$, where $a=|a|\cdot e^{-2\pi i\vartheta}$. In view of  Lemma \ref{ropet1}  proved in Appendix \ref{calc-lemma}  we have the following estimates.
\begin{lem}\label{lkj}
For every multi-index $\alpha=(\alpha_1, \alpha_2)$,  there exists a constant $C$ such that
$$
\abs{D^{\alpha}R(a)}\leq C\cdot R(a)\cdot [\ln(R(a))]^{2|\alpha|}
$$
for $0<\abs{a}<\frac{1}{2}$.
\end{lem}

Let us continue with the study of the map
$$
\Gamma_1:B_\frac{1}{2}\oplus L\to  L, \quad \Gamma_1(a,u)=f_1\left(\cdot -\frac{R}{2}\right)\cdot u.
$$
As already done before we assume that $f_1(-\infty)=1$ and  study, 
rather than $\Gamma_1$,  the map $\Gamma(a,u)=u-\Gamma_1(a,u)$ which  has the form 
$$
\Gamma(a,u)=f\left(\cdot -\frac{R}{2} \right)\cdot u
$$
where $f=1-f_1$.

Let us denote by ${\bf R}(a)$, for $a\neq 0$, any product of derivatives of the kind 
$$
{\bf R}(a)(a_1, \ldots ,a_n)=(D^{n_1}R)(a)(a_1,\ldots ,a_{n_1})\cdot\ldots \cdot (D^{n_l}R)(a)(a_{k_{l-1}+1}, \ldots ,a_{n})
$$
where $n=n_1+\ldots +n_l$ and $(a_1, \ldots , a_n)\in C^n$. We call the integer $n$  the order of ${\bf R}(a)$. We define  ${\bf R}(a)$ of order $0$ to be the constant function equal to $1$. 
To  prove the sc-smoothness of $\Gamma$,  we need a structural statement about the form of $T^k\Gamma$ for  $a\neq 0$. 
\begin{lem}\label{ind-lem}
Assume that $\pi:T^kL\to  L^j$ is the projection onto a factor of $T^kL$. Then, for $a\neq 0$,  the composition $\pi\circ T^k\Gamma:T^kL\to L^j$ is a linear combination of maps 
$$A:(B_\frac{1}{2}\setminus\{0\})\oplus {\mathbb C}^n\oplus L^m\to  L^j$$
of the form
\begin{equation}\label{hedgehog}\\
(a,h,w)\mapsto  {\bf R}(a)(h)\cdot f^{(p)}\left(s-\frac{R}{2}\right)\cdot w, 
\end{equation}
$h=(h_1, \ldots ,h_n)\in \C^n$ and ${\bf R}(a)$ has order $n$. Moreover,  the following inequalities hold
$$
p\leq m-j\quad \text{and}\quad n\leq k.
$$
\end{lem}
\begin{proof}
We prove the lemma  by induction with respect to $k$ starting with $k=0$.
In this case,  the statement is trivially satisfied since $T^0\Gamma=\Gamma$  so that  ${\bf R}(a)=1$, $n=p=0$, and $j=m=0$. 
We assume that the statement has been proved for  $k$ and verify   that it holds for $k+1$.
If $\pi:T^{k+1}L\to  L^j$ is a projection onto one of the first $2^k$ factors we know that terms of $\pi\circ T^{k+1}\Gamma$ have the required form by induction hypothesis.
The only thing which is different is that the indices $m$ and $j$ are both raised by one (recall the definition of the tangent). If $\pi:T^{k+1}L\to  L^j$  projects onto one of the last $2^k$ factors,  the terms of $\pi\circ T^{k+1}\Gamma$  are the linear combinations of derivatives of maps guaranteed by the induction hypothesis.
Hence  we take a  map of the form (\ref{hedgehog}) and differentiate in the sc-sense.
 For $(a,h,w,\delta a,\delta h,\delta w)\in T((B_\frac{1}{2}\setminus\{0\})\oplus {\mathbb C}^n\oplus L^m)=
(B_\frac{1}{2}\setminus\{0\})\oplus {\mathbb C}^n\oplus L^{m+1}\oplus {\mathbb C}\oplus{\mathbb C}^n\oplus L^m$ we obtain 
a linear combination of maps of the following  four types:
\begin{itemize}
\item[(1)] $B_\frac{1}{2}\oplus {\mathbb C}^{n+1}\oplus L^{m+1}\to  L^j,$
$$
(a,(\delta a,h),w)\mapsto  {\bf R}'(a)(\delta a,h)f^{(p)}\left(s-\frac{R}{2}\right)\cdot w
$$
where all the occurring ${\bf R}'(a)$ have order $n'=n+1$. Here $j'=j$, $m'=m+1$, and $n'=n+1$  so that $p'\leq m'-j'$ and $n'\leq k+1$.
\item[(2)] $B_\frac{1}{2}\oplus {\mathbb C}^n\oplus L^{m+1}\to  L^j$, \
$$(a,(h_1,\ldots ,\delta h_i,\ldots ,h_n),w)\mapsto  A(a,(h_1,\ldots, \delta h_i,\ldots ,h_n),w).$$
Here $j'=j$, $p'=p$, $m'=m+1$ and  $n'=n+1$ so that  $p'\leq m'-j'$ and $n'\leq k+1$.
\item[(3)] $B_\frac{1}{2}\oplus {\mathbb C}^{n+1}\oplus L^{m+1}\to  L^j,$
$$
(a,(\delta a,h),w)\mapsto  {\bf R}_1(a)(\delta a,h)\cdot f^{(p+1)}\left(s-\frac{R}{2}\right)\cdot w
$$
where  $ {\bf R}_1(a)(\delta a,h)=(DR(a)(\delta a )) {\bf R}(a)(h)$ so that its order is equal to $n'=n+1$. Moreover, 
$p'=p+1$,  $m'=m+1$,  and $j'=j$. Again,  $p'\leq m'-j'$ and $n'\leq k+1$.
\item[(4)] $B_\frac{1}{2}\oplus \C^n\oplus L^m\to  L^j,$
$$ 
(a,h,\delta w)\mapsto A(a,h,\delta w).
$$
Here $j'=j$, $m'=m$, $p'=p$,  and $n'=n$ satisfy $p'\leq m'-j'$ and $n'\leq k+1$.
\end{itemize}
This completes the proof of the lemma.
\end{proof}

Let us observe that any map  $A$ in Lemma \ref{ind-lem} has a continuous extension
to points $(0,h,w)$ by defining $A(0,h,w)=0$. Indeed, if $m=j$ so that $p=0$
we have $A(a,w)=\Gamma(a,w)$ as a map $B_\frac{1}{2}\oplus L^m\to  L^m$
and we already know that this is $\ssc^0$.  If, on the other hand,  $m-j>0$, then we combine the estimates in Lemmata \ref{lkk} and \ref{lkj} and obtain  the estimate
\begin{equation}\label{est-eq}
\abs{A(a, h, w)}_j\leq Ce^{-d_{m, j}R}\cdot R^{3k}\cdot \abs{h}^k \cdot \abs{w}_m
\end{equation}
with a constant $C$ depending on $m$, $j$, $p$, and $n$,  but not on $a$.
Recalling that $R=e^{\frac{1}{\abs{a}}}-e$,  the right-hand side converges to $0$ as $\abs{a}\to 0$ keeping $(h, w)$ bounded.

 At this point we have proved the following lemma.
\begin{lem}\label{nhj}
The map $\Gamma:(B_\frac{1}{2}\setminus\{0\})\oplus L\to  L$ is $\ssc$-smooth. Moreover,  its iterated tangent map $T^k\Gamma$ can be extended continuously by $0$ over all  points containing  $a=0$.
\end{lem}
It remains to show the approximation property at points $(a, H)\in T^k( B_\frac{1}{2}\oplus L) $ where $a=0\in B_\frac{1}{2}$. Of course, the candidate is  the $0$ map.
We have to show that,  given $(0, H)\in T^k( B_\frac{1}{2}\oplus L )$, 
$$
\frac{1}{ \norm{(\delta a,\delta H) }_1} \cdot \norm{T^k\Gamma(\delta a, H+\delta H) }_0\to 0\quad \text{as $ \norm{(\delta a,\delta H) }_1\to 0$},
$$
where $\norm{\cdot }_0$ ($\norm{\cdot}_1$)  is the norm on the level $0$ ($1$) of the iterated tangent  $T^k( B_\frac{1}{2}\oplus L)$ and is equal to the sum of the norms on each of the factors of $(T^k( B_\frac{1}{2}\oplus L))_0$ ($(T^k( B_\frac{1}{2}\oplus L))_1$).

We know from  Lemma \ref{ind-lem}  that the maps $\pi\circ T^k\Gamma$  on the factors of $T^k( B_\frac{1}{2}\oplus L)$ are of particular forms. Hence it suffices to consider the maps defined in Lemma \ref{ind-lem}.  More precisely, given a map 
$$A: B_\frac{1}{2}\oplus \C^n\oplus L^{m}\to L^j, \quad (a, h, w)\mapsto (a,h, w)\mapsto  {\bf R}(a)(h)\cdot f^{(p)}\left(s-\frac{R}{2}\right)\cdot w
$$
and  $(0, h, w)\in B_\frac{1}{2}\oplus \C^n\oplus L_{m+1}$ (recall that in order to linearize we have to raise the index $m$ by $1$), we have to show that 
\begin{equation}\label{limit-eq}
\frac{1}{\abs{\delta a}+\abs{\delta h}+\abs{\delta w}_{m+1}}\abs{A(\delta a,h+\delta h,w+\delta w)}_j \to 0
\end{equation}
as $\abs{\delta a}+\abs{\delta h}+\abs{\delta w}_{m+1}\to 0$.
If $\delta a=0$, then $A(\delta a , \cdot , \cdot )=0$ and so we may assume that $\delta a\neq 0$. In view of the estimate \eqref{est-eq} in which  $R=\varphi (\abs{\delta a})$ for the exponential gluing profile $\varphi$, the left-hand side of \eqref{limit-eq} is less or equal to 
\begin{equation*}
\begin{split}
&\frac{C}{\abs{\delta a}+\abs{\delta h}+\abs{\delta w}_{m+1}}e^{-d_{m, j}R}\cdot R^{3k}\cdot \abs{h+\delta h}^k \cdot \abs{w+\delta w}_m\\
&\phantom{===}=\frac{C\abs{\delta a}}{\abs{\delta a}+\abs{\delta h}+\abs{\delta w}_{m+1}}\cdot {\frac{e^{-d_{m, j}R}\cdot R^{3k}}{\abs{\delta a}}}\cdot \abs{h+\delta h}^k \cdot \abs{w+\delta w}_m\
\end{split}
\end{equation*}
which converges to $0$ as $\abs{\delta a}+\abs{\delta h}+\abs{\delta w}_{m+1}\to 0$.
This proves the approximation property and completes  the proof that $\Gamma$, and hence $\Gamma_1$,  is sc-smooth.

Next we consider the map $\Gamma_2$. Again we start with the prove of structural result about $\pi\circ T^k\Gamma_2$.  Before we do that,  we state the  estimate which we shall subsequently use. It follows immediately from Lemma \ref{polj}.
\begin{lem}\label{thetaes-prop}
For all multi-indices $\alpha$,  there exists a constant $C>0$ so that 
$$
\abs{D^{\alpha}\vartheta(a)}\leq C\cdot \abs{  \ln R(a)    }^{|\alpha|}
$$
if $a\neq 0$. 
\end{lem}
Similar to ${\bf R}(a)$ we introduce for $a\neq 0$ the expression ${\bf {\Theta}}(a)$ of products of derivatives of the form
$$
{\bf {\Theta}}(a) (h_1,\ldots ,h_k)=D^{k_1}\vartheta(a)(h_1,\ldots , h_{k_1} )\cdot \ldots \cdot D^{k_l} \vartheta (h_{k_{l-1}+1},\ldots  ,h_k),
$$
where $k=k_1+\ldots +k_l$ is the order of $\Theta (a)$. By Lemma  \ref{thetaes-prop}, 
$$
\abs{ {\bf {\Theta}} (a)(h_1,\ldots ,h_k)}\leq C\cdot ( \ln R )^k\cdot \abs{h}^k
$$
for sufficiently small  $a\neq 0$ with  a constant $C>0$ independent of $a$.

Here is the necessary structural statement  for the map $\Gamma_2$. 

\begin{lem}\label{bigg}
Assume that $\pi:T^kL\to  L^j$ is the projection onto a factor of $T^kL$. Then, for $a\neq 0$,  the composition $\pi\circ T^k\Gamma_2:T^kL\to L^j$ is a linear combination of maps 
$$A:(B_\frac{1}{2}\setminus\{0\})\oplus {\mathbb C}^{p+\alpha_1}\oplus {\mathbb C}^{\alpha_2}\oplus L^m\to  L^j$$
of the form 
\begin{equation*}
(a,h,e,w) \mapsto  {\bf R}(a)(h)\cdot {\bf {\Theta}}(a)(e)\cdot f^{(p)}\left(\cdot -\frac{R}{2}\right)(D^{\alpha}w)(\cdot -R,\cdot -\vartheta)]
\end{equation*}
where $p+|\alpha|\leq m-j$ and $|p|+|\alpha|\leq k$. Moreover,  the sum of the orders of ${\bf R}(a)$ and ${\bf {\Theta}}(a)$ does not exceed $k$.
\end{lem}
\begin{proof}
Clearly the statement is true for $k=0$. Assume it has been proved for
$k$. Consider the composition  $\pi\circ T^{k+1}\Gamma_2$ where $\pi$ is the projection  onto one of the first $2^k$ factors,. Then the result follows from the induction hypothesis raising the indices in the domain and  the target by $1$. If $\pi$ is a projection  onto one of the last $2^k$ factors, then the composition $\pi\circ T^{k+1}\Gamma_2$ is the linear combination of derivatives of terms guaranteed by the induction hypothesis. More precisely,
assume that we consider a map  $A$ of the above form.  If we differentiate $A$ at a point where  $a\neq 0$, then the linearization is a linear combination of the following types of maps:\\[0.7ex]

\noindent  {\bf ($1$) }\quad 
$(B_\frac{1}{2}\setminus\{0\})\oplus  {\mathbb C}^{n+\alpha_1+1}\oplus{\mathbb C}^{\alpha_2}\oplus L^{m+1}\to  L^j,$
$$
(a,(\delta a,h),e,w)\mapsto  {\bf R}'(a)(\delta a,h)\cdot {\bf {\Theta}}(a)(e)\cdot f^{(p)}\left(\cdot -\frac{R}{2}\right)(D^{\alpha}w)(\cdot -R,\cdot -\vartheta)
$$
where ${\bf R}'(a)$  is a linear combination of terms of the form ${\bf R}(a)$ of order $n+\alpha_1$. 
\mbox{}\\
\noindent  {\bf ($2$) }\quad This  map is obtained by only differentiating ${\bf R}(a)(h)$ with respect to $h$. This gives the   map  
$B_\frac{1}{2}\oplus {\mathbb C}^{n+\alpha_1}\oplus{\mathbb C}^{\alpha_2}\oplus L^{m+1}\to  L^j,$
$$(a,(h_1,\ldots ,\delta h_i,\ldots ,h_n),e, w)\mapsto  A(a,(h_1,\ldots, \delta h_i,\ldots ,h_n),e, w).$$
\mbox{}\\
\noindent  {\bf ($3$) }\quad $(B_\frac{1}{2}\setminus\{0\})\oplus  {\mathbb C}^{n+\alpha_1}\oplus{\mathbb C}^{\alpha_2+1}\oplus L^{m+1}\to  L^j,$
$$
(a,h,(\delta a, e),w)\mapsto  {\bf R}(a)(h)\cdot {\bf {\Theta}}'(a)(\delta a, e)\cdot f^{(p)}\left(\cdot -\frac{R}{2}\right)(D^{\alpha}w)(\cdot -R,\cdot -\vartheta)
$$
where ${\bf \Theta}'(a)$  is a linear combination of terms of the form ${\bf \Theta}(a)$ each of order $\alpha_1$ so that ${\bf \Theta}'(a)$ is of order $\alpha_2'=\alpha_2+1$.
\mbox{}\\

\noindent  {\bf ($4$) }\quad $B_\frac{1}{2}\oplus {\mathbb C}^{n+\alpha_1}\oplus {\mathbb C}^{\alpha_2}\oplus L^{m+1}\to  L^j,$
$$(a,h,(e_1, \ldots ,\delta E_m, \ldots ,e_{\alpha_2}), w)\mapsto  A(a,h,(e_1, \ldots ,\delta E_m, \ldots ,e_{\alpha_2}), w).$$
\mbox{}\\
\noindent  {\bf ($5$) }\quad 
$(B_\frac{1}{2}\setminus\{0\})\oplus  {\mathbb C}^{1+n+\alpha_1}\oplus{\mathbb C}^{\alpha_2}\oplus L^{m+1}\to  L^j,$
$$
(a,(\delta a,h),e,w)\mapsto  {\bf R}_1(a)(\delta a,h)\cdot {\bf {\Theta}}(a)(e)\cdot f^{(p+1)}\left(\cdot-\frac{R}{2}\right)(D^{\alpha}w)(\cdot -R,\cdot -\vartheta)
$$
where  ${\bf R}_1(a)(\delta a,h)=(DR(a)\delta a)){\bf R}_1(a)(h)$.
\mbox{}\\

\noindent  {\bf ($6$) }\quad 
If we differentiate with respect to $w$ we have to replace $w$ 
by $\delta w$ but with $\delta w\in L_m$. The gives the map
$B_\frac{1}{2}\oplus {\mathbb C}^{n+\alpha_1}\oplus{\mathbb C}^{\alpha_2}\oplus L^{m}\to  L^j,$
$$(a,h,e, \delta w)\mapsto  A(a,h,e, \delta w).$$

\noindent  {\bf ($7$) }\quad Lastly, the map obtain by differentiating  $(D^{\alpha}w)(s-R,t-\vartheta)$ with respect to $a$. This leads to the map 
which is a linear combination  of the two following maps
\begin{gather*}
(B_\frac{1}{2}\setminus\{0\})\oplus  {\mathbb C}^{n+\alpha_1+1}\oplus{\mathbb C}^{\alpha_2}\oplus L^{m+1}\to  L^j,\\
(a,(\delta a,h),e,w)\mapsto  {\bf R}_1(a)(\delta a,h)\cdot {\bf {\Theta}}(a)(e)\cdot f^{(p+1)}\left(\cdot -\frac{R}{2}\right)(D^{\alpha'}w)(\cdot -R,\cdot -\vartheta)
\end{gather*}
where $ {\bf R}_1(a)(\delta a,h)=(DR(a)\delta a)){\bf R}_1(a)(h)$ and $\alpha'=\alpha +(1, 0)$, and 
\begin{gather*}
(B_\frac{1}{2}\setminus\{0\})\oplus  {\mathbb C}^{n+\alpha_1}\oplus{\mathbb C}^{\alpha_2+1}\oplus L^{m+1}\to  L^j,\\
(a,(\delta a,h),e,w)\mapsto  {\bf R}_1(a)(\delta a,h)\cdot {\bf {\Theta}}_1(a)(e)\cdot f^{(p+1)}\left(\cdot -\frac{R}{2}\right)(D^{\alpha''}w)(\cdot -R,\cdot -\vartheta)
\end{gather*}
where  $ {\bf \Theta}_1(a)(\delta a,h)=(D\Theta (a)\delta a)) {\bf \Theta} (a)(e)$ and $\alpha''=\alpha +(0, 1)$.
\mbox{}\\

Hence this  derivative is a finite linear combination of terms of the required  form.
Using the previous result and the previously derived exponential decay estimates,  we see that every $T^k\Gamma_2$ can be extended in a $\ssc^0$-continuous way by $0$ over points containing $a=0$. Arguing as before (just after Lemma \ref{nhj}) we can also verify  the approximation property. Hence the map $\Gamma_2$ is sc-smooth. This completes the proof of Proposition \ref{qed}.
\end{proof}

\subsection{Gluing, Anti-Gluing and Splicings, Proof of  Theorem \ref{first-splice}}\label{helmut}
In this section we present the proof of  Theorem \ref{first-splice}. In order to do this  we recall the  formula for the projection  map $\pi_a:E\to  E$.
If $a=0$, this projection is the identity since $\ominus_a$ is  the zero map. So,   we assume that  $0<|a|<\frac{1}{2}$ and 
set 
$$\pi_a(\xi^+,\xi^-):=(\eta^+,\eta^-).$$   The pair $(\eta^+,\eta^-) $ is  found  by solving the following
 system of equations 
 \begin{align*}
&\oplus_a(\eta^+,\eta^-)=\oplus_a(\xi^+,\xi^-)&\\
& \ominus_a(\eta^+,\eta^-)=0.&
 \end{align*}
 Recalling that $\beta_a (s)=\beta_R(s)=\beta (s-\frac{R}{2})$ and setting $\gamma_a =\beta^2_a +(1-\beta_a)^2$,  we have derived in Section \ref{arisingx} the formula
 \begin{equation*}
 \begin{split}
 \eta^+(s,t)&=
\left(1-\frac{\beta_a}{\gamma_a}(s) \right)\cdot
\frac{1}{2}\cdot ([\xi^+]_R+[\xi^-]_{-R})\\
 &\phantom{=}+\frac{\beta_a^2}{\gamma_a}(s)\xi^+(s,t)  +\frac{\beta_a(1-\beta_a)}{\gamma_a}(s) \xi^-(s-R,t-\vartheta),
 \end{split}
 \end{equation*}
where 
  $$
 [\xi^+]_R=\int_{S^1} \xi^+\left(\frac{R}{2},t \right)\ dt\quad \text{and}\quad 
 [\xi^-]_{-R}=\int_{S^1} \xi^-\left(-\frac{R}{2},t \right)\ dt.
 $$
 A similar formula holds for $\eta^-$. In order to study the sc-smoothness we  consider the map
  $$
 (a,\xi^+,\xi^-)\mapsto  \eta^+,
 $$
 the sc-smoothness of the map $ (a,\xi^+,\xi^-)\mapsto  \eta^-$ is verified the same  way.  If we write $\xi^\pm =c+r^\pm$,  where $c$ is the common asymptotic constant, then  the formula for $\eta^+$ takes the form
 \begin{equation}\label{pof}
 \begin{split}
 \eta^+(s,t)
 &=c+\frac{1}{2}\left(1-\frac{\beta_a}{\gamma_a}(s)\right)
 \cdot( [r^+]_R+[r^-]_{-R})\\
 &
 \phantom{=}+\frac{\beta_a^2}{\gamma_a}(s)\cdot r^+(s,t)+
 \frac{\beta_a(1-\beta_a)}{\gamma_a}(s)\cdot r^-(s-R,t-\vartheta)\\
\end{split}
\end{equation}

We shall study  the following five mappings:
\begin{itemize}
\item[{\bf M1.}]  The map 
$$H^{3,\delta_0}_c(\R^+\times S^1,\R^{N})\to  \R^{N}, \quad \xi^+\mapsto c$$
 which associates with  $\xi^+$  its  asymptotic constant $c$.
 \item[{\bf M2.}]  The map 
 $$B_\frac{1}{2}\times H^{3,\delta_0}(\R^+\times S^1,\R^{N})\to  \R^{N},\quad (a,r^+)\mapsto  [r^+]_R.$$ 
\item[{\bf M3.}]   The map 
$$B_\frac{1}{2}\times H^{3,\delta_0}(\R^+\times S^1,\R^{N})\to  H^{3,\delta_0}(\R^+\times S^1,\R^{N}),  \quad 
(a,r^+)\mapsto  \frac{\beta_a}{\gamma_a}(\cdot )[r^+]_{R}.$$
\item[{\bf M4.}]  The map $B_\frac{1}{2}\oplus H^{3,\delta_0}(\R^+\times S^1,\R^{N})\to  H^{3,\delta_0}(\R^+\times S^1,\R^{N}),$
$$(a,r^+)\mapsto \frac{\beta_a^2}{\gamma_a}\cdot r^+.$$
 \item[{\bf M5.}]  The map $B_\frac{1}{2}\oplus H^{3,\delta_0}(\R^-\times S^1,\R^{N})\to  H^{3,\delta_0}(\R^+\times S^1,\R^{N})$,
$$ (a,r^-)\mapsto   \frac{\beta_a(1-\beta_a)}{\gamma_a} r^-(\cdot -R,\cdot -\vartheta).$$
\end{itemize}
 In view of the formula for the projection map $\pi_a$  the  sc-smoothness of the map $(a,(\xi^+,\xi^-))\mapsto  \pi_a(\eta^+,\eta^-)$  in Theorem \ref{first-splice},
 is a consequence of the following proposition.
 \begin{prop}\label{klopx}
 The maps {\bf M1}-{\bf M5} listed above (and suitably defined at the parameter value $a=0$) are sc-smooth in a neighborhood of $a=0$.
 \end{prop}
The proof  of the proposition follows  from a sequence of lemmata.

\begin{lem}\label{lem2-18}
The map $H^{3,\delta_0}_c(\R^+\times S^1,\R^{N})\to  \R^{N}, \quad \xi^+\mapsto c$, 
 which associates with  $\xi^+$  its  asymptotic constant $c$ is $\ssc$-smooth.
 \end{lem}
\begin{proof}
The map $\xi^+\mapsto  c$ is an sc-projection and therefore $\ssc$-smooth.
\end{proof}

\begin{lem} \label{prop-xx}
The map $\Phi:B_\frac{1}{2}\oplus H^{3, \delta_0}(\R^+ \times S^1, \R^N)\to  \R^{N}$ , defined by $\Phi (0, h)=0$  for $a=0$ and 
$$\Phi (a,h)=
[h]_R=\int_{S^1} h\left(\frac{R}{2},t\right) \ dt 
$$
for $a\neq 0$,  is $\ssc$-smooth.
\end{lem}
\begin{proof} We abbreviate in the proof $F=H^{3, \delta_0}(\R^+ \times S^1, \R^N)$.
Using the Sobolev embedding theorem for bounded domains into continuously differentiable functions we see that the map
$$
(0,\infty)\times F_m\to  C^0(S^1,\R^{N}), \quad (R,h)\mapsto   h\left(\frac{R}{2},\cdot \right)
$$
is of class $C^{m+1}$ for every $m\geq 0$. In view of Corollary \ref{ABC-y} this implies that the map
$$
\wh{\Phi}:(0,\infty)\times F\to  \R^{N},  \quad (R,h)\mapsto [h]_R
$$
is $\ssc$-smooth. Since the map $a\mapsto  R(a):=\varphi(|a|)$ is obviously smooth  if  $a\neq 0$, we conclude, using the chain rule for sc-smooth maps,  that the  map
$$
\Phi:(B_\frac{1}{2}\setminus\{0\})\oplus F \to  \R^{N}, \quad (a,h)\mapsto 
[h]_R
$$
is $\ssc$-smooth and we claim that the map $\Phi$ is $\ssc^0$ at every  point $(0, h)\in B_{\frac{1}{2}}\oplus F$. Indeed,   assume that $(a_k, h_k)\in (B_{\frac{1}{2}}\setminus \{0\})\oplus F_m$ is a sequence converging to $(0, h)$.  Abbreviating $\Sigma_k=(\frac{R_k}{2}-1, \frac{R_k}{2}+1)\times S^1$ where $R_k=\varphi (\abs{a_k})$, we show that $\abs{\Phi (a_k, h_k)}=\abs{[h_k]_{R_k}}\to 0$.  By the Sobolev embedding theorem on bounded domains and using the bound  $\abs{h_k}_m\leq C'$, we estimate
$$\abs{e^{\delta_m \cdot }\cdot h_k}_{C^0(\Sigma_k)}\leq C\abs{e^{\delta_m \cdot }\cdot h_k}_{H^{m+3}(\Sigma_k)}\leq C''.$$
This  implies 
\begin{equation}\label{sob-eq}
\abs{[h_k]_{R_k}}\leq C''\cdot e^{-\delta_m R_k/2}
\end{equation}
and the claim follows.

At this point we know that the map 
$$\Phi:B_\frac{1}{2}\oplus F\to  \R^{N}$$
is $\ssc^0$ and,  when restricted to $(B_\frac{1}{2}\setminus \{0\})\oplus F$ it is $\ssc^{\infty}$. 
We shall denote points in $T^k(B_\frac{1}{2}\oplus F)$ by $(a,H)$ where $a\in B_\frac{1}{2}$. 
We shall prove inductively the following statements:

\begin{induction}
The map $\Phi:B_\frac{1}{2}\oplus F\to  \R^{N}$ is of class $\ssc^k$ and  $T^k\Phi(0,H)=0$ for every $(0,H)\in T^k(B_\frac{1}{2}\oplus F)$.  Moreover, if $\pi:T^k(\R^{N})\to  \R^{N}$
is  the projection onto a  factor of $T^k\R^N$,  then  the composition $\pi\circ T^k\Phi$ is a  linear combination of maps  $\Gamma$ of the  of the following types,
$$
\Gamma:B_\frac{1}{2}\oplus {\mathbb C}^n\oplus F_m\to  \R^{N},\quad (a,b,v)\to  {\bf R}(a)(b_1,\ldots ,b_n)\cdot [\partial^j_sv]_R
$$
for $a\neq 0$ and $\Gamma(0, b, v)=0$. Here $j\leq m, n\leq k$,  and ${\bf R}(a)$ is the product of derivatives of the function $R(a)=e^{\frac{1}{\abs{a} }}-e$ of  the form
$$
{\bf R}(a)(b_1,\ldots ,b_n)=D^{n_1}R(a)(b_1,\ldots ,b_{n_1})\cdot\ldots \cdot D^{n_l}R(a)(b_{n_1+\ldots +n_{l-1}+1},\ldots ,b_n),
$$
where the integer $n=n_1+\ldots +n_l$ is called the order of ${\bf R}(a)$. We set  ${\bf R}(a)=1$ if $n=0$.
\end{induction}

We begin by verifying  that $({\bf S_0})$ holds. In this case,  the projection $\pi:T^0\R^N=\R^N\to \R^N$ is the identity map, the indices $j, k, m$ and $n$ are equal to $0$, and the composition $\pi\circ T^0\Phi$ is just the map $\Phi:B_\frac{1}{2}\oplus F\to  \R^{N}$ given by 
$$
(a,v)\to  [v]_R.
$$
The map has the required form with ${\bf R}(a)=1$ of order $0$.  With $\Phi (0, v)=0$, we already know that $\Phi$ is $\ssc^0$. So, the assertion $({\bf S_0})$  holds.

Assuming  that $({\bf S_k})$ holds, we show that $({\bf S_{k+1}})$ also holds.  By induction hypothesis, 
the map $\Phi$ is $\ssc^k$, so that  $T^k\Phi$ is $\ssc^0$.  Moreover, $T^k\Phi(0,H)=0$,  $T^{k+1}\Phi$ is $\ssc$-smooth at points $(a,H)$ with $a\neq 0$, and $\pi\circ T^k\Phi$ can be written as a linear combination of maps  of a certain form.

Setting $DT^k\Phi (0, H)=0$, we prove  the approximation property of $T^k\Phi$ at the points $(0,H)\in (T^k(B_\frac{1}{2}\oplus F)^1$. That is, recalling that  $T^k\Phi (0, H)=0$, we show that 
\begin{equation}\label{maps-sp}
\frac{1}{\norm{(\delta a,\delta H)}_1}\abs{T^k\Phi(\delta a,H+\delta H)}_0\to 0\quad 
\text{as $\norm{(\delta a,\delta H)}_1\to  0$}.
\end{equation}
where   the subscripts $0$ and $1$ refer to the levels of the iterated tangents.  By the inductive assumption $({\bf S_{k}})$, we know that  the compositions $\pi \circ T^k\Phi$ with projections $\pi$ on different factors of $T^k\R^N$ are linear combinations of maps $A$ described in $({\bf S_{k}})$.  Hence to prove \eqref{maps-sp} amounts to showing that  at the point $(0, h, v)\in B_\frac{1}{2}\oplus {\mathbb C}^j\oplus F^{m+1}$ we have 
\begin{equation}\label{qer}
\frac{1}{\abs{\delta a}+\abs{\delta h}+\abs{\delta v}_{m+1}}\abs{A(\delta a,b+\delta b, v+\delta v)}\to 0 
\end{equation}
as $\abs{\delta a}+\abs{\delta b}+\abs{\delta v}_{m+1}\to  0$ for the maps 
$$
A:B_\frac{1}{2}\oplus {\mathbb C}^n\oplus F^m\to  \R^{N},\quad (a,b,v)\mapsto   {\bf R}(a)(b_1,\ldots ,b_n)\cdot [\partial^j_sv]_R
$$ 
defined in $({\bf S_{k}})$.

Using as above the Sobolev estimate on the bounded domain $\Sigma_R=\left(\frac{R}{2}-1, \frac{R}{2}+1\right)\times S^1$, we obtain
$$\abs{ e^{\delta_{m+1}\cdot } \partial_s^j (v+\delta v)}_{C^0(\Sigma_R)} \leq C
\abs{e^{\delta_{m+1}\cdot} \partial_s^j (v+\delta v) }_{H^{m+3}(\Sigma_R)},
$$
where $j\leq m$, and estimate
$$[\partial_s^j (v+\delta v)]_R\leq Ce^{-\delta_{m+1}\frac{R}{2}}\abs{v+\delta v}_{m+1}.$$
Therefore, in view of the estimate of $R(a)$ in Lemma \ref{lkj},
\begin{equation*}
\begin{split}
\abs{A(\delta a,b+\delta b, v+\delta v)}
\leq Ce^{-\delta_{m+1}\frac{R}{2}}\abs{R}^{3n}\cdot \abs{b+\delta b}^n\cdot \abs{v+\delta v}_{m+1}
\end{split}
\end{equation*}
where $R=\varphi (\abs{\delta a} )$ and  $\delta a\neq 0$. 
Consequently,
\begin{equation}\label{ger-eq}
\frac{\abs{A(\delta a,b+\delta b, v+\delta v)}}{\abs{\delta a}+\abs{\delta h}+\abs{\delta v}_{m+1} }
 \leq C\frac{ e^{-\delta_{m+1}\frac{R}{2}} \abs{R}^{3n}}{\abs{\delta a}} \cdot \abs{b+\delta b}^n\cdot \abs{v+\delta v}_{m+1}.
\end{equation} 

If $R=\varphi (\abs{\delta a})$ is large (or $\abs{\delta a}$ is small), then 
$2R\geq 2\ln R\geq \frac{1}{\abs{\delta a}}$ so 
that the left hand-side of \eqref{ger-eq} is smaller than 
$$
Ce^{-\delta_{m+1}\frac{R}{2}}\abs{R}^{4n}\cdot \abs{b+\delta b}^n\cdot \abs{v+\delta v}_{m+1}
$$ 
which converges to $0$ as $(\delta a, \delta b, \delta v)\to (0, 0, 0)$ in $\C\oplus \C^n\oplus F^{m+1}$. 
Summing up our discussion so far,   we have proved the approximation property  for the map $T^k\Phi$  and  
$$
D(T^k\Phi)(0,H)=0
$$
for all $(0,H)\in {(T^k(B_\frac{1}{2}\oplus F))}_1$.  To complete the proof, it remains to show that 
$T^{k+1}\Phi$ is of class $\ssc^0$ (which will imply that $\Phi$ is of class $\ssc^k$)  and to show that the compositions 
$\pi\circ T^{k+1}\Phi$ have the required form. 

We consider $\pi\circ T^{k+1}\Phi (a,H)$ where $a\neq 0$.
If $\pi$ is the projection onto one of the first $2^k$ factors, then $\pi\circ T^{k+1}\Phi $ has the form of the map $A$ in  $({\bf S_{k}})$.
The only thing is that the indices are raised by $1$. Denoting  the new indices by $j'$, $m'$, and $n'$, we have $j'=j$, $m'=m+1$,  and $n'=n$ which obviously satisfy 
$j'\leq m'\leq k+1$,  and $n'\leq k+1$. If $\pi$ is the projection onto one of the remaining $2^k$ factors, then 
$\pi\circ T^{k+1}\Phi $ is equal to  the sum  of derivatives of maps in the induction hypothesis $({\bf S_{k}})$. 
So,   if the map $A:B_\frac{1}{2}\oplus {\mathbb C}^n\oplus F^m\to  \R^{N},$ given by 
$$
(a,b,v)\mapsto  {\bf R}(a)(b_1,\ldots ,b_j)\cdot [\partial^j_sv]_R
$$
for $a\neq 0$ and $A(0, b, v)=0$,  is one of the maps from  $({\bf S_{k}})$ and  if we take the sc-derivative of $A$ (which we already know exists at every point), we obtain 
 a linear combination of maps  of the following types:
\begin{itemize}
\item[(1)]  $B_\frac{1}{2}\oplus {\mathbb C}^n\oplus F^{m+1}\to  \R^{N}$ defined by 
$$
(a,b_1,\ldots \delta b_i,\ldots, b_n,b)\mapsto {\bf R}(a)(b_1,\ldots ,\delta b_i\ldots, b_j)\cdot [\partial^j_sv]_R$$
for every $1\leq i\leq n$.
\item[(2)]  $B_\frac{1}{2}\oplus {\mathbb C}^{n+1}\oplus F^{m+1}\to  \R^{N}$ defined by 
$$
(a,(\delta a, b),v)\mapsto {\bf R}'(a)(\delta a, b)\cdot [\partial^j_s v]_R
$$
and  obtained by differentiation of  ${\bf R}(a)$ with respect to $a$. 
\item[(3)]  $B_\frac{1}{2}\oplus {\mathbb C}^n\oplus F^m\to  \R^{N}$ defined by 
$$
(a,b,\delta v)\to  {\bf R}(a)(b)\cdot [\partial^j_s\delta v]_R
$$
and obtained by differentiating with respect to $v$.
\item[(4)]$ B_\frac{1}{2}\oplus {\mathbb C}^{n+1}\oplus F^{m+1}\to  \R^{N}$ defined by 
$$
(a,(\delta a,b),v)\to  {\bf R}_1(a)(\delta a, b)\cdot [\partial^{j+1}_sv]_R
$$
which is obtained by differentiating $R (a)$ in the term $ [\partial^{j}_sv]_R$ with respect to $a$. Hence 
${\bf R}_1(a)(\delta a, b)=(DR (a)\delta a)\cdot {\bf R}(a)(b)$.
\end{itemize}
Note that in all of the above cases the new indices $j', m'$ and $n'$  stay the same or are raised by $1$ so that we have 
$j'\leq m'\leq k+1$ and $n'\leq k+1$. We have verified that the statement $({\bf S_{k+1}})$ holds true.  This completes the proof of Lemma \ref{prop-xx}.
\end{proof}

\begin{lem}
The map
$\Psi:B_\frac{1}{2}\oplus (H^{3, \delta_0}(\R^+ \times S^1, \R^N)\times H^{3, \delta_0}(\R^- \times S^1, \R^N))\to  H_c^{3, \delta_0}(\R^+ \times S^1, \R^N)$,  defined by $\Phi (0, r^+, r^-)=0$ at $a=0$ and 
$$
\Psi(a,r^+,r^-)= \left(1-\frac{\beta_a}{\gamma_a}\right) \cdot \bigl( [r^+]_{R}+[r^-]_{-R} \bigr), 
$$
for $a\neq 0$, is $\ssc$-smooth.
\end{lem}

\begin{proof} We shall abbreviate $G=H^{3, \delta_0}(\R^+ \times S^1, \R^N)\times H^{3, \delta_0}(\R^- \times S^1, \R^N)$ and $F= H^{3, \delta_0}(\R^+ \times S^1, \R^N)$. 
We already know that the maps
$$
B_\frac{1}{2}\oplus G\to  \R^{N}, \quad 
(a,r^+,r^-)\mapsto  [r^+]_{R},\  [r^-]_{-R}
$$
are $\ssc$-smooth. It suffices  to consider
only the map
$$
\Phi:B_\frac{1}{2}\oplus F\to  F, \quad (a,r)\mapsto  \frac{\beta_a}{\gamma_a}\cdot [r]_R
$$
if $a\neq 0$ and $\Psi(0, r)=0$ at $a=0$.  
The similar map for  $(a,r^-)$ can be dealt with  the same way.
Clearly the map $\Phi$ is $\ssc$-smooth on the set $(B_\frac{1}{2}\setminus\{0\})\oplus F$ and we shall prove the sc-smoothness at the  points $(0, r)\in B_\frac{1}{2}\oplus F$. We set 
$\sigma :=\frac{\beta }{\gamma }$ and $\sigma_a=\sigma_R=\sigma (\cdot -\frac{R}{2})$ where $R=\varphi (\abs{a})$ for the exponential gluing profile $\varphi$. We shall prove the following statements $({\bf S_{k}})$ by induction:

\begin{induction}
The map $\Phi$ is of class $\ssc^k$ and $T^k\Phi (0, H)=0$.  
Moreover, if $\pi:T^k(B_\frac{1}{2}\oplus F)\to  \R^N$ is a projection onto a  factor of $T^kF$,  then the composition $\pi\circ T^k\Phi$ is the linear combination of maps  of the following type,
\begin{gather*}
A:B_\frac{1}{2}\oplus \C^n\oplus F^m\to  F^j\\
(a,h,v)\to  {\bf R}(a)(h)\cdot \sigma_a^{(p)}\cdot [\partial_s^q v]_R 
\end{gather*}
for $a\neq 0$ and $A(0, h,v)=0$. In addition, the indices satisfy $p+q=n$, $m\leq k$ and $l\leq m-j$.
\end{induction}

We start with  $({\bf S_0})$.  In this case there is only one projection $\pi:T^0F=F\to F$, namely,  $\pi =\id$. Clearly $\pi\circ\Phi=\Phi$ has the required form with ${\bf R}(a)=1$ of order $0$ and all indices $j,l, p$, and $q$ equal to $0$.
Hence we only need to show that the map $\Phi$ has $\ssc^0$-property. This is  clearly  true  at  points 
$(a, v)\in B_{ \frac{1}{2}}\oplus F^m$ where  $a\neq 0$. We carry out the proof of the $\ssc^0$-property for  the map $\Phi$ at $(0, v)$.  We take a sequence $(a_k, v_k)$  converging to $(0, v)$ in $B_{\frac{1}{2}}\oplus F_m$ and we claim that $\Phi (a_k, v_k)\to 0$ in $F_m$. Since $\sigma_R$ vanishes on $[\frac{R}{2}+1,\infty )$, we can estimate  
\begin{equation*}
\begin{split}
\abs{\Phi (a_k, v_k)}_m=\abs{\sigma_{R_k} \cdot [v_k]_{R_k}}_m^2&=\sum_{\abs{\alpha}\leq m}\abs{ [v_k]_{R_k}}^2\int_{\R^+\times S^1}\abs{D^{\alpha}\sigma_{R_k} (s)}^2e^{2\delta_m s}\ ds dt\\
&\leq \sum_{\abs{\alpha}\leq m}C_{\alpha}\abs{ [v_k]_{R_k}}^2e^{2\delta_m (\frac{R_k}{2}+1)}
\end{split}
\end{equation*}
with constants $C_{\alpha}$ depending only on $\sigma$, the multi-index $\alpha$, and $m$. So,  to  prove them  claim we have to show that $[v_k]_{R_k}\to 0$ in $\R^N$. 
Abbreviate $\Sigma_k=[\frac{R_k}{2}-1, \frac{R_k}{2}+1]\times S^1$.  By  the Sobolev embedding theorem on bounded domains, 
$$\abs{e^{\delta_m \cdot }v}_{C^0(\Sigma_k)}\leq C\abs{e^{\delta_m \cdot }v}_{H^m(\Sigma_k)}=:\varepsilon_k$$
with the constant $C$ independent of $v$ and $k$. This  shows that 
\begin{equation}\label{sob-eq1}
\abs{[v]_{R_k}}\leq \varepsilon_k\cdot e^{-\delta_m R_k}.
\end{equation}
Also note that since $v$ belongs to $E_m$, the sequence $\varepsilon_k$ converges to $0$.
Similarly, we have 
$$
\abs{e^{\delta_m \cdot } (v_k-v)}_{C^0( \Sigma_k) }\leq C\abs{e^{\delta_m \cdot }(v_k-v)}_{H^m(\Sigma_k)}\leq C\abs{v_k-v}_m=:\varepsilon_k'
$$
which implies that 
$$\abs{[v_k]_{R_k}-[v]_{R_k}}= \abs{[v_k-v]_{R_k}}\leq \varepsilon_k'\cdot e^{-\delta_m R_k}.$$
By assumption, $\abs{v-v_k}_m=\varepsilon_k\to 0$.  Consequently, 

$$
\abs{[v_k]_{R_k} }e^{\delta_mR_k}\leq \abs{ [v_k]_{R_k}-[v]_{R_k} }\cdot e^{\delta_m R_k}+ \abs{[v]_{R_k} }\cdot e^{\delta_m R_k}\leq 
 \varepsilon_k'+ \varepsilon_k\to 0
$$
which proves our claim. At this point we have proved the assertion $({\bf S_0})$.
 
Now we assume that $({\bf S_k})$ has been established and  prove  that $({\bf S_{k+1}})$ holds.  
By induction hypothesis,  the map $\Phi$ is of class $\ssc^k$, so that  $T^k\Phi$ is $\ssc^0$, and  $T^k\Phi(0,H)=0$.   Moreover, 
$\pi\circ T^k\Phi$ can be written as a linear combination of maps  of a certain form. We also know that $T^{k+1}\Phi$ is $\ssc$-smooth at points $(a,H)$ with $a\neq 0$.
 
We begin by verifying the approximation property  at points $(0,H)$. As before it suffices to do this for  the maps $A$ described in $({\bf S_k})$ as follows,
\begin{gather*}
A:B_\frac{1}{2}\oplus \C^n\oplus F^m\to  \R^N\\
(a,h,v)\to  {\bf R}(a)(h)\cdot \sigma_a^{(p)}\cdot [\partial_s^q v]_R 
\end{gather*}
for $a\neq 0$ and $A(0, h,v)=0$. More precisely, we show that  if $(0, h, v)\in B_{\frac{1}{2}}\oplus \C^n\oplus E^{m+1}$, then  
$$
\frac{1}{\abs{\delta a}+\abs{\delta h}+\abs{\delta v}_{m+1}}\abs{A(\delta a, h+\delta h, v+\delta v)}_j\to 0$$
as $\abs{\delta a}+\abs{\delta h}+\abs{\delta v}_{m+1}\to 0$ which will prove  that  $A$ has the approximation property at $(0, h, v)$ with respect to the linearized map $DA(0, h, v)=0$. 

Proceeding as in the proof of Lemma \ref{prop-xx}, one obtains the estimate 
\begin{equation*}
\begin{split}
\frac{\abs{A(\delta a, h+\delta h, v+\delta v)} }{ \abs{\delta a}+\abs{\delta h}+\abs{\delta v}_{m+1} }\leq C\frac{e^{-\delta_{m+1}\frac{R}{2} }R^{3n}}{\abs{\delta a}}\abs{h+\delta h}^n\abs{v+\delta v}_{m+1}
\end{split}
\end{equation*}
which converges to $0$ as $\abs{\delta a}+\abs{\delta h}+\abs{\delta v}_{m+1}$ converges to $0$.

Finally,  we need to show that $\pi\circ T^{k+1}\Phi$ is a linear combination of the maps of the required form and which have the required continuity properties at points with vanishing $a$. The terms making up $\pi\circ T^{k+1}\Phi$ are the terms guaranteed by
$({\bf S_k}) $ provided $\pi$ is the  projection  onto one of the first $2^k$ factors. In this case the indices $m$ and $j$ are raised by one. If $\pi$ is the projection
onto one of the last $2^k$ factors, then $\pi\circ T^{k+1}\Phi$  is a linear combination of derivatives of maps guaranteed by $({\bf S_k})$. This leads to a case by case study quite similarly to that of the previous lemma and is left to the reader.

\end{proof}

\begin{lem} \label{lem2.21-nnew}
The map $\Phi:B_\frac{1}{2}\oplus H^{3,\delta_0}(\R^+\times S^1,\R^{N})\to  H^{3,\delta_0}(\R^+\times S^1,\R^{N})$, defined by $\Phi (0, r)=r$ if $a=0$, and by 
$$\Phi(a,r)= \frac{\beta_a^2}{\gamma_a}\cdot r$$
if $a\neq 0$, is sc-smooth.
\end{lem}
\begin{proof} We  only have to prove the sc-smoothness at points  $(0,r)\in B_\frac{1}{2}\oplus H^{3,\delta_0}(\R^+\times S^1)$.
Hence we may assume that $a$ is small. We choose a  smooth function
$\chi_1:\R^+\to [0,1]$ satisfying $\chi_1(s)=1$ for $s\in [0,1]$ and $\chi_1(s)=0$ for $s\geq 2$, and  set 
$\chi_2=1-\chi_1$. Then the map
$$
(a,r)\mapsto   \frac{\beta_a^2}{\gamma_a}\cdot \chi_1\cdot r
$$
is obviously sc-smooth since for $|a|$ small it is equal to  the map 
$$
(a,r)\to  \chi_1\cdot r
$$
which is independent of $a$. It remains to  deal with the map
$$
(a,r)\mapsto  \frac{\beta_a^2}{\gamma_a}\cdot \chi_2\cdot r.
$$
This map can be factored as follows. First,  we apply the map
$$
B_\frac{1}{2}\oplus H^{3,\delta_0}(\R^+\times S^1,\R^{N})\to   B_\frac{1}{2}\oplus H^{3,\delta_0}(\R\times S^1,\R^{N}), 
\quad (a,r)\mapsto  (a,\chi_2r)
$$
which obviously is an sc-operator and  hence sc-smooth. Then we compose this map with the map
$$
B_\frac{1}{2}\oplus  H^{3,\delta_0}(\R\times S^1,\R^{N})\to  H^{3,\delta_0}(\R\times S^1,\R^{N}), \quad (a,u)\to  \frac{\beta_a^2}{\gamma_a}\cdot  u
$$
which is sc-smooth by Proposition \ref{qed}. Finally,  we  take the restriction map
$$
H^{3,\delta_0}(\R\times S^1,\R^{N})\to 
H^{3,\delta_0}(\R^+\times S^1,\R^{N}).
$$
 which as an sc-operator is also  sc-smooth. Since the composition of sc-smooth maps is an sc-smooth map, the proof is complete.
\end{proof}

\begin{lem}\label{lem2-22}
 The map $\Phi:(B_\frac{1}{2}\oplus H^{3,\delta_0}(\R^-\times S^1,\R^{N})\to  H^{3,\delta_0}(\R^+\times S^1,\R^{N}),$ defined by $\Phi (0, r)=0$ if $a=0$,  and by 
 $$\quad (a,r)\mapsto   \frac{\beta_a(1-\beta_a)}{\gamma_a}\cdot r(\cdot -R,\cdot -\vartheta)$$
if $a\neq 0$,  is sc-smooth.
 \end{lem}
\begin{proof}
Again we only have to study  the map for $a$  small. We choose  a smooth map
 $ \chi:\R^-\to  [0,1]$
 satisfying $\chi(s)=1$ for $s\leq -1$ and $\chi(s)=0$ for $s\in [-\frac{1}{2},0]$.  If  $\abs{a}$ is small, the map $\Phi$ is the composition of the following three maps. The first map is defined by
 $$
 B_\frac{1}{2}\oplus H^{3,\delta_0}(\R^-\times S^1,\R^{N})\to  B_\frac{1}{2}\oplus H^{3,\delta_0}(\R\times S^1,\R^{N}),  \quad (a,u)\mapsto (a, \chi\cdot  u).
 $$
It  is an sc-operator and hence sc-smooth.
The second map is  defined by 
\begin{gather*}
 B_\frac{1}{2}\oplus H^{3,\delta_0}(\R\times S^1,\R^{N})\to  H^{3,\delta_0}(\R\times S^1,\R^{N})\\
 (a,u)\mapsto  \frac{\beta(1-\beta_a)}{\gamma_a}\cdot u(\cdot-R,\cdot-\vartheta).
\end{gather*}
 By  Proposition \ref{qed},  this map is sc-smooth. The last map is the restriction map $ H^{3,\delta_0}(\R\times S^1,\R^{N})\to  H^{3,\delta_0}(\R^+\times S^1,\R^{N})$
 which is an sc-operator. This completes the proof of Lemma \ref{lem2-22}.
 \end{proof}
 
 In view of the above lemmata \ref{lem2-18}--\ref{lem2-22}, the proof of Proposition \ref{klopx} is finished. Hence the proof of Theorem \ref{first-splice} is complete.

\subsection{Estimates for the  Total Gluing Map}\label{esttotalgluing}

In section \ref{arisingx}, we have introduced the space $E$ consisting  of pairs $(\eta^+,\eta^-)$ of 
maps $\eta^\pm:\R^\pm \times S^1\to \R^N$ the Sobolev class  $(3,\delta_0)$ having  common asymptotic limits. Taking  a strictly increasing sequence $(\delta_m)_{m\in \N_0}$ starting with $\delta_0>0$,  we equip the Hilbert space $E$ with the sc--structure $(E_m)_{m\in \N_0}$  where $E_m$  consists of those  pairs  $(\eta^+, \eta^-)$ in $E$ of Sobolev  class $(3+m, \delta_m)$.  
We shall later impose boundary conditions, but this is not important for the moment.

If  $(\eta^+,\eta^-)$ is a pair in  $E$, then $\eta^\pm=c+r^\pm$ where $c$ is the  common asymptotic limit and $r^\pm\in H^{3, \delta_0}(\R^\pm\times S^1, \R^N)$. The $E_m$-norm of the pair  $(\eta^+,\eta^-)$ is defined as
$$
\abs{(\eta^+,\eta^-)}_{E_m}^2 =\abs{c}^2+\abs{r^+}^2_{H^{3+m, \delta_m}}+\abs{r^-}^2_{H^{3+m, \delta_m}}.
$$

For every gluing parameter $a\in B_{\frac{1}{2}}$, we introduce the space $G^a$ as follows.  If $a=0$, we set 
$$G^0=E\oplus \{0\},$$
and  if  $a\neq 0$, we  define  
$$
G^a =Q^a\oplus P^a= H^3(Z_a)\oplus H^{3,\delta_0}_c(C_a).
$$
The sc-structure of $G^a$ is  given by the sequence $H^{3+m}(Z_a)\oplus H_c^{3+m, \delta_m}(C_a)$ for all $m\in \N_0$. 

The total gluing map $\boxdot_a=(\oplus_a, \ominus_a):E\to G^a$ is an sc-linear isomorphism for every $a\in B_{\frac{1}{2}}$ in view of Theorem \ref{propn-1.27}. 

For a pair $(k,\delta)$ in which $k$ is a non-negative integer and $\delta$ and a map 
$q:Z_a\rightarrow {\mathbb R}^N$ defined on the finite cylinder $Z_a$, we introduce  the norm  $\nr q\nr_{k,\delta}$ by 
$$
\nr q\nr_{k,\delta}^2:=\sum_{\abs{\alpha}\leq k} \int_{[0,R]\times S^1} \abs{D^{\alpha}q(s,t)}^2\cdot e^{2\delta |s-\frac{R}{2}|} \ ds dt
$$
where $q(s,t):=q([s,t])$.

We recall that the average of  the map $q:Z_a\rightarrow {\mathbb R}^N$, denoted  by $[q]_a$ or $[q]_R$, is defined as   the  integral over the middle loop, 
$$
[q]_a=\int_{S^1} q\left(\frac{R}{2},t\right)\  dt.
$$

If $p:C_a\to \R^N$ is a map with vanishing asymptotic constants we set,  with $p(s,t):=p([s,t])$, 
$$
\nr p\nr_{m,\delta}^2= \sum_{|\alpha|\leq m} \int_{\R\times S^1} \abs{(D^{\alpha}p)(s,t)}^2 e^{2\delta_m |s-\frac{R}{2}|} \ dsdt
$$ 

Observe that the center loop is located at $s=\frac{R}{2}$, which explains the occurrence of the $\frac{R}{2}$ in the exponential weight.

Now we define a norm for the pairs $(q,p)\in G^a$ as follows. With a map $p\in H^{3, \delta_0}_c(C_a)$ which has the antipodal constants  $p_{\infty}=\lim_{s\to \infty}p(s, t)$ and $p_{-\infty}=-p_{\infty}$,  we associate the  map $\wh{p}:C_a\to \R^N$  defined by 
$$
\wh{p}([s,t]) = p([s,t]) - (1-2\cdot\beta_a(s))p_\infty. 
$$
Then $\lim_{s\to \pm \infty}\wh{p}(s, t)=0$ and $\wh{p}$ belongs to $H^{3,\delta_0}(C_a,\R^N)$.  Now  we introduce a norm on the level $m$ of $G^a=H^{3}(Z_a)\oplus H^{3,\delta_0}_c(C_a)$ by 
$$
| (q,p)|^2_{G^a_m} =\abs{[q]_a -p_\infty}^2 +e^{\delta_m R}\cdot \left(\nr q-[q]_a+p_\infty\nr^2_{3+m,-\delta_m}+\nr \wh{p}\nr^2_{m+3,\delta_m} \right).
$$
To get a better understanding of the norm $\abs{(q,p)}^2_{G^a_m}$ we take the  unique pair $(\eta^+, \eta^-)\in E$ satisfying $q=\oplus_a(\eta^+, \eta^-)$ and $p=\ominus_a(\eta^+, \eta^-)$ by using Theorem \ref{propn-1.27}. We write 
$\eta^\pm=c+r^\pm$ where $c$ is the common asymptotic limit and $r^\pm \in H^{3,\delta_0}(\R^\pm\times S^1, \R^N)$.
Then 
\begin{align*}
q&=\oplus_a(\eta^+, \eta^-)=c+\wh{\oplus}_a(r^+, r^-)\\
p&=\ominus_a(\eta^+, \eta^-)=\wh{\ominus}_a(r^+, r^-)+(1-2\beta_a)\av (r^+, r^-).
\end{align*}
The mean value $[q]_a$ of $q$ is equal to $c+\av (r^+, r^-)$ and the  positive asymptotic constant $p_{\infty}$ of  $p$  is equal to $p_{\infty}=\av (r^+, r^-)$. Hence $[q]_a -p_{\infty}=c$ and  $q-[q]_a+p_{\infty}=\wh{\oplus}_a(r^+, r^-)$, and $\wh{p}=p-(1-2\beta_a)\cdot p_{\infty}=\wh{\ominus}_a(r^+, r^-)$.
Consequently, the $G^a_m$-norm of $(q, p)$  becomes,
\begin{equation*}
\begin{split}
| (q,p)|^2_{G^a_m}&=\abs{[q]_a -p_\infty}^2 +e^{\delta_m R}\cdot \left( \nr q-[q]_a+p_\infty\nr^2_{3+m,-\delta_m}+\nr \wh{p}\nr^2_{m+3,\delta_m} \right) \\
&=\abs{c}^2+e^{\delta_m R}
\left( 
\nr\wh{\oplus}_a(r^+, r^-)\nr^2_{3+m,-\delta_m}+\nr\wh{\ominus}_a(r^+, r^-)\nr^2_{m+3,\delta_m} \right).
\end{split}
\end{equation*}

\begin{thm}\label{boxdot-est1}
For every level $m$ there exists a constant $C_m>0$ independent of $|a|<\frac{1}{2}$ so that the total gluing map $ \boxdot_a :E\rightarrow G^a$, defined by 
$$
 (\eta^+,\eta^-)\mapsto \boxdot_a (\eta^+,\eta^-):=(\oplus_a(\eta^+,\eta^-),\ominus_a(\eta^-,\eta^-)), 
$$
is an sc-isomorphism and satisfies the estimate
$$
C^{-1}_m\cdot \abs{(\eta^+,\eta^-)}_{E_m}\leq | \boxdot_a (\eta^+, \eta^-) |_{G^a_m}\leq C_m\cdot 
 \abs{(\eta^+,\eta^-)}_{E_m}.
$$
\end{thm}
\begin{proof}
We take a pair $(\eta^+, \eta^-)$ belonging to the space $E$ and represent it by  $\eta^\pm=c+r^\pm$ where $c$ is the common asymptotic limit and $r^\pm\in H^{3,\delta_0}(\R^\pm \times S^1, \R^N)$. Then we introduce $(q, p)$ by $(q,p)=(\oplus_a(\eta^+, \eta^-), \ominus_a(\eta^+, \eta^-)$ and compute the $G^a_m$-norm,  
\begin{equation*}
| (q,p)|^2_{G^a_m}=\abs{c}^2+e^{\delta_m R}
\left( 
\nr \wh{\oplus}_a(r^+, r^-)\nr^2_{3+m,-\delta_m}+\nr \wh{\ominus}_a(r^+, r^-)\nr^2_{m+3,\delta_m} \right).
\end{equation*}
In view of Lemma \ref{input-xy} below applied to $r^+=u$, $r^-=v$ and $\wh{\oplus}_a(r^+, r^-)=U$ and $\wh{\ominus}_a(r^+, r^-)=V$,  there exists a constant $C$ depending only on $m$ and $\delta_m$ so that
\begin{equation*}
\begin{split}
&\frac{1}{C}\cdot [ \nr\wh{\oplus}_a(r^+, r^-)\nr^2_{3+m,-\delta_m}+\nr\wh{\ominus}_a(r^+, r^-)\nr^2_{m+3,\delta_m}] \\
&\phantom{====} \leq e^{-\delta_m R}[\abs{r^+}^2_{H^{3+m,\delta_m}}+\abs{r^-}^2_{H^{3+m,\delta_m}}],
\end{split}
\end{equation*}
and
\begin{equation*}
\begin{split}
&e^{-\delta_m R}[\abs{r^+}^2_{H^{3+m,\delta_m}}+\abs{r^-}^2_{H^{3+m,\delta_m}}]\\
&\phantom{====} \leq C\cdot [ \nr\wh{\oplus}_a(r^+, r^-)\nr^2_{3+m,-\delta_m}+\nr \wh{\ominus}_a(r^+, r^-)\nr^2_{m+3,\delta_m}] .
\end{split}
\end{equation*}
We have denoted by $\abs{\cdot}_{H^{3+m,\delta_m}}$ our standard weighted Sobolev norms. Consequently, our  desired estimate  follows since the  $E_m$-norm of the pair  $(\eta^+, \eta^-)$  is defined by 
\begin{equation*}\label{norm0}
\abs{(\eta^+, \eta^-)}_{E_m}^2=\abs{c}^2+\abs{r^+}^2_{H^{3+m,\delta_m}}+
\abs{r^-}^2_{H^{3+m,\delta_m}}.
\end{equation*}
The proof of the proposition is complete.

\end{proof}

We introduce the sc-Hilbert space $\wh{E}$ consisting of pairs $(h^+,h^-)$ where $h^\pm =h^\pm_{\infty}+r^\pm$ with $r^\pm \in H^{3,\delta_0}(\R^\pm\times S^1, \R^N)$. We do not require that the asymptotic constants $h^\pm_{\infty}$ are equal. In addition, we also impose the following  boundary conditions,
$$h^\pm(0,0)=(0,0)\quad \text{and}\quad h^\pm(0,t)\in \{0\}\times{\mathbb R}.$$
Abbreviating the maps
$$\wt{h}^\pm =h^\pm\mp\frac{1}{2}(h^+_{\infty}-h^-_{\infty}),$$
we note that 
$$\wt{h}^\pm=h^\pm_{\infty}\mp\frac{1}{2}(h^+_{\infty}-h^-_{\infty})+r^\pm=\frac{1}{2}(h^+_{\infty}+h^-_{\infty})+r^\pm,$$
so that  the maps $\wt{h}^\pm$ have the same asymptotic limits 
equal to $\frac{1}{2}(h^+_{\infty}+h^-_{\infty})$. Consequently, 
the pair $(\wt{h}^+,\wt{h}^-)$ belongs to the previously defined sc-Hilbert space $E$.

The $\wh{E}_m$-norm of the pair $(h^+, h^-)$ is defined   by 
\begin{equation}\label{normwh}
\abs{(h^+,h^-)}_{\wh{E}_m}^2:=\abs{h^+_{\infty}}^2+\abs{h^-_{\infty}}^2+\abs{r^+}^2_{H^{3+m,\delta_m}}+
\abs{r^-}^2_{H^{3+m,\delta_m}}.
\end{equation}

Considering $\R^N\oplus E_m$ with the product norm,  the norm of 
$(h^+_{\infty}-h^-_{\infty}, (\wt{h}^+,\wt{h}^-))$ is equal to 
\begin{equation}\label{normwh1}
\begin{split}
\abs{(h^+_{\infty}-h^-_{\infty},& (\wt{h}^+,\wt{h}^-)}_{\R^N\oplus E_m}^2=\abs{h^+_{\infty}-h^-_{\infty}}^2+\abs{ (\wt{h}^+,\wt{h}^-)}_{E_m}^2\\
&=\abs{h^+_{\infty}-h^-_{\infty}}^2+\frac{1}{4}\abs{h^+_{\infty}+h^-_{\infty}}^2+\abs{r^+}^2_{H^{3+m,\delta_m}}+\abs{r^-}^2_{H^{3+m,\delta_m}}.
\end{split}
\end{equation}
It follows from \eqref{normwh} and \eqref{normwh1} that there exists a universal constant  $C$ 
so that 
\begin{equation}\label{normwh2}
\frac{1}{C}\cdot \abs{(h^+,h^-)}_{\wh{E}_m}^2\leq  \abs{(h^+_{\infty}-h^-_{\infty}, (\wt{h}^+,\wt{h}^-))}_{\R^N\oplus E_m}^2\leq C\cdot \abs{(h^+,h^-)}_{\wh{E}_m}^2.
\end{equation}

Now as a consequence of Theorem \ref{boxdot-est1} we obtain the following corollary.

\begin{cor}\label{est-cor1}
For every level $m$ there exists a constant $C_m>0$ independent of  the gluing parameter $|a|<\frac{1}{2}$, so that  for   $(h^+,h^-)\in\wh{E}_m$, the following estimate holds, 
$$
C_m^{-1}\cdot \abs{(h^+,h^-)}^2_{\wh{E}_m}\leq \left[ |h^+_\infty-h^-_\infty|^2+|\boxdot_a(\wt{h}^+,\wt{h}^-)|^2_{G^a_m}\right]\leq
C_m\cdot \abs{(h^+,h^-)}^2_{\wh{E}_m}
$$
where 
$\boxdot_a(\wt{h}^+,\wt{h}^-)=(\oplus_a(\wt{h}^+,\wt{h}^-), \ominus_a(\wt{h}^+, \wt{h}^-))$.

\end{cor}
\begin{proof}
Using \eqref{normwh},  \eqref{normwh1}, \eqref{normwh2}, and Theorem \ref{boxdot-est1}, one obtains for a generic constant $c_m$ depending on $m$ and not on $\abs{a}<\frac{1}{2}$,
\begin{equation*}
\begin{split}
\abs{(h^+,h^-)}^2_{\wh{E}_m}&\leq c_m\cdot \abs{(h^+_{\infty}-h^-_{\infty}, (\wt{h}^+,\wt{h}^-)}_{\R^N\oplus E_m}^2\\
&=c_m\cdot \left[\abs{h^+_{\infty}-h^-_{\infty}}^2+\abs{ (\wt{h}^+,\wt{h}^-)}_{E_m}^2\right]\\
&\leq c_m\cdot \left[ |h^+_\infty-h^-_\infty|^2+|\boxdot_a(\wt{h}^+,\wt{h}^-)|^2_{G^a_m}\right]\\
&\leq c_m\cdot  \left[\abs{h^+_{\infty}-h^-_{\infty}}^2+\abs{ (\wt{h}^+,\wt{h}^-)}_{E_m}^2\right]\\
&\leq c_m\cdot \abs{(h^+,h^-)}^2_{\wh{E}_m}
\end{split}
\end{equation*}
as claimed.
\end{proof}
\begin{rem}
Later on we deal with the case  $N=2$  where  the pair $(h^+,h^-)\in\wh{E}$ satisfies the boundary condition $h^\pm(0,0)=(0,0)$ and
$h^\pm(0,t)\in \{0\}\times {\mathbb R}$. Then the map $q=\oplus_a(\wh{h}^+,\wh{h}^-):Z_a\to \R^2$ will satisfy the following  boundary conditions, 
$$
q([0,0])=-\frac{1}{2}(h^+_\infty-h^-_\infty),\quad q([0,0]')=\frac{1}{2}(h^+_\infty-h^-_\infty)
$$
and, in addition, 
$$
q([0,t])\in -\frac{1}{2}(h^+_\infty-h^-_\infty) + (\{0\}\times {\mathbb R})\, \quad q([0,t]')\in \frac{1}{2}(h^+_\infty-h^-_\infty) + (\{0\}\times {\mathbb R}).
$$
\end{rem}

Later on we will need the following variant  variant of Theorem \ref{boxdot-est1} with respect to the hat gluing and hat anti-gluing.

We denote by  $F$ the sc-Hilbert space $F=H^{2, \delta_0}(\R^+\times S^1, \R^N)\oplus H^{2, \delta_0}(\R^-\times S^1, \R^N)$ whose  sc-structure is given by the sequence 
$F_m=H^{2+m, \delta_m}(\R^+\times S^1, \R^N)\oplus H^{2+m, \delta_m}(\R^-\times S^1, \R^N)$. 
The $F_m$-norm of the pair $(\xi^+, \xi^-)\in F$ is given by
\begin{equation*}
\abs{(\xi^+,\xi^-)}_{F_m}^2:=\abs{\xi^+}^2_{H^{m+2,\delta_m}}+\abs{\xi^-}^2_{H^{m+2,\delta_m}}.
\end{equation*}

We introduce the space $\wh{G}_a$ as follows.
If $a=0$, set  we $\wh{G}_a=F\oplus\{0\}$ and if $0<\abs{a}<\frac{1}{2}$,  then we define 
$$
\wh{G}_a=\wh{Q}^a\oplus  \wh{P}^a= H^{2}(Z_a,{\mathbb R}^{N})\oplus H^{2,\delta_0}(C_a,{\mathbb R}^{N}).
$$

The sc-structure of $\wh{G}^a$ is  given by the sequence $H^{2+m}(Z_a)\oplus H^{2+m, \delta_m}(C_a)$ for all $m\in \N_0$.  The total gluing map 
$\wh{\boxdot}_a=(\wh{\oplus}_a, \wh{\ominus}_a):F\to \wh{G}^a$ is an sc-linear isomorphism for every $a\in B_{\frac{1}{2}}$ in view of Theorem \ref{second-splicing}.

As in the case of  the total gluing $\boxdot$ it is useful to introduce families of norms. 
We introduce the $\wh{G}^a_m$-norm of the pair $(q, p)\in \wh{G}_a$ 
by setting 
$$
| (q,p)|_{\wh{G}^a_m}^2:=e^{\delta_m R}\cdot \left[\nr q\nr_{m+2,-\delta_m}^2+\nr p\nr_{m+2,\delta_m}^2\right]
$$
where these norms are defined above. 
Recall that if  
$(q,p)=\wh{\boxdot}_a(\xi^+,\xi^-)= (\wh{\oplus}_a(\xi^+, \xi^-),
\wh{\ominus}_a(\xi^+, \xi^-))$, then $(q, p)$ and $(\xi^+, \xi^-)$ are related as follows,
$$
\begin{bmatrix}q(s, t)\\p(s, t)\end{bmatrix}=
\begin{bmatrix}
\phantom{-}\beta_a&1-\beta_a\\
-(1-\beta_a)&\beta_a
\end{bmatrix}\cdot
\begin{bmatrix}
\xi^+(s, t)\\ \xi^-(s-R, t-\vartheta)
\end{bmatrix}
$$
where,  as usual,  $\beta_a=\beta_a(s)$. 
Then, in view of the definition of the  norm on $\wh{G}^a_m$, one derives from the estimates of  Lemma \ref{input-xy} the following  theorem.
\begin{thm}\label{poker1}
Given the level $m$ there exists a constant $C_m$ not depending on $a\in B_\frac{1}{2}$ so that
the following estimate holds, 
$$
C_m^{-1}\cdot\abs{(\xi^+,\xi^-)}_{F_m}\leq | \wh{\boxdot}_a(\xi^+,\xi^-)|_{\wh{G}^a_m}\leq C_m\cdot \abs{(\xi^+,\xi^-)}_{F_m}.
$$
\end{thm}

\begin{rem} The study of pairs $(\eta^+,\eta^-)$ in $F$ on level $m$ with respect to the norm $\abs{\cdot}_{F_m}$ is for every $a\in B_\frac{1}{2}$ completely equivalent
(up to a multiplicative constant independent of $a$) to  the study of the  associated pairs $(q,p)\in \wh{G}^a$ on  the level $m$ with respect to the norm 
$|\cdot |_{\wh{G}^a_m}$.
\end{rem}

The two theorems can be deduced from the following lemma.
\begin{lem}\label{input-xy}
There exists a constant $C$ depending on $m$ and $\delta>0$ so that
 for maps $(U,V)\in H^m(Z_a,{\mathbb R}^{N})\oplus H^{m,\delta}(C_a,{\mathbb R}^{N})$ and $(u,v)\in H^{m,\delta}({\mathbb R}^+\times S^1,{\mathbb R}^{N})\oplus H^{m,\delta}({\mathbb R}^-\times S^1,{\mathbb R}^{N})$ satisfying 
  \begin{equation}\label{newrel0}
 \begin{bmatrix}U(s, t)\\V(s, t)\end{bmatrix}=\begin{bmatrix}\phantom{-}\beta_a&1-\beta_a\\-(1-\beta_a)&\beta_a\end{bmatrix}\cdot \begin{bmatrix}u(s, t)\\ v(s-R, t-\vartheta)\end{bmatrix}
 =\begin{bmatrix}\wh{\oplus}_a(U, V)\\ \wh{\ominus}_a(U, V)\end{bmatrix},
 \end{equation}
the following estimate holds true,
 \begin{equation}\label{newest1}
\frac{1}{C}\left[ \nr U \nr_{m,-\delta}^2+\nr V \nr_{m,\delta}^2\right]\leq e^{-\delta R}\left[\abs{u}_{H^{m,\delta}}^2+\abs{v}_{H^{m,\delta}}^2\right]\leq
 C\left[\nr U \nr_{m,-\delta}^2 +\nr V\nr_{m,\delta}^2\right].
 \end{equation}
\end{lem}

\begin{proof}
In view of \eqref{newrel0}  and recalling  $\beta_a(s)=0$ for $s\geq \frac{R}{2}+1$ and $1-\beta_a(s)=0$  for $s\leq  \frac{R}{2}-1$, it follows that $\abs{U}_{m,-\delta}^2+\abs{V}_{m,\delta}^2$   is  bounded above by a constant $C$ times the sum of integrals of the following type:
\begin{itemize}
\item[$\bullet$]\quad $I_1= \int\limits_{[0, \frac{R}{2}+1]\times S^1}\abs{D^{\alpha}u(s, t)}^2e^{-2\delta \abs{s-\frac{R}{2} }},$
\item[$\bullet$]\quad $I_2=\int\limits_{[\frac{R}{2}-1, R]\times S^1}\abs{D^{\alpha}v(s-R, t-\vartheta)}^2e^{-2\delta \abs{s-\frac{R}{2}}},$
\item[$\bullet$]\quad $I_3=\int\limits_{[\frac{R}{2}-1, \infty)\times S^1}\abs{D^{\alpha}u(s, t)}^2e^{-2\delta \abs{s-\frac{R}{2} }},$
\item[$\bullet$]\quad $I_4=\int\limits_{(-\infty, \frac{R}{2}+1]\times S^1}\abs{D^{\alpha}v(s-R, t-\vartheta)}^2e^{-2\delta \abs{s-\frac{R}{2}}},$
\end{itemize}
where the multi-indices $\alpha$ satisfy $\abs{\alpha}\leq m$. The constant $C$ depends on $m$ and the function $\beta$.  We estimate each of the above integrals.  To estimate the integrals $I_1$ and $I_3$ we use the fact that $-2\delta \abs{s-\frac{R}{2}}-2\delta s\leq -\delta R$ for all $s\in \R$. Then 
\begin{equation*}
\begin{split}
I_1&= \int\limits_{[0, \frac{R}{2}+1]\times S^1}\abs{D^{\alpha}u(s, t)}^2e^{-2\delta \abs{s-\frac{R}{2} }}=\int\limits_{[0, \frac{R}{2}+1]\times S^1}\abs{D^{\alpha}u(s, t)}^2e^{2\delta s}\cdot e^{-2\delta \abs{s-\frac{R}{2} }-2\delta s}\\
&\leq e^{-\delta R}\int\limits_{[0, \frac{R}{2}+1]\times S^1}\abs{D^{\alpha}u(s, t)}^2e^{2\delta s} \leq 
e^{-\delta R}\cdot \abs{u}^2_{H^{m, \delta}}
\end{split}
\end{equation*}
and 
\begin{equation*}
\begin{split}
I_3&=\int\limits_{[\frac{R}{2}-1, \infty)\times S^1}\abs{D^{\alpha}u(s, t)}^2e^{-2\delta \abs{s-\frac{R}{2} }}=\int\limits_{[\frac{R}{2}-1, \infty)\times S^1}\abs{D^{\alpha}u(s, t)}^2e^{2\delta s}\cdot e^{-2\delta \abs{s-\frac{R}{2} }-2\delta s}\\
&\leq e^{-\delta R}\int\limits_{[\frac{R}{2}-1, \infty)\times S^1}\abs{D^{\alpha}u(s, t)}^2e^{2\delta s}\leq  e^{-\delta R}\cdot \abs{u}^2_{H^{m, \delta}}.
\end{split}
\end{equation*}

To estimate the integrals $I_2$ and $I_4$ we use $-2\delta \abs{s+\frac{R}{2}}+2\delta s\leq -\delta R$ for all $s$. Then abbreviating $\Sigma_R=\left[\frac{R}{2}-1, R\right]\times S^1$ we obtain for for the integral $I_2$,
\begin{equation*}
\begin{split}
I_2&=\int\limits_{\Sigma_R}\abs{D^{\alpha}v(s-R, t-\vartheta)}^2e^{-2\delta \abs{s-\frac{R}{2}}}=\int\limits_{\Sigma_{-R}}\abs{D^{\alpha}v(s, t)}^2e^{-2\delta \abs{s+\frac{R}{2}}}\\
&=\int\limits_{\Sigma_{-R}}\abs{D^{\alpha}v(s, t)}^2e^{-2\delta s}e^{-2\delta \abs{s+\frac{R}{2}}+2\delta s}\leq e^{-\delta R}\int\limits_{\Sigma_{-R}}\abs{D^{\alpha}v(s, t)}^2e^{-2\delta s}\\
&\leq e^{-\delta R}\cdot \abs{v}^2_{H^{m, \delta}},
\end{split}
\end{equation*}
and abbreviating $\Sigma_R=(-\infty, \frac{R}{2}+1]\times S^1$,  the integral $I_4$ can be estimated as 
\begin{equation*}
\begin{split}
I_4&=\int\limits_{\Sigma_R}\abs{D^{\alpha}v(s-R, t-\vartheta)}^2e^{-2\delta \abs{s-\frac{R}{2}}}=\int\limits_{\Sigma_{-R}}\abs{D^{\alpha}v(s, t)}^2e^{-2\delta \abs{s+\frac{R}{2}}}\\
&=\int\limits_{\Sigma_{-R}}\abs{D^{\alpha}v(s, t)}^2e^{-2\delta s}\cdot e^{-2\delta \abs{s+\frac{R}{2}}+2\delta s}\\
&\leq e^{-\delta R}\int\limits_{\Sigma_{-R}}\abs{D^{\alpha}v(s, t)}^2e^{-2\delta s}
\leq  e^{ -\delta R}\cdot \abs{v}^2_{H^{m, \delta}}.
\end{split}
\end{equation*}
Summing up the  estimates for the integrals $I_1,\ldots ,I_4$ for all multi-indices $\alpha$  satisfying  $\abs{\alpha}\leq m$, we get 
$$\nr U\nr_{m,-\delta}^2+\nr V\nr_{m,\delta}^2\leq Ce^{-\delta R}\cdot \left[\abs{u}^2_{H^{m,\delta}}+\abs{v}^2_{H^{m,\delta}}\right] $$
as desired.

In order to estimate $\abs{u}^2_{m,\delta}+\abs{v}_{m,\delta}^2$ we  multiplying both  sides of \eqref{newrel0} by the the inverse of the matrix and obtain 
\begin{equation}\label{newrel1}
 \begin{bmatrix}u(s, t)\\ v(s-R, t-\vartheta)\end{bmatrix}
 =
\frac{1}{\gamma_a} \begin{bmatrix}\phantom{-}\beta_a&-(1-\beta_a)\\1-\beta_a& \phantom{-}\beta_a\end{bmatrix}\cdot 
\begin{bmatrix}U(s, t)\\V(s, t)\end{bmatrix}
\end{equation}
where $\gamma_a=\gamma_a(s)$ is the the determinant of the matrix. 
Hence
$$v(s', t')=\frac{1-\beta_a(s'+R)}{\gamma'(s'+R)}U(s'+R, t'+\vartheta)+\frac{\beta_a(s'+R)}{\gamma_a(s'+R)}V(s'+R, t'+\vartheta).$$
Using the above relations and  $\beta_a(s'+R)=0$ for $s\geq -\frac{R}{2}+1$ and $1-\beta_a(s'+R)=0$  for $s\leq - \frac{R}{2}-1$, it follows that $\abs{u}_{m,\delta}^2+\abs{v}_{m,\delta}^2$   is  bounded above by a constant $C$ (depending only on $m$ and the function $\beta$)  times the sum of integrals of the following type:
\begin{itemize}
\item[$\bullet$]\quad $J_1= \int\limits_{[0, \frac{R}{2}+1]\times S^1}\abs{D^{\alpha}U(s, t)}^2e^{2\delta s},$
\item[$\bullet$]\quad $J_2=\int\limits_{[\frac{R}{2}-1, \infty)\times S^1}\abs{D^{\alpha}V(s,t)}^2e^{2\delta s},$
\item[$\bullet$]\quad $J_3=\int\limits_{[-\frac{R}{2}-1,0]\times S^1}\abs{D^{\alpha}U(s+R, t+\vartheta)}^2e^{-2\delta s},$
\item[$\bullet$]\quad $J_4=\int\limits_{(-\infty, -\frac{R}{2}+1]\times S^1}\abs{D^{\alpha}V(s+R, t+\vartheta)}^2e^{-2\delta s},$
\end{itemize}
where the multi-indices $\alpha$ satisfy $\abs{\alpha}\leq m$. 

The integral $J_1$ is estimated as follows, 
\begin{equation*}
\begin{split}
J_1&= \int\limits_{[0, \frac{R}{2}+1]\times S^1}\abs{D^{\alpha}U(s, t)}^2e^{2\delta s}\\
&=
 \int\limits_{[0, \frac{R}{2}+1]\times S^1}\abs{D^{\alpha}U(s, t)}^2e^{-2\delta \abs{s-\frac{R}{2}}}\cdot e^{2\delta \abs{s-\frac{R}{2}}+2\delta s}\\
&\leq e^{\delta(4+R)} \int\limits_{[0, \frac{R}{2}+1]\times S^1}\abs{D^{\alpha}U(s, t)}^2e^{-2\delta \abs{s-\frac{R}{2}}}\leq e^{4\delta +\delta R} \cdot \nr U\nr_{m,-\delta}^2,
\end{split}
\end{equation*}
using  $2\delta \abs{s-\frac{R}{2}}+2\delta s\leq \delta(4+R)$ for all $s\leq \frac{R}{2}+1$.

For the integral $J_2$ we obtain
\begin{equation*}
\begin{split}
J_2&=\int\limits_{[\frac{R}{2}-1, \infty)\times S^1}\abs{D^{\alpha}V(s,t)}^2e^{2\delta s}\\
&=
\int\limits_{[\frac{R}{2}-1, \infty)\times S^1}\abs{D^{\alpha}V(s,t)}^2e^{2\delta \abs{s-\frac{R}{2}}} \cdot e^{-2\delta \abs{s-\frac{R}{2}}+2\delta s}\\
&\leq e^{\delta R}\int\limits_{[\frac{R}{2}-1, \infty)\times S^1}\abs{D^{\alpha}V(s,t)}^2e^{2\delta \abs{s-\frac{R}{2}}} \leq e^{\delta R}\cdot \nr V\nr^2_{m,\delta},
\end{split}
\end{equation*}
using  $-2\delta \abs{s-\frac{R}{2}}+2\delta s\leq \delta R$ for all $s\in \R$.

For the integral $J_3$ we find
\begin{equation*}
\begin{split}
J_3&=\int\limits_{[-\frac{R}{2}-1,0]\times S^1}\abs{D^{\alpha}U(s+R, t+\vartheta)}^2e^{-2\delta s}\\
&=
\int\limits_{[\frac{R}{2}-1,R]\times S^1}\abs{D^{\alpha}U(s, t)}^2e^{-2\delta ( s-R)}\\
&=\int\limits_{[\frac{R}{2}-1,R]\times S^1}\abs{D^{\alpha}U(s, t)}^2e^{-2\delta \abs{s-\frac{R}{2}}}\cdot e^{2\delta \abs{s-\frac{R}{2}}-2\delta ( s-R)}\\
&\leq e^{\delta (4+R)}\int\limits_{[\frac{R}{2}-1,R]\times S^1}\abs{D^{\alpha}U(s, t)}^2e^{-2\delta \abs{s-\frac{R}{2}}}\leq e^{\delta (4+R)}\cdot \nr U\nr^2_{m,-\delta},
\end{split}
\end{equation*}
using  $2\delta \abs{s-\frac{R}{2}}-2\delta ( s-R)\leq \delta (4+R)$ for all $s\geq \frac{R}{2}-1$.

The last integral can be estimated as follows
\begin{equation*}
\begin{split}
J_4&=\int\limits_{(-\infty, -\frac{R}{2}+1]\times S^1}\abs{D^{\alpha}V(s+R, t+\vartheta)}^2e^{-2\delta s}\\
&=\int\limits_{(-\infty, \frac{R}{2}+1]\times S^1}\abs{D^{\alpha}V(s, t)}^2e^{-2\delta (s-R)}\\
&=\int\limits_{(-\infty, \frac{R}{2}+1]\times S^1}\abs{D^{\alpha}V(s, t)}^2e^{2\delta \abs{s-\frac{R}{2}}}\cdot e^{-2\delta \abs{s-\frac{R}{2}}-2\delta (s-R)}\\
&\leq e^{\delta R}\int\limits_{(-\infty, \frac{R}{2}+1]\times S^1}\abs{D^{\alpha}V(s, t)}^2e^{2\delta \abs{s-\frac{R}{2}}}\leq e^{\delta R}\cdot \nr V\nr^2_{m,\delta}, 
\end{split}
\end{equation*}
in view of  $-2\delta \abs{s-\frac{R}{2}}-2\delta (s-R)\leq \delta R$ for all $s\in \R$. 

Summing up the estimates for the integrals $J_1,\ldots ,J_4$  for all the multi-indices $\alpha$ satisfying $\abs{\alpha}\leq m$, we obtain 
$$\abs{u}_{H^{m,\delta}}^2+\abs{v}_{H^{m,\delta}}^2\leq Ce^{\delta R}\cdot \left[ \nr U\nr_{m,-\delta}^2+\nr V\nr_{m,\delta}^2\right], $$
as claimed. The proof of Lemma \ref{input-xy} is complete.
\end{proof}

\section{Long Cylinders and Gluing}\label{longcylinders}

In chapter \ref{longcylinders} we are going to construct an M-polyfold structure on the set $\ov{X}$ of diffeomorphisms between conformal cylinders which break apart. Using the technical results of chapter \ref{sec2} we shall first prove Proposition \ref{prop3.15-nnew} and Theorem \ref{thm-comp}. Then we shall illustrate the Fredholm theory outlined in chapter \ref{chapter1} by constructing a strong bundle over the M-polyfold $\ov{X}$ proving Theorem \ref{thm1.42-nnew}. We shall show that this strong bundle admits the Cauchy-Riemann operator as an sc-Fredholm section. Finally, the sc-implicit function theorem for polyfold Fredholm sections provides  the proofs of  the Theorems \ref{fred:thm0} and \ref{strong-x}. 

We start with the introduction of another retraction used later on.

\subsection{Another Retraction}\label{anothersplicing}
We choose a strictly  increasing sequence $(\delta_m)_{m\in \N_0}$ satisfying $0<\delta_m< 2\pi$. Recall that the Banach space $\wh{E}$ consists of pairs
$(h^+,h^-)$ of maps 
$$
h^\pm\in H^{3,\delta_0}_c(\R^\pm\times S^1,\R^N)
$$
for which there exist asymptotic constants  $c^\pm \in \R^N$ so that $h^\pm-c^\pm\in H^{3,\delta_0}(\R^\pm\times S^1,\R^N)$. The sc-structure $(\wh{E}_m)$  is defined by the sequence  
$$
\wh{E}_m=H^{3+m,\delta_m}_c(\R^+\times S^1,\R^N)\times H^{3+m,\delta_m}_c(\R^-\times S^1,\R^N).
$$
Given a pair $(h^+, h^-)\in \wh{E}$,  we denote the asymptotic constants of $h^{\pm}$  by $h^\pm_\infty$
and write 
$$
h^\pm=h^\pm_\infty+r^\pm
$$
so that $r^\pm \in H^{3,\delta_0}(\R^\pm\times S^1,\R^N).$ 

We point out that in contrast to Section \ref{arisingx} we do not require that the asymptotic constants $h^{+}_{\infty}$ and $h^{-}_{\infty}$ coincide.

We define for $a\in \C$ satisfying $0\leq |a|<\frac{1}{2}$ 
the mapping 
$$
\rho_a:\wh{E}\to  \wh{E}
$$
by
$$
\rho_a(h^+,h^-)=(h^+_\infty,h^-_\infty)+\pi_a(r^+,r^-),
$$
where $\pi_a$ is the projection defined by the formulae in Section \ref{arisingx}. Since the projection $\pi_0$ is equal to the identity map, also $\rho_0(h^+, h^-)=(h^+, h^-)$ for any pair $(h^+, h^-)\in \wh{E}$.  From $\pi_a\circ \pi_a=\pi_a$ we derive the following lemma.
\begin{lem}
The map $\rho_a:\wh{E}\to  \wh{E}$ is a projection.
\end{lem}
\begin{proof}
Given a $(h^+, h^-)\in \wh{E}$, we have to show that 
$$
\rho_a\circ \rho_a(h^+,h^-)=\rho_a(h^+,h^-).
$$
We write $h^\pm =h^\pm_{\infty}+r^\pm$ and set $\pi_a (r^+, r^-)=(\eta^+, \eta^-)$. 
Then 
\begin{equation}\label{proj-eq0}
\rho_a(h^+, h^-)=(h^+_\infty,h^-_\infty)+\pi_a(r^+,r^-)=(h^+_\infty,h^-_\infty)+(\eta^+, \eta^-).
\end{equation}
From the formula for the projection $\pi_a$ in Section \ref{arisingx}, we know that 
$$(\eta^+, \eta^-)=(c, c)+(\xi^+, \xi^-)$$
where $c$ is the common asymptotic constant of  the maps $\eta^+$ and $\eta^-$ which   is equal to $\text{av}_R(r^+, r^-)$. 
Thus, continuing  with \eqref{proj-eq0}, $$\rho_a(h^+, h^-)=(h^+_\infty+c,h^-_\infty+c)+(\xi^+, \xi^-).$$
Applying $\rho_a$ to both sides, we obtain 
\begin{equation}\label{proj-eq1}
\rho_a\circ \rho_a(h^+, h^-)=(h^+_\infty+c,h^-_\infty+c)+\pi_a(\xi^+, \xi^-).
\end{equation}
Using $\pi_a\circ \pi_a=\pi_a$ and $\pi_a(r^+, r^-)=(\eta^+, \eta^-)$, we obtain 
$\pi_a(r^+, r^-)=\pi_a(\eta^+, \eta^-)=(c, c)+\pi_a(\xi^+, \xi^-)$. Therefore, the  right-hand side of \eqref{proj-eq1} is equal to
$$
(h^+_\infty+c,h^-_\infty+c)+\pi_a(\xi^+, \xi^-)=(h^+_{\infty}, h^-_{\infty})+\pi_a(r^+, r^-)=\rho_a(h^+,h^-).$$
This finishes the proof of the lemma. 
\end{proof}
 From the already established sc-smoothness properties of  $\pi_a$  in Section \ref{arisingx} we deduce the sc-smoothness of the projection  $\rho_a$.
\begin{thm}
The map
$$
B_\frac{1}{2}\oplus \wh{E}\to  \wh{E},  \quad (a,(h^+,h^-))\mapsto  \rho_a(h^+,h^-)
$$
is sc-smooth.
\end{thm}
\begin{proof} If $h^\pm=h^\pm_{\infty}+r^\pm$, then 
the map
$$
(a,(h^+,h^-))\mapsto   (a, (r^+,r^-))
$$
as well as the map
$$
(a,(h^+,h^-))\mapsto   (h^+_\infty,h^-_\infty)
$$
are sc--smooth. 
By composing  the first map with the sc--smooth map $\pi: B_\frac{1}{2}\oplus E\to E$, $(a, (h^+, h^-))\mapsto \pi_a (h^+, h^-)$ and then  adding to it the second map we obtain an sc--smooth map.
\end{proof}

The geometric interpretation of  the projection $\rho_a$ is the following.
\begin{lem}\label{kernel-lem1}
The projection $\rho_a$ is the projection onto the kernel of the map 
$
P_a^-:\wh{E}\to  H^{3,\delta_a}_c(C_a),$ defined by 
$$
 (h^+,h^-)\mapsto   \ominus_a(h^+-h^+_\infty,h^--h^-_\infty), 
$$
along the kernel of the map
$
P_a^+:\wh{E}\to  \R^N\oplus H^3(Z_a),$ defined by 
$$
(h^+,h^-)\mapsto   (h^+_\infty-h^-_\infty,\oplus_a(h^+,h^-+h^+_\infty-h^-_\infty)).
$$
\end{lem}
\begin{proof} We first show that $\wh{E}=\ker P^+_a \oplus \ker P^-_a.$ To do this we consider the linear map 
$$
P:\wh{E}\to  \R^N\oplus H^3(Z_a)\oplus H^{3,\delta_0}_c(C_a)
$$
defined  by
$$
(h^+,h^-)\mapsto   (h^+_\infty-h^-_\infty,\oplus_a(h^+,h^-+h^+_\infty-h^-_\infty),\ominus_a(h^+-h^+_\infty, h^- -h^-_\infty)).
$$
Clearly, $P$ is linear and we claim that  it  is a bijection. Then the  injectivity of $P$ will show that $\ker  P^+_a\cap  \ker P^-_a=\{(0,0)\}$ while  the surjectity of 
$P$  will show that $\ker  P^+_a+\ker  P^-_a=\wh{E}$.

Now, if  $(h^+, h^-)$ belongs  to the kernel of the map $P$, then $h^+_{\infty}=h^-_{\infty}$ and consequently $\oplus_a (h^+, h^-)=0$ and $\ominus_a (h^+, h^-)=0$, and hence by Proposition  \ref{first-splice},  $h^+=h^-=0$.
In order to prove  that $P$ is a surjection we take $(\xi, \wh{u}, \wh{v})\in  \R^N\oplus H^3(Z_a)\oplus H^{3,\delta_0}_c(C_a).$  
By Proposition \ref{first-splice} again,  there is  a unique pair $(u^+, u^-)\in E$ solving the equations 
$$\oplus_a (u^+, u^-)=\wh{u}\quad \text{and}\quad \ominus_a (u^+, u^-)=\wh{v}$$
where we have denoted by $E$ the subspace of $\wh{E}$ consisting of pairs having  matching asymptotic limits.  If $c$ is the  common asymptotic constant of the maps $u^\pm$, then $u^\pm =c+r^\pm$ and 
$\ominus_a (u^+, u^-)=\ominus_a (r^+, r^-)=\wh{v}$. Given $\xi$ we introduce the maps 
$$h^+=u^+\quad \text{and}\quad h^-=-\xi+u^-$$
whose   asymptotic constants are equal to $h^+_{\infty}=c$ and $h^-_{\infty}=-\xi+c$,  so that $h^+_{\infty}-h^-_{\infty}=\xi$, $\oplus_ a(h^+, h^-+h^+_{\infty}-h^-_{\infty})=\oplus_a(u^+, u^-)=\wh{u}$, and 
$\ominus_a (h^+-h^+_{\infty}, h^- -h^-_{\infty})=\ominus_a(r^+, r^-)=\wh{v}$. We have proved that $P$ is surjective. This finishes the proof of our claim and hence   $\wh{E}=\ker  P^+_a\oplus \ker P^-_a.$ 

We split a given pair  $(h^+,h^-)\in\wh{E}$ where $h^\pm=h^\pm_{\infty}+r^\pm$ into the sum 
$$
(h^+,h^-)=(w^+,w^-)+(v^+,v^-)
$$
in which 
\begin{align*}
(w^+,w^-):&=(h^+_\infty,h^-_\infty)+ \pi_a(r^+,r^-)\\
\intertext{and}
(v^+,v^-):&= (\id -\pi_a) (r^+,r^-).
\end{align*}
We claim that $(w^+,w^-)\in \ker  \ominus_a$ and $(v^+, v^-)\in \ker  \oplus_a$.  
Since $\id -\pi_a$ is the projection onto $\ker  \oplus_a$ we conclude   that $\oplus_a(v^+,v^-)=0$. Moreover,  by the formula for the projection $\pi_a$ in Section \ref{arisingx},  the pair $(v^+,v^-)$ has a common asymptotic limit equal to $c=-\text{av}_a (r^+, r^-).$ This means that $(v^+,v^-)\in E$ and 
\begin{equation*}
\begin{split}
P_a^+(v^+, v^-)&=(v^+_{\infty}-v^-_{\infty}, \oplus_a (v^+, v^-+v^+_{\infty}-v^-_{\infty})\\
&=(0,  \oplus_a (v^+, v^-))=(0,0).
\end{split}
\end{equation*}
Hence  $(v^+, v^-)\in  \ker P_a^+$. Considering the term $(w^+,w^-)$, we note that  $w^\pm_{\infty}=h^\pm_\infty+c$. Using the formula for the projection map $\pi_a$ and abbreviating  $\pi_a (r^+, r^-)=(\eta^+, \eta^-)$,  we have $\pi_a (r^+-c, r^--c)=(\eta^+-c, \eta^--c)$.  Hence 
\begin{equation*}
\begin{split}
P^-_a(w^+, w^-)&=\ominus_a(w^+-w^+_\infty,w^--w^-_\infty)=\ominus_a(w^+-h^+_\infty-c,w^--h^-_\infty-c)\\
&=\ominus_a(\eta^+-c,\eta^--c) =\ominus_a\circ \pi_a (r^+-c,r^--c) =0, 
\end{split}
\end{equation*}
since $\pi_a$ is the projection onto the kernel of the map $\ominus_a$. Consequently, $(w^+, w^-)\in \ker P^-_a$. This completes the proof of Lemma \ref{kernel-lem1}.
\end{proof}

If $a\neq 0$ we abbreviate by $F_a$ the kernel of the map $P_a^-$.   Lemma \ref{kernel-lem1} implies that 
the  restriction of the  map $P_a^+$ to the kernel, 
$$P^+_a :F_a\to \R^N\oplus H^3(Z_a,\R^N), \quad (h^+,h^-)\to  (h^+_\infty-h^-_\infty,\oplus_a(h^+,h^-+h^+_\infty-h^-_\infty))$$
is a bijection.
If $a=0$, we put $\rho_0=\id$. In this  case, we set  $F_0=\wh{E}$ and define the map 
$P^+_0:\wh{E}\to  \R^N\oplus E$ by 
$$
(h^+,h^-)\mapsto   (h^+_\infty-h^-_\infty,h^+,h^-+h^+_\infty-h^-_\infty), 
$$
where  $E$ consists of those pairs in $\wh {E}$  which have matching asymptotic limits. This map is also a bijection.

Next  we consider the case of dimension $N=2$ and we restrict ourselves to the subspace $\wh{E}_0$
of $\wh{E}$ consisting of all  pairs $(h^+,h^-)$ of maps in $H^{3, \delta_0}_c(\R^\pm\times S^1, \R^2)$ satisfying 
$$h^\pm(0,0)=(0,0)\quad  \text{and} \quad h^\pm(0,t)\in \{0\}\times \R\subset \R^2$$
for all $t\in S^1$.  We abbreviate the image of $\wh{E}_0$ under the projection $\rho_a$ by 
$$
G^a =\rho_a(\wh{E}_0)=\{ (h^+, h^-)\in \wh{E}_0\vert \ominus_a(h^+-h^+_{\infty}, h^--h^-_{\infty})=0\}
$$
where we have used Lemma \ref{kernel-lem1}. 
Observe that $G^a$ is a  subspace of $\wh{E}_0.$

The map $\Delta:Z_a\to  \R^2$,  defined by 
$$
\Delta=\oplus_a(h^+,h^-+h^+_\infty-h^-_\infty),
$$ 
has the following  properties:
\begin{itemize}
\item[(1)]  $\Delta(p^+_a)=\Delta([0,0])=(0,0)\quad \text{and}\quad  \Delta(p_a^-)=\Delta([0,0]')=h^+_\infty-h^-_\infty$.
\item[(2)] $\Delta([0,t])\in \{0\}\times \R \quad \text{and}\quad   \Delta([R,t])\in (h^+_\infty-h^-_\infty)+(\{0\}\times \R)$.
 \end{itemize}
  
If  $0<\abs{a}<\frac{1}{2}$, we denote by  $H_a$  the  sc-subspace of $H^3(Z_a)$ consisting of those maps $u:Z_a\to \R^2$ which satisfy $u(p^+_a)=(0, 0)$ and 
 $$
 u([0,t])\in \{0\}\times \R\subset  \R^2\quad \text{and}\quad u([0,t]')\in u(p^-_a)+(\{0\}\times\R).
 $$
 If  $a=0$,  we  denote by $H_0$  the space  consisting  of all 
 pairs $(h^+,h^-)\in \wh{E}_0$ for which  there exists a constant $c\in \R^2$  satisfying   $(h^+,h^-+c)\in E$.

\begin{lem}
For $|a|<\frac{1}{2}$,   the map $ \Phi_a: G^a\to  H_a$,  defined by 
\begin{equation*}
 \Phi_a(h^+,h^-)=
 \begin{cases} 
 \oplus_a(h^+,h^-+h^+_\infty-h^-_\infty)&\quad \text{if $a\neq 0$,}\\
\phantom{a} (h^+,h^-+h^+_\infty-h^-_\infty)&\quad \text{if $a=0$,}
 \end{cases}
 \end{equation*}
 is a linear sc-isomorphism.
  \end{lem}
  \begin{proof}
We only consider the case $a\neq 0$ and begin by showing  the  injectivity of the map $\Phi$. We assume that 
 $ \oplus_a(h^+,h^-+h^+_\infty-h^-_\infty)=0$ for  a pair $(h^+, h^-)\in G^a$. Hence  the pair $(h^+, h^-)\in \wh{E}_0$  solves the system of two equations
 \begin{equation*}
 \begin{aligned}
 \oplus_a(h^+,h^-+h^+_\infty-h^-_\infty)&=0\\
\ominus_a(h^+,h^-+h^+_\infty-h^-_\infty)&=0
\end{aligned}
\end{equation*}
The system has a unique solution $h^+=0$ and $h^-+h^+_{\infty}-h^-_{\infty}=0$.  Since $h^\pm=h^\pm_{\infty}+r^\pm$ where $r^\pm \in H^{3, \delta_0}(\R^\pm \times S^1)$, we conclude that 
$h^+_{\infty}=0$ and $r^\pm=0$. Hence $h^-=h^-_{\infty}$. From $h^-(0, 0)=0$, we obtain that also the constant $h^-_{\infty}=0$ vanishes. Consequently, $(h^+, h^-)=(0, 0)$ as claimed. 

In order to  prove the surjectivity, we choose 
a map $u:Z_a\to  \R^2$ which belongs to $H_a$ and abbreviate  $c=u(p_a^-)=u(R, \vartheta)$. 
We have to solve the system of two equations
 \begin{equation}\label{both}
 \begin{aligned}
 \oplus_a(h^+,h^-+h^+_\infty-h^-_\infty)&=u\\
\ominus_a(h^+,h^-+h^+_\infty-h^-_\infty)&=0
\end{aligned}
\end{equation}
Integrating the first  equation at $s=\frac{R}{2}$ over the circle $S^1$ one obtains $\av (h^+,h^-+h^+_\infty-h^-_\infty)=[u]$. These averages are defined in Section \ref{arisingx}. 
Thus  the solutions of the above system are given by the formulae 
\begin{equation}\label{equationhp0}
h^+(s, t)=[u]+\frac{\beta_a}{\gamma_a}\cdot (u-[u])
\end{equation}
and
\begin{equation}\label{equationhm0}
h^-(s-R, t-\vartheta)+h^+_\infty-h^-_\infty=[u]+\frac{1-\beta_a}{\gamma_a}\cdot (u-[u])
\end{equation}
where, as usual, $\beta_a=\beta_a(s)$ and $\gamma_a=\beta_a^2+(1-\beta_a)^2$. 
Evaluating both sides of the first equation in \eqref{both} at the  point $(R, \vartheta)$, we  find 
$$h^+_{\infty}-h^-_{\infty}=u(R, \vartheta)=:c$$
so that  $h^-$ in \eqref{equationhm0} becomes 
\begin{equation*}
h^-(s-R, t-\vartheta)=[u]-c+\frac{1-\beta_a}{\gamma_a}\cdot (u-[u]).
\end{equation*}
In view of the properties of the function $\beta_a$,  the asymptotic constants $h^\pm_{\infty}$ are equal to $h^+_{\infty}=[u]$ and $h^-_{\infty}=[u]-c.$

It remains to show that $(h^+, h^-)\in \wh{E}_0$.  Using \eqref{equationhp0} and  the properties of the map $u$,  we find  $h^+(0,0)=u(0, 0)=(0, 0)$ and $h^+(0, t)=u(0, t)\in \{0\}\times \R$. 
For the map $h^-$ we use \eqref{equationhm0} and  find that
$h^-(0, 0)=u(R, \vartheta)=c$ and $h^-(0, t)=u(R, \vartheta +t)\in c+\{0\}\times \R$. Hence $(h^+, h^-)\in \wh{E}_0$ and the proof of the lemma is  finished.
 \end{proof}

\begin{rem}
We note that if $h^+$ and $h^-$ have matching asymptotic limits, then  the glued map $\Delta$ takes the boundary values in
 $\{0\}\times \R$. The difference of asymptotic values makes $\Delta$ take boundary values in an affine
 subspace on the right-hand boundary of $Z_a$.
 \end{rem}
In order to understand the upcoming context of the above construction, we consider   pairs 
$(u^+_0,u^-_0)$ of diffeomorphisms 
$$
u^\pm_0:\R^\pm\times S^1\to  
\R^\pm\times S^1
$$
of the half-cylinders, which on the covering spaces  are of the form $u^\pm=\id +h^\pm$ with $h^\pm \in \wh{E}$ specified below. 
The diffeomorphisms leave the boundaries 
$$\partial (\R^\pm \times S^1)=\{0\}\times S^1$$
invariant and we require that they keep  the distinguished points $(0, 0)\in \R^\pm \times S^1$ fixed, so that 
$$
u^\pm_0(0,0)=(0,0).
$$
On the covering spaces  we shall represent  the maps $u^\pm_0$ in the following form, 
\begin{equation}\label{write1}
u^+_0(s,t)=(s,t)+(d_0^+,\vartheta^+_0)+r^+_0(s,t), \quad s\geq 0
\end{equation}
and
\begin{equation}\label{write2}
u^-_0(s',t')=(s',t')+(d_0^-,\vartheta^-_0)+r^-_0(s',t'), \quad s'\leq 0.
\end{equation}
where  $r^\pm_0:\R^\pm\times \R\to  \R^2$ belong to $H^{3,\delta_0}(\R^\pm \times S^1, \R^2)$  and satisfy at the boundaries (where $s=0$)

$$(d_0^\pm, \vartheta^\pm_0)+r^\pm_0(0, t)\in \{0\}\times \R.$$

Our aim is to define   for a gluing parameter $a$ of sufficiently small modulus,  a glued map
$$
\boxplus_a(u^+_0,u^-_0):
Z_a\to  Z_b,
$$
 which maps the glued finite cylinder $Z_a$ introduced in Section \ref{arisingx}  diffeomorphically  onto the glued finite cylinder $Z_b$ belonging to $b=b(a, u^+_0,u^-_0)$.   In order to define the glued map we associate with  $a\neq 0$, the pair $(R, \vartheta)$ defined by 
 $$R=\varphi (\abs{a})\quad \text{and}\quad a=\abs{a}e^{-2\pi i \vartheta},$$
 where, as before, $\varphi$ is the exponential gluing profile $\varphi (r)=e^{\frac{1}{r}}-e$. In view of the above representations of the diffeomorphisms $u^\pm_0$ we define the pair $(R', \varphi')$ by 
 $$R'=R+d_0^+-d^-_0\quad \text{and}\quad \vartheta'=\vartheta+\vartheta_0^+-\vartheta^-_0.$$
 The pair $(R', \vartheta')$ is associated with the gluing parameter $b=b(a, u^+_0,u^-_0)$. 
 \begin{defn}\label{box-map}
 If  $\abs{a}$ is sufficiently small we define the map 
 $\boxplus_a(u^+_0,u^-_0):Z_a\to  Z_b$ as follows. 
 If $a=0$, we set 
$$\boxplus_0(u^+_0,u^-_0)=(u^+_0,u^-_0),$$
and if $a\neq 0$, we define
\begin{equation*}
\begin{split}
&\boxplus_a(u^+_0,u^-_0) ([s,t])\\
&\phantom{==}=[(s, t)+(d_0^+,\vartheta_0^+)+\oplus_a(r^+_0, r^-_0)([s,t])]\\
&\phantom{==}=[(s,t)+(d_0^+,\vartheta_0^+)+\beta_a(s) \cdot r^+_0(s,t)
+(1-\beta_a(s))\cdot r^-_0(s-R,t-\vartheta)].
\end{split}
\end{equation*}
\end{defn}

The coordinates in the domain of  definition of $\boxplus_a(u^+,u^-)$ are $[s,t]$ where $s\in [0,R]$ 
and  the coordinates in the target  cylinder $Z_b$ are $[S,T]$ where  $S\in [0,R']$  and where $R'=R+d^+_0-d^-_0$ as defined above.

Note that  since $\beta_a (s)=1$ if $s\leq \frac{R}{2}-1$, we have $\boxplus_0(u^+_0,u^-_0)([0, t])=[0, T(t)]$ and since $\beta_a (s)=0$ if $s\geq \frac{R}{2}+1$, we have   $\boxplus_0(u^+_0,u^-_0)([R, t])=[R+d^+_0-d^-_0, t+\vartheta^+_0-\vartheta_0^++\wh{T}(t)]$.

The map $\boxplus_a(u^+,u^-)$ involves  the asymptotic data of the two maps $u^\pm$ and incorporates them into the twist
of the target cylinder. The choices involved in the construction
are subject to further constraints later on.

If the gluing parameter $0<|a|$ is  sufficiently small (hence $R$ is sufficiently large),   the map $\boxplus_a(u^+,u^-)$ defines  a diffeomorphism between finite cylinders $Z_a$ and $Z_b$ mapping the marked points  $p^\pm_a$ onto the marked points  $p^\pm_b$. These are the points on $Z_a$ and $Z_b$ corresponding to the original boundary points $(0,0)$.

Now recall that the subspace $\wh{E}_0\subset \wh{E}$ consists of pairs $(h^+, h^-)$ of mappings in $H_c^{3, \delta_0}(\R^\pm\times S^1, \R^2)$ satisfying $h^\pm (0, 0)=(0, 0)$ and $h^\pm (0, t)\in \{0\}\times \R\subset \R^2$ for all $t\in S^1$.  If 
$$h^\pm =h^\pm_{\infty}+r^\pm, $$
with the asymptotic constants $h^\pm_{\infty}$ of  $h^\pm$,  we obtain the mappings $u_0^\pm+h^\pm:\R^\pm\times S^1\to \R^\pm\times S^1$
represented by 
$$(u_0^++h^+)(s, t)=(s, t)+(d^+_0, \vartheta^+_0)+h_{\infty}^++r^+_0(s, t) +r^+(s, t),\quad s\geq 0$$
and 
$$(u_0^-+h^-)(s', t')=(s', t')+(d^-_0, \vartheta^-_0)+h_{\infty}^-+r^-_0(s', t') +r^-(s', t'),\quad s'\leq 0.$$

If  the norms of $r^\pm$ are sufficiently small,  the maps  $u^\pm_0+h^\pm$  are  still diffeomorphisms of the cylinders $\R^\pm \times S^1$ leaving the boundaries 
$\partial (\R^\pm \times S^1)=\{0\}\times S^1$ invariant and fixing the points $(0, 0)$. Therefore, if $\abs{a}$ is sufficiently small, the glued map 
$$\boxplus_a(u^+_0+h^+,u^-_0+h^-):Z_a\to  Z_b$$
for the gluing parameter $b=b(a, u^+_0+h^+,u^-_0+h^-)$,  is a diffeomorphism between finite cylinders preserving the distinguished points. The gluing parameter 
$b(a, u^+_0+h^+,u^-_0+h^-)$ is associated with the pair $(R', \vartheta')$ and is given by 
$$
(R',\vartheta')=(R,\vartheta) + (d_0^+-d_0^-,\vartheta^+_0-\vartheta^-_0)+h^+_\infty-h^-_\infty.
$$
\begin{lem}\label{abmap}
The map $B_{\frac{1}{2}}\oplus H_c^{3,\delta_0}(\R^+\times S^1, \R^2)\oplus H_c^{3,\delta_0}(\R^-\times S^1, \R^2)\to \C$, defined by 
$$(a, h^+, h^-)\mapsto b(a, u^+_0+h^+, u^-_0+h^-),$$
is smooth on every level. In particular, it is sc-smooth.
\end{lem}
\begin{proof}
 If $a=\abs{a}e^{-2\pi i \vartheta}$ and $b=\abs{b}e^{-2\pi i \vartheta'}$, we have defined $R=\varphi (\abs{a})$ and $R'=\varphi (\abs{b})$, with the gluing profile $\varphi (r)=e^{\frac{1}{r}}-e$ for $0<r\leq 1$.  By construction 
$(R', \vartheta')=(R, \vartheta)+(d^+-d^-, \vartheta^+-\vartheta^-)+(h^+_{\infty}- h^-_{\infty})$. The maps $h^\pm \mapsto h^\pm_{\infty}$ are  linear projection and hence smooth on every level. 

On the other hand, the function $B$ defined by $B(r, c):=\varphi^{-1}(\varphi (r)+c)$ if $r>0$ and $B(0, c)=0$ is smooth by Lemma \ref{lem5.2} of the appendix. It follows that the map 
$(a, h^+, h^-)\mapsto b(a, u_1^++h^+, u_1^-+h^-)$ is indeed smooth on every level and, in particular, sc-smooth  as desired.
\end{proof}

Explicitly, the glued map is computed to be 
\begin{equation*}
\begin{split}
&\boxplus_a(u^+_0+ h^+,u^-_0+h^-)([s,t])\\
&=\left[ (s,t)+(d^+_0,\vartheta^+_0)+h^+_\infty+\oplus_a(r^+_0+r^+, r^-_0+r^-)([s, t]\right] \\
&=\left[\boxplus_a(u^+_0,u^-_0)([s,t]) + h^+_\infty +\oplus_a(r^+,r^-)([s,t])\right]\\
&=\left[\boxplus_a(u^+_0,u^-_0)([s,t])+\oplus_a(h^+,h^-+h^+_\infty-h^-_\infty)([s,t])\right]
\end{split}
\end{equation*}
where  $[\ , \ ]$  denotes  the equivalence class of coordinates in $Z_a$ as well as in and $Z_b$.

\begin{rem} The following remark will be made precise later on. If $w_{a_0}:Z_{a_0}\to  Z_{b_0}$
is the  diffeomorphism constructed by means of  $w_{a_0}=\boxplus_{a_0}(u^+_0,u^-_0)$ and $b_0=b(a_0,u^+_0,u^-_0)$, then
given $(a,b,w)$ close to $(a_0,b_0,w_{a_0})$ and abbreviating $w_a=\boxplus_{a}(u^+_0,u^-_0)$
we shall solve later on the equation
$$
w_a +\oplus_a(h^+,h^-+h^+_\infty-h^-_\infty) =w
$$
for $(h^+,h^-)\in \wh{E}_0$ small and satisfying, in addition, $\rho_a(h^+,h^-)=(h^+,h^-)$. As we shall see  this problem has a
unique solution which allows  to construct  M-polyfold charts  for suitable spaces  later on.
\end{rem}

\subsection{An M-Polyfold Construction, Proof of Theorem \ref{thm-comp}}\label{subsub}
The subsection \ref{subsub} is devoted to the proofs of Proposition \ref{prop1.32}  and Theorem \ref{thm-comp} about the  M-polyfold structures on the set $X$ and on its ``completion'' $\ov{X}$.

In order to recall the definition of the set $X$ we denote by 
$\Gamma$  the set of all pairs $(a,b)$ of complex numbers satisfying $a\cdot b\neq 0$ and $\abs{a}, \abs{b}<\varepsilon$.  The size of $\varepsilon$ will be adjusted during the proof. 

\begin{defn}\label{Xdef}
The set $X$ consists of  all triples $(a,b,w)$ in which $a, b\in \Gamma$ and $w:Z_a\rightarrow Z_b$ is a $C^1$-diffeomorphism  between the two finite glued cylinders  belonging  to  the Sobolev class $H^3$ and satisfying  $w(p^\pm_a)=p^\pm_b$. 
\end{defn}

In order to define a topology on $X$  we fix a point $(a_0,b_0,w_0)\in X$ and choose a family $a\mapsto  \phi_a$ of diffeomorphisms 
$$
\phi_a:Z_{a}\rightarrow Z_{a_0}
$$
defined for  gluing parameters $a$ close to $a_0$, mapping  $p^{\pm}_{a}$ to $p^{\pm}_{a_0}$, and satisfying $\phi_{a_0}=\id$.  We assume that the family $a\mapsto \phi_a$ is smooth in the following sense. 
On the finite cylinder $Z_{a_0}$ we have the global coordinates $Z_{a_0}\ni [s,t]\rightarrow (s,t)\in [0,R_0]\times S^1$ and $[s',t']'\rightarrow (s',t')\in [-R_0,0]\times S^1$. Similarly, we have  two  global coordinates on $Z_a$. If $a$ is close to $a_0$, we can express the map $
\phi_a:Z_{a}\rightarrow Z_{a_0}$  with respect to four choices
of coordinate systems, two in the domains and two in the target,  namely
\begin{itemize}
\item $(s,t)\mapsto  [s,t]\xrightarrow{\phi_a} [S,T]\mapsto  (S,T)$
\item $(s,t)\mapsto  [s,t]\xrightarrow{\phi_a} [S',T']'\mapsto  (S',T')$
\item $(s',t')\mapsto  [s',t']'\xrightarrow{\phi_a}[S,T]\mapsto  (S,T)$
\item $(s',t')\mapsto  [s',t']'\xrightarrow{\phi_a}[S',T']\mapsto  (S',T').$
    \end{itemize}
The family $a\mapsto \phi_a$ is called smooth if  all these coordinate expressions are smooth as maps of $(a,s,t)$, respectively of  $(a,s',t')$, for $a$ close to $a_0$. Similarly, 
we can choose the second smooth family $b\mapsto \psi_b$ of diffeomorphisms 
$$
\psi_b:Z_{b_0}\rightarrow Z_b
$$
defined for  gluing parameters $b$ close to $b_0$, mapping  $p^{\pm}_{b_0}$ to $p^{\pm}_{b}$, and satisfying $\psi_{b_0}=\id$.  Given an open neighborhood  $U(w_0)$ of the diffeomorphism $w_0$  in the space
of diffeomorphisms $Z_{a_0}\rightarrow Z_{b_0}$ of class $H^3$,  we introduce  the set 
$$
{\mathcal U}(a_0,b_0,w_0,U(w_0),\delta_0)
$$
consisting of triples $(a, b, u)$ satisfying 
$$\text{$\abs{a-a_0}<\delta_0,\quad  \abs{b-b_0}<\delta_0,\quad u=\psi_b\circ w\circ\phi_a$ and  $w\in U(w_0)$}.$$
\begin{equation*}
 \begin{CD}
Z_{a_0}@>w>> Z_{b_0}\\
           @A\varphi_aAA       @VV\psi_bV    \\
 Z_a@>>\psi_b\circ w\circ \varphi_a > Z_b.
  \end{CD}
 \end{equation*}
 \mbox{}\\[1ex]
Clearly,  the chosen  triple $(a_0, b_0, w_0)$ belongs  to the set ${\mathcal U}(a_0,b_0,w_0,U(w_0),\delta_0)$ so that the collection of the sets  ${\mathcal U}(a_0,b_0,w_0,U(w_0),\delta_0)$  covers the  set $X$. 

\begin{lem}\label{lem-top1}
We abbreviate 
$${\mathcal U}_0= {\mathcal U}(a_0, b_0,w_0, U(w_0),\delta_0 )\quad \text{and}\quad {\mathcal U}_1={\mathcal U}(a_1,b_1,w_1,U(w_1),\delta_1).$$
If $(a_2, b_2, w_2)\in  {\mathcal U}_0\cap {\mathcal U}_1$, then 
there exists a set ${\mathcal U}_2={\mathcal U}(a_2,b_2,w_2,U(w_2),\delta_2)$ satisfying 
$${\mathcal U}_2\subset {\mathcal U}_0\cap {\mathcal U}_1.$$
\end{lem}
\begin{proof}
We consider the set of diffeomorphisms  ${\mathcal U}_0= {\mathcal U}(a_0, b_0,w_0, U(w_0),\delta_0 )$ consisting  of points $(a, b, u)$ satisfying $\abs{a-a_0}<\delta_0, \abs{b-b_0}<\delta_0$ and $u=\psi^0_b\circ w\circ \varphi^0_a$ for a diffeomorphism $w\in  U(w_0)$. Here $\psi^0_b:Z_{b_0}\to Z_b$ and 
 $\varphi^0_a:Z_{b}\to Z_{a_0}$ are the associated smooth families of diffeomorphisms. Moreover, 
 $U(w_0)$ is an $H^3$-open neighborhood of $C^1$-diffeomorphisms of $w_0:Z_{a_0}\to Z_{b_0}$. Similarly, we consider the second set 
${\mathcal U}_1= {\mathcal U}(a_1, b_1,w_1,U(w_1),\delta_1 )$ of points 
$(a, b, u)$ satisfying 
$\abs{a-a_1}<\delta_1$, $ \abs{b-b_1}<\delta_1$ and $u=\psi^1_b\circ w\circ \varphi^1_a$ for  $w\in  U(w_1)$, 
where $w_1:Z_{a_1}\to Z_{b_1}$ is a $C^1$-diffeomorphism in $H^3$. If 
$$(a_2, b_2, w_2)\in  {\mathcal U}_0\cap {\mathcal U}_1$$
then 
$$w_2=\psi^0_{b_2}\circ v_0\circ \varphi^0_{a_2}=\psi^1_{b_2}\circ v_1\circ \varphi^1_{a_2}$$
for two diffeomorphisms $v_0\in U(w_0)$ and $v_1\in U(w_1)$. In view of the remarks (Theorem \ref{thm-125p}) about the action by diffeomorphisms we find an open neighborhoods $V(v_0)\subset U(w_0)$ and $V(v_1)\subset U(w_1)$ and a number $\delta_2>0$ sufficiently small so that 
\begin{equation}\label{bases1}
\psi^0_{b}\circ  V(v_0)\circ \varphi^0_{a}\subset \psi^1_{b}\circ  V(v_1)\circ \varphi^1_{a}
\end{equation}
for all $\abs{a-a_2}<\delta_2$ and $\abs{b-b_2}<\delta_2$. The smooth families of diffeomorphisms, defined by 
\begin{equation*}
\begin{aligned}
\wt{\varphi}_a&=(\varphi^0_{a_2})^{-1}\circ \varphi^0_{a}:Z_a\to Z_{a_2}\\
\wt{\psi}_a&=\psi^0_b\circ (\psi^0_{b_2})^{-1}:Z_{b_2}\to Z_{b}
\end{aligned}
\end{equation*}
for  $\abs{a-a_2}<\delta_2$ and $\abs{b-b_2}<\delta_2$ satisfy $\wt{\varphi}_{a_2}=\id$ on $Z_{a_2}$ and $\wt{\psi}_{b_2}=\id$ on $Z_{b_2}$ and map the points $p^\pm_{a}$ onto $p^\pm_{a_2}$, respectively $p^\pm_{b_2}$ onto $p^\pm_{b}$. Next we define the open neighborhood in $H^3$ of the $C^1$-diffeomorphism $w_2:Z_{a_2}\to Z_{b_2}$ by 
$$U(w_2)=\psi^0_{b_2}\circ V(v_0)\circ \varphi^0_{a_2},$$
and introduce the set 
\begin{equation*}
\begin{gathered}
{\mathcal U}_2= {\mathcal U}(a_2, b_2,w_2, U(w_2),\delta_2 )\\
=\{(a, b, w)\vert \, \text{$\abs{a-a_2}<\delta_2, \abs{b-b_2}<\delta_2$
and $w=\wt{\psi}^0_b\circ u\circ \wt{\varphi}^0_a$ for some  $u\in  U(w_2)$}\}
\end{gathered}
\end{equation*}
If $(a, b, w)\in {\mathcal U}_2$, then $w=\wt{\psi}^0_b\circ u\circ \wt{\varphi}^0_a$  and   $u\in  U(w_2)$ and hence there exists $v\in V(v_0)$ satisfying 
$$w=\wt{\psi}^0_b\circ\psi_{b_2}^0\circ  v\circ\varphi^0_{a_2}\circ  \wt{\varphi}^0_a=
\psi_{b}^0\circ  v \circ\varphi^0_{a}.$$
Since $v\in V(v_0)\subset U(w_0)$, we conclude that $(a, b, w)\in {\mathcal U}_0$. In view of \eqref{bases1}, we also conclude that 
$w=\psi^1_b\circ\wt{ v}\circ\varphi^1_a$ for some $\wt{v}\in V(v_1)\subset U(v_1)$ and hence $(a, b, w)\in {\mathcal U}_1$. Consequently, ${\mathcal U}_2\subset {\mathcal U}_0\cap {\mathcal U}_1$ as claimed in Lemma \ref{lem-top1}.
\end{proof}
%The easy proof, involving the results about actions by smooth maps, is left to the reader.}

 Lemma \ref{lem-top1} shows  that  the collection $\{{\mathcal U}(a,b,w,U(w),\delta)\}$  defines the basis for the topology ${\mathcal T}$ on $X$. This topology is second countable and parcompact,  and hence metrizable.

The construction used in the definition of  the topology of $X$ allows  to define M-polyfold charts on $X$ as follows. 

We choose a point  $(a_0,b_0,w_0)\in X$ where 
$$
w_0:Z_{a_0}\rightarrow Z_{b_0}
$$
is a $C^1$-diffeomorphism belonging to  $H^3$ and mapping $p^\pm_{a_0}$ to $p^\pm_{b_0}$. There exists an $\varepsilon_0>0$
so that for given $h\in H^3(Z_{a_0},{\mathbb R}^2)$ satisfying
$h(p^\pm_{a_0})=(0, 0)$ and $h([0,t])\in \{0\}\times{\mathbb R}$, $h([0,t]')\in \{0\}\times {\mathbb R}$, and
$|D^{\alpha}h([s,t])|<\varepsilon_0$ for $\abs{\alpha}\leq 1$,  the map
$$
[s,t]\rightarrow w_0([s,t])+h([s,t])
$$
is still  a $C^1$-diffeomorphism $Z_{a_0}\rightarrow Z_{b_0}$ between the  finite glued cylinders.
Let us denote by $\wh{H}^3(Z_a,{\mathbb R}^2)$ the closed subspace of $H^3(Z_a,{\mathbb R}^2)$ consisting of maps $h$ satisfying $h(p^{\pm}_a)=(0,0)$, 
$h([0,t])\in\{0\}\times {\mathbb R}$ and $h([0,t]')\in \{0\}\times {\mathbb R}$.
Moreover, every  diffeomorphism $Z_{a_0}\to Z_{b_0}$ in $H^3$  sufficiently close to
$w_0$ in the $C^1$-norm can be written in such a way for a uniquely determined $h$. 
\begin{rem}\label{remark2-nnew}
Let us note the following important fact.
Take $\varepsilon_1\in (0,\varepsilon_0)$. Then we find a 
diffeomorphism $w_1:Z_{a_0}\rightarrow Z_{b_0}$ which is smooth and close to $w_0$ so that
the previous discussion is valid for $w_1$ with $\varepsilon_0$ replaced by $\varepsilon_1$. In addition, 
$w_0=w_1+h_0$ for a suitable $h_0$ which is controlled by $\varepsilon_1$. Hence we may assume without loss of generality
that in the triple $(a_0,b_0,w_0)$ chosen above,  the map  $w_0:Z_{a_0}\to Z_{b_0}$ is a smooth diffeomorphism
and the collection of triples $(a_0,b_0,w_0+h)$ where $h$ satisfies the  conditions described above contains an a-priori given triple in $X$.
\end{rem}

Now, assuming that  $w_0:Z_{a_0}\to Z_{b_0}$ is  a  smooth  diffeomorphism, we choose 
 two  smooth (in the sense explained after the Definition \ref{Xdef}) families $a\mapsto \phi_a$ and $b\mapsto \psi_b$ of diffeomorphisms 
$$\phi_a:Z_a\to Z_{a_0}\quad \text{and}\quad \psi_b:Z_{b_0}\to Z_b$$ 
defined for $a$ close to $a_0$ and $b$ close to $b_0$  
and introduce the map
$$
(a,b,h)\mapsto  (a,b,\psi_b\circ (w_0+h)\circ \phi_a)\in X
$$
into the space $X$, defined  for  triples $(a,b,h)$ in which  $(a,b)$ is close to $(a_0,b_0)$ and $h$  is varying in  the subspace $\wh{H}^3(Z_{a_0}, \R^2)$ of $H^3(Z_{a_0}, \R^2)$  as described above. The domain of definition of the map is a set in a neighborhood of the point $(a_0, b_0, 0)$ in the sc-Hilbert space
$$\C\oplus \C\oplus \wh{H}^{3}(Z_{a_0}, \R^2).$$
By definition of the topology, 
this map is continuous and obviously a homeomorphism onto some open neighborhood of  the point $(a_0,b_0,w_0)\in X$. Hence we may view it as the inverse of a chart.

We consider two such charts around the points  $(a_0,b_0,w_0)$ and $(\wt{a}_0,\wt{b}_0,\wt{w}_0)\in X$. 
The inverse of the second chart is the map 
$$(\wt{a}, \wt{b}, \wt{h})\mapsto  (\wt{a}, \wt{b}, \wt{\psi}_{\wt{b}}\circ (\wt{w}_0+\wt{h})\circ \wt{\varphi}_{\wt{a}})
$$
with $\wt{h}\in \wh{H}^3(Z_{\wt{a}_0}, \R^2)$. At the points  where the two charts intersect we have $\wt{a}=a$ and $\wt{b}=b$ and hence 
$\psi_b\circ (w_0+h)\circ \varphi_a=\wt{\psi}_b\circ (\wt{w}_0+\wt{h})\circ \wt{\varphi}_a$. Therefore, the transition map of the charts is of the form 
$$
(a,b,h)\mapsto (a, b, \wt{h})
$$
where $\wt{h}(a, b, h)\in H^3(Z_{\wt{a}_0}, \R^2)$ is defined by 
$$
 \wt{h}(a, b, h):=(\wt{\psi}_b^{-1}\circ\psi_b\circ (w_0+h)\circ \phi_a\circ\wt{\phi}_a^{-1})-\wt{w}_0.
$$
The map $a\mapsto \varphi_a\circ  \wt{\varphi}^{-1}_{a}$ is a smooth family of diffeomorphisms $Z_{\wt{a}_0}\to Z_{a_0}$,  and $b\mapsto \wt{\psi}^{-1}_{b}\circ \psi_b$ is a smooth family of diffeomorphisms $Z_{b_0}\to Z_{\wt{b}_0}$. Recalling the result about diffeomorphism actions we conclude from Theorem \ref{thm-125} together with the chain rule that the map 
$(a, b, h)\mapsto \wt{h}(a, b, h)$ is an sc-smooth map.

Having proved that the transition maps between the M-polyfold charts are sc-smooth, we have equipped the set $X$ with the structure of an M-polyfold.  
\begin{rem}\label{rexremark}
If we equip the (classical) Hilbert manifold of diffeomorphisms
$Z_{a_0}\rightarrow Z_{b_0}$ of class $H^3$ preserving the distinguished points,  
with the filtration for which the  level $m$ corresponds to the Sobolev regularity $m+3$, 
we obtain an sc-manifold taking  as charts the ones
coming from exponential maps. We refer to  \cite{El} for the classical set-up. For this sc-manifold,  the map
$$
(a,b,u)\mapsto  (a,b,\psi_b\circ u\circ \phi_a)
$$
establishes  an sc-diffeomorphism between an open neighborhood
of the triple $(a_0,b_0,u_0)$ in which $u_0$ is viewed as an element in the latter defined space,  and an open neighborhood of the same triple in the former M-polyfold $X$.
\end{rem}

Let us consider two different points $(a_1,b_1,w_1)$ and $(a_2,b_2,w_2)$ in $X$. We can connect $(a_1,b_1,w_1)$ by a continuous path to an element of the form $(a_2,b_2,w_3)$. We have to keep in mind that
the space of end-point preserving diffeomorphisms is disconnected
(think about Dehn twists). However keeping $a_2$ fixed we can vary $b_2$ and $w_3$ and connect it with $(a_2,b_2,w_2)$. Therefore, the topological space $X$ is connected. Since it carries a second countable paracompact topology, the proof of Theorem \ref{prop1.32} is complete.  \hfill $\blacksquare$
\mbox{}\\[0.5ex]

Compared to the general setting, the M-polyfold structure for $X$ is quite special insofar as the local models are open subsets of sc-Hilbert spaces. This will change in the next step where we 
 ``complete''  the space $X$ to the  space
$\ov{X}$ by adding the elements  corresponding to the gluing  parameter
values  $a=0$ and $b=0$.

We fix a number $\delta_0\in (0,2\pi)$ and denote by ${\mathcal D}={\mathcal D}^{3,\delta_0}$  the space consisting of pairs
$(u^+,u^-)$ in which the maps  $u^\pm:{\mathbb R}^\pm\times S^1\rightarrow {\mathbb R}^\pm\times S^1$ are $C^1$-diffeomorphisms belonging  to $H^{3}_{\loc}$ and satisfying 
$$
u^\pm(0,0)=(0,0).
$$
Moreover, we assume that   there exist asymptotic constants $(d^\pm,\vartheta^\pm)\in {\mathbb R}\times S^1$,
so that 
$$
u^{\pm}(s,t)=(s+d^\pm,t+\vartheta^\pm)+r^\pm(s,t),
$$
where the maps  $r^\pm$ belong to $H^{3,\delta_0}({\mathbb R}^\pm\times S^1,{\mathbb R}^2)$. We define an sc-structure on ${\mathcal D}$ by declaring the  level $m$ to consist of elements of regularity $(m+3,\delta_m)$ where 
$(\delta_m)$ is a strictly increasing sequence of real numbers contained in $(0,2\pi)$ and starting with the previously chosen $\delta_0$.

We recall from Section \ref{illustration} that  the set  $\ov{X}$ is defined  as the disjoint union 
$$
\ov{X}= X\coprod \bigl(\{(0,0)\}\times {\mathcal D}\bigr).
$$
We shall  construct charts around points of the form $(0,0,u^+,u^-)$, which are compatible with the M-polyfold structure already defined for $X$ and, of course, are compatible among themselves. Before we do that
we define the topology on $\ov{X}$.

We consider a sufficiently small open neighborhood
$U(u^+_0,u^-_0)$ in ${\mathcal D}$ for the $H^{3,\delta_0}$-topology on ${\mathcal D}$.
Then there exists  $\sigma_0>0$ so that for
$0<|a|<\sigma_0$ the glued maps  $\boxplus_a(u^+,u^-)$ are  diffeomorphisms  $Z_a\rightarrow Z_b$ between finite cylinders, where $(u^+,u^-)\in U(u^+_0,u^-_0)$ and 
 $b=b(a,u^+,u^-)$ and $|a|,|b|<\varepsilon$. If $a=0$,  we have,  by definition,  
$\boxplus_0(u^+,u^-)=(u^+,u^-)$ and $b(0, u^+, u^-)=0$.  We  introduce  for $\sigma\in(0,\sigma_0)$ the set  ${\mathcal U}_\sigma$ of triples 
\begin{equation}\label{topx}
{\mathcal U}_\sigma=\{(a,b(a,u^+,u^-),\boxplus_a(u^+,u^-))\ \vert \, 0\leq \abs{a}<\sigma,\ (u^+,u^-)\in U(u^+_0,u^-_0)\}.
\end{equation}
We  have already equipped the set $X$ with a metrizable topology.
The following lemma shows that the collection of subsets of type ${\mathcal U}_{\sigma}$ in $\ov{X}$ as 
defined in \eqref{topx}, is compatible with the topology on $X$.
\begin{prop}\label{topologyc}
Let ${\mathcal U}={\mathcal U}_\sigma$ be as defined in (\ref{topx}). Then the set $X\cap {\mathcal U}$ is open in $X$. Given two sets ${\mathcal U}$ and ${\mathcal V}$ as defined in (\ref{topx}) and a point
$(0,0,u^+_0,u^-_0)\in\bar{X}$ which belongs to ${\mathcal U}\cap{\mathcal V}$, there exists a third set ${\mathcal W}$ constructed according to the recipe (\ref{topx}) centered at $(0,0,u^+_0,u^-_0)$ so that
$$
(0,0,u^+_0,u^-_0)\in {\mathcal W}\subset {\mathcal U}\cap {\mathcal V}.
$$
Hence there exists a unique topology on $\bar{X}$ for which the open sets in $X$ and the new sets just introduced form  a basis. Moreover, this topology is second countable and metrizable. Further,  $\ov{X}$ equipped with this topology is connected and $X$ is open and dense in $\bar{X}$.
\end{prop}
\begin{proof}
We first consider the case that $(a_0, b_0, w_0)\in X\cap {\mathcal U}_{\sigma}$, so that 
$$(a_0, b_0, w_0)=(a_0, b(a_0, u^+_0, u^-_0), \boxplus_{a_0}(u^+_0, u^-_0))\in 
 {\mathcal U}_{\sigma}$$
 where $0<\abs{a_0}<\sigma$ and $w_0= \boxplus_{a_0}(u^+_0, u^-_0)$  is a diffeomorphism
$Z_{a_0}\to Z_{b_0}$ for $b_0=b(a_0,  u^+_0, u^-_0)$. We  have to show that a small open neighborhood ${\mathcal O}$ of $(a_0, b_0, w_0)$ in $X$ is contained in $ {\mathcal U}_{\sigma}$. We  recall that 
an open neighborhood  ${\mathcal O}$ of the point $(a_0, b_0, w_0)\in X$ may be assumed  of the form
$${\mathcal O}=
\{(a, b, \psi_b\circ u\circ \phi_a)\ \vert \,\text{$\abs{a-a_0}<\varepsilon, \abs{b-b_0}<\varepsilon$ and $u\in U(w_0)$}\}$$
where $\varepsilon>0$ and where $U(w_0)$ is  a small $H^3$-neighborhood of $w_0$ in the set of diffeomorphisms $Z_{a_0}\to Z_{b_0}$ of Sobolev class $H^3$ and fixing  the distinguished points.

If $a=a_0$ and $b=b_0$, then $u=w_0$ so that 
$$\psi_{b_0}\circ w_0\circ \phi_{a_0}=w_0.$$
We have used that  $\phi_{a_0}=\id$ on $Z_{a_0}$ and $\psi_{b_0}=\id$ on $Z_{b_0}$.

We recall that  the pair $(h^+, h^-)\in \wh{E}_0$ consists of maps $h^\pm:\R^\pm \times S^1\to \R^2$ of the form 
$$h^\pm=h^\pm_{\infty}+r^\pm$$
where $r^\pm\in H^{3, \delta_{0}}(\R^\pm\times S^1, \R^2)$ and where $h^\pm (0, 0)=(0, 0)$ and 
$h^\pm (0, t)\in \{0\}\times \R$ for all $t\in \R$.

 If $\varepsilon>0$ in the  definition of ${\mathcal O}$ is sufficiently small and $(a, b, \psi_{b}\circ u\circ \phi_a)\in {\mathcal O}$ we look for a solution $(h^+, h^-)\in \wh{E}_0$ of the following equations
\begin{equation}\label{nbox-eq}
\boxplus_a(u_0^++h^+, u_0^-+h^-)=\psi_b\circ u\circ \phi_a
\end{equation}
and 
\begin{equation}\label{b-eq}
b(a, u^+_0+h^+, u^-_0+h^-)=b
\end{equation}
so that 
$(a, b(a, u^+_0+h^+, u^-_0+h^-), \boxplus_a(u_0^++h^+, u^-+h^-))\in {\mathcal U}_{\sigma}$. We recall that $u^\pm_0:\R^\pm \times S^1\to \R^\pm \times S^1$ are diffeomorphisms of the form 
\begin{align*}
u^+_0(s, t)&=(s, t)+(d^+_0, \vartheta^+_0)+r_0^+(s, t),\quad s\geq 0\\
u^-_0(s', t')&=(s', t')+(d^-_0, \vartheta^-_0)+r_0^-(s', t'),\quad s'\leq 0
\end{align*}
and $r^\pm_0\in H^{3, \delta_0}(\R^\pm\times S^1, \R^2)$. From Section \ref{anothersplicing} we know  that 
\begin{equation*}
\begin{split}
\boxplus_a&(u_0^++h^+, u^-_0+h^-)([s,t])\\
&=\left[ \boxplus_a(u_0^+, u^-_0)([s,t])+\oplus_a(h^+, h^-+h^+_{\infty}-h^-_{\infty})([s,t])\right].
\end{split}
\end{equation*}
Hence,  abbreviating
$$ \boxplus_a(u_0^+, u^-_0)=w_a,$$
we have to solve the equation
\begin{equation}\label{eq35}
\oplus_a(h^+, h^-+h^+_{\infty}-h^-_{\infty})= \psi_b\circ u\circ \phi_a-w_a
\end{equation}
for the unknown maps $(h^+, h^-)\in \wh{E}_0$. The right hand side vanishes if $a=a_{a_0}$, $b=b_{0}$, and $u=w_0$ since $w_{a_0}=w_0$.
It actually suffices to solve for  $(q^+,q^-)\in E$  the equation 
\begin{equation}\label{klm1}
\oplus_a(q^+,q^-)= \psi_b\circ u\circ \phi_a-w_a.
\end{equation}
because the solution $(h^+, h^-)\in \wh{E}_0$ of \eqref{eq35} is then given by the formula
$$
(h^+, h^-)=(q^+,q^-+(h^-_{\infty}-h^+_{\infty})).
$$
Recall that $h^-_{\infty}-h^+_{\infty}$ can be computed from $a$ and $b$, so that $(h^+,h^-)$ is uniquely determined.
The equation \eqref{klm1} becomes uniquely solvable once we impose, in addition, the equation $\ominus_a(q^+, q^-)=0$. Abbreviating the right hand side  \eqref{klm1}  by
$$g=\psi_b\circ u\circ \phi_a-w_a,$$
the two equations 
\begin{equation}\label{newbox-eq}
\begin{aligned}
\oplus_a(q^+,q^-)=g\\
\ominus_a(q^+,q^-)=0
\end{aligned}
\end{equation}
have the following  unique solution.  Explicitly, the equations \eqref{newbox-eq} are represented by 
\begin{equation*}
\beta_a(s)\cdot q^+(s, t)+(1-\beta_a(s))\cdot q^-(s-R, t-\vartheta)=g(s, t)
\end{equation*}
\begin{equation*}
-(1-\beta_a(s))\cdot q^+(s, t)+\beta_a(s)\cdot  q^-(s-R, t-\vartheta)=(2\beta_a(s)-1)\av (q^+, q^-).
\end{equation*}
Integrating the first equation at $s=\frac{R}{2}$ over the circle $S^1$, we find in view of $\beta_a\left(\frac{R}{2}\right)=\frac{1}{2}$ that the average
$$
\av (q^+, q^-)=\int_{S^1}g\left(\frac{R}{2}, t\right)\ dt=[g]
$$
agrees with the mean value of the  function $g$. In matrix form the above two equations are now written as
\begin{equation*}
\begin{bmatrix}
\phantom{-}\beta_a(s)&(1-\beta_a(s))\\
-(1-\beta_a(s)&\beta_a(s)
\end{bmatrix}
\cdot 
\begin{bmatrix}
q^+(s, t)\\
 q^-(s-R, t-\vartheta)
\end{bmatrix}=
\begin{bmatrix}
g\\
(2\beta_a(s)-1)[g]
\end{bmatrix}.
\end{equation*}
Denoting by $\gamma_a(s)=\beta_a(s)^2+(1-\beta_a(s))^2$ the determinant of the matrix, one arrives at the following formulae for the solution,
\begin{equation}\label{eqqp}
q^+(s, t)=[g]+\frac{\beta_a(s)}{\gamma_a(s)}\cdot (g(s, t)-[g])
\end{equation}
and
\begin{equation}\label{eqqm}
q^{-}(s-R, t-\vartheta)=[g]+\frac{1-\beta_a(s)}{\gamma_a(s)}\cdot (g(s, t)-[g]).
\end{equation}
The asymptotic constants of $q^\pm$ are both equal to $[g]$. In particular, we find for the asymptotic constant of $h^+=q^+$ the value
\begin{equation}\label{asymptconst-new}
h^+_{\infty}=[g].
\end{equation}
In order to to solve the equation \eqref{b-eq} we associate with the gluing parameter $b$ the pair 
$(R', \vartheta')$ consisting of the gluing length and the gluing angle 
as the solution of the equation
\begin{equation}\label{rprime}
(R', \vartheta')=(R, \vartheta )+(d^+_0-d^-_0, \vartheta^+_0-\vartheta^-_0)+h^+_{\infty}-h^-_{\infty},
\end{equation}
then $b=b(a, u^+_0+h^+, u_0^-+h^-_{\infty})$ as desired  and we have proved that  the equations \eqref{nbox-eq} together with  \eqref{b-eq}  and $\ominus_a(h^+, h^-+h^+_{\infty}-h^-_{\infty})=0$ have a unique solution $(h^+, h^-)\in \wh{E}$.  Actually, $(h^+, h^-)\in \wh{E}_0$ as the next lemma shows.
\begin{lem}
The pair $(h^+, h^-)$ belongs to $\wh{E}_0$.
\end{lem}
 \begin{proof}
 We first calculate the values of $g=\psi_b\circ u\circ \phi_a-w_a$ at the points $(0, 0)$ and $(R, \vartheta)$ where $(R, \vartheta)$ are the gluing length and the gluing angle associated with the parameter $a$.  We recall that $\phi_a:Z_a\to Z_{a_0}$ is a  
 diffeomorphism mapping the distinguished points of the cylinder $Z_{a_0}$ onto the distinguished points of $Z_{a_0}$. The same holds true for the diffeomorphism  $\psi_b:Z_{b_0}\to Z_b$ and the diffeomorphism $u:Z_{a_0}\to Z_{b_0}$ fixes the distinguished points. Hence, 
 $\psi_b\circ u\circ \phi_a(0, 0)=(0, 0)$ and $\psi_b\circ u\circ \phi_a(R, \vartheta)=(R', \vartheta')$ where $(R', \vartheta')$ are the parameters associated with $b$. 
 In addition, since the  boundaries of $Z_a$ are mapped onto the corresponding boundaries of $Z_b$, we have $\psi_b\circ u\circ \phi_a(0, t)\in \{0\}\times \R$ and $\psi_b\circ u\circ \phi_a(R, \vartheta +t)\in \{(R', \vartheta')\}+(\{0\}\times \R)$ for all $t\in \R$. 
 
 To evaluate the  diffeomorphism $w_a=\boxplus_a (u^+_0, u^-_0)$ at the points $(0, 0)$ and $(R, \vartheta)$, we recall that  the diffeomorphisms $u^\pm_0:\R^\pm \times S^1\to \R^\pm \times S^1$ are of the form $u^\pm_0(s, t)=(s, t)+(d^\pm_0, \vartheta^\pm_0)+r^\pm_0(s, t)$ and satisfy $u^\pm_0(0, 0)=(0, 0)$.
 Using that  $\boxplus_a (u^+_0, u^-_0)(s, t)=(s, t)+(d^+_0, \vartheta^+_0)+\beta_ar_0^+(s, t)+(1-\beta_a)r_0^-(s-R, t-\vartheta)$, we find $w_a(0, 0)=(0, 0)$ and $w_a(0, t)\in \{0\}\times \R$. 
 Evaluating at $(R, \vartheta)$ we obtain,
 $w_a(R, \vartheta)=(R, \vartheta)+(d^+_0, \vartheta^+_0)+r_0^-(0, 0)=(R, \vartheta)+(d^+_0-d^-_0, \vartheta^+_0-\vartheta^-_0)+u^-_0(0, 0)=(R, \vartheta)+(d^+_0-d^-_0, \vartheta^+_0-\vartheta^-_0)$ and similarly 
 $$w_a(R, \vartheta+t)\in 
( (R, \vartheta)+(d^+_0-d^-_0, \vartheta^+_0-\vartheta^-_0))+(\{0\}\times \R).$$
Consequently,
$g(0,0)=(0, 0)$ and $g(R, \vartheta)=(R', \vartheta')-( (R, \vartheta)+(d^+_0-d^-_0, \vartheta^+_0-\vartheta^-_0))$. In addition,
$g(0, t)\in \{0\}\times \R$ and 
$$g(R, \vartheta+t)\in ((R', \vartheta')- (R, \vartheta)-(d^+_0-d^-_0, \vartheta^+_0-\vartheta^-_0))+
(\{0\}\times \R).$$ 

Now, since $h^+=q^+$, it follows from \eqref{eqqp} that $h^+(0, 0)=q^+(0,0)=g(0, 0)=(0, 0)$ and 
$h^+(0, t)=g(0, t)\in \{0\}\times \R$ for all $t\in \R$. Using \eqref{rprime} and \eqref{eqqm} and $h^-=q^-+h^-_{\infty}-h^+_{\infty}$,  we find that 
$h^-(0, 0)=q^-(0, 0)+h^-_{\infty}-h^+_{\infty}=g(R, \vartheta)+h^-_{\infty}-h^+_{\infty}=
(R', \vartheta')-( (R, \vartheta)+(d^+_0-d^-_0, \vartheta^+_0-\vartheta^-_0))+h^-_{\infty}-h^+_{\infty}=(0,0)$ and $h^-(0, t)=g(R, \vartheta)+h^-_{\infty}-h^+_{\infty}\in \{0\}\times \R$. 
Consequently,  the pair  $(h^+, h^-)$ belongs to $\wh{E}_0$ and the  proof of the lemma is complete.

 \end{proof}

 We recall that if 
$a=a_0$, $b=b_0$ and $u=w_0$, then $g=0$ and hence $h^+=0$ and $h^-=0$ and we conclude from the above formulae that if $\varepsilon>0$ is sufficiently small, then indeed ${\mathcal O}\subset U_{\sigma}$.

It remains  to consider the case in which $a_0=b_0=0$. We claim that 
given two sets $U_{\sigma_1}$  and $U_{\sigma_2}$  as defined in (33) and  a point $(0, 0, u^+_0, u^-_0)\in \ov{X}$ which belongs to $U_{\sigma_1}\cap U_{\sigma_2}$,  then there exists  a third set $U_{\sigma_0}$ such that 
$$(0, 0, u^+_0, u^-_0)\in U_{\sigma_0}\subset U_{\sigma_1}\cap U_{\sigma_2}.$$
Indeed, the sets $U_{\sigma_1}$ and $U_{\sigma_2}$ have the  form
\begin{align*}
U_{\sigma_1}&:=\{(a, b(a, u^+, u^-), \boxplus_a(u^+, u^-))\vert \, 0\leq \abs{a}<\sigma_1,\; (u^+, u^-)\in U(u_1^+, u^-_1)\}\\
U_{\sigma_2}&:=\{(a, b(a, u^+, u^-), \boxplus_a(u^+, u^-))\vert \, 0\leq \abs{a}<\sigma_2,\; (u^+, u^-)\in U(u_2^+, u^-_2)\}
\end{align*}
where $U(u^+_1, u^-_1)$ is an open neighborhood of the pair $(u^+_1, u^-_1)$ in ${\mathcal  D}$ for the $H^{3, \delta_0}$-topology on ${\mathcal D}$,  and similarly for $U(u_2^+, u^-_2)$. 
Since $(0, 0, u^+_0, u^-_0)\in U_{\sigma_1}\cap U_{\sigma_2}$,  
$$(u^+_0, u^-_0)\in U(u_1^+, u^-_1)\cap U(u_2^+, u^-_2).$$
We choose an open neighborhood $U(u^+_0, u^-_0)$ of the pair $(u^+_0, u^-_0)$ in ${\mathcal  D}$ for the $H^{3, \delta_0}$-topology  satisfying 
$$(u^+_0, u^-_0)\in U(u^+_0, u^-_0)\subset U(u_1^+, u^-_1)\cap U(u_2^+, u^-_2),$$
and set  $\delta_0=\min \{\delta_1, \delta_2\}$. Then   the set 
$U_{\sigma_0}$ , defined by 
$$U_{\sigma_0}:=\{(a, b(a, u^+, u^-), \boxplus_a(u^+, u^-))\vert \, \abs{a}<\sigma_0,\quad (u^+, u^-)\in U(u_0^+, u^-_0)\}, $$
satisfies
$$(0, 0, u^+_0, u^-_0)\in U_{\sigma_0}\subset U_{\sigma_1}\cap U_{\sigma_2}$$
as claimed. The proof of Proposition \ref{topologyc} is complete.

\end{proof}

\begin{lem}\label{new-lemma1}
 If the pair $(q^+, q^-)\in E$ is the unique solution of \eqref{newbox-eq}
 \begin{align*}
\oplus_a(q^+, q^-)&=g\\
\ominus_a(q^+, q^-)&=0,
\end{align*}
then the pair $(h^+, h^-):=(q^+, q^-+h^-_{\infty}-h^+_{\infty})\in \wh{E}_0$ satisfies 
$$\rho_a(h^+, h^-)=(h^+, h^-).$$
\end{lem}
\begin{proof}
The condition $\ominus_a( (q^+, q^-)=\ominus_a(h^+, h^-+h^+_{\infty}-h^-_{\infty})=0$ is the same as the condition $\ominus_a(h^+-h^+_{\infty}, h^- -h^-_{\infty})=0$ which is equivalent to the relation $\rho_a(h^+, h^-)=(h^+, h^-)$,  in view of Lemma \ref{kernel-lem1}.
\end{proof}

\begin{lem}\label{new-lemma2}
We assume that  $(u^+_0, u^-_0)\in {\mathcal D}^{\infty}$ are smooth diffeomorphisms of  the half-cylinders  $\R^\pm \times S^1$ and fix the smooth point $(a_0, b_0, w_0)\in X$ in which $w_0=\boxplus_{a_0}(u^+_0, u^-_0)$. Setting $g=\psi_b\circ u\circ \phi_a-w_a=w-w_a$, the mappings $X\to E$,
$$(a, b, w)\mapsto (q^+(a, b, w),q^-(a,b,w))\in E,$$
defined on a neighborhood of $(a_0, b_0, w_0)$ in $X$ as the unique solution of 
\eqref{newbox-eq}, are sc-smooth.
\end{lem}
\begin{proof}
We use the local sc-coordinates $(a, b, h)\mapsto (a, b, \psi_b\circ (w_0+h)\circ \phi_a)$ of $X$ near $(a_0, b_0, w_0)\in X$, where $h$ varies in the sc-Hilbert space $\wh{H}^{3, \delta_0}(Z_{a_0}, \R^2)$ of functions on the fixed cylinder $Z_{a_0}$. Then $g= \psi_b\circ (w_0+h)\circ \phi_a-w_a$ and the solutions $q^\pm$ of \eqref{newbox-eq} are in these coordinates mappings 
$(a, b, h)\mapsto q^\pm (a, b, h)$. Recalling the proof of Proposition \ref{topologyc}, the map $q^+=d+r^+$ is defined by $d=[g]$ and 
$$r^+(s, t)=\frac{\beta_a(s)}{\gamma_a(s)}\left(g-[g]\right)$$
for $s\geq 0$. Since $\beta_a(s)=0$ for $s\geq \frac{R}{2}+1$, the map $q^+$ is obtained from  the sole knowledge of $g$ on the finite piece  $[0,\frac{R}{2}+1]\times S^1$. Clearly, $w_a$ as a function of $a$  near $a_0$ is sc-smoothly depending on $a$ on the fixed cylinder $[0,\frac{R_0}{2}+2]\times S^1$. In view of Theorem \ref{thm-125} about diffeomorphisms actions, the function $g$ depends sc-smoothly on $(a, b, h)$ in the neighborhood of $(a_0, b_0, 0)$. The sc-smoothness of the map 
$(a, b, h)\mapsto q^+(a, b, h)$ now follows from the above formula for $r^+(s, t)$ using Proposition 
\ref{qed} and the chain rule. Similar arguments show that also $(a, b, h)\mapsto q^-(a, b, h)$ is an sc-smooth function.
\end{proof}

From Lemma \ref{new-lemma2} we deduce immediately the following result.
\begin{lem}\label{316-lem}
Fix the smooth point $(a_0, b_0, w_0)\in X$. Then there exists an open neighborhood $U\subset X$ of the point and an sc-smooth map 
$$\Gamma:U\to \wh{E}_0,$$
defined by $\Gamma (a, b, w)=(h^+, h^-)$ where $(h^+, h^-)=(q^+,q^-+h^+_{\infty}-h^+_{\infty})\in \wh{E}_0$ is the unique solution  of \eqref{b-eq} and \eqref{newbox-eq}.
\end{lem}
\begin{proof}
This follows immediately from Lemma \ref{new-lemma2} together with Lemma \ref{prop-xx} applied to $h^+_{\infty}=[g]$ and Lemma \ref{abmap} applied to $h^-_{\infty}.$
\end{proof}

\begin{lem}\label{317-lem}
We assume that $(u^+_0, u^-_0)\in {\mathcal D}^\infty$ are smooth diffeomorphisms of $\R^\pm \times S^1$. Then the map 
$$U(a_0, 0, 0)\subset \C\times \wh{E}_0\to X,$$
defined by
$$
(a, h^+, h^-)\mapsto (a, b(a, u^+_0+h^+, u^-_0+h^-), \boxplus_a(u^+_0+h^+,u^-_0+h^-))\in X$$
in a small open neighborhood of $(a_0, h^+, h^-)=(a_0, 0,0)$ where $a_0\neq 0$, is an sc-smooth map.
\end{lem}
\begin{proof}
In the local sc-chart around the point $(a_0, b_0, w_0)\in X$ which is defined by 
$(a, b, h)\mapsto (a, b, \psi_b\circ (w_0+h)\circ \phi_a)$, our map is represented by the formulae 
$b=b(a, u^+_0+h^+, u^-_0+h^-)$ and 
$$h=\psi_b\circ \boxplus_a(u^+_0+h^+, u^-_0+h^-)\circ \phi_a^{-1}-w_0.$$
Arguing now as in the proof of Lemma \ref{new-lemma2} one verifies that $(a, h^+, h^-)\mapsto h$ is an sc-smooth map.
\end{proof}

For the following it is important to observe that if $(a_0, b_0, \wt{w}_0 )\in X$ is a point in which
$\wt{w}_0=\boxplus_{a_0}(\wt{u}^+_0,\wt{u}^-_0 )$ for a pair $(\wt{u}^+_0, \wt{u}^-_0)$ of diffeomorphisms of $\R^\pm \times S^1$ belonging to the space ${\mathcal D}^{3, \delta_0}$, then there exists a smooth point $(a_0, b_0, w_0)\in X$ is which $w_0=\boxplus_{a_0}(u^+_0,u^-_0 )$  for a pair $(u^+_0, u^-_0)$ of smooth diffeomorphisms of $\R^\pm \times S^1$ belonging to the space ${\mathcal D}^{\infty}$, and an open neighborhood $U\subset X$ of $(a_0, b_0, w_0)$, which contains the original point $(a_0, b_0, \wt{w}_0)$.

In view of the above discussion, we have establish the following result. 
\begin{prop}
If $(u^+_0, u^-_0)\in {\mathcal D}^{\infty}$ is a pair of smooth diffeomorphisms of $\R^\pm \times S^1$ and if   $(a_0,b_0,w_0)$ is the smooth point in which
$
 w_0=\boxplus_{a_0}(u_0^{+},u_0^{-}),$
 then there exists an open neighborhood $U\subset X$ of this point and an sc-smooth map
 $$
\Gamma: U\rightarrow \wh{E}_0,
$$
defined by $\Gamma (a, b, w)=(h^+, h^-)$ and having the following properties.
\begin{itemize}
\item[$\bullet$] $\Gamma (a_0, b_0, w_0)=(0, 0)$
\item[$\bullet$]  $\boxplus_a[ (u^+_0,u^-_0)+\Gamma(a,b,w)]=w$.
\item[$\bullet$] $\rho_a (h^+, h^-)=(h^+, h^-)$.
\end{itemize}
\end{prop}

In view of our observation above, it is sufficient to introduce chart maps of $\ov{X}$ centered at smooth points $(0, 0, u^+_0, u^-_0)\in \ov{X}\setminus X$ where 
$(u^+_0,u^-_0)\in {\mathcal D}^\infty$ are smooth diffeomorphisms of $\R^\pm \times S^1$.
\begin{defn}\label{charts-new-def}
Given $(u^+_0,u^-_0)\in {\mathcal D}^\infty$ we define the map
$$\varphi:\{(a,h^+,h^-)\ \vert \,  \rho_a(h^+,h^-)=(h^+,h^-),|a|<\sigma_0, (h^+,h^-)\in U\}\rightarrow \ov{X}, $$
in which $U\subset \wh{E}_0$ is a small open neighborhood of $(0, 0)\in \wh{E}_0$, by 
$$\varphi(a,h^+,h^-)=(a,b(a,u^+_0+h^+,u^-_0+h^-),\boxplus_a(u^+_0+h^+,u^-_0+h^-))$$
if $a\neq 0$, and by the formula,
$$\varphi(0,h^+,h^-)=(0, 0, u_0^++h^+, u_0^-+h^-)$$
 in the case $a=0$.
\end{defn}
\begin{rem}
We would like to emphasize that due to the assumption $\rho_a(h^+, h^-)=(h^+, h^-)$ the map $\varphi$ is invertible. Indeed, since this assumption is equivalent to the requirement 
$\ominus_a(h^+-h^+_{\infty}, h^--h^-_{\infty})=0$, in view of Lemma \ref{kernel-lem1}, the injectivity of $\varphi$ follows by the arguments already used in the proof of Proposition \ref{topologyc}. Namely, if 
$$\varphi (a, h^+, h^-)=\varphi (\wt{a}, k^+, k^-),$$
then $a=\wt{a}$ and so $b(a, u^+_0+h^+, u^-_0+h^-)=b(a, u^+_0+k^+, u^-_0+k^-)$ and 
$\boxplus_a (u^+_0+h^+, u^-_0+h^-)=\boxplus_a (u^+_0+k^+, u^-_0+k^-)$ and the additional equations $\ominus_a(h^+-h^+_{\infty}, h^--h^-_{\infty})=0$ and $\ominus_a(k^+-k^+_{\infty}, k^--k^-_{\infty})=0$ imply that $h^+=k^+$ and $h^-=k^-$.

\end{rem}
From the previous results  we deduce the following result. 

\begin{prop}\label{319-prop}
The map $\varphi$ in Definition \ref{charts-new-def} restricted to triples $(a, h^+, h^-)$ satisfying 
$a\neq 0$ is an sc-diffeomorphism onto an open subset of the M-polyfold $X$.

\end{prop}

The above maps of Definition \ref{charts-new-def} cover the set $\ov{X}\setminus X$ and are compatible with the sc-msooth structure of the  M-polyfold $X$. Hence in order to establish an sc-smooth structure on $\ov{X}$  it remains to show that the chart transformations are sc-smooth in a neighborhood of $a=0$. We shall make use of the sc-smoothness results in chapter \ref{sec2}, in particular of Proposition \ref{qed} and of Lemma \ref{lem2-18}-\ref{lem2.21-nnew}.

\begin{prop}\label{prop3.14-new}
If  $\phi_1$ and $\phi_2$ are two chart maps of the topological space $\ov{X}$ as introduced in Definition  \ref{charts-new-def}, then the chart transformation 
 $$\phi_2^{-1}\circ\phi_1$$
  is an sc-smooth map.
\end{prop}
\begin{proof}
We consider two chart maps $\phi_1$ and $\phi_2$ into $\ov{X}$ according to 
Definition  \ref{charts-new-def} which are defined in a neighborhood of $a=0$. The first one is the map $\phi_1$ 
$$\{(a,h^+,h^-)\ \vert \,  \rho_a(h^+,h^-)=(h^+,h^-),0\leq 0\leq  |a|<\sigma_1, (h^+,h^-)\in U_1\}\rightarrow \ov{X},$$
defined by
$$
\phi_1(a,h^+,h^-):=(a,b(a,u^+_1+h^+,u^-_1+h^-),\boxplus_a(u^+_1+h^+,u^-_1+h^-))
$$ 
if $a\neq 0$, and by
$$\phi_1(a,h^+,h^-):=(0, 0, u^+_1+h^+,u^-_1+h^-)
$$ 
if $a=0$. Here $u_1^\pm:\R^\pm \times S^1\to\R^\pm \times S^1$ are two smooth diffeomorphisms of the positive resp. negative cylinder (in the covering spaces) of the form 
\begin{equation}\label{eq0-nnew}
\begin{split}
&\, u_1^+(s, t)\, =(s, t)+(d^+_1, \vartheta_1^+)+r_1^+(s, t),\quad s\geq 0\\
&u_1^-(s', t')=(s', t')+(d^-_1, \vartheta_1^-)+r_1^-(s', t'),\quad s'\leq 0
\end{split}
\end{equation}
where the pair $(r^+_1,  r^-_1)\in {\mathcal D}^{\infty}$ is a smooth point in the sc-Banach space $H^{3, \delta_0}(\R^+\times S^1, \R^2)\times H^{3, \delta_0}(\R^-\times S^1, \R^2)$. In addition, $U_1\subset \wh{E}_0$ is an open neighborhood of $(0, 0)\in \wh{E}_0$ consisting of pairs $(h^+, h^-)$ of maps $h^\pm :\R^\pm \times S^1\to  \R^2$ of the form 
$h^\pm (s, t)=h_{\infty}^\pm +r^\pm (s, t)$ and $r^\pm \in H^{3, \delta_0}(\R^+\times S^1, \R^2).$
Similarly, the second chart map $\varphi_2$, 
$$\{(a',k^+,k^-)\ \vert \,  \rho_{a'}(k^+,k^-)=(k^+,k^-),0\leq |a'|<\sigma_2, (k^+,h^-)\in U_2\}\rightarrow \ov{X}$$
is defined by 
$$
\phi_2(a',k^+,k^-):=(a',b'(a',u^+_2+k^+,u^-_2+k^-),\boxplus_{a'}(u^+_2+k^+,u^-_2+k^-))
$$
if $a' \neq 0$,  and by 
$$
\phi_2(a',k^+,k^-):=(0, 0, u^+_2+k^+,u^-_2+k^-)
$$
if $a'=0$.
The regions where the charts overlap are defined by 
\begin{equation}\label{eq1-nnew}
\phi_1(a,h^+,h^-)=\phi_2(a', k^+, k^-).
\end{equation}
If $a=0$, then it follows that also $a'=0$, and hence  $b=0$ and $b'=0$, so that 
\begin{equation*}
\begin{split}
\phi_1(0,h^+,h^-)&=(0, 0, u^+_1+h^+,u^-_1+h^-)\\
&=(0, 0, u^+_2+k^+,u^-_2+k^-)=\phi_2(0,k^+,k^-).
\end{split}
\end{equation*}
The map 
$\phi^{-1}_2\circ \phi_1(0, h^+, h^-)=(0, k^+, k^-)$ is therefore given by 
\begin{equation}\label{eq0-new}
\begin{aligned}
k^+&=h^++(u^+_1-u^+_2)\\
k^-&=h^-+(u^-_1-u^-_2).
\end{aligned}
\end{equation}
If $a\neq 0$, then $a'=a$ and hence 
$b(a,u^+_1+h^+,u^-_1+h^-)=b'(a,u^+_2+k^+,u^-_2+k^-)$ and 
\begin{equation}\label{eq1-new}
\boxplus_a(u^+_1+h^+,u^-_1+h^-)=\boxplus_{a}(u^+_2+k^+,u^-_2+k^-).
\end{equation}
In particular, the gluing lengths and gluing angles $(R_1, \vartheta_1)$ and $(R_2, \vartheta_2)$ associated with the gluing parameters $a$ and $a'$ are  equal. Also  the gluing lengths and gluing angles  $(R_1', \vartheta_1')$ and $(R_2', \vartheta_2')$ associated with the gluing parameters $b(a,u^+_1+h^+,u^-_1+h^-)$ and $b'(a,u^+_2+k^+,u^-_2+k^-)$ are the same. Hence, in view of $(R_1', \vartheta_1')=(R_1, \vartheta_1)+(d^+_1-d^-_1, \vartheta^+_1-\vartheta^-_1)+h^+_{\infty}-h^-_{\infty}$ and  $(R_2', \vartheta_2')=(R_2, \vartheta_2)+(d^+_2-d^-_2, \vartheta^+_2-\vartheta^-_2)+k^+_{\infty}-k^-_{\infty}$ , we obtain the relation 
\begin{equation}\label{eq1a-new}
k^+_{\infty}-k^-_{\infty}=(h^+_{\infty}-h^-_{\infty})+(d^+_1-d^-_1-d^+_2+d^-_2, \vartheta_1^+-\vartheta_1^--\vartheta_2^++\vartheta_2^-).
\end{equation}
Abbreviating 
\begin{equation*}\label{eq2-nnew}
w_a=\boxplus_a(u_1^+, u_1^-)-\boxplus_a(u_2^+, u_2^-),
\end{equation*}
the equation \eqref{eq1-new} becomes 
\begin{equation}\label{eq2-new}
\oplus_a(k^+, k^-+k^+_{\infty}-k^-_{\infty})=\oplus_a(h^+, h^-+h^+_{\infty}-h^-_{\infty})+w_a.
\end{equation}
Recall that  the pair $(k^+, k^-)$ satisfies $\rho_{a}(k^+,k^-)=(k^+,k^-)$ which implies by  Lemma \ref{kernel-lem1} that 
$\ominus_a(k^+-k^+_{\infty}, k^--k^-_{\infty})=0$.    Observing that  the anti-gluing operation
 $\ominus_a$ satisfies
 $$\ominus_a(v^++A, v^-+A)=\ominus_a(v^+, v^-)$$
 for every constant $A$, we conclude that  $\ominus_a(k^+, k^-+k^+_{\infty}-k^-_{\infty})=0.$
Setting $q^+=k^+$ and $q^-=k^-+k^+_{\infty}-k^-_{\infty}$ and abbreviating 
\begin{equation}\label{eq2a-new}
g:=\oplus_a(h^+, h^-+h^+_{\infty}-h^-_{\infty})+w_a,
\end{equation}
we have therefore to solve the following system of equations  for $(q^+, q^-)$
\begin{equation}
\begin{aligned}\label{eq4-nnew}
\oplus_a (q^+, q^-)&=g\\
\ominus_a(q^+, q^-)& =0.
\end{aligned}
\end{equation}
Integrating the first equation at $s=\frac{R}{2}$ over the circle $S^1$ and recalling that $\beta_a(\frac{R}{2})=\frac{1}{2}$, we find 
\begin{equation}\label{eq3-new}
\text{av}_a (q^+, q^-)=[g].
\end{equation}
In matrix form the system \eqref{eq4-nnew} is expressed  as follows,
$$\begin{bmatrix}
\phantom{-}\beta_a&1-\beta_a\\
-(1-\beta_a)&\beta
\end{bmatrix}\cdot
\begin{bmatrix}q^+\\q^-\end{bmatrix}=
\begin{bmatrix}g\\
(2\beta_a-1)\text{av}_a (q^+, q^-)
\end{bmatrix}= 
\begin{bmatrix}g\\
(2\beta_a-1)[g]
\end{bmatrix} .
$$
Denoting by $\gamma_a=\beta+a^2+(1-\beta_a)^2$ the determinant of the matrix and abbreviating $\gamma_a=\gamma_a(s)$ and $\beta_a=\beta_a(s)$, the unique solution of \eqref{eq4-nnew} is given by 
\begin{equation}\label{eq4-new}
q^+(s,t)=\left(1-\frac{\beta_a}{\gamma_a}\right)\cdot [g]+ \frac{\beta_a}{\gamma_a}\cdot g
\end{equation}
for $s\geq 0$, and 
\begin{equation}\label{eq5-new}
q^-(s-R, t-\vartheta)=\left(1-\frac{1-\beta_a}{\gamma_a}\right)\cdot [g]+\frac{1-\beta_a}{\gamma_a}\cdot g
\end{equation}
for all $s\leq R$. 

In order to analyze the behavior at $a=0$ we write down the solutions in detail.  To do this we  represent $h^\pm=h^\pm_{\infty}+r^\pm_3$ where $r^\pm_3 \in H^{3, \delta_0}(\R^\pm \times S^1, \R^2)$ and use the explicit  representations  of 
\begin{equation*}
\begin{split}
g(s, t)&=\oplus_a(h^+, h^-+h^+_{\infty}-h^-_{\infty})(s, t)+w_a(s, t)\\
&=\oplus_a (h^+_{\infty}+r^+_3, h^+_{\infty}+r^-_3)(s, t)+w_a(s, t)\\
&=h^+_{\infty}+\beta_a \cdot r^+_3(s, t)+(1-\beta_a)\cdot r^-_3(s-R, t-\vartheta)+w_a(s, t)
\end{split}
\end{equation*}
and of 
\begin{equation*}
\begin{split}
w_a(s, t)&=(s, t)+(d_1^+, \vartheta^+_1)+\beta_a\cdot  r^+_1(s, t)+(1-\beta_a)\cdot r_1^-(s-R, t-\vartheta)\\
&\phantom{=}-(s, t)-(d_2^+, \vartheta^+_2)-\beta_a\cdot  r^+_2(s, t)-(1-\beta_a)\cdot r_2^-(s-R, t-\vartheta)\\
&=(d_1^+-d_2^+, \vartheta^+_1-\vartheta^+_2)\\
&\phantom{=}+\beta_a\cdot (r^+_1-r^+_2)(s, t)+(1-\beta_a)\cdot (r^-_1-r^-_2)(s-R, t-\vartheta),
\end{split}
\end{equation*}
so that 
\begin{equation}\label{eq6-nnew}
\begin{split}
g(s,t)&=h^+_{\infty}+(d_1^+-d_2^+, \vartheta^+_1-\vartheta^+_2)\\
&\phantom{=}+\beta_a\cdot r^+(s,t)+(1-\beta_a)\cdot r^-(s-R, t-\vartheta)
\end{split}
\end{equation}
where we have abbreviated 
\begin{equation}\label{eqr-nnew}
r^\pm=r^\pm_3+r^\pm_1-r^\pm_2.
\end{equation}
The mean value 
$[g]=\int_{S^1}g\left(\frac{R}{2},  t\right)\ dt$ is computed to be
\begin{equation}\label{eq7-nnew}
[g]=h^+_{\infty}+(d_1^+-d_2^+, \vartheta^+_1-\vartheta^+_2)+\av (r^+, r^-).
\end{equation}

With \eqref{eq6-nnew} and \eqref{eq7-nnew} the solution $q^+(s, t)$ for all $s\geq 0$ is equal to 
\begin{equation}\label{eq8-nnew}
\begin{split}
q^+(s, t)&=h^+_{\infty}+(d_1^+-d_2^+, \vartheta^+_1-\vartheta^+_2)+\left(1-\frac{\beta_a}{\gamma_a}\right)\cdot \av (r^+, r^-)\\
&\phantom{=}+\frac{\beta_a^2}{\gamma_a}\cdot r^+(s, t)+\frac{\beta_a(1-\beta_a)}{\gamma_a}\cdot r^-(s -R,t-\vartheta), 
\end{split}
\end{equation}
where, as usual,  we have abbreviated $\beta_a=\beta_a(s)$.
Since $q^+=k^+$ we read off the  solution $q^+(s, t)$ the asymptotic constant 
\begin{equation}\label{eqq1-nnew}
k^+_{\infty}=\lim_{s\to \infty}q^+(s, t)=h^+_{\infty}+(d_1^+-d_2^+, \vartheta^+_1-\vartheta^+_2)+
\av (r^+, r^-)
\end{equation}
using that  $\beta_a(s)=0$ for $s\geq \frac{R}{2}+1$.

In order to represent the solution $q^-(s-R, t-\vartheta)$ for all $s\leq R$ we introduce the variables $s'=s-R$ and $t'=t-\vartheta$. From $\beta (s)=1-\beta (-s)$ one deduces $\beta_a (-s')=\beta (-s'-\frac{R}{2})=1-\beta(s'+\frac{R}{2})=1-\beta_a (s'+R)$ and $\gamma_a (-s')=\gamma_a(s'+R)$. Using this, the solution $q^-(s', t')$ is represented by
\begin{equation*}
\begin{split}
q^-(s', t')&=k^-(s', t')+(k^+_{\infty}-k^-_{\infty})\\
&=h^+_{\infty}+(d_1^+-d_2^+, \vartheta^+_1-\vartheta^+_2)+\left(1-\frac{\beta_a(-s')}{\gamma_a(-s')}\right)\cdot \av (r^+, r^-)\\
&\phantom{=}+\frac{(1-\beta_a(-s'))\beta_a(-s')}{\gamma_a(-s')}\cdot r^+(s'+R, t'+\vartheta)+
\frac{\beta_a(-s')^2}{\gamma_a(-s')}\cdot r^-(s', t')
\end{split}
\end{equation*}
for all $s'\leq 0$.  In view of the relation \eqref{eq1a-new}, the solution $k^-(s', t')$ has the following representation,
\begin{equation}\label{eq10-nnew}
\begin{split}
k^-(s', t')&=h^-_{\infty}+(d^-_1-d^-_2\, \vartheta^-_1-\vartheta^-_2)+\left(1-\frac{\beta_a(-s')}{\gamma_a(-s')}\right)\cdot \av (r^+, r^-)\\
&\phantom{=}\frac{\beta_a(-s')}{\gamma_a(-s')}\cdot r^+(s'+R, t'+\vartheta)+
\frac{\beta_a(-s')^2}{\gamma_a(-s')}\cdot r^-(s', t').
\end{split}
\end{equation}
The asymptotic constant $k^-_{\infty}=\lim_{s'\to -\infty}k^-(s', t')$  of the solution $k^-$ is equal to 
\begin{equation}\label{eqq2-nnew}
k^-_{\infty}=h^-_{\infty}+(d^-_1-d^-_2\, \vartheta^-_1-\vartheta^-_2)+\av (r^+, r^-).
\end{equation}

From the expressions   \eqref{eqq1-nnew} and \eqref{eqq2-nnew} one deduces using the Lemmata   \ref{lem2-18} and \ref{prop-xx},    that the asymptotic constants $k^\pm_{\infty}$ depend sc-smoothly on $(a, h^+, h^-)$.

 To sum up our computations, we represent the chart transformation 
$\Phi=\phi^{-1}_2\circ \phi_1$ by the formula
$$
\Phi (a, h^+, h^-)=\begin{cases}(0, h^++(u_1^+-u_2^+), h^-+(u_1^--u_2^-))\quad &\text{if $a=0$}\\
(0, k^+ (a, h^+, h^-),  k^-(a, h^+, h^-))\quad &\text{if $a\neq 0$}
\end{cases}
$$
where,  
for $a\neq 0$, the maps $k^+ (a, h^+, h^-)=q^+$ and 
$k^-(a, h^+, h^-)=q^-+k^-_{\infty}-k^+_{\infty}$ are defined by 
\eqref{eq8-nnew} and \eqref{eq10-nnew}. We define the maps $k^\pm$ at $a=0$ by 
\begin{align*}
k^+ (0,  h^+, h^-)&:=h^++(u_1^+-u_2^+)\\
k^- (0,  h^+, h^-)&:=h^-+(u_1^--u_2^-), 
\end{align*}
and  observe that  
\begin{equation*}
\begin{split}
h^++(u_1^+-u_2^+)&=h^+_{\infty}+r^+_3+(d_1^+-d_2^+, \vartheta^+_1-\vartheta^+_2)+(r^+_1-r^+_2)\\
&=h^+_{\infty}+(d_1^+-d_2^+, \vartheta^+_1-\vartheta^+_2)+r^+
\end{split}
\end{equation*}
and 
$$h^-+(u_1^--u_2^-)=h^-_{\infty}+(d^-_1-d^-_2\, \vartheta^-_1-\vartheta^-_2)+r^-$$
where as before $r^\pm=r^\pm_3+r_1^\pm-r^\pm_2$. 

Checking every term in \eqref{eq8-nnew} and in \eqref{eq10-nnew}, applying Proposition \ref{qed} and the Lemmata 
 \ref{lem2-18}--\ref{lem2.21-nnew}, one sees that the maps 
$(a, h^+, h^-)\mapsto k^\pm (a, h^+, h^-)$ are sc-smooth in a neighborhood of $a=0$.

Finally, in view of Lemma \ref{abmap}, also the map $(a,h^+, h^-)\mapsto b(a, u^+_0+h^+, u^-_0+h^-)$ is sc-smooth.

The proof of the sc-smoothness of the chart transformations is complete.

\end{proof}

As a   consequence we obtain the following result. 
\begin{prop}\label{prop3.15-nnew}
The topological space $\ov{X}$ has in a natural way the structure of an  M-polyfold which induces on $X$ the previously defined M-polyfold structure.
\end{prop}

With Proposition \ref{prop3.15-nnew},  the proof of Theorem \ref{thm-comp} is complete.

\subsection{A Strong Bundle, Proof of Theorem \ref{thm1.42-nnew}}\label{subsect3.3}

Continuing with the illustration of the polyfold theory, we are going to construct a strong bundle over the M-polyfold $V\times \ov{X}$.

 In order to define the Cauchy-Riemann operator as an sc-smooth Fredholm section we shall first equip the cylinders $Z_a$ with complex structures.

We assume that $\ov{X}$ has the M-polyfold structure defined above, using the gluing profile $\varphi (r)=e^{\frac{1}{r}}-e$ and the sequence $(\delta_m)$ of weights in the open  interval $(0, 2\pi)$. We choose a smooth  family 
$$v\mapsto j^\pm (v)$$
 of complex structures on the half cylinders $\R^\pm \times S^1$ parametrized by $v$ belonging to an open neighborhood $V$ of $0$ in some finite dimensional vector space. We assume that 
$j^\pm(v)=i$ is the standard complex structure outside of a compact neighborhood of the boundaries $\partial (\R^\pm \times S^1)$.  In order to arrange that the gluing of the half-cyliners $\R^\pm\times S^1$ takes place in a region where the complex structures $j^\pm (v)$ are the standard structures we shall not glue the half-cylinders along the pieces $[0,R]\times S^1$ resp. $[-R, 0]\times S^1$ as we did so far but along much shorter pieces and obtain the new finite cylinders $Z_a$ and the new infinite cylinders  $C_a$, which we denote by the same letters because they are biholomorphically equivalent to the old cylinders we have considered  so far.

We assume that $j^+(v)=i$ for $s\geq \frac{1}{2}s_0$ and $j^-(v)=i$ for $s\leq -\frac{1}{2}s_0$ and choose the gluing parameter $a$ so small that $R=\varphi (\abs{a})$ satisfies $R-s_0>s_0$. We then identify the points $(s, t)\in [s_0, R-s_0]\times S^1$ of the cylinder $\R^+\times S^1$ with the points $(s', t')\in [-R+s_0, -s_0]\times S^1$ of the negative cylinder $\R^-\times S^1$ if 
\begin{align*}
s'&=s-R\\
t'&=t-\vartheta, 
\end{align*}
as illustrated in  Figure 5. We have to keep in mind that $Z_a$ possesses the distinguished points $p^\pm_a$.

\begin{figure}[htbp]
\mbox{}\\[2ex]
\centerline{\relabelbox \epsfxsize4truein \epsfbox{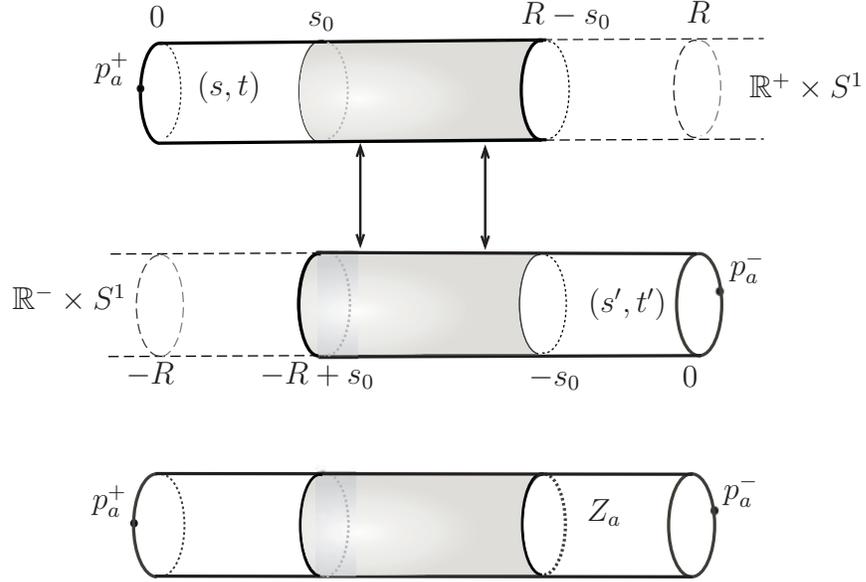}
\relabel {a1}{$(s', t')$}
\relabel {a2}{$(s,t)$}
\relabel {a}{$\R^-\times S^1$}
\relabel {b}{$\R^+\times S^1$}
\relabel {c}{$Z_a$}
\relabel {mr1}{$-R$}
\relabel {mrs}{$-R+s_0$}
\relabel {ms0}{$-s_0$}
\relabel {01}{$0$}
\relabel {02}{$0$}
\relabel {s0}{$s_0$}
\relabel {rms}{$R-s_0$}
\relabel {r1}{$R$}
\relabel {p1}{$p_a^+$}
\relabel {p2}{$p_a^-$}
\relabel {p3}{$p_a^+$}
\relabel {p4}{$p_a^-$}
\endrelabelbox}
\caption{Glued finite cylinders $Z_a$}\label{Fig5}
\end{figure}

Using  the same identification we also redefine the infinite cylinders  $C_a$ as illustrated in Figure 6.

\begin{figure}[htbp]
\mbox{}\\[2ex]
\centerline{\relabelbox \epsfxsize 4truein \epsfbox{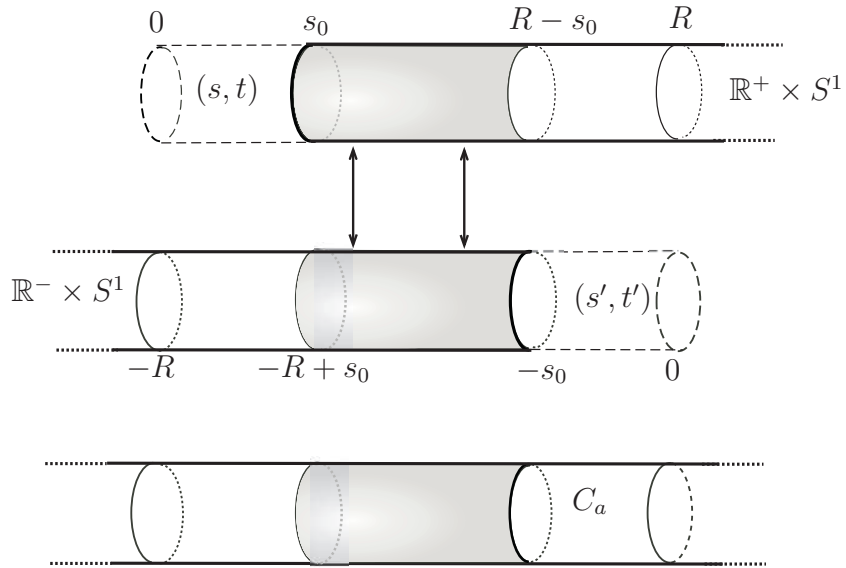}
\relabel {a1}{$(s', t')$}
\relabel {a2}{$(s,t)$}
\relabel {a}{$\R^-\times S^1$}
\relabel {b}{$\R^+\times S^1$}
\relabel {c}{$C_a$}
\relabel {mr1}{$-R$}
\relabel {mrs}{$-R+s_0$}
\relabel {ms0}{$-s_0$}
\relabel {01}{$0$}
\relabel {02}{$0$}
\relabel {s0}{$s_0$}
\relabel {rms}{$R-s_0$}
\relabel {r1}{$R$}
\endrelabelbox}
\caption{Glued  infinite cylinders $C_a$}\label{Fig6}
\end{figure}

The complex structures $j^\pm (v)$ induce the complex structures $j(a, v)$ on the glued finite cylinders $Z_a$ for sufficiently small gluing parameters $a$.  We equip the  glued 
 infinite  cylinder $C_a$ with the standard complex structure denoted by $i$. It is clearly biholomorphic to the standard complex cylinder $\R\times S^1$.
 
 With these new cylinders $Z_a$ and $C_a$, the gluing formula $\oplus_a$ and the anti-gluing formula $\ominus_a$ for the maps on the cylinders remain unchanged and all the definitions and corresponding results proved so far hold true also for  the new cylinders with the identical proofs.

Clearly with $\ov{X}$,  also $V\times \ov{X}$ is an M-polyfold.  If  the point $(v, a, b, w)\in V\times \ov{X}$ satisfies  $a\neq 0$ and $b\neq 0$, then $w:Z_a\to Z_b$ is a diffeomorphism and associated with this point we consider mappings $\xi$ which are defined on the cylinder $Z_a$ and whose images
$$\xi (z):(T_z Z_a, j(a, v))\to (T_{w(z)}Z_b, i)$$
are complex anti-linear mappings belonging to the Sobolev space $H^2$. The complex anti-linearity requires that 
$$\xi (z)\circ j(a, v)=-i\circ \xi (z)$$
for all $z\in Z_a$. In the following we identify the tangent spaces of the cylinders $Z_a$ and $\R^\pm \times S^1$ with $\R^2$.

If $a=0$ (and consequently $b=0$), then $Z_0$ is the disjoint union
$$(\R^+\times S^1)\coprod  (\R^-\times S^1)$$
and recalling the two diffeomorphisms $u^\pm :\R^\pm\times S^1\to \R^\pm\times S^1$, we are 
familiar with from the previous section, we associate with $Z_0$ two maps $z\mapsto \xi^\pm (z)$ defined on the half cylinders $\R^\pm \times S^1$  whose images 
$$\xi^\pm (z):(T_z\R^\pm \times S^1), j^\pm (v))\to (T_{u^\pm(z)}\R^\pm \times S^1),i)$$
are complex anti-linear and belong to $H^{2, \delta_0}(\R^\pm \times S^1, \R^4)$.

The collection $E$ of all multiplets $(v, a, b, w, \xi)$ in which $(v, a, b, w)\in V\times \ov{X}$ and 
$\xi$ is the associated complex anti-linear map,  possesses the projection map
$$E\to V\times \ov{X},\quad (v, a, b, w, \xi)\mapsto (v, a, b, w)$$
where fibers (containing $\xi$) have in a natural way the structure of a Hilbert  space.

On  $E$ we introduce the double-filtration $(E_{m,k})$ whose   indices run over $m\geq 0$ and 
$0\leq k\leq m+1$, as explained in Section \ref{Fredholm-section}. If $ab\neq 0$, then an  element $(v,a,b,w,\xi)$ belongs to $E_{m, k}$ if  $(v,a,b,w)\in V\times \ov{X}_m$, where $w$ belongs to the level $m$ if it  has the Sobolev regularity $(m+3,\delta_m)$, and if $\xi$ belongs to the class  $(k+2,\delta_k)$. 

This subsection \ref{subsect3.3} is devoted to the proof of the following theorem announced in the introduction.
\begin{thm}\label{thm42.-nnnew}
Having fixed the exponential gluing profile $\varphi$ and the increasing sequence $(\delta_m)_{m\in \N_0}$ of real numbers satisfying $0<\delta_m<2\pi$,   the  set  $E$ admits in a natural way the structure of a strong bundle over the M-polyfold $V\times  \ov{X}$.
\end{thm}
\begin{proof}
In order to prove the theorem we have to define strong bundle charts. We first construct charts on $E\vert V\oplus X$. To do so, 
we fix a  smooth  point $ (v_0,a_0,b_0,w_0)\in V\times X$ so that $w_0:Z_{a_0}\to Z_{b_0}$ is a smooth diffeomorphism preserving the marked points. We recall that a chart of $X$ around 
$(v_0,a_0,b_0,w_0)\in  V\times X$ has been previously constructed by the map
$$(a, b, h)\mapsto (a, b, \psi_b \circ (w_0+h)\circ \phi_a)$$
for  $(a, b)$  close to $(a_0, b_0)$ and $h\in H^{3}(Z_{a_0}, \R^2)$ satisfying $h(p^\pm_{a_0})=(0, 0)$ and $h([0,t])\in \{0\}\times \R$ and $h([0,t]')\in \{0\}\times \R$. Moreover, the derivative 
$\abs{Dh([s, t])}<\varepsilon_0$ is so small that $w_0+h$ is still a diffeomorphism $Z_{a_0}\to Z_{b_0}$. Then $\psi_b \circ (w_0+h)\circ \phi_a$ is a diffeomorphism $Z_a\to Z_b$.

Given now a point $(v, a, b, \psi_b \circ (w_0+h)\circ \phi_a)\in V\times X$, there is a one-to-one correspondence between complex anti-linear maps 
$$\xi (z):(T_z Z_a, j(a, v))\to (T_{\psi_b \circ (w_0+h)\circ \phi_a(z)}Z_b, i)$$
for $z\in Z_a$, and elements of the Hilbert space $H^2(Z_{a_0}, \R^2)$ on the fixed cylinder $Z_{a_0}$,  defined by the relation 
\begin{equation}\label{eeq1-nnew}
\eta (\phi_a(z))=\xi(z)\cdot \frac{\partial}{\partial s}.
\end{equation}
This follows from the complex anti-linearity of $\xi (z)$. Recall that $\phi_a$ maps the cylinder $Z_a$ diffeomorphically  onto the cylinder $Z_{a_0}$. The chart  of $E\vert V\oplus X$ around the point $(v_0, a_0, b_0, w_0)$ is now defined by the map
$$\Phi:(v, a, b, h,\eta )\mapsto (v, a, b, \psi_b \circ (w_0+h)\circ \phi_a, \xi)\in E\vert V\oplus X$$
where $\eta$ and $\xi$ are related by \eqref{eeq1-nnew}.
If 
$$
\wt{\Phi}:(\wt{v}, \wt{ a},\wt{ b}, \wt{h},\wt{\eta} )\mapsto(\wt{v}, \wt{ a},\wt{ b}, \wt{\psi}_{\wt{b}} \circ (\wt{w}_0+\wt{h})\circ  \wt{\phi}_{\wt{a}}, \wt{ \xi})
$$
is a second such chart, then we have in the overlapping region  $\wt{v}=v$, $\wt{a}=a$, and 
$\wt{b}=b$, and hence 
$ \wt{\psi}_b \circ (\wt{w}_0+\wt{h})\circ  \wt{\phi}_{a}= \psi_{b} \circ (w_0+h)\circ  \phi_a$ and 
$\wt{\xi}=\xi$ so that the formula for the chart transformation is as follows,
$$\wt{\Phi}^{-1}\circ \Phi (v, a, b, h,\eta)=(v, a, b, \wt{h},\wt{\eta})$$
where
$$
\wt{h}=\wt{\psi}_{b}^{-1} \circ \psi_b\circ (\wt{w}_0+\wt{h})\circ\phi_a\circ   \wt{\phi}^{-1}_{a}-\wt{w}_0
$$
and
$$
\wt{\eta}(z)=\eta (\phi_a\circ \wt{\phi}^{-1}_a(z))
$$
for all $z\in Z_{a_0}$.

Recalling from Section \ref{Fredholm-section} the definition of an $\ssc_{\triangleleft}$-smooth bundle map, we see from the above formulae that the chart transformation is an  $\ssc_{\triangleleft}$-smooth bundle map, in view of our results  about  the  action by diffeomorphisms  (Theorem \ref{thm-125p}).

We need to define strong bundle charts also for $E\to V\times \ov{X}$. To do so we first construct the local models for the bundle charts.

We take the collection $K'$ of all tuples $(v, a, h^+, h^-, \eta^+, \eta^-)$ in which $v\in V$ and $\abs{a}<\varepsilon$, moreover, $(h^+, h^-)\in \wh{E}_0$ and $\eta^\pm\in H^{2, \delta_0}(\R^\pm \times S^1, \R^2)=F$. They satisfy the relation
$$\rho_a(h^+, h^-)=(h^+, h^-)\quad \text{and}\quad \wh{\pi}_a(\eta^+, \eta^-)=(\eta^+, \eta^-).$$
The projection $\rho_a:\wh{E}_0\to \wh{E}_0$ has been introduced in Section \ref{anothersplicing} and the projection $\wh{\pi}_a:F\to F$ in Section \ref{arisingx}. 
 The Hilbert space $F$ is equipped with the sc-smooth structure $(F_m)$ defined by $F_m=H^{2+m, \delta_m}(\R^\pm \times S^1, \R^2)$.
 
 By $O'$ we denote the collection of all tuples $(v, a, h^+, h^-)$ satisfying $v\in V$, $\abs{a}<\varepsilon$, and $\rho_a(h^+, h^-)=(h^+, h^-)$. The natural projection
 $$K'\to O', \quad (v, a, h^+, h^-, \eta^+, \eta^-)\mapsto   (v, a, h^+, h^-)$$
 defines a local model for a strong bundle. Indeed, if we define the map 
 \begin{gather*}
 R:(V\oplus B_\varepsilon\oplus \wh{E}_0)\triangleleft F\to (V\oplus B_\varepsilon\oplus \wh{E}_0)\triangleleft F\\
 \intertext{as}
 R(v, a, (h^+, h^-), ( \eta^+, \eta^-))=(v, a, \rho_a (h^+, h^-), \wh{\pi}_a (\eta^+, \eta^-)),
 \end{gather*}
 then  the map $R$ is an $\ssc_{\triangleleft}$-smooth strong bundle retraction satisfying 
 $$K'=R((V\oplus B_\varepsilon\oplus \wh{E}_0)\triangleleft F)$$
 in view of Theorem \ref{propn-1.27} and Theorem \ref{second-splicing}. Moreover, the map $R$ covers the retraction map 
 $r:V\oplus B_\varepsilon\oplus \wh{E}_0\to V\oplus B_\varepsilon\oplus \wh{E}_0$ defined by  
  $r(v, a, (h^+, h^-))=(v, a, \rho_a (h^+, h^-))$ whose image is the set 
 $$O'=r(V\oplus B_\varepsilon\oplus \wh{E}_0).$$
 We next use these local models to define an atlas of charts of the bundle $E\to V\times \ov{X}$. 
 
 Recall that  $E$ is the collection of tuples $(v, a, b, w, \xi)$ in which for $ab\neq 0$ the map $w:Z_a\to Z_b$ is a diffeomorphism and $\xi (z): (T_zZ_a, j(a, v))\to (T_{w(z)}Z_b, i)$  a complex anti-linear map defined for every $z\in Z_a$. If $a=0$ and hence $b=0$, the point in $E$ is defined as 
 $$(v, 0, 0, (u^+, u^-), (\xi^+, \xi^-))$$
 where $u^\pm:\R^\pm\times S^1\to \R^\pm \times S^1$ are  two diffeomorphisms of the half cylinders  satisfying
 $u^\pm(0,0)=(0,0)$. Moreover, 
 $\xi^\pm  (z): (T_z(\R^\pm\times S^1), j^\pm (v))\to (T_{u^\pm (z)}( \R^\pm \times S^1), i)$  are complex anti-linear maps defined  for  every $z\in \R^\pm \times S^1$.
 
 In order to define strong bundle charts for $E\to V\times \ov{X}$ we start with the chart maps $\Psi:O\to V\oplus \ov{X}$ defined by 
 $$
 \Psi(v, a, h^+, h^-)=(v, a, b(a, u^+_0+h^+, u^-_0+h^-), \boxplus_a (u^+_0+h^+, u^-_0+h^-)),
 $$
 around a base pair $(u^+_0,u^-_0)$, if $a\neq 0$, and by 
 $$\Psi(v, 0, h^+, h^-)=(v, 0, 0,  u^+_0+h^+, u^-_0+h^-)$$
 if $a=0$.  The domain of definition of the map $\Psi$ is the set
\begin{gather*}
O=\\
 \{(v, a, h^+, h^-)\vert \, \text{$\rho_a(h^+, h^-)=(h^+, h^-)$, $v\in V$, $\abs{a}<\sigma_0$, and $(h^+, h^-)\in U$}\}
\end{gather*}
 where $U\subset \wh{E}_0$ is an open neighborhood of $(0, 0)\in \wh{E}_0$. Now  we observe that 
 $O\subset O'$ is an open subset and define $K\subset K'$ as the subset of $K'$ lying above $O$. We shall show that $K\to O$ is a local model for the bundle $E\to V\oplus \ov{X}$. 
 
 We define the bundle chart $\Gamma:K\to E$ which covers the chart $\Psi$ of $V\oplus \ov{X}$ by the following map. If $a\neq 0$,  we set
 \begin{equation*}
 \begin{split}
 \Gamma (v, a, h^+, h^-, \eta^+, \eta^-)&=(\Psi (v, a, h^+, h^-), \xi)\\
 &=(v, a, b(a, u^+_0+h^+, u^-_0+h^-), \boxplus_a(u^+_0+h^+, u^-_0+h^-), \xi)
 \end{split}
 \end{equation*}
where $(h^+, h^-)\in \wh{E}_0$ satisfies $\rho_a(h^+, h^-)=(h^+, h^-)$ and $(\eta^+, \eta^-)\in F$ satisfies $\wh{\pi}_a(\eta^+, \eta^-)=(\eta^+, \eta^-)$. Moreover, abbreviating the diffeomorphism $w_a=\boxplus_a(u^+_0+h^+, u^-_0+h^-):Z_a\to Z_b$, the fiber part $\xi$ is the complex anti-linear map 
 $$\xi (z):(T_zZ_a, j(v, a))\to (T_{w(z)}, i)$$
 defined by 
 $$\xi (z)\frac{\partial }{\partial s}=\wh{\oplus}_a(\eta^+, \eta^-)(z), \quad z\in Z_a.$$
 We recall that $\wh{\oplus}_a(\eta^+, \eta^-)(s, t)=\beta_a(s)\cdot \eta^+(s, t)+(1-\beta_a(s))\cdot \eta^-(s-R, t-\vartheta)$ for $(s, t)\in [0,R]\times S^1$, and $\wh{\ominus}_a(\eta^+, \eta^-)(s, t)=-(1-\beta_a(s))\cdot \eta^+(s, t)+\beta_a(s)\cdot \eta^-(s-R, t-\vartheta)$  for $(s, t)\in \R\times S^1$.
 
 If $a=0$, the map $\Gamma:K\to E$ is defined as 
  \begin{equation*}
 \begin{split}
 \Gamma (v,0, h^+, h^-, \eta^+, \eta^-)&=(\Psi (v, 0, u_0^++h^+, u^-_0+h^-),  (\xi^+, \xi^-))\\
 &=(v, 0, 0, u^+_0+h^+, u^-_0+h^-;  (\xi^+, \xi^-))
 \end{split}
 \end{equation*}
 where the complex anti-linear maps
 $$
 \xi^\pm (z):(T_z(\R^\pm \times S^1), j^\pm (v))\to (T_{(u^\pm_0+h^\pm)(z)}(\R^\pm\times S^1), i)$$ are defined as 
 $$\xi^\pm (z)\frac{\partial }{\partial s}=\eta^\pm (z)$$
for all $z=(s, t)\in \R^\pm \times S^1$. It is easy to see that these charts are $\ssc_{\triangleleft}$-smoothly equivalent. Namely,  abbreviating two chart maps by 
 $ \Gamma (v, a, h^+, h^-, \eta^+, \eta^-)=(\Psi (v, a, h^+, h^-), \xi)$ and $\wt{ \Gamma} (\wt{v}, \wt{a}, \wt{h}^+, \wt{h}^-, \wt{\eta}^+, \wt{\eta}^-)=(\wt{\Psi} (\wt{v}, \wt{a}, \wt{h}^+,\wt{h}^-), \wt{\xi})$
 the chart transformation is given by 
 \begin{gather*}
 \wt{\Gamma}^{-1}\circ \Gamma (v, a, h^+, h^-, \eta^+, \eta^-)=(\wt{\Psi}^{-1}\circ \Psi (v, a, h^+, h^-), \eta^+, \eta^-),\\
 \intertext{if $a\neq 0$, and by}
 \wt{\Gamma}^{-1}\circ \Gamma (v, 0, h^+, h^-, \eta^+, \eta^-)= (v, 0, (u_0^+-\wt{u}_0^+) +h^+,  (u_0^--\wt{u}_0^-) +h^-, \eta^+, \eta^-),
 \end{gather*}
 if $a=0$. From the previous considerations we know that $\wt{\Psi}^{-1}\circ \Psi$ is an sc-smooth diffeomorphism. In the fibers the transformation is the identity. Indeed, at a point of intersection 
 we have $\xi=\wt{\xi}$. It follows that 
 $\wh{\oplus}_a(\eta^+, \eta^-)(z)=\wh{\oplus}_a(\wt{\eta}^+,\wt{ \eta}^-)(z)$ for all $z\in Z_a$. Since, by definition, $\wh{\ominus}_a(\eta^+, \eta^-)=0$ and $\wh{\ominus}_a(\wt{\eta}^+, \wt{\eta}^-)=0$, we conclude from the uniqueness of the solutions of the system of  two equations that  $\eta^+=\wt{\eta}^+$ and $\eta^-=\wt{\eta}^-$. 
 
 We have proved that $ \wt{\Gamma}^{-1}\circ \Gamma $ is an $\ssc_{\triangleleft}$-bundle isomorphism between the local models of a strong bundle. 
 
It remains to prove that the chart maps $\Gamma$ and $\Phi$ are  $\ssc_{\triangle}$-smoothly equivalent. We assume that we are given two chart maps into $E$. The first  one is 
$$\Phi(v, a, b, h, \eta)=(v, a, b, \psi_b \circ (w_0+h)\circ \phi_a, \xi),$$
defined for $(v, a, b)$ close to $(v_0, a_0, b_0)$ and   $\eta$ and $\xi$ are related via  $\xi (z)\frac{\partial }{\partial s}=\eta (\varphi_a (z))$. The second chart map is 
$$\Gamma(\wt{v}, \wt{a}, h^+, h^-, \eta^+, \eta^-)=(\wt{v},\wt{ a}, b(\wt{a}, u^+_0+h^+, u^-_0+h^-), \boxplus_{\wt{a}} (u^+_0+h^+, u^-_0+h^-),\wt{ \xi})$$
defined  for $(\wt{v},\wt{a}) $ close to $(\wt{v}_0, \wt{a}_0)$ and where $\wt{\xi}(z):T_zZ_{\wt{a}}\to T_{w(z)}Z_{b'}$ is a complex anti-linear map uniquely defined by 
$$\wt{\xi }(z)\frac{\partial }{\partial s}=\wh{\oplus}_{\wt{a}}(\eta^+, \eta^-)(z).$$
Here $w=\boxplus_{\wt{a}} (u^+_0+h^+, u^-_0+h^-)$, $b'= b(\wt{a}, u^+_0+h^+, u^-_0+h^-)$,  and 
$z=(s, t)$.
From $\Phi (v, a, b, h, \eta)=\Gamma (\wt{v}, \wt{a}, h^+, h^-, \eta^+, \eta^-)$
one concludes that
$a=\wt{a}$ and  $b= b(\wt{a}, u^+_0+h^+, u^-_0+h^-)$ and, moreover,  
\begin{equation}\label{eeq1}
\psi_b \circ (w_0+h)\circ \phi_a=\boxplus_{a} (u^+_0+h^+, u^-_0+h^-)
\end{equation}
and
\begin{equation}\label{eeq2}
 \xi=\wt{\xi}.
 \end{equation}
The first equation can be solved for $h$ as map of $(a, h^+, h^-)$ resulting in 
$$h=\psi_b^{-1}\circ \boxplus_{a} (u^+_0+h^+, u^-_0+h^-)\circ \phi_a^{-1} -w_0$$
where 
$b= b(a, u^+_0+h^+, u^-_0+h^-).$
From the second equation $ \xi=\wt{\xi}$ we conclude that 
$\eta (\phi_a(z))=\wh{\oplus}_a (\eta^+, \eta^-)(z)$
where $z=[s,t]\in Z_a$. 
Consequently, the transition map $\Phi^{-1}\circ \Gamma$ has the following form
$$\Phi^{-1}\circ \Gamma (v, a, h^+, h^-, \eta^+, \eta^-)=(v, a, b, h, \eta)$$
where
\begin{gather*}
b=b(a, u^+_0+h^+, u^-_0+h^-)\\
h=\psi_b^{-1}\circ \boxplus_{a} (u^+_0+h^+, u^-_0+h^-)\circ \phi_a^{-1} -w_0\\
\eta (z)=\wh{\oplus}_a (\eta^+, \eta^-)(\phi^{-1}_a(z)), \quad z\in Z_{a_0}.
\end{gather*}
For the  transition map $\Gamma\circ \Phi^{-1}$ we obtain  
$$\Gamma^{-1}\circ \Phi (v, a, b, h, \eta) =(v, a, h^+, h^-, \eta^+, \eta^-)$$
where 
the pair $(h^+, h^- )$ is obtained by solving \eqref{eeq1} for $(h^+, h^-)$ in terms of $(a, h)$. 
The formulae for $h^\pm$ are given in the proof of Proposition 3.10. Namely,
$$h^+=q^+\quad \text{and}\quad h^-=q^-+(h^-_{\infty}-h^+_{\infty})$$
where 
\begin{align*}
q^+(s, t)&=[g]+\frac{\beta_a(s)}{\gamma_a(s)}\cdot (g(s, t)-[g])\\
q^{-}(s-R, t-\vartheta)&=[g]+\frac{1-\beta_a(s)}{\gamma_a(s)}\cdot (g(s, t)-[g])
\end{align*}
and 
$$g=\psi_b\circ (w_0+h)\circ \phi_a-w_a,\quad w_a=\boxplus_a(u^+_0, u^-_0).$$
The pair $(\eta^+, \eta^-)$ is obtained by solving the system 
\begin{align*}
\wh{\oplus}_a (\eta^+, \eta^-)&=\eta \circ \phi_a\\
\wh{\ominus}_a(\eta^+, \eta^-)&=0.
\end{align*}
We have used that $\wh{\pi}_a(\eta^+, \eta^-)=(\eta^+, \eta^-)$ is equivalent to $\wh{\ominus}_a(\eta^+, \eta^-)=0$. Solving the above system leads to the formulae
\begin{align*}
\eta^+(s, t)&=\frac{ \beta_a (s)}{\gamma_a(s) }\cdot \eta \circ \phi_a (s, t),\quad s\geq 0\\
\eta^-(s-R, t-\vartheta)&=\frac{1-\beta_a(s)}{\gamma_a (s)}\cdot \eta \circ \phi_a(s-R, t-\vartheta), \quad s\leq R.
\end{align*}

From the formulae for $\Phi^{-1}\circ \Gamma$ and $\Gamma^{-1}\circ \Phi $ we conclude, using Theorem \ref{thm-125p}, Proposition \ref{qed}, and the Lemmta \ref{lem2-18}--\ref{lem2.21-nnew}, arguing as in the proof of Proposition \ref{prop3.14-new}, that the chart transformations are $\ssc_{\triangleleft}$-smooth bundle isomorphisms. 

Consequently, the bundle $E\to V\oplus\ov{X}$ admits the structure of a strong bundle over the M-polyfold $V\oplus \ov{X}$ and the proof of Theorem \ref{thm1.42-nnew} is complete.

 \end{proof}

\subsection{The Cauchy-Riemann Operator}\label{nsection3.4}

We define the Cauchy-Riemann section $\ov{\partial}$ of the M-polyfold bundle
$E\rightarrow V\times\ov{X}$
by
$$\ov{\partial}(v,a,b,w)=(v,a,b,w;\ov{\partial}_{i, j(a, v)}w)$$
if $ab\neq 0$. Here $w:Z_a\to Z_b$ is a diffeomorphism, and the complex anti-linear map 
$\ov{\partial}_{i, j(a, v)}w(z):(T_zZ_a, j(a, v))\to (T_{w(z)}Z_b, i)$ is defined by 
$$
\ov{\partial}_{i, j(a, v)}w(z):=\frac{1}{2}[Tw+i\circ(Tw)\circ j(a,v)](z),\quad z\in Z_a.
$$
If $a=0$ and hence $b=0$, the section $\ov{\partial }$ is defined by 
$$
\ov{\partial}(v,0,0,(u^+,u^-))=(v,0,0,(u^+,u^-); (\ov{\partial }_{i, j^+(v)}u^+, \ov{\partial }_{i, j^-(v)}u^-)).
$$
Our aim is to prove that the section $\care$  is a polyfold Fredholm section.
\begin{prop}\label{prop3.19-nnew}
The Cauchy-Riemann section $\ov{\partial}$ is sc-smooth and regularizing.
\end{prop}
\begin{proof}
In order to verify the regularizing property we assume that 
$$
\ov{\partial}(v,a,b,w)\in E_{m,m+1}.
$$
 It then follows from elliptic regularity theory that $w\in H^{m+4}_{loc}$. Hence, if $a\neq 0$ we conclude that $(v,a,b,w)\in X_{m+1}$. Considering now the case  $a=0$ in which case also  $b=0$, 
we have  $u=(u^+,u^-)$ with the two diffeomorphisms $u^\pm$ of $\R^+\times S^1$ of the form $u^\pm (s, t)=(s, t)+(d^\pm_0, \vartheta^\pm_0)+r^\pm (s, t)$, where, by the assumption $r^\pm\in H^{3+m, \delta_m}(\R^\pm \times S^1, \R^2)$.  On $\R^+\times S^1$, the complex structure is equal to $j(a, v)(s, t)=i$ if $s\geq s_0$ and hence we conclude from the assumption,  that 
$$\partial_sr^+ +i\partial_tr^+\in H^{3+m,\delta_{m+1}}([s_0, \infty )\times S^1, \R^2).$$

Now we recall that the weights $\delta_m$ have been chosen in the open interval $(0, 2\pi)$. Since the asymptotic structure of the differential  equation is $\frac{\partial }{\partial s}u=Au$, with the self-adjoint operator $A=-i\frac{d}{dt}$ in the Hilbert space $L^2(S^1)$ having the spectrum $2\pi\Z$, we therefore conclude by the standard arguments  going back to Lockhardt and  McOwen  \cite{Lockhardt} that $r^+\in H^{4+m,\delta_{m+1}}([s_0, \infty )\times S^1, \R^2)$.

By the same arguments,   $r^-\in H^{4+m,\delta_{m+1}}((-\infty, -s_0]\times S^1, \R^2)$ and hence 
$(u^+, u^-)\in X_{m+1}$.

In order to verify that the section $\ov{\partial}$ is sc-smooth, we have to study its coordinate representation. We recall the strong bundle chart $\Gamma:K\to E$,
$$
\Gamma(v, a,h^+,h^-, \eta^+, \eta^-)=(v, a,b(a,u^+_0+h^+,u^-_0+h^-),\boxplus_a(u^+_0+h^+,u^-_0+h^-) ; \xi)
$$
where, abbreviating the diffeomorphism $w_a=\boxplus_a(u^+_0+h^+,u^-_0+h^-) :Z_a\to Z_b$, the complex anti-linear map $\xi (z):(T_zZ_a, j(a, v))\to (T_{w_a(z)}Z_b, i)$  is defined by $\xi (z)=\wh{\oplus}_a(\eta^+, \eta^-)(z)$ for $z\in Z_a$.

Consequently, the Cauchy-Riemann section $\ov{\partial }$ is,   in these local coordinates,   represented by the map
$$(v, a, h^+, h^-)\mapsto (v, a, h^+, h^-; \eta^+, \eta^-),$$
where $(\eta^+, \eta^-)$ is the unique solution of the system 
\begin{align*}
\wh{\oplus}_a(\eta^+,\eta^-)&=\left( \ov{\partial}_{i,j(a,v)}w_a\right)\frac{\partial }{\partial s}\\
\wh{\ominus}_a(\eta^+,\eta^-)&=0.
\end{align*}
The solution $\eta^+$ is equal to 
$$
\eta^+(s,t)=\frac{\beta_a}{\gamma_a}\left(\ov{\partial}_{i,j(a,v)}w\right)\left(\frac{\partial}{\partial s}\right),
$$
where,  as usual,  we abbreviate 
$\beta_a=\beta_a(s)=\beta\left(s-\frac{R}{2}\right)$ and $\gamma_a=\beta_a^2+(1-\beta_a)^2$
Recalling that $u^\pm_0(s, t)=(s, t)+(d_0^\pm, \vartheta_0^\pm)+r^\pm_0$and representing $h^\pm=h^\pm_{\infty}+r^\pm_1$ where  $r^\pm_0$ and $r^\pm_1$ belong to  $H^{3, \delta_0}(\R^\pm \times S^1, \R^2)$, the map $w_a= \boxplus_a(u^+_0+h^+,u^-_0+h^-) :Z_a\to Z_b$ is equal to
\begin{equation*}
\begin{split}
w_a(s, t)&=(s, t)+(d_0^+, \vartheta^+_0)+h^+_{\infty}\\
&\phantom{=}+\beta_a\cdot r^+(s, t)+(1-\beta_a)\cdot r^-(s-R, t-\vartheta).
\end{split}
\end{equation*}
for $0\leq s\leq R$, where $r^\pm:=r^\pm_0+r_1^\pm$.  
Observing  that on $Z_a$, in the coordiates $s\geq 0$, the complex structure is equal to $j(a, v)(s, t)=j^+(v)$ while in the coordinates $s'\leq 0$ we have  $j(a, v)(s', t')=j^-(v)$, one obtains for the solution $\eta^+(s, t)$ for $(s, t)\in \R^+\times S^1$, 
\begin{equation*}
\begin{split}
\eta^+(s,t)&=\frac{\beta_a}{\gamma_a}\left(\ov{\partial}_{i,j^+(v)}\id \right)\left(\frac{\partial}{\partial s}\right)(s,t)\\
&\phantom{=}+ \frac{\beta_a^2}{\gamma_a}
\left[ {\partial}_{i,j^+(v)}r^+(s,t)\right] \left(\frac{\partial}{\partial s}\right)\\
&\phantom{=}+\frac{\beta_a(1-\beta_a)}{\gamma_a}\left[ \ov{\partial}_{i,j^-(v)}r^-(s-R,t-\vartheta)\right]\left(\frac{\partial}{\partial s}\right)\\
&\phantom{=}+\frac{\beta_a\beta_a'}{\gamma_a}[r^+(s,t)-r^-(s-R,t-\vartheta) ].
\end{split}
\end{equation*}

A similar formula holds for $\eta^-(s', t')$ on $\R^-\times S^1$.   
Recalling that the complex structure $j^+(v)$ is equal to the standard complex structure $i$ for $s\geq s_0$, we see that 
$(\ov{\partial}_{i,j^+(v)}\id ) \left(\frac{\partial}{\partial s}\right)(s,t)=0$ for $s\geq s_0$. 
Recall that $\frac{\beta_a(s)}{\gamma_a(s)}=1$ if $s\leq \frac{R}{2}-1$.  If $\abs{a}$ is so small that $s_0\leq \frac{R}{2}-1$, then for all $(s, t)\in \R^+\times S^1$,  the function
$$\frac{\beta_a}{\gamma_a}\left(\ov{\partial}_{i,j^+(v)}\id \right)\left(\frac{\partial}{\partial s}\right)(s,t)
=\left(\ov{\partial}_{i,j^+(v)}\id \right)\left(\frac{\partial}{\partial s}\right)(s,t)$$
has a compact support, hence belongs to $H^{2,\delta_0}(\R^+\times S^1, \R^2)$ and is independent of $a$. Hence as a function of $a$ into $H^{2,\delta_0}(\R^+\times S^1, \R^2)$, the map is constant. 
Applying the chain rule, Proposition \ref{qed}, and the fact that the operators $\partial_s, \partial_t:H^{3,\delta_0}\to H^{2,\delta_0}$ are sc-linear, one concludes  that the  remaining terms  in the formula for $\eta^+$ define maps which depend sc-smoothly on $(v, a, h^+, h^-)$. Consequently, 
the maps $(v, a, h^+, h^-)\mapsto \eta^\pm (v, a, h^+, h^-)$ are sc-smooth and we have proved that the section $\ov{\partial }$ is sc-smooth.
\end{proof}

In the next step we shall prove that $\ov{\partial}$
is a polyfold Fredholm section. By definition this means that
around every smooth point $(v,(a,b,w))\in V\oplus \ov{X}$ the Cauchy-Riemann operator has a germ which can be brought into a `nice' form by a suitable coordinate change.   

We consider two cases. The easy case  concerns points in $X$ where  an open neighborhood is sc-diffeomorphic to an open subset of an sc-Hilbert space. The interesting case concerns points in 
$\ov{X}\setminus X$ whose local description is that of a nontrivial retract. Here the  concept of a filler  will play a decisive role. 

The following lemma  takes care of the easy case.

\begin{lem}\label{lem3.19-nnew}
Around a smooth point $(v, x)\in V\times X$ the germ $(\ov{\partial}, (v, x))$ of the Cauchy-Riemann section is an sc-Fredholm germ (in the sense of Section \ref{Fredholm-section}).
\end{lem}
 \begin{proof}
 We choose the smooth point $(v_0, a_0, b_0, w_0)\in V\times X$. Then $w_0:Z_{a_0}\to Z_{b_0}$ is a diffeomorphism  preserving the distinguished points at  the boundaries.
Around this point we choose a chart of $V\times X$ as in Theorem \ref{thm42.-nnnew} by means of the map 
$$
\Phi( v, a, b, h)=(v, a, b, \psi_{b}\circ (w_0+ h)\circ \phi_{a})
$$
where $( v,  a,  b)$ belongs to an open neighborhood of the origin in $V\times \C\times \C$ and where $ h$ belongs to the subspace $\wh{H}^3(Z_{a_0}, \R^2)$ of 
$H^3(Z_{a_0}, \R^2)$, which consists  of all $h$ satisfying $h(p^{\pm}_a)=(0,0)$, $h([0,t])\in \{0\}\times {\mathbb R}$, and $h([0,t]')\in\{0\}\times {\mathbb R}$. Recalling  the diffeomorphisms
\begin{align*}
&\psi_{b}:Z_{b_0}\to Z_{b}\\
&\phi_{a}:Z_{a}\to Z_{a_0},
\end{align*}
the Cauchy-Riemann operator takes  the form
$$
( v, a, b,  h)\mapsto \ov{\partial}_{i,j(a,v)}(\psi_{b}\circ (w_0+h)\circ\phi_{a}).
$$
On  the cylinder  $Z_{b_0}$ we introduce the parameter depending complex structure $k( b)$,  defined by 
$$
k( b):= (T\psi_{ b})^{-1}\circ i \circ ( T\psi_{ b}),
$$
and on  the cylinder $Z_{a_0}$ the complex structure  
$$
\wh{j}( a, v) :=( T\phi_{ a})\circ j(a,v)\circ (T\phi_{a})^{-1}, 
$$
so that the Cauchy-Riemann operator can be written as 
$$
( v, a, b,  h)\mapsto T\psi_{ b} \circ \left(\frac{1}{2}\left[ T(w_0+ h)+k( b)\circ T(w_0+ h)\circ \wh{j}( a, v)\right]\right)
\circ T\phi_{ a}.
$$
Now we can introduce the obvious strong bundle coordinates in which  the local expression of the Cauchy-Riemann  section  is as follows,
\begin{equation}\label{coo}
( v, a, b, h)\mapsto 
\frac{1}{2}[ T(w_0+ h)+k( b)\circ T(w_0+ h)\circ \wh{j}( a, v)]\cdot \left(\frac{\partial}{\partial s}\right).
\end{equation}
As usual we identify the tangent spaces at points of the cylinder $Z_{b_0}$,  as real vector spaces, with $\R^2$. Hence the above section associates with $( v, a, b, h)$
a  function  in $H^2( Z_{a_0}, \R^2)$.

In order to verify the Fredholm property of the Cauchy-Riemann section one needs to show 
 near the smooth point $(v_0,a_0,b_0,w_0)$ the contraction germ property.  This  requires to study the map \eqref{coo} for   small data $(a, b, h)$.
 
Observing that for $ h=0$  the right hand side of \eqref{coo} is an sc-smooth map, we define the $\ssc^+$-section $s$ of the strong local bundle by its principal part 
$$
s( v, a, b, h)=\frac{1}{2}[Tw_0+k( b)\circ Tw_0\circ\wh{ j}( a, v)]\cdot \left(\frac{\partial}{\partial s} \right).
$$
Here again  we identify the tangent fibers of  of the tangent bundle $TZ_{b_0}$ of the cylinder with ${\mathbb R}^2$. Denoting the Cauchy-Riemann section  by $f$ we now study the section 
$f-s$  whose principal part is given by 
\begin{equation}\label{coo1}
(f-s)( v, a, b, h)=
\frac{1}{2}[ T( h)+k( b)\circ T( h)\circ \wh{j}( a, v)]\cdot \left(\frac{\partial}{\partial s}\right).
\end{equation}
We shall abbreviate the parameters by $\lambda=( v, a,  b)$. They  vary in the finite dimensional vector space $\Lambda :=V\times  \C\times \C$ near its origin. 
Moreover, we abbreviate  the sc-spaces $E=\wh{H}^3(Z_{a_0}, \R^2)$ and $F=H^{2}(Z_{a_0}, \R^2)$ with the sc-structures $E_m=\wh{H}^{3+m, \delta_m}(Z_{a_0}, \R^2)$ and  $F_m=H^{2+m}(Z_{a_0}, \R^2)$ and denote the right hand side of \eqref{coo1} by $T_{\lambda}( h)$. 
Then $T_{\lambda}:E\to F$ is a family of bounded linear operators which are Fredholm operators 
$E_m\to F_m$ for all $m$ and depend in their operator norms smoothly on the parameters for all levels $m$. Setting $\lambda=\lambda_0=(v_0, a_0, b_0)$, the sc-Fredholm operator $T_{\lambda_0}\equiv T_0:E\to F$ defines the sc-splitings $E=Y\oplus K$ where $K$ is the kernel of $T_0$ and $F=R(T_0)\oplus C$ with the range  $R(T_0)$ and the cokernel  $C$ of $T_0$. Let 
$$P:F=R(T_0)\oplus C\to R(T_0)$$
be the projection map  and choose  a linear isomorphism $\phi:\R^N\to C$ onto the cokernel of $T_0$. Then the restriction $T_0\vert Y:Y\to R(T_0)$ is an sc-isomorphism. Moreover, the linear map 
$$PT_0+\phi:Y\oplus \R^N\to F,$$
defined by 
$$(PT_0+\phi )(y+r)=PT_0(y)+\phi (r)$$
is an sc-isomorphism. Introducing the sc-bundle isomorphism
$$\Phi:(\Lambda\oplus E)\triangleleft F\to  (\Lambda\oplus K\oplus Y)\triangleleft (Y\oplus \R^N),$$
defined by 
$\Phi (\lambda, k+y, f)=(\lambda, k, y; (PT_0+\phi )^{-1}f)$ (for small $(\lambda, k+y)$)
and denoting by 
$g(\lambda, \delta h)=T_{\lambda}(\delta h)$ the principal part of the Cauchy-Riemann section, we obtain for the push-forward  section 
$$\Phi_{\ast}(g):  \Lambda\oplus K\oplus Y\to  Y\oplus \R^N$$
the expression $\Phi_{\ast}(g)(\lambda, k, y)= (PT_0+\phi )^{-1}T_{\lambda}(k+y).$ With the projection 
$$Q:Y\oplus \R^N\to Y,$$  the sc-smooth germ 
$(\lambda, k, y)\mapsto Q\Phi_{\ast}(g)(\lambda, k, y)$  satisfies, by construction 
$Q\Phi_{\ast}(g)(0,0, y)=y$ and hence is of the form 
$$Q\Phi_{\ast}(g)(\lambda, k, y)=y-B(\lambda, k, y)$$
where $B$ satisfies $B(0, 0, y)=0$ and $\abs{B(\lambda, k, y)-B(\lambda, k, y')}_m\leq \varepsilon\cdot \abs{y-y'}_m$ for every $m\geq 0$ and every $\varepsilon>0$ if only $(\lambda, k)$ is sufficiently small (depending on $\varepsilon$ and $m$).

Having found coordinates in which the  Cauchy-Riemann section near smooth points possesses the contraction germ property, the proof of Lemma \ref{lem3.19-nnew} is complete.

\end{proof}

Next we turn to  the more interesting case and  fix in $V\times (\ov{X}\setminus X)$ the smooth point
$(v_0,0,0,(u^+_0,u^-_0))$.  We are going to construct a filler for the Cauchy-Riemann section $\ov{\partial }$ near this point. The construction is based on the Cauchy-Riemann operator
$$\ov{\partial }_0:H^{3, \delta_0}_c(C_a, \R^2)\to H^{2, \delta_0}(C_a, \R^2)$$
defined by
$$(\ov{\partial }_0\xi )(s, t)=\frac{1}{2}\left( \frac{\partial }{\partial s}\xi +i\frac{\partial }{\partial t}\xi \right)(s, t)=\frac{1}{2}[ T\xi +i\circ T\xi\circ i](s, t)\left(\frac{\partial}{\partial s}\right)$$
where we have used the coordinates $(s, t)\in \R\times S^1$ for the glued infinite cylinder $C_a$, which is equipped with the standard complex structure $i$. We recall that the Hilbert space $H^{3, \delta_0}_c(C_a, \R^2)$ consists of all maps $u$ in $H^3_{\text{loc}}(C_a, \R^2)$ for which there exists a constant $c\in \R^2$ having the property that $u(s, t)-c$ has  weak partial derivatives up to order $3$ which, if weighted by $e^{\delta_0\abs{s}}$ belong to the space $L^2([0, \infty )\times S^1)$  and $u(s, t)+c$ has the same properties with respect to $L^2((-\infty , 0]\times S^1)$. Consequently, $u$ converges at $\pm \infty$ exponentially fast to the antipodal points $c$ and $-c$. We equip $H^{3, \delta_0}_c(C_a, \R^2)$ with the sc-structure for which the level $m$ corresponds to the Sobolev regularity $(m+3, \delta_m)$ and the  sc-Hilbert space  $H^{2,\delta_0}(C_a, \R^2)$ with the sc-structure for which the level $m$  corresponds to the  regularity $(2+m, \delta_m)$.
The norms are defined in Section \ref{esttotalgluing}.

\begin{lem}\label{isomo-lem}
The operator $\ov{\partial}_0$ is an sc-isomorphism.
\end{lem}

Identifying the cylinder with the Riemann sphere with two antipodal points removed, the lemma follows from the results about the Cauchy-Riemann operator acting on  maps on the Riemann sphere  and discussed for example in \cite{HZ}, and from the  asymptotic  study of the operator near the ends.

In order to introduce a filler  for  the Cauchy-Riemann section $\ov{\partial }$ we first recall the local model for the strong M-polyfold bundle and start with the bundle 
$$(V\oplus B_{\varepsilon}\oplus \wh{E}_0)\triangleleft F\to V\oplus B_{\varepsilon}\oplus \wh{E}_0$$
where $F$ is the sc-Hilbert space consisting of pairs $(\eta^+, \eta^-)$ of functions in 
$H^{2, \delta_0}(\R^+\times S^1, \R^2)\oplus H^{2, \delta_0}(\R^-\times S^1, \R^2)$ equipped, as usual, with the sc-structure $F_m=H^{2+m, \delta_m}(\R^+\times S^1, \R^2)\oplus H^{2+m, \delta_m}(\R^-\times S^1, \R^2)$.

The retraction $R:(V\oplus B_{\varepsilon}\oplus \wh{E}_0)\triangleleft F\to (V\oplus B_{\varepsilon}\oplus \wh{E}_0)\triangleleft F$ is defined by 
$$R(v, a, (h^+, h^-), (\eta^+, \eta^-))=(v, a, \rho_a(h^+, h^-), \wh{\pi}_a(\eta^+, \eta^-))$$
and covers the retraction 
$r:V\oplus B_{\varepsilon}\oplus \wh{E}_0\to V\oplus B_{\varepsilon}\oplus \wh{E}_0$, defined by 
$$r(v, a, (h^+, h^-))=(v, a, \rho_a(h^+, h^-)).$$
With the retracts $K:=R((V\oplus B_{\varepsilon}\oplus \wh{E}_0)\triangleleft F)$ and $O=r(V\oplus B_{\varepsilon}\oplus \wh{E}_0)$ we obtain  the local model 
$$K\to O$$
of the strong M-polyfold bundle. In these local coordinates the Cauchy-Riemann section $\ov{\partial }:O\to K$ near the point $(v, a, h^+, h^-)=(v_0, 0, 0, 0)$  becomes
$$
\ov{\partial }(v,  a,  h^+,  h^-)=(v,  a,  h^+,  h^-; \eta^+, \eta^-)
$$
where $(\eta^+, \eta^-)$ is a solution of 
\begin{align*}
\wh{\oplus}_{ a}(\eta^+, \eta^-)&=\left(\ov{\partial}_{i, j(a, v)}(w)\right)\cdot \left(\frac{\partial }{\partial s}\right)\\
\wh{\ominus}_{ a}(\eta^+, \eta^-)&=0
\end{align*}
in the case $ a\neq 0$, where $w=\boxplus_{ a}(u_0^++ h^+, u_0^-+ h^-):Z_{ a}\to Z_b$ is a diffeomorphism and $b=b( a, u_0^++ h^+, u_0^-+ h^-).$ If $ a=0$, then 
\begin{equation*}
\begin{split}
&\ov{\partial }(v, 0,  h^+,  h^-)=\\
&(v,0,  h^+,  h^-;  \ov{\partial  }_{i, j^+(v)}(u^+_0+ h^+), 
\ov{\partial  }_{i, j^-(v)}(u^-_0+h^-).
\end{split}
\end{equation*}

A filler  for $\ov{\partial}$ is an extension of the local section $\ov{\partial }:O\to K$ to a section
$$g:V\oplus B_{\varepsilon}\oplus \wh{E}_0\to (V\oplus B_{\varepsilon}\oplus \wh{E}_0)\triangleleft F$$
of the original bundle which is a section defined on  an open set in an sc-Hilbert space. 
It is defined as follows. If $ a=0$, we set $g( v, 0,  h^+,  h^-)=
\ov{\partial }(v, 0,  h^+,  h^-).$  If $ a\neq 0$, we define the section 
$$g(v,  a,  h^+,  h^-)=(v,  a,  h^+,  h^-; \eta^+, \eta^-)$$
the following way. We take $( h^+,  h^-)\in \wh{E}_0$ and construct the associated 
diffeomorphism $\boxplus_{ a}(u_0^++ h^+, u_0^-+ h^-):Z_{ a}\to Z_b$  between the glued cylinders where $b=b( a, u_0^++ h^+, u_0^-+ h^-).$ In addition, we define the function 
$\xi\in H_c^{2, \delta_0}(C_{ a}, \R^2)$ by 
$$\xi=\ominus_{ a}( h^+-h^+_{\infty},  h^--h^-_{\infty}).$$
In view of Theorem \ref{second-splicing}, we have the sc-isomorphism
$$(\wh{\oplus}_{ a}, \wh{\ominus}_{ a}):F\to H^{2}(Z_{ a}, \R^2)\oplus H_c^{2, \delta_0}(C_{ a}, \R^2)$$
and define the principal part of the section $g$ as the unique solution of the two equations
\begin{align*}
\wh{\oplus}_{ a}(\eta^+, \eta^-)&=\left(\ov{\partial}_{i, j( a, v)}(w)\right)\cdot 
\left(\frac{\partial }{\partial s}\right)\\
\wh{\ominus}_{ a}(\eta^+, \eta^-)&=\ov{\partial }_0\xi.
\end{align*}
\begin{lem}\label{llem3.22-nnew}
The above section $g$ whose  principal part  is given by
$$
(v, 0,  h^+,  h^-)\mapsto  (\eta^+, \eta^-), 
$$
and which is  defined on  an open neighborhood of $(v_0, 0, (0, 0))$ in $V\oplus \C\oplus \wh{E}_0$ and has  its  image in $F$,  is a filled version of the local Cauchy-Riemann section $\ov{\partial }:O\to K$ near the point  $(v_0, 0, 0,u^+_0, u^-_0)\in V\times (\ov{X}\setminus X).$ 
\end{lem}
\begin{proof}
As in Proposition \ref{prop3.19-nnew} one sees that the section is sc-smooth. Recalling Definition \ref{filled-def} of a filler in Section \ref{Fredholm-section}, we have to verify the three defining properties. 

If we  restrict
 $g$ to the sc-smooth retract $O$, then,  in view of Lemma  \ref{kernel-lem1},   it follows from $\rho_{ a}( h^+,  h^-)=( h^+,  h^-)$ that
 $$
 \xi=\ominus_{ a}( h^+-h^+_{\infty},  h^--h^-_{\infty})=\ominus_{ a}( h^+,  h^-+ h^+_{\infty}- h^-_{\infty})=0, 
 $$
 and therefore 
$\wh{\ominus}_{ a}(\eta^+, \eta^-)=0$ so that $(v,  a,  h^+,  h^-; \eta^+, \eta^-)\in K$. Consequently, $g\vert O$ is a section of the bundle $K\to O$. By construction, it is the local coordinate representation  of the Cauchy-Riemann section $\ov{\partial }:O\to K$ near the originally given point. This proves the property (1) of the requirements to be a filled version.

In order to verify property (2) of a filler, we assume that 
$$g(y)=\phi(r(y))\cdot g(y)$$
for a point $y=(v,  a,  h^+,  h^-)\in V\oplus B_{\varepsilon}\oplus \wh{E}_0$ close to $(v_0,0, 0, 0)$. We conclude that 
$$(\eta^+, \eta^-)=\wh{\pi}_{ a}(\eta^+, \eta^-)$$
and hence $\wh{\ominus}_{ a}(\eta^+, \eta^-)=0$. Since $\ov{\partial}_0$ is an isomorphism in view of Lemma \ref{isomo-lem}, we obtain 
$\ominus_{ a} ( h^+ - h^+_{\infty},  h^-  - h^-_{\infty})=\ominus_{ a} ( h^+,  h^- + h^+_{\infty}- h^-_{\infty})=0$ so that  by Lemma \ref{kernel-lem1}
again, 
$\rho_{ a}( h^+,  h^-)=( h^+,  h^-)$ and consequently, 
$(v,  a,  h^+,  h^-)\in O$. We have verified that  the section $g$ satisfies property (2) of a filler. 

Finally, the third property of a  filler  is easily verified. At the point 
$y_0=(v_0, 0, (0, 0))\in V\oplus \C\oplus \wh{E}_0$, the derivative of the retraction $Dr (y_0)=\id$ is the identity map so that $\text{kernel}\ Dr (y_0)=\{0\}$. Since $\phi(y_0)=\wh{\pi}_0=\id$ we also have that  $\text{kernel}\   \phi(y_0)=\{0\}$. Consequently, the linearization of the map 
$y\mapsto [\id -\phi(r(y))]\cdot g(y)$ at the point $y_0$ restricted to $\ker Dr(y_0)$ defines trivially an isomorphism $\ker Dr(y_0)\to \ker \phi(y_0)=\{0\}$. 
The proof of Lemma \ref{llem3.22-nnew} is complete.
\end{proof}

We are going to verify the Fredholm property of the section $g$ near the smooth point $(v_0, 0, 0, u^+_0, u^-_0)\in V\times \ov{X}_{\infty}$, where 
$$u^\pm_0(s, t)=(s, t)+(d^\pm_0, \vartheta^\pm_0)+r^\pm_0(s, t)$$
are diffeomorphisms  of the half-cylinders $\R^+ \times S^1$ and $\R^- \times S^1$ which  satisfy
$r^\pm_0\in \bigcap_{m\geq 0}H^{3+m, \delta_m}(\R^\pm \times S^1, \R^2).$

For this purpose we introduce, in our local coordinates near the above distinguished point, the 
$\ssc^+$-section 
$$g_0(v, a, h^+, h^-)=g(v, a, 0,0)\in F_{\infty}$$
for  $(v, a, h^+, h^-)$ near $(v_0, 0, 0, 0)$. 

\begin{lem}\label{localexp}
There exists  $\tau\in (0,\frac{1}{2})$ which depends only  on the size of the support of $j^\pm(v)$ defined by the parameter $s_0>0$  for which 
$j^+(v)=i$ on $[s_0,\infty)\times S^1$ for $v\in V$,  and $j^-(v)=i$ on $(-\infty, -s_0]\times S^1$, so that the following holds.
If $0<|a|<\tau$, then the  principal part  $(v, a, h^+, h^-)\mapsto (g-g_0)(v, a, h^+, h^-)\in F$ of the sc-smooth section $g-g_0$ is given by the formula,
\begin{equation*}
\begin{split}
(g-g_0)(v, a, h^+, h^-)&=
\begin{bmatrix}
\frac{1}{2}[ Th^+ +i\circ Th^+\circ j^+(v)] \left( \frac{\partial }{\partial s} \right) \\[1.5ex]
\frac{1}{2}[ Th^- +i\circ Th^+\circ j^-(v)] \left( \frac{\partial }{\partial s} \right)
\end{bmatrix}\\
&\phantom{=}+\frac{\beta_a'}{\gamma_a}\cdot 
\begin{bmatrix}
\beta_a&\beta_a-1\\
1-\beta_a&\beta_a
\end{bmatrix}
\cdot
\begin{bmatrix}
h^+ -h^- -h^+_{\infty}+h^-_{\infty}\\
h^+ +h^- -2 \av  (h^+, h^-)
\end{bmatrix}
\end{split}
\end{equation*}
where, as usual, $\beta_a=\beta_a(s)=\beta (s-\frac{R}{2})$ and $h^+=h^+(s, t)$ and $h^-=h^-(s-R, t-\vartheta)$ and $0\leq s\leq R$. Moreover, $\beta_a'$ stands for the derivative of the function $\beta_a$.

If $a=0$,  then the principal part $(v, 0, h^+, h^-)\mapsto (g-g_0)(v, 0, h^+, h^-)\in F$ of the sc-smooth section $g-g_0$ is given by the formula,
\begin{equation*}
\begin{split}
(g-g_0)(v, 0, h^+, h^-)&=
\begin{bmatrix}
\frac{1}{2}[ Th^+ +i\circ Th^+\circ j^+(v)] \left( \frac{\partial }{\partial s} \right) \\[1.5ex]
\frac{1}{2}[ Th^- +i\circ Th^+\circ j^-(v)] \left( \frac{\partial }{\partial s} \right)
\end{bmatrix}.
\end{split}
\end{equation*}

\end{lem}
\begin{proof}
In the following proof we use the  abbreviated notations 
\begin{align*}
\ov{\partial}_vw\equiv=\care w:&=\frac{1}{2}[ Tw+i\circ Tw\circ j(a, v)]\left(\frac{\partial }{\partial s}\right)\\
\carem w&:=\frac{1}{2}[ Tw+i\circ Tw\circ i]\left(\frac{\partial }{\partial s}\right)
\end{align*}
for the Cauchy-Riemann operators. Recalling that  the section $g$ has the principal part $(\eta^+, \eta^-)\in F$ defined by the equations
\begin{align*}
\wh{\oplus}_{a}(\eta^+, \eta^-)&=\care  (w)\\
\wh{\ominus}_{a}(\eta^+, \eta^-)&=\carem ( \ominus_a(h^+-h^+_{\infty}, h^--h^-_{\infty}))
\end{align*}
$$w=\boxplus_a(u_0^+ +h^+, u_0^-+h^-):Z_a\to Z_b,$$
the principal part of the $\ssc^+$-section $g_0(a, v, h^+, h^-)=(\eta^+_0, \eta^-_0)$ is determined by 
the equations
\begin{align*}
\wh{\oplus}_{a}(\eta_0^+, \eta_0^-)&=\bar{\partial} (w_0)\\
\wh{\ominus}_{a}(\eta_0^+, \eta_0^-)&=0
\end{align*}
$$w_0=\boxplus_a(u_0^+, u_0^-).$$
Note that $(\eta_0^+,\eta^-_0)$ is a function of $(a,v)$ which we suppress in the notation. Also $\bar{\partial}$ is depending on $(a,v)$.
Observing that $w-w_0=\boxplus_{a}(u^+_0+h^+, u^-_0+h^-)- \boxplus_{a}(u^+_0, u^-_0)=\oplus_a(h^+, h^-+h^+_{\infty}-h^-_{\infty})$ and recalling the definitions of $\ominus_a$ and $\wh{\ominus}_a$ and $\wh{\oplus}_a=\oplus_a$, we see that the principal part $(\eta^+-\eta^+_0, \eta^--\eta^-_0)$ of the section $g-g_0$ is characterized by the equations
\begin{equation*}
\begin{split}
\begin{bmatrix}
\wh{\oplus}_{a}(\eta^+-\eta^+_0, \eta^--\eta^-_0)\\ 
\wh{\ominus}_{a}(\eta^+-\eta^+_0, \eta^--\eta^-_0)
\end{bmatrix}
&=
\begin{bmatrix}
\wh{\oplus}_{a}(\care h^+, \care h^-)\\
\wh{\ominus}_{a}(\carem h^+, \carem h^-)
\end{bmatrix}\\
&\phantom{=\ }
+\beta_a' 
\cdot
\begin{bmatrix}
h^+ -h^- -h^+_{\infty}+h^-_{\infty}\\
h^+ +h^- -2 \av  (h^+, h^-)
\end{bmatrix}.
\end{split}
\end{equation*}
In view of the definition of $\beta_a$ and $j(a, v)$ we conclude that 
$$\wh{\ominus}_a(\carem h^+, \carem h^-)=\wh{\ominus}_a(\care h^+, \care h^-)$$
and hence obtain
\begin{equation}\label{cr-formula}
\begin{split}
\begin{bmatrix}
\eta^+ -\eta^+_0\\
\eta^-_0 -\eta^-_0
\end{bmatrix}
&=
\begin{bmatrix}
\care h^+\\
\care h^-
\end{bmatrix}\\
&+\frac{\beta_a'}{\gamma_a}\cdot 
\begin{bmatrix}
\beta_a&\beta_a-1\\
1-\beta_a&\beta_a
\end{bmatrix}
\cdot
\begin{bmatrix}
h^+ -h^- -h^+_{\infty}+h^-_{\infty}\\
h^+ +h^- -2 \av  (h^+, h^-)
\end{bmatrix}.
\end{split}
\end{equation}
Since $j(a, v)=j^\pm(v)$ on $\R^\pm \times S^1$, the lemma is proved.

\end{proof}
Let us define $L(a,v)$ by \eqref{cr-formula} via
\begin{equation}\label{eq73}
(a,v,h^+,h^-)\rightarrow L(a,v)(h^+,h^-):=(g-g_0)(a,v,h^+,h^-).
\end{equation}
In view of the above lemma,  the section $g-g_0:V\oplus B_{\varepsilon}\oplus \wh{E}_0\to F$ is given by a family $L(\lambda):\wh{E}_0\to F$ of linear operators (not continuous as a family of operators on any level) parametrized by $\lambda =(v, a)$ and of the form 
$$
L(\lambda )=D_v+\Delta_a,
$$
where $D_v:\wh{E}_0\to F$ is the Cauchy-Riemann operator. From the formulae in the above lemma, we see that 
at $a=0$ we have $ \Delta_0 (h^+, h^-)=0$ and that the map $(a, h^+, h^-)\mapsto \Delta_a (h^+, h^-)$ is the sc-smooth 
in view of  the Lemmata \ref{lem2-18}-\ref{lem2-22}.
Hence we have to study for data $(\lambda,h^+,h^-)$ near $(\lambda_0,0,0)$ the  sc-smooth map
\begin{equation}\label{eq74}
(\lambda,h^+,h^-)\rightarrow L(\lambda)(h^+,h^-).
\end{equation}
For the later discussion we  note the following formula
\begin{equation}\label{CRUCIALLL}
\begin{split}
(\wh{\oplus}_a,\wh{\ominus}_a)&(L(a,v)(h^+,h^-))\\
& = (\ov{\partial}_v(\oplus_a(h^+-h^+_\infty)),\bar{\partial}_0(\ominus_a(h^+-h^+_\infty,h^--h^-_\infty)))
\end{split}
\end{equation}
which follows immediately from the proof of  the previous lemma.
\begin{rem}\label{rem-x}The Cauchy-Riemann operator $D_{v}$ is a classical Fredholm operator between $(\wh{E}_0)_m$ and $F_m$ for every $m\geq 0$. One  verifies readily  that $D_v$ is bijective. Indeed , this follows from the classical fact that the standard Cauchy-Riemann operator
acting on ${\mathbb C}$-valued functions on the disk with real boundary conditions  is surjective with a 1-dimensional kernel.
Then in view of the boundary conditions of the functions in the domain $\wh{E}_0$, the kernel of $D_v$ is equal to $\{0\}$ and hence $D_v$ is an isomorphism. The linear operator $\Delta_a$ is compact and hence  the operators $L(\lambda)$ are all Fredholm operators  of index $0$.
\end{rem}
\begin{prop}\label{family} We consider the sc-smooth map 
$(\lambda, h^+, h^-)\mapsto L(\lambda)(h^+, h^-)$ defined in \eqref{eq74} in a neighborhood of the parameter value $\lambda_0=(v_0, a_0)=(v_0, 0)$.
There exists a constant $\sigma>0$  so that the following holds. 
\begin{itemize}
\item[(1)] If  $\lambda =(v, a)$ satisfies $\abs{\lambda-\lambda_0}<\sigma$, then the linear operator $L(\lambda):\wh{E}_0\to F$  is an sc-isomorphism.
\item[(2)]   For every $m\geq 0$ there exists a constant $C_m$ independent of $\lambda$ so that the  norm of the inverse operator
$$L(\lambda)^{-1}:F_m\to (\wh{E}_0)_m$$
is  bounded by $C_m$ for every   $\abs{\lambda-\lambda_0}<\sigma$.
\end{itemize}
\end{prop}

We postpone the proof of this nontrivial proposition  to the appendix and use it in order to verify that the section $g-g_0$ is a Fredholm germ.  Let $B_{\sigma}(\lambda_0)$ be  the open ball in $V\oplus \C$ centered at $\lambda_0=(v_0, 0)$  of radius $\sigma$.   In view of Proposition \ref{family}, the map 
$B_{\sigma}(\lambda_0) \oplus \wh{E}_0\to  (V\oplus \C)\oplus F$, defined by,  
$$ (\lambda, (h^+, h^-))\mapsto (\lambda, L(\lambda)(h^+, h^-)),$$
satisfies the assumptions of  Proposition \ref{solmap} in the appendix, from which we conclude that the inverse map 
$B_{\sigma}(\lambda_0)\oplus F\to B_{\sigma}(\lambda_0)\oplus \wh{E}_0$, 
$$(\lambda ,(\xi^+,\xi^-))\rightarrow L(\lambda)^{-1}(\xi^+,\xi^-),$$
is sc-smooth.  This allows to introduce the local strong bundle  coordinate change 

$$
\Phi:(B_{\sigma}(\lambda_0) \oplus U)\triangleleft F\rightarrow (B_{\sigma}(\lambda_0) \oplus U)\triangleleft \wh{E}_0, 
$$
defined by 
$$
(\lambda,(h^+,h^-),(\xi^+,\xi^-))\mapsto  (\lambda,( h^+,h^-),L(\lambda)^{-1}(\xi^+,\xi^-)).
$$
Since the map $(\lambda, \xi^+, \xi^-)\mapsto L(\lambda)^{-1}(\xi^+, \xi^-)$  between 
$B_{\sigma}(\lambda_0) \oplus F$ and $\wh{E}_0$ is sc-smooth, 
it is also sc-smooth as a map between $B_{\sigma}(\lambda_0) \oplus F^1$ and $\wh{E}^1_0$, in view of Proposition \ref{ias}. Therefore,  the map $\Phi$ is  $\ssc_\triangleleft$-smooth. 

Now we consider the push-forward  section $\Phi_\ast(g-g_0)$.  By construction,  its principal part is the map   $\Phi_\ast (g-g_0): (B_{\sigma}\oplus U)\to \wh{E}_0$, given by
$$
(\lambda,(h^+,h^-))\mapsto  ( h^+,h^-).
$$
Obviously, it  is an $\ssc^0$-contraction germ where the contraction term 
$$B(\lambda, h^+, h^-)
$$
 vanishes identically.  

To sum up, we have studied the Cauchy-Riemann section 
$$
f:(v, a, h^+, h^-)\mapsto (v, a, h^+, h^-, \eta^-, \eta^-)
$$
 in  local coordinates near the smooth point $(v_0, 0, 0, 0)$ and have constructed a filled section $g$ of $f$. Moreover, we have established an $\ssc^+$-section $g_0$ satisfying $(g-g_0)(v, 0, 0, 0)=0$ and found a bundle isomorphism $\Phi$ so that the push-forward $\Phi_\ast (g-g_0)$ is a germ belonging to $\mathfrak{C}_{basic}$. This shows that $f$ is a Fredholm germ at the point $(v_0, 0, 0, 0)$. 

The proof that the Cauchy-Riemann section $\care$ of the strong polyfold bundle $E\to V\oplus \ov{X}$ is a polyfold Fredholm section is complete.

\subsection{Application of the Sc-Implicit Function Theorem}

We are going to prove  Theorems \ref{fred:thm0} and Theorem \ref{strong-x}. 
The previous sections demonstrated that the Cauchy-Riemann operator 
$$\care :V\oplus \ov{X}\to E$$
is an sc- smooth Fredholm section of the strong M-polyfold bundle $E\to V\oplus \ov{X}$. Therefore, the sc-implicit function theorem from \cite{HWZ3} can be applied to our situation, and it follows that if the smooth point $x_0=(v_0, a_0, b_0, w_0)\in V\oplus \ov{X}$ is a solution of 
$$\care (v_0, a_0, b_0, w_0)=0$$
and if the linearized map at this point is surjective, then the solution set $\care (x)=0$ nearby is a smooth manifold of the dimension of the kernel of the linearized map at the reference solution $x_0$. 

We now consider the distinguished solution
$$\care (v_0, 0, 0, u^+_0, u^-_0)=0$$
at $a=0$ where 
$$u^\pm_0:(\R^\pm \times S^1, j^\pm (v_0))\to (\R^\pm\times S^1, i)$$
are the unique biholomorphic mappings of the half-cylinders which fix the boundary points 
$(0, 0)\in \partial (\R^\pm \times S^1)$. These  biholomorphic mappings are guaranteed by the uniformization theorem. Indeed, the maps $z=e^{2\pi (s+it)}$ and $z=e^{-2\pi (s+it)}$ are diffeomorphisms  between the  half-cylinders $\R^\pm \times S^1$ and the closed  unit disc $D\setminus \{0\}$ with the origin removed. Now there is a unique biholomorphic map $(D, j)\to (D, i)$ leaving a boundary point and the  origin fixed. By assumption on the complex structures  $j^\pm(v_0)$, the induced complex structure $j$ in $D$ agrees in an open neighborhood of the origin with the standard structure $i$. Therefore, since $h(0)=0$ and hence $h(z)=az+\cdots $ near the origin, the corresponding maps of the half-cylinders are of the form 
$$u^\pm_0(s, t)=(s, t)+(d^\pm_0, \vartheta^\pm_0)+r^\pm_0(s, t)$$
where the maps $r^\pm_0:\R^\pm \times S^1\to \R^2$ decay with all their  derivatives to $0$ at every rate bounded by $Ce^{-\varepsilon\abs{s}}$ for every $0\leq \varepsilon<2\pi$. 

Around the distinguished point $p_0=(v_0, 0,0, u^+_0, u^-_0)\in V\oplus \ov{X}$ at $a=0$ we take 
our chart  $\Psi$ of the M-polyfold, 
$$\Psi(v, a, h^+, h^-)=(v, a, b(a, u^+_0+h^+, u^-_0+h^-), \boxplus_a (u^+_0+h^+, u^-_0+h^-))$$
where $\rho_a(h^+, h^-)=(h^+, h^-)$, and take the associated bundle chart. In these local coordinates, the principal part of the Cauchy-Riemann section $\care$ is expressed by the map 
$$f(v, a, h^+, h^-)=(\eta^+, \eta^-)\in F$$
where 
\begin{align*}
\wh{\oplus}_a(\eta^+, \eta^-)&=\care_{i, j(a, v)}(\boxplus_a(u^+_0+h^+, u_0^-+h^-))\\
\wh{\ominus}_a(\eta^+,\eta^-)&=0.
\end{align*}
In Lemma \ref{localexp} we have computed the linearization $D(g-g_0)(p_0)$ where $g$ is the filled section of $f$.  It is a surjective Fredholm operator. The partial linearizations with respect to the variables $(h^+, h^-)$, denoted by $D_2(g-g_0)(p_0)$, is an sc-isomorphims from $E$ onto $F$. The same holds true for the linearization $Dg(p_0)$. In view of the definition of the filler in section \ref{Fredholm-section} one concludes that the linearization $Df(p_0)$ of the section $f$ is a surjective Fredholm operator and $D_2f(p_0)$ an sc-isomorphism.

From the sc-implicit function theorem, Theorem 4.6 in \cite{HWZ3}, we therefore conclude that there exists a unique sc-smooth map defined near $(v_0, 0)$,
$$
\sigma (v, a)=(v, a, h^+(v, a), h^-(v, a))
$$
satisfying $\sigma (v_0, 0)=(v_0, 0, 0, 0)$
and 
$$f(v, a, h^+(v, a), h^-(v, a))=0.$$
This means  for the section $\care$ on the M-polyfold $V \oplus \ov{X}$, that 
$$\care (\Phi \circ \sigma (v, a))=0$$
for the solutions near the reference solution.  Moreover, these are all the  solutions of the Cauchy-Riemann equations near $(v_0, 0, 0, u^+_0, u^-_0)\in V\oplus \ov{X}$. This completes the proof of Theorem \ref{fred:thm0}. \hfill $\blacksquare$

\noindent Rather than taking a chart around the reference solution $p_0=(v_0, 0, 0, u^+_0, u^-_0)$ we can take a chart  around the smooth point $(v_0, 0, 0, w^+_0, w^-_0)$ nearby so that 
the  chart contains the solution $p_0$ and satisfies, moreover,  
$$
w^\pm_0(s,t) = (s+d^\pm_0,t+\vartheta_0^\pm)
$$
for all large $\abs{s}\geq s_0$. 
Applying the sc-implicit function theorem to the reference solution $p_0$ we obtain in the new coordinates the sc-smooth map
$$\wh{\sigma}(v, a)=(v, a, \wh{h}^+(v, a), \wh{h}^-(v, a))$$
near  $(v_0, 0)$ and the associated biholomorphic maps between the finite cylinders
$$
(Z_a,j(a,v),p^+_a,p^-_a)\rightarrow (Z_{b(a,v)},i,p^+_{b(a,v)},p^-_{b(a,v)})
$$
given by 
$$
w(v, a)=\boxplus_a(w^+_0+\wh{h}^+(v,a), w^-_0+\wh{h}^-(v,a)).
$$
In view of the uniqueness of the solutions near $(v_0, 0)$ we know, in particular, that at $a=0$
$$
u^\pm_0 =w^\pm_0 + \wh{h}^\pm(0,v_0).
$$
It follows  for $(v, a)$ close to $(v_0, 0)$ that the mappings 
$$
w^\pm_0+\wh{h}^\pm(a,v):{\mathbb R}^\pm\times S^1\rightarrow {\mathbb R}^\pm\times S^1
$$
are diffeomorphisms. By  definition of the chart,  
$$
\rho_a(\wh{h}^+(a,v),\wh{h}^-(a,v))=(\wh{h}^+(a,v),\wh{h}^-(a,v))
$$
and hence we derive from the representation formula for $\rho_a$ the following asymptotics
\begin{align*}
\wh{h}^+(a,v)(s,t) &= \wh{h}^+(a,v)_\infty\quad \text{if $s\geq \frac{R}{2}+1$}\\
\wh{h}^-(a,v)(s',t') &= \wh{h}^-(a,v)_\infty\quad \text{if $s'\leq -\frac{R}{2}-1$},
\end{align*}
from which  Theorem \ref{strong-x} follows  in the case $\Delta\leq 1$.
 
In order to obtain the general case we proceed as follows. By construction, 
$$
w(a,v) =\boxplus_a(w^+_0,w^-_0) + \wh{h}^+(a,v) +\oplus_a(\wh{r}^{ +}(a,v),\wh{r}^{ -}(a,v)).
$$
We define for $|a|$ small enough the diffeomorphism
$$
\wt{w}(a,v):{\mathbb R}^+\times S^1\rightarrow {\mathbb R}^+\times S^1
$$
by
\begin{equation}\label{bnm}
\wt{w}(a,v)(s,t)=w^+_0(s,t) + \wh{h}^+_\infty(a,v) + \beta\left(s-\frac{R}{2}-\Delta-1\right)\cdot\oplus_a(\wh{r}^{+}(a,v),\wh{r}^{-}(a,v)).
\end{equation}
Observe
that for $s\in [0,\frac{R}{2}+\Delta]$ if  $|a|$ is small enough,  
\begin{equation*}
w^+_0(s,t) = \boxplus_a(w^+_0,w^-_0)([s,t]).
\end{equation*}
Therefore, 
$$
\wt{w}(a,v)(s,t)=w(a,v)([s,t])
$$
for all   $s\in [0,\frac{R}{2}+\Delta]$. Finally,  we note that
$$
(a,v)\mapsto  \wt{w}(a,v)
$$
defines a smooth map into every ${\mathcal D}^{m,\varepsilon}$ for all $m\geq 2$ and $\varepsilon\in (0,2\pi)$. Since in our set-up for the implicit function theorem we can take any sequence $(\delta_m)$ as long as it is strictly increasing and stays below $2\pi$, we see that the maps
$$
(a,v)\mapsto  \wh{ h}^\pm(a,v)
$$
as maps from an open neighborhood of $(0,v_0)$ into $H^{m,\varepsilon}_c({\mathbb R}^\pm\times S^1,{\mathbb R}^2)$ are smooth.
The same is true,  for $\wh{r}^\pm(a,v)$. Now using Proposition \ref{qed} we see that the map
$$
(a,v)\mapsto  \beta\left(\cdot-\frac{R}{2}-\Delta-1\right)\cdot\oplus_a(\wh{r}^{+}(a,v),\wh{r}^{-}(a,v))([s,t]).
$$
is sc-smooth for every  choice of  the sequence of $(\delta_m)$ as described above.
Choosing $(\delta_m)\in (\varepsilon,2\pi)$ we deduce that the map
$$
(a,v)\mapsto   \wt{w}(a,v), 
$$
which is defined near $(0,v_0)$,  is smooth into every space ${\mathcal D}^{m,\varepsilon}$.
This completes the proof of Theorem \ref{strong-x}. \hfill $\blacksquare$

\section{Appendix}
In the appendix we shall prove the sc-smoothness of the shift-map. Moreover, we collect informations about the gluing profile and provide proofs of several technical results about families of sc-isomorphisms and estimates for the Cauchy-Riemann section  which are  used in our constructions.

\subsection{The Shift-Map}\label{app1}
Fixing  a strictly increasing sequence $(\delta_m)_{m\geq 0}$  of real numbers starting with $\delta_0=0$,  
we consider the Hilbert space $E=L^2(\R\times S^1)$ equipped    with the sc-structure defined by the sequence $E_m=H^{m,\delta_m}$ for all $m\geq 0$.
The shift-map $\Phi:\R^2\oplus E\to  E$ is defined as  
$$\Phi: (R,\vartheta , u)\to  (R,\vartheta)\ast u:= u(s+R,t+\vartheta).$$
Our first result concerns the $\ssc^0$-property of $\Phi$.
\begin{prop}\label{sc}
The shift-map $\Phi$ is $\ssc^0$-continuous.
\end{prop}
\begin{proof}
Fix a level $m$ and take $u\in E_m$.
We estimate the norm $\abs{ (R,\vartheta)\ast u}_m$ as follows,
\begin{equation*}
\begin{split}
\abs{ (R,\vartheta)\ast u}^2_m &=\sum_{|\alpha|\leq m} \int |(D^{\alpha}u)(s+R,t+\vartheta ) |^2 e^{2\delta_m|s|}\  ds dt\\
&\leq \sum_{|\alpha|\leq m} \int |(D^{\alpha}u)(s+R,t+\vartheta ) |^2 e^{2\delta_m|s+R|} e^{2\delta_m \abs{R}}\  ds dt\\
&=e^{2\delta_m \abs{R}} \abs{u}_m^2.
\end{split}
\end{equation*}
Thus, 
$$
\abs{ (R,\vartheta)\ast u}_m \leq e^{\delta_m \abs{R}} \abs{u}_m.
$$
 Since the set of  smooth compactly supported functions $\R\times S^1\to \R$ is  dense in $H^m(\R\times S^1)$,  it is also dense in $E_m$.   We claim that $\Phi$ is continuous at the point $((0, 0), u_0)\in \R^2\times E_m$.  To see this, first note that if 
$v$ is smooth and compactly supported,  then $\Phi(R,\vartheta, v)\to v$ in $C^\infty$
as $(R,\vartheta)\to  0$, which immediately  implies the convergence in the $m$-norm, 
$$
\lim_{(R,\vartheta )\to  (0,0)} \abs{\Phi(R,\vartheta , v)-v}_m=0.
$$
Now, if $u\in E_m$ and  if $v$ is a compactly supported smooth function, we estimate,  
\begin{equation*}
\begin{split}
\abs{(R, \vartheta)\ast u-u_0}_m&=\abs{ (R, \vartheta)\ast (u-v)  +  ((R, \vartheta)\ast v -v)  + (v-u_0)}_m\\
&\leq e^{\delta_m|c|}\cdot \abs{u-v}_m + \abs{(R, \vartheta)\ast v-v}_m+\abs{v-u_0}_m
\end{split}
\end{equation*}
Given $\varepsilon>0$,  we chose $v$ so that $\abs{v-u_0}_m<\varepsilon$. Then for all $u\in E_m$ satisfying $\abs{u-u_0}_m<\varepsilon$ and $(R, \vartheta)$ close to $(0, 0)$, we have   
$$\abs{u-v}_m<2\varepsilon, \quad  \abs{(R, \vartheta)\ast v-v}_m<\varepsilon,\quad e^{\delta_mR}<9,$$
so that the above estimate gives
$$
\abs{(R, \vartheta)\ast u-u_0}_m< 10\varepsilon,
$$
proving the continuity at the point $(0,0, u_0)\in \R^2\oplus E_m$. 
Since  for fixed $(R_0,\vartheta_0)$  the map $E\to E$ defined by   $u\mapsto  (R_0,\vartheta_0)\ast u$ is clearly  an $\ssc^0$-operator and since 
$$
(R, \vartheta)\ast u -(R_0, \vartheta_0)\ast u_0= (R-R_0,\vartheta-\vartheta_0)\ast ( (R_0,\vartheta_0)\ast u) -(R_0,\vartheta_0)\ast u_0,
$$
the previous discussion  shows that $\Phi$ is continuous at $(R_0,\vartheta_0)\ast u_0\in \R^2\oplus E_m$. Consequently, $\Phi$ is $\ssc^0$ as claimed.
\end{proof}

Having proved that the shift-map is of class $\ssc^0$, we show that it is sc-smooth.
\begin{prop}\label{sc-sm}
If  $E=L^2(\R\times S^1)$ is equipped with the sc-structure $(E_m)_{m\in \N_0}$ as described before,  
the shift-map $\Phi:\R^2\oplus E\to E$  is $\ssc$-smooth.
\end{prop}
\begin{proof}
By Proposition \ref{sc}, the map $\Phi$ is $\ssc^0$ and we  first show that  it  is $\ssc^1$.  Take  a point $(R, \vartheta, u)\in \R^2\oplus E_1$.  We want to find a linear bounded operator 
$$D\Phi (R, \vartheta , u): \R^2\oplus E_0\to E_0$$ 
satisfying points (1) and (2) of Definition \ref{sc-def}. Our  candidate for the linearization $D\Phi (R, \vartheta, u)$ of $\Phi$ at the point $(R, \vartheta, u)$  is the formal derivative  
\begin{equation*}\label{cand-eq1}
\Psi: (\R^2\oplus E)^1\oplus (\R^2\oplus E)\to  E
\end{equation*}
defined as the map 
\begin{equation*}
\begin{split}
\Psi  (R, \vartheta , u)(R_1, \vartheta_1, v)&=(R, \vartheta)\ast (R_1\cdot u_s+\vartheta_1\cdot u_t+v )\\
&=\Phi (R, \vartheta , R_1\cdot u_s+\vartheta_1\cdot u_t+v).
\end{split}
\end{equation*}
We note that  the map $\Psi$ is $\ssc^0$. Indeed, we already know that the shift-map $\Phi$ 
is $\ssc^0$ and the two maps $E^1\to  E$ defined by  $u\mapsto u_s, u_t$ are $\ssc^0$-operators.  Also,  a scalar multiplication $\R\oplus E\to  E$ 
$\ssc^0$. Hence, the map $\Psi$ can be written as the composition of  $\ssc^0$-maps and, therefore,  is $\ssc^0$. It remains to prove the approximation property.
To do so we have to show that
$$
\frac{ \abs{\Phi(R+R_1,\vartheta+\vartheta_1, u+v)-\Phi(R, \vartheta, u) -\Phi (R, \vartheta , R_1\cdot u_s+\vartheta_1\cdot u_t+v)}_0}
{\abs{(R_1, \vartheta_1, v)}_1}.
$$
converges to $0$ as $\abs{(R_1, \vartheta_1, v)}_1\to 0$. Once this is proved, we will know by our earlier discussion that $\Phi$ is $\ssc^1$ and its linerization at the point 
$(R, \vartheta, u)\in R^2\oplus  E$ is given by 
$$
D\Phi (R, \vartheta, u)(R_1, \vartheta_1, v) =\Phi (R, \vartheta, R_1\cdot u_s+\vartheta_1\cdot u_t+v).
$$
We compute, 
\begin{equation*}
\begin{split}
&\Phi(R+R_1,\vartheta+\vartheta_1, u+v)-\Phi(R, \vartheta, u) -\Phi (R, \vartheta, R_1\cdot u_s+\vartheta_1\cdot u_t+v)\\
&\phantom{===}=\Phi(R+R_1,\vartheta+\vartheta_1, u)-\Phi(R, \vartheta , u) -\Phi (R, \vartheta , R_1\cdot u_s+\vartheta_1\cdot u_t)\\
&\phantom{====} + \Phi( R+R_1,\vartheta+\vartheta_1 , v)-\Phi( R,\vartheta , v).
\end{split}
\end{equation*}
We first show that 
$$
\frac{ \abs{\Phi(R+R_1,\vartheta+\vartheta_1, u)-\Phi(R, \vartheta, u) -\Phi (R, \vartheta, R_1\cdot u_s+\vartheta_1\cdot u_t)}_0}{\abs{(R_1, \vartheta_1, v)}_1}
$$
converges to $0$ as $\abs{(R_1, \vartheta_1, v)}_1\to 0$.  Since the map $\Phi (R,\vartheta, \cdot ):E\to E$ is an  isometry on level $0$, it suffices to show that 
$$
\lim_{\abs{(R_1, \vartheta_1)}\to 0} \frac{ \abs{\Phi(R_1,\vartheta_1, u)-u - R_1\cdot u_s+\vartheta_1\cdot u_t}_0}{\abs{(R_1, \vartheta_1)} }=0
$$
If $u$ is compactly supported smooth function, then we have 
\begin{equation*}
\begin{split}
\Phi(R_1,\vartheta_1, u)-&u - R_1\cdot u_s+\vartheta_1\cdot u_t\\
&=\int_0^1 [ \Phi(\tau R_1,\tau \vartheta_1,  R_1\cdot u_s+\vartheta_1\cdot u_t )- R_1\cdot u_s+\vartheta_1\cdot u_t]\ d\tau.
\end{split}
\end{equation*}
Now if $u\in E_1$, then we find a sequence $(u_n)$ of compactly supported smooth functions so that $\abs{u_n-u}_1\to 0$. The above equality holds for each of the functions $u_n$
and letting $n\to \infty$ we find that it also holds for $u$.  Moreover, note that 
\begin{equation*}
\begin{split}
\abs{\Phi(\tau R_1,\tau \vartheta_1 , &R_1\cdot u_s+\vartheta_1\cdot u_t)- R_1\cdot u_s+\vartheta_1\cdot u_t}_0\\
&\leq \abs{R_1}\abs{\Phi( \tau R_1,\tau \vartheta_1)  u_s)-u_s}_0
 +\abs{\vartheta_1}\abs{\Phi( \tau R_1,\tau \vartheta_1 , u_t)-u_t}_0
 \end{split}
 \end{equation*}
and that each summand of the right hand side divided by $\abs{(R_1, \vartheta_1)}$ converges uniformly  for $\tau\in [0,1]$ to $0$ as  $\abs{(R_1, \vartheta_1)}\to 0$. 
Consequently,
\begin{equation*}
\begin{split}
&\frac{ \abs{ \Phi(R_1,\vartheta_1, u)-u - R_1\cdot u_s-\vartheta_1\cdot u_t }_0    }{ \abs{ (R_1, \vartheta_1) } } \\
&\phantom{===}=\left| \int_0^1\frac{   \Phi( \tau R_1,\tau \vartheta_1,  R_1\cdot u_s+\vartheta_1\cdot u_t  )- R_1\cdot u_s+\vartheta_1\cdot u_t  }{\abs{(R_1, \vartheta_1)} }  \  d\tau   \right|_0 \\
& \phantom{===}\leq \int_0^1\frac{  \abs{   \Phi(\tau R_1,\tau \vartheta_1,  R_1\cdot u_s+\vartheta_1\cdot u_t )- R_1\cdot u_s+\vartheta_1\cdot u_t }_0}{\abs{(R_1, \vartheta_1)} }\ d\tau 
\end{split}
\end{equation*}
and the right hand side converges to $0$ as $\abs{(R_1, \vartheta_1)} \to 0$.
Next we show that 
$$\frac{ \abs{\Phi(R+R_1,\vartheta+\vartheta_1, v)-\Phi(R,\vartheta, v)}_0}{\abs{(R_1, \vartheta_1, v)}_1}\to 0$$
as $\abs{(R_1, \vartheta_1, v)}_1\to 0$.  This is more tricky and relies on the compactnees of  the inclusion $E_1\to  E_0$.

Arguing indirectly
we find  an $\varepsilon>0$ and a sequence $(R_1^n,\vartheta_1^n,v_n)\to  (0,0,0)$ in $\R^2\oplus E_1$  satisfying 
\begin{equation}\label{eq7}
\frac{\abs{ \Phi( R+R_1^n,\vartheta+\vartheta^n_1,v^n)-\Phi( R,\vartheta,v^n) }_0 }{\abs{(\R_1^n,\vartheta_1^n,v_n)}_1}\geq \varepsilon.
\end{equation}
The sequence 
$$w^n=\frac{ v^n }{ \abs{ (R^n_1, \vartheta^n_1, v^n) }_1 }$$
is bounded in $E_1$ and we may assume, using that the inclusion $E_1\to E_0$ is compact, that   $w^n\to  w$ in $E_0$. By the already established $\ssc^0$-continuity of $\Phi$
we conclude that
\begin{equation*}
\begin{split}
&\frac{ \Phi( R+R_1^n,\vartheta+\vartheta^n_1,v^n)-\Phi ( R,\vartheta,v^n) }{\abs{(R_1^n,\vartheta_1^n,v_n)}_1 }\\
&\phantom{=====}=  \Phi ( R+R_1^n,\vartheta+\vartheta^n_1, w^n)-\Phi( R,\vartheta,w^n ) \\
&\phantom{=====} \to \Phi( R,\vartheta,w)-\Phi( R,\vartheta,w)=0,
\end{split}
\end{equation*}
contradicting \eqref{eq7}.  

At this point we have proved  that $\Phi$ is of class $\ssc^1$ and its tangent map  
$$T\Phi : T(\R^2\oplus E)=\R^2\oplus E^1\oplus R^2\oplus E\to TE=E^1\oplus E$$  is given by
\begin{equation}\label{shift-eq1}
T\Phi(R, \vartheta, u, R_1, \vartheta_1, v)=(\Phi (R, \vartheta, u), \Phi (R, \vartheta , R_1u_s+\vartheta_1u_t +v)).
\end{equation}
This will allow us to give an a inductive argument to show that $\Phi$ is of class $\ssc^k$. We prove  the following statements  by induction:\\

\begin{induction}
The map $\Phi$ is of class $\ssc^k$ and for every projection $\pi:T^kE\to  E^j$ onto one of the factors of $T^kE$, the composition  $\pi\circ T^k\Phi$
is a  linear combination of maps  of the form
$$
A: \R^2\oplus E^m\oplus \R^{|\alpha|}\to  E^j, \quad (R,\vartheta,u,h)\to  \Phi(R,\vartheta, h_1\cdot \ldots \cdot h_{|\alpha|}\cdot D^{\alpha}u)
$$
with a multi-index $\alpha$ satisfying  $|\alpha|\leq m-j$.
\end{induction}

We already know that $\Phi$ is $\ssc^1$. If $\pi: TE=E^1\oplus E\to E^1$ is the projection onto  the first factor, then, by \eqref{shift-eq1},  the map
$\pi\circ T\Phi$  is given by  
$$\R^2\oplus E^1\to  E^1, \quad (R, \vartheta), v) \mapsto \Phi ((R, \vartheta), v).$$
If $\pi: TE=E^1\oplus E^0\to E^0$ is the projection onto the second factor, then  the composition 
$\pi\circ T\Phi$ is given by 
\begin{gather*}
 \R^2\oplus E^1\oplus \R^2\oplus E^0\to E^0\\
 (R, \vartheta, u, R_1, \vartheta_1, v)\mapsto  \Phi ((R, \vartheta) , R_1u_s+\vartheta_1u_t +v)
 \end{gather*}
which can be written as a sum of the  maps of the following types, 
\begin{align*}
&\R^2\oplus E\to  E, \qquad (R,\vartheta, v)\mapsto  \Phi(R,\vartheta, v)\\
\intertext{and}
&\R^2\oplus E^1\oplus \R\to  E, \qquad  (R,\vartheta, u, h) \mapsto  \Phi( (R,\vartheta,hD^{\alpha}u),
\end{align*}
where   $\alpha=(1,0)$ and  $\alpha=(0,1)$.
Hence  the assertion $({\bf S_1})$ holds.

Now assume that we have proved  $({\bf S_k})$. We shall prove $({\bf S_{k+1}})$.  We first  show that the map $T^k\Phi:T^k(\R^2\oplus E)\to T^kE$ is $\ssc^1$.   It suffices to show that maps described in $({\bf S_k})$
are all of class $\ssc^1$. So, we consider the map 
$$
A:\R^2\oplus E^m\oplus \R^{|\alpha|}\to  E^j, \qquad (R,\vartheta,u,h)\mapsto 
\Phi(R,\vartheta, h_1\cdot\ldots \cdot h_{|\alpha|}\cdot D^{\alpha}u),
$$
where  $|\alpha|\leq m-j$.

We observe that this map is a composition of the following maps. The first one,   defined by 
$$
\R^2\oplus E^m\oplus\R^{|\alpha|}\to  \R^2\oplus E^j,\quad 
 (R,\vartheta,u,h)\mapsto  (R,\vartheta, h_1\cdot \ldots \cdot h_{|\alpha|}D^\alpha u)
$$
is  clearly  $\ssc^\infty$. The second map is our shift-map 
$$
\Phi:\R^2\oplus E^j\to  E^j.
$$
By induction hypothesis the shift-map $\Phi:\R^2\oplus E\to E$ is $\ssc^k$ so that  by Proposition \ref{ias} the map 
$\Phi:\R^2\oplus E^j\to E^j$ is, in particular,  $\ssc^1$. Applying the chain rule, we conclude that the map $A$ is $\ssc^1$.
At this point we know that $\Phi$ is $\ssc^{k+1}$ and it remains to show that the iterated tangent map $T^{k+1}\Phi:T^{k+1}(\R^2\oplus E)\to T^{k+1}E$ satisfies the remaining statements of  $({\bf S_{k+1}})$. 
 
We consider the composition $\pi\circ T^{k+1}\Phi$  where $\pi$ is the projection onto one of the factors of $T^{k+1}E$. If $\pi$ is a projection onto one of the first $2^k$ factors, then $\pi\circ T^{k+1}\Phi$ is  a sum of maps  obtained from the maps $A$ described in $({\bf S_k})$ by raising the index of the domain by $1$.   If $\pi$ is a projection onto one of the last  $2^k$ factors, then $\pi\circ T^{k+1}\Phi$ 
is a linear combination of derivatives of maps $A$ in  $({\bf S_k})$.
Hence we have to consider the map 
$$
A:\R^2\oplus E^m\oplus \R^{|\alpha|}:\to  E^j, \quad (R,\vartheta,u,h)\mapsto   \Phi(R,\vartheta,h_1\cdot \ldots \cdot h_{|\alpha|}\cdot D^{\alpha}u).
$$
Its derivative after raising the index of the domain is the  map
$$
\R^2\oplus E^{m+1}\oplus \R^{|\alpha|}\oplus \R^2\oplus E^{m}\oplus \R^{|\alpha|}\to  E^j,
$$
defined by
\begin{equation*}
\begin{split}
&  (R,\vartheta,u,h),\delta R,\delta \vartheta,\delta u,\delta h)\mapsto  \Phi(R,\vartheta,u, h_1\cdot \ldots \cdot h_{|\alpha|}\cdot D^{\alpha}\delta u)\\
&\phantom{====}+\sum_{l=1}^{|\alpha|}\Phi(R,\vartheta,u,h_1\cdot \ldots \cdot \delta h_l\cdot\ldots \cdot h_{|\alpha|}\cdot D^{\alpha}u)\\
&\phantom{====}+\Phi(R,\vartheta,u,h^1\cdot\ldots \cdot h^{|\alpha|} (\delta R \cdot \partial_s D^{\alpha}u +\delta\vartheta \cdot  \partial_t D^{\alpha}u ).
\end{split}
\end{equation*}
This is the sum of maps of the form:
$$
\R^2\oplus E^{m}\oplus \R^{|\alpha|}\to  E^j, \quad (c,d,v,h)\to  \Phi(c,d,h_1\cdot\ldots \cdot h_{|\alpha|}\cdot D^{\alpha}v)
$$
with $|\alpha|\leq m-j$, and
\begin{gather*}
\R^2\oplus E^{m+1}\oplus \R^{|\alpha|}\to  E^j\\
(c,d,u,(h_1,\ldots ,\delta h_l,\ldots h_{|\alpha|}))\to  \Phi(c,d,h_1\cdot\ldots \cdot \delta h_l\cdot\ldots \cdot h_{|\alpha|}D^{\alpha}u)
\end{gather*}
with $|\alpha|\leq m-j\leq m+1-j$, and finally
$$
\R^2\oplus E^{m+1}\oplus \R^{|\alpha|+1}\to  E^j,\quad (c,d,u,(h,\gamma))\to  \Phi(c,d,h_1\cdots \ldots \cdot h_{|\alpha|}\gamma D^\beta u),
$$
where $\beta=\alpha+(1,0)$ or $\beta=\alpha+(0,1)$. Then $|\beta|=|\alpha|+1\leq m+1-j$ and  we have verified  that  the map $\Phi$ satisfies $({\bf S_{k+1}})$. The proof of Proposition \ref{sc-sm} is complete.
\end{proof}

\subsection{The Gluing Profile}\label{calc-lemma}
\begin{defn}
A gluing profile is a smooth diffeomorphism
$$
\varphi:(0,1]\to [0,\infty).
$$
\end{defn}
See Figure 7 for the graph of a gluing profile.
In order to construct a polyfold structure on our moduli spaces, the gluing profile has to satisfy additional properties, which hold true for the special profile $\varphi (x)=e^{\frac{1}{x}}-e$ we have chosen. 
\begin{lem}\label{lem5.2}
We consider the gluing profile
$$
\varphi(x)=e^{1/x}-e
$$
and define the  function $B:[0,r)\times [-R,R]\rightarrow {\mathbb R}$
for sufficiently small $0<r<1$ by
\begin{equation*}
B(x,c)=\begin{cases}
 \varphi^{-1}[\varphi(x)+c]&
\text{if  $ x\in (0,r)$}\\
0&\text{if $x=0$.}
\end{cases}
\end{equation*}
Then $B$ is smooth and satisfies
\begin{equation*}
D_xB(0,c)=1\quad \text{and}\quad D^{\alpha}B (0, c)=0
\end{equation*}
for all multi-indices $\alpha=(\alpha_1, \alpha_2)$ with $\alpha_1\geq 2$ and $\alpha_2\geq 0$.
\end{lem}
\begin{proof}
The inverse of  the  function
$
\varphi (x)=e^{1/x}-e,
$
defined on the domain $(0,1]$, 
is the function 
\begin{equation*}
\varphi^{-1}(y)=\dfrac{1}{\ln [e+y]}.
\end{equation*}
Hence   the function
$B$ satisfies
\begin{equation*}
B(x, c)=\dfrac{1}{\ln [e^{1/x}+c]}
\end{equation*}
 for $x>0$.
To prove our claim we have to show that
\begin{equation}\label{xeq1}
B(x,c)\to 0,\\
 D_xB(x,c)\to 1\quad \text{ and}\quad  D^{n,m}B(x,c)\to
0
\end{equation}
as $x\to 0$ uniformly in $c$,
for all $n\geq 2$. Writing
$
\ln [e^{1/x}+c]=
\ln \bigl[ e^{1/x}\cdot \bigl(1 + c\cdot e^{-1/x}\bigr)\bigr]= \dfrac{1}{x}+\ln \bigl[1 +c\cdot e^{-1/x}\bigr]$,
we  represent $B(x,c)$ as
\begin{equation*}
B(x, c)=x\cdot \dfrac{1}{ 1+x\ln \bigl[1 +c\cdot e^{-1/x}\bigr]}=
x\cdot f(x,c)\end{equation*}
 where
$$f(x,c)=\dfrac{1}{ 1+x\ln \bigl[1 +c\cdot e^{-1/x}\bigr]}.$$
Clearly,  $f(x,c)\rightarrow 1$ as $x\rightarrow 0$,
uniformly in $c$. Defining  the function $g$ by
$$
f(x,c)=\frac{1}{1+g(x,c)}, 
$$
it suffices to show that $D^{\alpha}g(x,c)\rightarrow
0$
 for $|\alpha|\geq 1$ uniformly in $c$ as $x\rightarrow 0$.
Since  $g(x,c)=x\ln[1+ce^{-1/x}]$, this follows from the fact  that  the function $h$,  defined by
 $
 h(x,c)=ce^{-1/x}
 $
 satisfies $D^{\alpha}h(x,c)\rightarrow 0$,  uniformly in $c$,  as
 $x\rightarrow 0$. In order to prove the second assertion for $B$,  we observe that a derivative
 of order $n$ of $e^{1/x}$ is a product of
 $e^{1/x}$ with a polynomial in the variable $1/x$ from which the desired assertion follows.
\end{proof}
\begin{center}
\begin{figure}[htbp]
\mbox{}\\[2pt]
\centerline{\relabelbox \epsfxsize 2.5truein \epsfbox{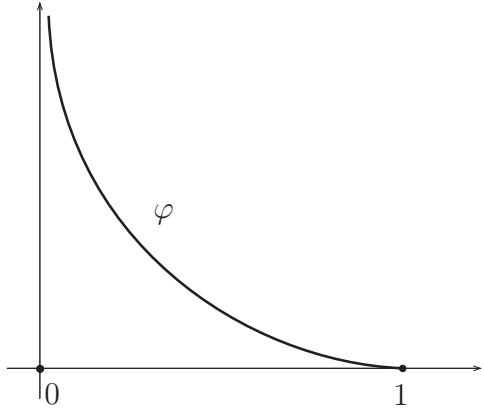}
\relabel {0}{$0$} 
\relabel {1}{$1$} 
\relabel {v}{$\varphi$}
\endrelabelbox}
\caption{A graph of a gluing profile $\varphi$}
\mbox{}\\[2pt]
\end{figure}
\end{center}
The function $B(x, c)$ defined in  Lemma  \ref{lem5.2} meets the assumptions of the next statement.

\begin{lem}\label{thm5.3}
Let $B:[0,1)\times {\mathbb R}\to {\mathbb R}$ be a smooth function satisfying
\begin{equation*}
B(0,c)=0,\quad D_1B(0,c)=1\: \: \text{and}\: \: D^n_1B(0,c)=0\:\text{for
$n\geq 2$}.
\end{equation*}
Then the function $f:D \setminus \{0\}\times {\mathbb R}\to {\mathbb C}$, defined by   
\begin{equation*}
f(z,c)=B(\abs{z},c)\dfrac{z}{\abs{z}},
\end{equation*}
is smooth and satisfies as $z\to 0$, uniformly in $c$ belonging to a bounded subset of ${\mathbb R}$, 
\begin{equation}\label{xeqq1}
\begin{aligned}
&f(z,c)\to 0\\\
&D_1f(z,c)\to \text{Id}\\
&D^n_1f(z,c)\to 0\:\:\: \text{for $n\geq 2$}.
\end{aligned}
\end{equation}
In particular, the function $f$ extends smoothly over $\{z=0\}\times {\mathbb R}$.
\end{lem}
\begin{proof}
If $g=g(s,c)$ we shall write $g^{(n)}$ for the $n-th$ derivative with respect to $s$.
We begin with the following simple calculus observation. Assuming that $g(s,c)$ is a smooth function on $[0,1]\times {\mathbb R}$ 
whose derivatives with respect to $s$ at $s=0$  all vanish so  that $g^{(n)}(0,c)=0$ for all $n\geq 0$,   the function $h(s,c)$,  defined by $h(s,c)=g(s,c)/s$ for $s>0$ and $h(0,c)=0$,  is smooth and  all its derivatives at $s=0$ vanish so that $h^{(n)}(0,c)=0$ for all $n\geq 0$. Indeed,  we may represent $g(s)$ in the form 
$$
g(s,c)=sR(s,c)\quad \text{with}\quad R(s,c)=\int_0^1g'(\tau s,c)d\tau .
$$
Observing  that $R(s,c)$ is smooth for $s\geq 0$  and satisfies $R^{(n)}(0,c)=0$ for all $n\geq 0$, we conclude that 
$h(s,c)=R(s,c)$  is smooth and  satisfies  $h^{(n)}(0,c)=0$ for all $n\geq 0$.

 We use this observation to obtain the following conclusions. Set $C_0(s,c)=B(s,c)/s$  for $s>0$ and $C_0(0,c)=1$,  and define a sequence of functions  $C_{n}(s,c)=C_{n-1}'(s,c)/s$  for $n\geq 1$ and $s>0$  where the prime stands for the first derivative. We claim that every  function $C_n(s,c)$ for $n\geq 1$  smoothly extends over $s=0$ and the derivatives $C_n^{(k)}(0,c)=0$ vanish for all $k\geq 0$. To see this we represent $B(s,c)$
using the assumptions  as 
$$B(s,c)=s+s^3\what{R}(s,c)\quad \text{with}\quad \what{R}(s)=\dfrac{1}{2}\int_0^1(1-\tau )^2B^{(3)}(\tau s,c)d\tau .$$
Then
$C_0(s,c)=1+s^2\what{R}(s,c)$ and its first derivative is given by $C_0'(s,c)=2s\what{R}(s,c)+s^2\what{R}'(s,c)$. Hence 
$C_1(s,c)=2\what{R}(s,c)+s\what{R}'(s,c)$ showing that $C_1(s,c)$ is smooth for $s\geq 0$ and $C_1^{(k)}(0,c)=0$ for all $k\geq 0$. 
Our  claim follows by 
applying  this procedure successively to all the functions $C_n(s,c)$ for $n\geq 2$.     

With the above definition of $C_0(s,c)$,   the function $f(z,c)$ is defined as 
$$f(z,c)=B(\abs{z},c)\dfrac{z}{\abs{z}}=C_0(\abs{z},c)z. $$
Differentiating this expression  at $z\neq 0$ we obtain 
\begin{equation*}\label{1smooth}
D_1f(z,c)h_1=\dfrac{C'_0(\abs{z},c)}{\abs{z}}\langle z, h_1\rangle+C_0(\abs{z},c)h_1
=C_1(\abs{z},c)\langle z, h_1\rangle+C_0(\abs{z},c)h_1.
\end{equation*}
Using the properties of $C_0(s,c)$ and $C_1(s,c)$  we conclude that 
$D_1f(z,c)\to \text{Id}$ as $z\to 0$. In general, the higher derivatives of $f(z,c)$ are of  the form
\begin{equation*}
\begin{split}
D^{(2n)}_1f(z,c)(h_1,\ldots ,h_{2n})&=C_{2n}(\abs{z},c)\langle z,h_1\rangle \cdots \langle z, h_{2n}\rangle z\\
&\phantom{=}+\sum_{j=0}^{n-1}C_{n+j}(\abs{z},c)M_{2n, 2j+1}(z,c, h_1,\ldots ,h_{2n})
\end{split}
\end{equation*}
\begin{equation*}
\begin{split}
D_1^{(2n+1)}f(z,c)(h_1,\ldots ,h_{2n+1})&=C_{2n+1}(\abs{z},c)\langle z,h_1\rangle \cdots \langle z, h_{2n+1}\rangle z\\
&\phantom{=}+\sum_{j=0}^n C_{n+j}(\abs{z},c)M_{2n+1, 2j}(z,c, h_1,\ldots ,h_{2n+1}),
\end{split}
\end{equation*}
where $M_{2n, 2j+1}(z,c, \cdots )$ and $M_{2n+1, 2j}(z,c, \cdots )$ are $2n$-  and $(2n+1)$-linear maps
which are smoothly depending on $(z,c$ and satisfying  
$$
\abs{M_{2n, 2j+1}(z,c, h_1,\ldots,h_{2n})}\leq \abs{z}^{2j+1}\abs{h_1}\cdots \abs{h_{2n}}
$$
 and 
$$
\abs{M_{2n+1, 2j}(z,c,h_1,\ldots,h_{2n})}\leq \abs{z}^{2j}\abs{h_1}\cdots \abs{h_{2n+1}}.
$$ 
Consequently, in view of the properties of the functions $C_n(s,c)$,  the derivatives  $D_1^{(n)}f(z,c)$  tend to $0$ for all $n\geq 2$ as $z\to 0$,  and the proof of the lemma is complete.
\end{proof}
\mbox{}\\

In our constructions we have to study the gluing length $R$ and
the gluing angle $\vartheta$ associated to a nonzero gluing
parameter $a\in B\setminus\{0\}$ via the gluing profile $\varphi$
defined by $\varphi(x)=e^{\frac{1}{x}}-e$. The pair $(R,\vartheta)$
is defined by
$$
R=\varphi(|a|)\quad \text{and}\quad 
a= |a|\cdot e^{-2\pi i\vartheta}.
$$
It is important to have  estimates for the map $a\mapsto 
(R(a),\vartheta (a))$. In order to derive the appropriate estimates, we use the identification $a=x+iy=(x,y)$.
If $\beta=(\beta_1,\beta_2)$ is a multi-index,  we write
$$
a^\beta=x^{\beta_1}\cdot y^{\beta_2}.
$$
We shall prove the following statement.
\begin{lem}\label{ropet1}
If  $0<|a|<1$,  the partial derivative $(D^\alpha R)(a)$ of order  $|\alpha|\geq 1$ is a linear combination
of terms of the form
$$
e^{\frac{1}{|a|}} \cdot \frac{1}{|a|^k} a^\beta
$$
with the integers $k$ and the multi-indices $\beta$ satisfying $k\leq 3\cdot |\alpha| $ and $|\beta|\leq |\alpha|$.
\end{lem}
\begin{proof}
In the case $\abs{\alpha}=1$, we may take without loss of generality $\alpha=(1, 0)$. Then, in view of  $a=x+iy\equiv (x,y)$,
$$
\frac{\partial}{\partial x}R(a) = -e^{\frac{1}{|a|}}\cdot \frac{1}{|a|^2}\cdot\frac{x}{|a|}= -e^{\frac{1}{|a|}}\cdot \frac{x}{|a|^3}.
$$
Hence  $\beta=(1,0)$ and $3\leq 3\abs{(1, 0)}$ and $\abs{\beta}\leq \abs{\alpha}$. The same holds for the partial derivative with respect to $y$. 
Assuming  the result to hold  for all multi-indices $\alpha$ with $|\alpha |\leq l$, we  consider the partial derivative $D^{\alpha}R$ for a  multi-index $\alpha$  of order $|\alpha|=l+1$. Without loss of generality we assume  that  $\alpha=\alpha_0+(1,0)$. We know that  $D^{\alpha_0}R$ is the  linear combination of terms of the form
$$
e^{\frac{1}{|a|}} \frac{1}{|a|^k} a^\beta
$$
where  $|\beta|\leq l$ and $k\leq 3l$. Applying $\frac{\partial }{\partial x}$, we obtain 
\begin{equation*}
\frac{\partial}{\partial x}\left( e^{\frac{1}{|a|}} \frac{1}{|a|^k} a^\beta
\right)= e^{\frac{1}{|a|}}\cdot\left(  \frac{-1}{\phantom{-}|a|^2}\cdot \frac{x}{|a|}\cdot \frac{1}{|a|^k} a^\beta
- \frac{k}{\phantom{-}|a|^{k+1}} \cdot \frac{x}{|a|} a^\beta
+\beta_1\cdot  \frac{1}{|a|^k} a^{\beta'}\right).
\end{equation*}
Here $a^{\beta'}=0$ if $\beta_1=0$ and otherwise $\beta'=\beta-(1,0)$.
This derivative is a linear combination of at most three terms, namely
\begin{equation*}
-e^{\frac{1}{|a|}}\cdot \frac{1}{|a|^{k+3}} \cdot a^{\beta+(1,0)},\quad  
-ke^{\frac{1}{|a|}}\cdot \frac{1}{|a|^{k+2}}\cdot a^{\beta+(1,0)},\quad \mbox{and}\quad
e^{\frac{1}{|a|}} \cdot \frac{1}{|a|^k} a^{\beta'}.
\end{equation*}
We see that  $\beta$ is increased by at most one order and $k$ is increased by at most $3$ so that our statement is proved.
\end{proof}
Next we consider for  $a\neq 0$,  the function  $\vartheta (a)$ defined  by 
$$\frac{a}{\abs{a}}=e^{-2\pi i\vartheta (a)}.$$
\begin{lem}\label{polj}
For every multi-index $\alpha$ satisfying  $|\alpha|\geq 1$,  the partial derivative $D^\alpha\vartheta(a)$ at a point $a\neq 0$,  is a linear combination of terms of the form
$\frac{a^\beta}{|a|^k}$
 with $k\leq 2\abs{\alpha}$ and $\abs{\beta}\leq \abs{\alpha}$. 
 \end{lem}
\begin{proof}
First we assume that  $\alpha$ has order one. As long as $a=x+iy$ satisfies  $x\neq 0$ we have $2\pi\cdot \vartheta(a)=\arctan(\frac{y}{x})$ so that 
$$
2\pi\cdot \frac{\partial\vartheta}{\partial x}(a)=-\frac{y}{|a|^2}\quad \text{and}\quad 
2\pi \cdot\frac{\partial\vartheta}{\partial y}(a)=\frac{x}{|a|^2}.
$$
This has the required form with $k=2$ and $\beta=(0,1)$ or $\beta(1,0)$.
We assume  that the assertion has been  proved for all   $\alpha$ of order $l$ and compute the derivatives of  order $l+1$. Without loss of generality we may assume  that  $\alpha=\alpha_0+(1,0)$.  The derivative
$$\frac{\partial }{\partial x} \left(\frac{a^\beta}{|a|^k}\right)$$
is equal to 
$$
\frac{\partial}{\partial x}\left(\frac{a^\beta}{|a|^k}\right).
=\beta_1\cdot\frac{a^{\beta'}}{|a|^{k}}-k\cdot\frac{a^\beta}{|a|^{k+2}}x=
\beta_1\cdot\frac{a^{\beta'}}{|a|^{k}}
-k\frac{a^{\beta+(1,0)}}{|a|^{k+2}}.
$$
where we put  $a^{\beta'}=0$ if $\beta_1=0$ and otherwise $\beta'=\beta-(0,1)$. 
This shows that $\beta$ increased at most by order one  and $k$ in the denominator by $2$.
This proves the assertion.
\end{proof}

\subsection{Families of Sc-Isomorphisms}

We assume that  $V$ is an open subset of a finite-dimensional vector space $H$ and $E$ and $F$ are sc-Banach spaces and  consider a family  $v\mapsto L(v)$ of linear operators parametrized by $v\in V$ 
such that  $L(v):E\rightarrow F$ are  sc-isomorphisms having the following properties.
\begin{itemize}\label{property-99}
\item[(1)] The map $\wh{L}:V\oplus E\rightarrow F$,  defined by 
$$
\wh{L}(v, h):=L(v)h,
$$
is sc-smooth.
\item[(2)] There exists for every $m$ a constant $C_m$ such that
$$
\abs{L(v)h}_m\geq C_m\cdot \abs{h}_m
$$
for all $v\in V$  and all $h\in E_m$.
\end{itemize}
Let us note that $ L(v)$ is not assumed to be continuously depending on $v$ as an operator.
Since the map $L(v):E\to F$ is an sc-isomorphism, the  equation
$$
L(v)h=k
$$
has  for every pair $(v,k)\in V\oplus F$  a unique solution $h$ denoted by 
$$h=L(v)^{-1}k=:f(v, k),$$
 so that $\wh{L}(v, f(v, k))=k$. 
\begin{prop}\label{solmap}
The map $f:V\oplus F\rightarrow E$ defined above is sc-smooth.
\end{prop}
\begin{proof}
We start by proving that $f$ is sc-continuous.
We fix a level $m$ and recall  that $\abs{L(v)h}_m\geq C_m\cdot \abs{h}_m$ for all $v\in V$ and $h\in E_m$.
We take a point $(v_0,k_0)\in V\oplus F_m$ and a sequence $(v_n,k_n)\in V\oplus F_m$ converging to $(v_0,k_0)$.  Setting $f(v_n,k_n)=h_n$ and $f(v_0,k_0)=h_0$ we compute,  
\begin{equation*}
\begin{split}
L(v_n)(h_n-h_0)&= k_n -L(v_n)h_0\\
&=k_n-L(v_0)h_0+L(v_0)h_0-L(v_0)h_0\\
&=k_n-k_0 + L(v_0)h_0-L(v_n)h_0=:\delta_n.
\end{split}
\end{equation*}         
Since the map $(v, h)\mapsto L(v)h$ is sc-smooth, it follows that  $\delta_n\rightarrow 0$ in $F_m$,  and using property (2) above,
\begin{equation*}
\begin{split}
\abs{f(v_n, k_n)-f(v_0, k_0)}_m&=\abs{h_n-h_0}_m\\
&\leq \frac{1}{C_m}\cdot \abs{L(v_n)(h_n-h_0)}_m= \frac{1}{C_m}\cdot \abs{\delta_n}_m\to 0.
\end{split}
\end{equation*}
Consequently, $f$ is sc-continuous.

Next we shall show that $f$ is  a map of class  $\ssc^1$. In order to define the candidate for the linearization  $Df(v,k): H\oplus F\to E$ of  the map $f$ at the point $(v, k)\in V\oplus F_1$, we formally differentiate the equation 
$$
\wh{L}(v, f(v, k))=k
$$
and obtain,
\begin{equation*}
\begin{split}
\delta k&=D_1\wh{L}(v, f(v, k))\cdot \delta v+D_2\wh{L}(v, f(v, k))\cdot Df(v, k)\cdot[\delta v, \delta k]\\
&=D_1\wh{L}(v, f(v, k))\cdot \delta v+\wh{L}(v, Df(v, k)\cdot[\delta v, \delta k])\\
&=D_1\wh{L}(v, f(v, k))\cdot \delta v+L(v)Df(v, k)\cdot[\delta v, \delta k]
\end{split}
\end{equation*}
where $[\delta v, \delta k]\in H\oplus F_1$. 
Hence
\begin{equation*}
\begin{split}
Df(v, k)\cdot[\delta v, \delta k]&=L(v)^{-1}(\delta k-D_1\wh{L}(v, f(v, k))\cdot \delta v)\\
&=f(v, \delta k-D_1\wh{L}(v, f(v, k))\cdot \delta v)
\end{split}
\end{equation*}
for all $[\delta v, \delta k]$ in $H\oplus E_1$. Observe that  for fixed $(v, k)\in V\oplus F_1$, the map  $(\delta v, \delta k)\mapsto  \delta k-D_1\wh{L}(v, f(v, k))$ defines  a bounded linear operator  between 
$H\oplus F_0$ and $F_0$. Then,  since $L(v):E_0\to F_0$ is a linear isomorphism,  the right-hand side in the first line  of the identity above defines  a bounded linear operator between $H\oplus F_0$  and $E_0$. In addition,  since $f$ is an $\ssc^0$-map, it follows that the map $(v,k,\delta v,\delta k)\in V\oplus F_{m+1}\oplus H\oplus F_m\to E_m$,  given by 
$$(v, k, \delta v, \delta k)\mapsto Df(v, k)\cdot[\delta v, \delta k], $$
is continuous so that  $Df$ is    $\ssc^0$.

It remains to verify the  approximation property.  We take $(v, k)\in V\oplus F_1$ and  $(\delta v,\delta k)\in H\oplus F_1$ and abbreviate 
$$\delta h=f(v+\delta v,k+\delta k)-f(v,k)-f(v,\delta k-D_1\wh{L}(v, f(v, k))\cdot \delta v)\in E_0.$$
Since  $C_0\abs{\delta h}_0\leq \abs{L(v+\delta v)\delta h}_0$, it suffices to show that 
$$\frac{1}{\abs{\delta v}+\abs{\delta k}_1}\abs{L(v+\delta v)\delta h}_0\to 0.$$

In order to prove  this,  we  compute at the point $(v, k)\in V\oplus F_1$ for $(\delta v,\delta k)\in H\oplus F_1$, 
\begin{equation*}
\begin{split}
&L(v+\delta v)\bigl[ f(v+\delta v,k+\delta k)-f(v,k)-f(v,\delta k-D_1\wh{L}(v, f(v, k))\cdot \delta v)\bigr]\\
&=k+\delta k - L(v+\delta v)\cdot f(v,k)-L(v+\delta v)\cdot f(v,\delta k-D_1\wh{L}(v, f(v, k))\cdot \delta v)\\
&=\delta k-D_1\wh{L}(v, f(v, k))\cdot \delta v\\
&-[L(v+\delta v)\cdot f(v,k)-L(v)\cdot f(v,k)-D_1\wh{L}(v, f(v, k))\cdot \delta v]\\
&-L(v+\delta v)\cdot f(v,\delta k-D_1\wh{L}(v, f(v, k))\cdot \delta v)\\
&=-[L(v+\delta v)\cdot f(v,k)-L(v)\cdot f(v,k)-D_1\wh{L}(v, f(v, k))\cdot \delta v]\\
&\phantom{=}-[L(v+\delta v)-L(v)]\cdot f(v, \delta k-D_1\wh{L}(v, f(v, k))\cdot \delta v)\\
\end{split}
\end{equation*}
using 
$ \delta k-D_1\wh{L}(v, f(v, k))\cdot \delta v=L(v)\cdot f(v,  \delta k-D_1\wh{L}(v, f(v, k))\cdot \delta v).$

The sc-smoothnes of the map 
$(v,h)\mapsto \wh{L}(v, h)=L(v)h$ implies for $k\in F_1$ and $h:=f(v,k)\in E_1$, 
\begin{equation*}
\frac{1}{\abs{\delta v}+\abs{\delta k}_1}\cdot\abs{(L(v+\delta v)\cdot f(v, k)-L(v)\cdot f(v,k)-D_1\wh{L}(v, f(v, k))\cdot \delta v }_0\\
\rightarrow 0
\end{equation*}
as $\abs{\delta v}+\abs{\delta k}_1\to 0$.
We next consider the second term. The map $f(v, k)$ is linear  with respect to the  variable $k$ so that  for fixed $(v,k)$ 
\begin{equation*}
\begin{split}
&\frac{1}{\abs{\delta v}+\abs{\delta k}_1}
\cdot \left| [ L(v+\delta v)-L(v)]\cdot  f(v,\delta k-D_1\wh{L}(v, f(v, k))\cdot \delta v )\right|_0\\
&\phantom{=}\leq  \biggl| [ L(v+\delta v)-L(v)]\cdot f\left(v,\frac{\delta k}{\abs{\delta v}+\abs{\delta k}_1}\right)\biggr|_0\\
&\phantom{=}+ \biggl| [ L(v+\delta v)-L(v)]\cdot  f\left(v,D_1\wh{L}(v, f(v, k))\frac{\delta v}{\abs{\delta v}+\abs{\delta k}_1}\right)\biggr|_0\\
&\phantom{=}=I_1+I_2
\end{split}
\end{equation*}
Since the   embedding $E_1\rightarrow E_0$ is compact  and $\delta k\in F_1$, we may assume without loss of generality that 
$\frac{\delta k}{\abs{\delta v}+\abs{\delta k}_1}$  converges in $F_0$ and since the maps $f$ is $\ssc^0$ and $(v, h)\mapsto \wh{L}(v, h)$ is sc-smooth, we conclude that $I_1\to 0$ as $\abs{\delta v}+\abs{\delta k}_1\to 0$.

Since  $\delta v$ is an element of the finite-dimensional space $H$, we may assume that 
$\frac{\delta v}{\abs{\delta v}+\abs{\delta k}_1}$  converges in $F_0$. Using again that  $f$ is $\ssc^0$ and $(v, h)\mapsto \wh{L}(v, h)$ is sc-smooth, we conclude that also the second term $I_2$ converges to $0$ as $\abs{\delta v}+\abs{\delta k}_1\to 0.$  Summing up, we have proved that 
$$\frac{1}{\abs{\delta v}+\abs{\delta k}_1}\abs{L(v+\delta v)\delta h}_0\to 0. $$
Consequently,  the map  $f:V\oplus F\to E$ is of class $\ssc^1$.

We have also proved that the tangent map $Tf:V\oplus E_1\oplus H\oplus E_0\to F_1\oplus F_0$ has the form
$$
Tf(v,k,\delta v,\delta k)=(f(v,k),f(v,\delta k-D_1\wh{L}(v, f(v, k))\cdot \delta v))
$$
which  is an $\ssc^0$-map.

To prove that $f$ is of class $\ssc^2$, it suffices to show that $f$  is of class $\ssc^2$ on $V'\oplus F$ for every subset $V'\subset V$ having  compact closure in $ V$.
We 
introduce the family $(v, \delta v)\mapsto L^1(v,\delta v)$,  parametrized by $(v,\delta v)\in TV$,  of sc-operators  $L^1(v, \delta v):TE\rightarrow TF$  defined by
\begin{equation*}
\begin{split}
(h,\delta h)\mapsto  &(L(v)h, D\wh{L}(v, h)\cdot [\delta v, \delta h])\\
&=(L(v)h, L(v)\delta h+ D_1\wh{L}(v, h)\cdot \delta v).
\end{split}
\end{equation*}
It  follows from the properties of  the family $v\mapsto L(v)$ that the map  $L^1(v, \delta v)$ is an sc-isomorphism and the map 
$$
TV\oplus TE\rightarrow TF, \quad (v,\delta v,h,\delta h)\mapsto  L^1(v,\delta v)(h,\delta h)
$$
is sc-smooth. 

We fix an open subset $V'\subset V$  having compact closure in $V$.  Since by our assumption the map $\wh{L}:V'\oplus E\to F$ is sc-smooth, it follows that also the map $$V'\oplus H\oplus E \to F,\quad (v, \delta v, h)\mapsto D_1\wh{L}(v, h)\cdot \delta v$$
is sc-smooth. We conclude, by a compactness argument, that for every level $m$ there exists a positive constant $d_m$   so that
$$
\abs{D_1\wh{L}(v,h)\delta v}_m\leq d_m\cdot \abs{h}_{m+1}\cdot \abs{\delta v}
$$
for $v\in V'$, $\delta v\in H$,  and $h\in F_{m+1}$.  This implies that given level $m$ there exists a constant $C'_m$ such that 
\begin{equation}\label{est-again0}
\begin{split}
\abs{  L^1(v,\delta v)(h,\delta h) }_m\geq C_m'\cdot \abs{(h,\delta h)}_m
\end{split}
\end{equation}
for all $v\in V'$ and all $\delta v\in H$  satisfying  $|\delta v|<1$. 
Indeed, if $0<\varepsilon\leq C_m$, where $C_m$ is the constant required in property (2) at the beginning of this section,  then 
\begin{equation*}
\begin{split}
\abs{  L^1(v,\delta v)(h,\delta h) }_m&= |L(v) h|_{m+1} + |L(v)\delta h +D_1\wh{L}(v,h)\delta v)|_m\\
&\geq C_{m+1} \cdot |h|_{m+1} + C_m\cdot |\delta h +L(v)^{-1}D_1\wh{L}(v,h)\delta v)|_m\\
&\geq C_{m+1} \cdot |h|_{m+1} +\varepsilon\cdot  |\delta h +L(v)^{-1}D_1\wh{L}(v,h)\delta v)|_m\\
&\geq  C_{m+1} \cdot |h|_{m+1} + \varepsilon\cdot |\delta h|_m - \varepsilon\cdot |L(v)^{-1}D_1\wh{L}(v,h)\delta v)|_m\\
&\geq C_{m+1} \cdot |h|_{m+1} + \varepsilon\cdot |\delta h|_m - \varepsilon\cdot C_m\cdot d_m \cdot |h|_{m+1}
\end{split}
\end{equation*}
for all $v\in V'$ and all $\delta v\in H$  satisfying  $|\delta v|<1$. 
Hence, taking $\varepsilon>0$ small enough we obtain the desired estimate \eqref{est-again0}.
Now our previous discussion applied to $L^1$ shows that the map 
$$
f^1:V'\oplus \{\delta v\in H\vert \ |\delta v|<1\}\oplus TF\rightarrow TE,
$$
 defined  as the solution of 
$$
L^1(v,\delta v)f^1(v,\delta v,k,\delta k)=(k,\delta k), 
$$
is of class  $\ssc^1$. Now we observe that
$$
f^1(v,\delta v,k,\delta k)= (f(v,k),f(v,\delta k-D_1\wh{L}(v, f(v, k))\cdot \delta v)).
$$
This  shows that the tangent map $Tf$ is of class $\ssc^1$ implying that $f$ is of class $\ssc^2$.
The result now  follows by induction.
\end{proof}
The above result remains true if $V$ is relatively open in
the partial quadrant of a finite-dimensional vector space.

\subsection{Cauchy-Riemann Operators}
The crucial point in the discussion of the Fredholm property of our Cauchy-Riemann operator
is the behavior of  the operator under  the  gluing of  the half-cylinders.
We start with  the linear Cauchy-Riemann operator and first recall some standard facts.  As usual we use the  symbol $Z_a$ to denote the finite glued cylinders  and $C_a$  to denote the infinite glued cylinders.

For the first result we work on the sc- Hilbert space $H^{3,\delta_0}_c(C_a,{\mathbb R}^2)$ which consists of all maps $u:C_a\to \R^2$ in  $H^3_{loc} $ so that there exists a constant $c\in {\mathbb R}^2$ for which  the map $u-c$ has weak  partial derivatives up to order 3 which weighted by $e^{\delta_0|s|}$ belong to  $L^2([0,\infty)\times S^1)$  while the weak partial derivatives up to order $3$ of the map $u+c$ belong to $L^2((-\infty,0]\times S^1)$. 
The level $m$ of  the sc-Hilbert space $H^{3,\delta_m}_c(C_a,{\mathbb R}^2)$  corresponds to the Sobolev regularity $(m+3,\delta_m)$ and the level $m$ of $H^{2,\delta_0}_c(C_a,{\mathbb R}^2)$  corresponds to  the regularity  $(2+m,\delta_m)$ where $(\delta_m)_{m\geq 0}$ is a strictly increasing sequence satisfying $0<\delta_m<2\pi$. The norms of these Hilbert spaces are defined in Section \ref{esttotalgluing}.
\begin{lem}\label{cr-asym}
The Cauchy-Riemann operator 
$$
\ov{\partial}_0:H^{3,\delta_0}_c(C_a,{\mathbb R}^2)\rightarrow H^{2,\delta_0}(C_a,{\mathbb R}^2), \quad \xi\mapsto \ov{\partial}_0\xi
$$
is an sc-isomorphism. In particular, for every $m\geq 0$ there exists a constant $C_m$ such that 
$$
\frac{1}{C_m}\cdot \abs{\xi}_{H^{3+m,\delta_m}_c}\leq \abs{\ov{\partial}_0\xi}_{ H^{2+m,\delta_m}}\leq C_m\cdot \abs{\xi}_{H^{3+m,\delta_m}_c}.
$$
\end{lem}

We also need the Cauchy Riemann operator in a different set-up. 
Assume that we have two copies of the cylinder ${\mathbb R}\times S^1$, which we denote by $\Sigma^\pm$. Then viewing
${\mathbb R}^+\times S^1\subset \Sigma^+$ and ${\mathbb R}^-\times S^1\subset \Sigma^-$ the original gluing construction carried out for the half-cylinders 
results as before in $C_a$, besides $Z_a$  we also have an infinite cylinder denoted by $Z_a^\ast$,  as illustrated in Figure \ref{Fig8}.  It is important that  $Z_a^\ast$ and $C_a$ both contain 
$Z_a$ in a natural way and the two different holomorphic coordinates extend to $Z_a^\ast$ as well as to $C_a$. 
\begin{center}
\begin{figure}[htbp]
\mbox{}\\[2pt]
\centerline{\relabelbox \epsfxsize 3.2truein \epsfbox{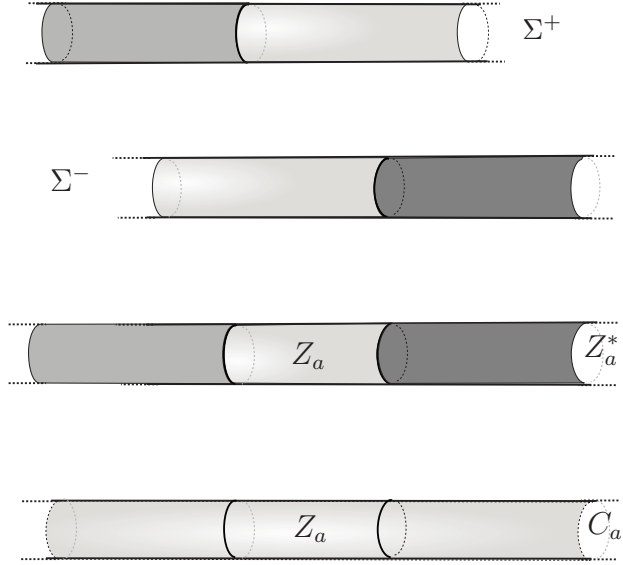}
\relabel {a}{\small{$\Sigma^+$}} 
\relabel {b}{$\Sigma^-$} 
\relabel {c}{$Z^\ast_a$}
\relabel {d}{$C_a$}
\relabel {e}{$Z_a$}
\relabel {f}{$Z_a$}
\endrelabelbox}
\mbox{}\\[2pt]
\caption{The extended cylinder $Z_a^\ast$.}\label{Fig8}
\end{figure}
\end{center}

The Hilbert spaces  $H^{3+m,-\delta_m}(Z_a^\ast,{\mathbb R}^2)$ for $m\geq 0$ consist of maps $u:Z_a^\ast\rightarrow {\mathbb R}^2$ for which the associated maps $v:{\mathbb R}\times S^1\rightarrow {\mathbb R}^2$, defined by $ (s,t)\rightarrow u([s,t])$ have partial derivatives up to order $3+m$ which, if  weighted by $e^{-\delta_m|s-\frac{R}{2}|}$ belong to $L^2(\R\times S^1)$. The spaces $H^{2+m,-\delta_m}(Z_a^\ast,{\mathbb R}^2)$  are defined analogously. The norms are denoted by $\nr u\nr_{3+m,-\delta_m}^{\ast}$, respectively $\nr u\nr_{2+m,-\delta_m}^{\ast}$.
The spaces $H^{3+m,-\delta_m}(Z_a,{\mathbb R}^2)$ and  $H^{2+m,-\delta_m}(Z_a,{\mathbb R}^2)$ are defined the same way. Their norms are denoted by $\nr u\nr_{3+m,-\delta_m}$, respectively by $\nr u\nr_{2+m,-\delta_m}$. We would like to point out that there is  no sc-structures on $H^{3,-\delta_0}(Z^*_a, \R^2) $, respectively $H^{2,-\delta_0}(Z^*_a, \R^2)$,  where the level $m$ corresponds  to the regularity $(3+m,-\delta_m)$, respectively $(2+m, -\delta_m)$.  
We denote  by
$[u]_a$ the average of a map $u:{\mathbb  R}\times S^1\rightarrow {\mathbb R}^2$ over the  circle at $[\frac{R}{2},t]$,   defined by 
$$
[u]_a:=\int_{S^1} u\left(\left[\frac{R}{2},t\right]\right) dt.
$$
\begin{lem}\label{cr-mean}
The Cauchy-Riemann operator
$$
\ov{\partial}_0:H^{3+m,-\delta_m}(Z_a^{\ast},{\mathbb R}^2)\rightarrow H^{2+m,-\delta_m}(Z_a^{\ast},{\mathbb R}^2),\quad \quad \xi\mapsto \carem u
$$
is a surjective Fredholm operator of real Fredholm index equal to $2$ for all $m\geq 0$. The kernel consists of the constant functions. Moreover,  there exists a constant $C_m>0$ such that
$$
\frac{1}{C_m}\cdot\nr u-[u]_a\nr_{m+3,-\delta_m}^{\ast}\leq \nr \bar{\partial}_0( u-[u]_a)\nr_{m+2,-\delta_m}^{\ast}\leq C_m\cdot\nr u-[u]_a\nr_{m+3,-\delta_m}^{\ast}.
$$
\end{lem}
\subsection{Proof of Proposition \ref{family}}
Recalling  Remark \ref{rem-x} and the definition of the family $\lambda\rightarrow L(\lambda)$ from Section \ref{nsection3.4}, we shall prove the Proposition \ref{family}.

We recall that  the space $\wh{E}$ consists of pairs $(h^+,h^-)$ with $h^{\pm}\in H^{3,\delta_0}_c({\mathbb R}^{\pm}\times S^1,{\mathbb R}^2)$ equipped with the obvious sc-structure and $\wh{E}_0$ is the closed subspace of $\wh{E}$ consisting of those pairs$(h^+, h^-)$  which satisfy $h^\pm(0,0)=(0,0)$ and $h^\pm(0,t)\in \{0\}\times {\mathbb R}$. 
Using the decomposition $h^\pm=h^\pm_{\infty}+r^\pm$ in which $h^\pm_{\infty}$ are the  asymptotic constants and  $r^\pm\in H^{3+m, \delta_m}(\R^\pm \times S^1)$, the $\wh{E}_m$-norm of $(h^+, h^-)$ is  defined as 
\begin{equation}\label{normeq1}
\abs{(h^+, h^-)}^2_{\wh{E}_m}=\abs{h^+_{\infty} }^2+ \abs{h^-_{\infty} }^2+\abs{r^+}^2_{H^{3+m, \delta_m}}+\abs{r^-}^2_{H^{3+m, \delta_m}}.
\end{equation}
We also recall the space $F$  consists of pairs $(\eta^+,\eta^-)\in H^{2, \delta_0}(\R^+\times S^1, \R^2)\oplus H^{2, \delta_0}(\R^-\times S^1, \R^2)$  equipped with the sc-structure $H^{2+m, \delta_m}(\R^+\times S^1, \R^2)\oplus H^{2, \delta_0}(\R^-\times S^1, \R^2)$. The $F_m$-norm of the pair $(\eta^+, \eta^-)\in F$ is given by
\begin{equation*}
\abs{(\eta^+, \eta^-)}_{F_m}^2=\abs{\eta^+}^2_{H^{m+2,\delta_m}}+\abs{\eta^-}^2_{H^{m+2,\delta_m}}.
\end{equation*}

\noindent{\bf Proposition 3.32}
{\em 
If $(\lambda, h^+, h^-)\to L(\lambda)(h^+, h^-)$ and $\lambda=(v, a)$ is the sc-smooth map defined  in \eqref{eq74} in a neighborhood of $\lambda_0=(v_0, 0)$, then there exists a constant $\sigma>0$  so that the following holds. 
\begin{itemize}
\item[(1)] If  $\lambda =(v, a)$ satisfies $\abs{\lambda-\lambda_0}<\sigma$, then the linear operator $L(\lambda):\wh{E}_0\to F$  is an sc-isomorphism.
\item[(2)]   For every $m\geq 0$ there exists a constant $C_m$ independent of $\lambda$ so that the  norm of the inverse operators
$$L(\lambda)^{-1}:F_m\to (\wh{E}_0)_m$$
is  bounded by $C_m$ for every   $\abs{\lambda-\lambda_0}<\sigma$.
\end{itemize}
}

\begin{proof} We already know  from Remark \ref{rem-x} that for $\lambda_0=(v_0, 0)$ the operator 
$L(\lambda_0):(\wh{E}_0)\to F$ is an sc-isomorphism. We also know that $\lambda\mapsto L(\lambda)$ is a  family of Fredholm operators of index $0$, 
$$
\text{ind}\ L(\lambda )=\text{ind}\ L(\lambda_0)=0.
$$
The family is not continuous in the operator norm.
Therefore, our  main task is  to prove the following  injectivity estimate.\\[1ex]
\noindent $(\ast)$ \:  There exists a constant $\sigma>0$ independent of $m$,
so that for every level $m$ there exists a constant $C_m$  
such that 
\begin{equation}\label{est-newm1}
\abs{L(\lambda )(h^+, h^-)}_{F_m}\geq C_m\abs{(h^+, h^-)}_{\wh{E}_m}
\end{equation}
for all $(h^+, h^-)\in (\wh{E_0})_m$ and all $\lambda$ satisfying $\abs{\lambda-\lambda_0}<\sigma$. \\

\noindent We claim that $(\ast)$ is a consequence of the following statement $(\ast \ast)$.\\[1ex]
\noindent $(\ast\ast )$ \:  
There exists an open neighborhood $U$ of  the point $\lambda_0=(v_0,0)$ in $H\times {\mathbb C}$ such that the following holds: 
Given a level $m$,  a point $\lambda\in U$, and two  sequence $\lambda_k\rightarrow \lambda$, and $(h^+_k,h^-_k)\in {(\wh{E}_0)}_m$ satisfying  $|(h^+_k,h^-_k)|_{\wh{E}_m}=1$
and 
$$
|{L(\lambda_k)(h^+_k,h^-_k)}|_{\wh{E}_m}\rightarrow 0, 
$$
the sequence $((h^+_k,h^-_k))$ has a convergent subsequence in $(\wh{E}_0)_m$.\\

Let us show that the statement $(\ast \ast)$ implies the statement $(\ast)$. Assuming that $(\ast\ast)$ holds, we  first consider the  level $0$ and claim that there are positive constants $\sigma'$ and $C_0$ such that 
\begin{equation}\label{lest0}
\abs{L(\lambda )(h^+, h^-)}_{F_0}\geq C_0\abs{(h^+, h^-)}_{\wh{E}_0}
\end{equation}
for all $\abs{\lambda-\lambda}<\sigma'$ and all $(h^+, h^-)\in  (\wh{E}_0)_0$.   Indeed, otherwise,  we find sequences $\lambda_k\to \lambda_0$ and $(h^+_k,h^-_k)\in {(\wh{E}_0)}_0$ satisfying $\abs{(h^+_k,h^-_k)}_{\wh{E}_0}=1$ and 
$L(\lambda_k)(h_k^+, h_k^-)\to  0$ in $F_0$. Applying $(\ast \ast)$,  we may assume that  the sequence  $(h^+_k, h^-_k)$ converges in $\wh{E}_0$ to  a pair $ (h^+, h^-)\in (\wh{E}_0)_0$ whose $\wh{E}_0$-norm is equal to $1$.
Consequently, 
$$L(\lambda_k)(h^+_k, h^-_k)\to L(\lambda_0)(h^+, h^-)=(0, 0).$$
This is impossible since $L(\lambda_0)$ is an isomorphism and so our claimed is proved. 
Since $\text{ind}\ L(\lambda )=\text{ind}\ L(\lambda_0)=0$, it follows  from  \eqref{lest0}  that the linear operators  $L(\lambda):\wh{E}_0\to F$  are sc-isomorphism for all $\abs{\lambda-\lambda_0}<\sigma'$.  This proves the statement (1) of  Proposition \ref{family}.

Next we fix a level $m$ and a parameter $\lambda$ satisfying $\abs{\lambda-\lambda_0}<\sigma'$.  
Arguing as above and using that $L(\lambda):(\wh{E}_0)_m\to F_m$ is an isomorphism for this $\lambda$,  we find an open neighborhood $U_{\lambda, m}$ of $\lambda$ in $H\times \C$ and a positive constant $c_{\lambda, m}$ such that 
\begin{equation}\label{localest1}
\abs{L(\lambda')(h^+, h^-)}_{F_m}\geq  c_{\lambda, m}\abs{(h^+, h^-)}_{\wh{E}_m}
\end{equation}
for all $\lambda'\in U_{\lambda, m}$ and $(h^+, h^-)\in (\wh{E}_0)_m$.

We choose  $0<\sigma<\sigma'$. Since the closed ball $\ov{B}_{\sigma}(\lambda_0)$ is compact in a finite dimensional space $H\times \C$, we find finitely many open sets $U_{\lambda_1, m}, \ldots ,U_{\lambda_{k_m},m}$ covering $B_{\sigma}(\lambda_0)$ such that the estimate \eqref{localest1} holds for all $\lambda\in U_{\lambda_j, m}$ with   constants $c_{\lambda_j, m}$ replacing $c_{\lambda,m}$. Choosing  $C_m:=\min\{c_{\lambda_1,m},\ldots, c_{\lambda_{k_m},m}\}$,  we conclude
$$\abs{L(\lambda')(h^+, h^-)}_{F_m}\geq  C_m\abs{(h^+, h^-)}_{\wh{E}_m}
$$
for all $\lambda\in B_{\sigma}(\lambda_0)$ and $(h^+, h^-)\in (\wh{E}_0)_m$.  Since the constant $\sigma$ is independent of level $m$, 
we have proved that $(\ast\ast)$ indeed implies $(\ast)$.

It remains to prove $(\ast \ast)$. We first define the set $U\subset H\times \C$.
We recall that on the half-cylinders $\R^\pm\times S^1$ we are given  smooth families $v\mapsto j^\pm (v)$ of complex structures  satisfying $j^+(v)=i$ on $[s_0-1, \infty )\times S^1$ and $j^-(v)=i$ on $(-\infty, -s_0+1]\times S^1$.  We recall the abbreviations
\begin{align*}
\ov{\partial}_vh&=\frac{1}{2}\left[ Th +i\circ Th\circ j(v)\right]\frac{\partial}{\partial s}\\
\ov{\partial}_0h&=\frac{1}{2}\left[ Th +i\circ Th\circ i\right]\frac{\partial}{\partial s}
\end{align*}
of the Cauchy-Riemann operators.

 In view of the standard elliptic estimates we have 
$$\abs{\ov{\partial}_{v_0}h}_{H^{2+m}}\geq C\abs{h}_{H^{3+m}}$$ 
for some positive constant $C$ and all $h\in H^{3+m}([0, s_0]\times S^1)$ having compact supports in $[0, s_0]\times S^1$  and satisfying $h(\{0\}\times S^1)\subset \{0\}\times \R$. 
Observe that since   $0<\delta_m<2\pi$, the norms on the Sobolev spaces  $H^{2+m} ([0, s_0]\times S^1)$ and $H^{2+m, \delta_m} ([0, s_0]\times S^1)$ are equivalent so that the  above estimate can be restated as 
$$\abs{\ov{\partial}_{v_0}h}_{H^{2+m, \delta_m}}\geq C\abs{h}_{H^{3+m,\delta_m}}.$$ 
Since 
\begin{equation*}
\begin{split}
\care_{v}h=\care_{v_0}h+\frac{1}{2} [ i\circ (T h)\circ (j^+(v)-j^+(v_0)]\frac{\partial }{\partial s}
\end{split}
\end{equation*}
and the family $v\mapsto  j^+(v)$ of complex structures on the half-cylinder $\R^+\times S^1$ is smooth, we conclude that 
$$
\abs{\care_{v}h}_{H^{2+m, \delta_m}}\geq \abs{\care_{v_0}h}_{H^{2+m, \delta_m}} -c(v)\abs{h}_{H^{3+m, \delta_m}}
$$
where $c(v)$ is a function converging to $0$ as $v\to v_0$.   Consequently, we may choose positive constants $C_m$ and $\rho$ so that 
\begin{equation}\label{crestaimaten1}
\abs{\care_{v}h}_{H^{2+m, \delta_m}}\geq C_m\abs{h}_{H^{3+m, \delta_m}}
\end{equation}
for  all $v\in B_{\rho}(v_0)$ and all maps  $h\in H^{3+m}([0, s_0]\times S^1)$ having compact supports in $[0, s_0]\times S^1$  and satisfying $h(\{0\}\times S^1)\subset \{0\}\times \R$. 
Taking, if necessary, smaller constants $C_m$ and $\rho$, we have also  the estimate 
$$\abs{\care_{v}h}_{H^{2+m, \delta_m}}\geq C\abs{h}_{H^{3+m, \delta_m}}$$
for the  Cauchy-Riemann operator $\care_{v}$ acting on maps $h\in H^{3+m}([-s_0,0]\times S^1)$ having compact supports in $[-s_0,0]\times S^1$  and satisfying $h(\{0\}\times S^1)\subset \{0\}\times \R$ for all $v\in B_{\rho}(v_0)$. Having defined  $\rho$,  we choose a sufficiently small number  $\tau>0$ satisfying  $2s_0+4<\varphi (\tau)=e^{\frac{1}{\tau}}-e$ and  set  
$$U=B_{\rho}(v_0)\times B_{\tau}(0)\subset H\times \C.$$

With this choice of  the set $U$ we are ready to prove the  statement  $(\ast \ast)$. 
We fix a point $\lambda=(v,a)\in U$, a level $m$,  and take  two sequences $(h^+_k, h^-_k)\in (\wh{E}_0)_m$ and $\lambda_k=(v_k, a_k)\in U$ satisfying 
 $\abs{(h^+_k, h^-_k)}_{\wh{E}_m}=1$ and $\lambda_k\to \lambda=(v, a)$ and 
\begin{equation}\label{kling}
\abs{L(\lambda_k)(h^+_k,h^+_k)}_{F_m}\rightarrow 0.
\end{equation}
We shall show  that the sequence $((h^+_k, h^-_k))$ has a converging subsequence in $(\wh{E}_0)_m$. 
Since 
$$1=\abs{(h^+_k, h^-_k)}^2_{\wh{E}_m}=\abs{h^+_{k,\infty} }^2+ \abs{h^-_{k,\infty} }^2+\abs{r^+_k}^2_{H^{3+m, \delta_m}}+\abs{r^-_k}^2_{H^{3+m, \delta_m}},$$
where  $h^{\pm}_k=h^\pm_{k, \infty}+r^\pm_k$ in which $h^\pm_{k, \infty}$ are  the asymptotic constants and $r^\pm_k\in H^{3+m, \delta_m}(\R^\pm \times S^1)$,   the sequences   $h^+_{k, \infty}$ and $ h^-_{k,\infty} $ are bounded  in $\R^2$ so that  we  can assume  without loss of generality the convergence $(h^+_{k, \infty},  h^-_{k,\infty} )\to (h^+_{\infty},  h^-_{\infty})$.  In addition, the  compact embedding theorem implies that there  are  subsequences again denoted by $h^+_k$ and $h^-_k$  converging in $H^{2+m,\delta_m}$ on $[0, s_0+1]\times S^1$ and $[-s_0-1,0]\times S^1$. 

Next we choose  a smooth function $\alpha_+:\R\to [0,1]$  having derivative $\alpha_+'\leq 0$ and 
satisfying  $\alpha_+(s)=1$ on $(-\infty, s_0-1]$ and $\alpha(s)=0$ on  $[s_0,\infty)$.  We define 
$\alpha_-(s)=\alpha_+(-s)$ and set $\gamma_\pm=1-\alpha_\pm$.  

In order to prove our claim we shall show, using the decomposition $h^\pm_k=h^\pm_{k,\infty}+r^\pm_k$,  that the sequences $(\alpha_+r^+_k, \alpha_-r^-_k)$ and $(\gamma_+r^+_k, \gamma_-r^-_k)$ posses subsequences satisfying 
\begin{equation}\label{contwo1}
(\alpha_+r^+_k, \alpha_-r^-_k)\to (f^+, f^-)\quad \text{and}\quad (\gamma_+r^+_k, \gamma_-r^-_k)\to (g^+, g^-)
\end{equation}
in $\wh{E}_m$ for two pairs $(f^+, f^-)$ and $(g^+, g^-)$ belonging to $\wh{E}_m$.  We already know that the sequence $(h^+_{k,\infty}, h^-_{k,\infty})$ of asymptotic constants converges to 
$(h^+_{\infty}, h^-_{\infty})$.  Then  assuming  the convergence in \eqref{contwo1} and  setting 
$h^\pm=h^\pm_{\infty}+f^\pm+g^\pm$, we conclude, using 
$h^\pm_{k}=h^+_{k,\infty}+\alpha_\pm r^\pm_k+\gamma_\pm r^\pm_k$,  that 
$$(h^+_k, h^-_k)\to (h^+, h^-)\quad \text{in $\wh{E}_m$},$$
as claimed.

We begin with the sequence $(\alpha_+r^+_k, \alpha_-r^-_k)$, and abbreviate 
\begin{equation}\label{eqeta1}
L(\lambda_k)(h_k^+, h^-_k)=(\eta^+_k, \eta^-_k).
\end{equation}
It  follows from  formula \eqref{cr-formula} for the operator $L(\lambda )$ that 
\begin{equation*}
\begin{bmatrix}\eta^+_k\\ \eta^-_k\end{bmatrix}=
\begin{bmatrix}\ov{\partial}_{v_k}h^+_k \\ \ov{\partial}_{v_k}h^-_k
\end{bmatrix}+\begin{bmatrix}\Phi^+_{a_k}(h^+_k, h^-_k)\\ \Phi^-_{a_k}(h^+_k, h^-_k)\end{bmatrix}
\end{equation*}
where the maps $\Phi^\pm_{a_k}(h^+_k, h^-_k)$ vanish outside of the finite cylinders $[\frac{R_k}{2}-1, \frac{R_k}{2}+1]\times S^1$ and $[-\frac{R_k}{2}-1, -\frac{R_k}{2}+1]\times S^1$, respectively. We recall that $\abs{a_k}<\tau$.  Since 
$2s_0+4<\varphi (\tau)$, it follows that $2s_0+4<R_k$ for all $k$. Since the functions $\alpha_\pm$ vanish for $s\geq s_0$ and $s\leq -s_0$, respectively, we conclude that 
$$\alpha_+\eta^+_k=\alpha_+\ov{\partial}_{v_k}h^+_k\quad \text{and}\quad \alpha_-\eta^-_k=
\alpha_-\ov{\partial}_{v_k}h^-_k.$$
Our assumption $L(\lambda_k)(h_k^+, h^-_k)=(\eta^+_k, \eta^-_k)\to (0, 0)$ in $F_m$
implies  the convergence $(\alpha_+\eta^+_k, \alpha_-\eta^-_k)\to (0, 0)$  in $F_m$.  Consequently,  
$$
(\alpha_+\ov{\partial}_{v_k}h^+_k, \alpha_-\ov{\partial}_{v_k}h^-_k)\to (0, 0)\quad \text{in $F_m$.  }
$$
Then using 
$$\ov{\partial}_{v}h^+_k=\ov{\partial}_{v_k}h^+_k+\frac{1}{2}\left[i\circ Th^+_k\circ (j^+(v)-j^+(v_k))\right]\circ \frac{\partial }{\partial s}$$
and the fact that  the family $v\mapsto j^+(v)$ is smooth and that $\abs{h^+_k}_{H^{3+m,\delta_m}}\leq 1$,  we find  also that 
\begin{equation}\label{conttwo2}
(\alpha_+\ov{\partial}_{v}h^+_k, \alpha_-\ov{\partial}_{v}h^-_k)\to (0, 0)\quad \text{in $F_m$}.
\end{equation}

Since $\alpha_+'$ vanishes outside of the interval $[s_0-1, s_0]$ and $\abs{h^+}_{H^{3+m,\delta_m}}\leq 1$,  the sequence $(\alpha_+h^+_k)$ is bounded in the $H^{3+m,\delta_m}$-norm and, by the Sobolev compact embedding theorem, we may assume that the sequence  converges in  the $H^{2+m,\delta_m}$-norm.  Now the equality $\ov{\partial}_{v}(\alpha_+h^+_k)=\alpha_+\ov{\partial}_{v}h^+_k+\alpha_+'h^+_k$ for all $k$  and the estimate \eqref{crestaimaten1} applied to $\alpha_+h^+_k-\alpha_+h^+_n$ show that
\begin{equation*}
\begin{split}
C_m\abs{\alpha_+h^+_k-\alpha_-h^-_n}_{H^{3+m,\delta_m}}&
\leq \abs{\ov{\partial}_{v}(\alpha_+h^+_k-\alpha_+h^+_n)}_{H^{2+m,\delta_m}}\\
&\leq \abs{ \alpha_+\ov{\partial}_{v}h^+_k-\alpha_+\ov{\partial}_{v}h^+_n}_{H^{2+m,\delta_m}}\\
&\phantom{\leq }+\abs{ \alpha_+'h^+_k- \alpha_+'h^+_n}_{H^{2+m,\delta_m}}.
\end{split}
\end{equation*}
Since we  already know that the sequences $(\alpha'_+h^+_k)$ and $(\alpha_+\ov{\partial}_{v}h^+_k)$ converge in $H^{2+m, \delta_m}$ on $[0,s_0+1]\times S^1$,  we conclude that  the right-hand side converges to $0$. Hence the sequence $(\alpha_+h^+_k)$ is a Cauchy sequence  and so converges  in the $H^{3+m,\delta_m}$-norm. The same argument applied to the sequence $(\alpha_-h^-_k)$ shows that  also  the sequence 
$(\alpha_-h^-_k)$ converges in $H^{3+m,\delta_m}$ and we conclude the convergence 
$(\alpha_+h^+_k, \alpha_-h^-_k)\to (\wt{f}^+, \wt{f}^-)$ in $\wh{E}_m$ to  some $(\wt{f}^+, \wt{f}^-)\in \wh{E}_m$.  Moreover, 
%in view of the embeddings $H^{3+m,\delta_m}([0, s_0+1]\times S^1)\subset C^0([0, s_0+1]\times S^1)$ and $H^{3+m,\delta_m}([-s_0-1]\times S^1)\subset  C^0([-s_0-1]\times S^1)$, we conclude that 
$\wt{f}^\pm (0, 0)=(0, 0)$ and 
$\wt{f}^\pm (\{0\}\times S^1)\subset  \{0\}\times \R$. Since $\alpha_\pm r^\pm_k=\alpha_\pm h^\pm_{k}-\alpha_\pm h^\pm_{k,\infty}$ and $h^\pm_{k, \infty}\to h^\pm_\infty$ and $\alpha_\pm$ is equal to $0$ if $s\geq s_0$ and $s\leq -s_0$, respectively, we conclude that also  the sequence $(\alpha_+r^+_k, \alpha_-r^-_k)$ converges to $(f^+, f^-)=(\wt{f}^+-\alpha_+h^+_{\infty}, \wt{f}^- -\alpha_-h^-_{\infty})$ in $\wh{E}_m$. This finishes  the proof of the first convergence in \eqref{contwo1}.

Next we prove  the convergence of the sequence $(\gamma_+r^+, \gamma_-r^-)$. In order to do this, we consider the maps $\oplus_{a_k}(r^+_k, r^-_k)$ and $\ominus_{a_k}(r^+_k, r^-_k)$.
Recall the convergence
\begin{equation*}\label{eqeta1}
L(\lambda_k)(h_k^+, h^-_k)=(\eta^+_k, \eta^-_k)\to (0, 0)\quad \text{in $F_m$}.
\end{equation*}
Hence $(\gamma_+\eta^+_k, \gamma_-\eta^-_k)\to 0$ in $F_m$, and 
the estimate  for the total hat-gluing in  Theorem \ref{poker1}  leads to   
\begin{equation*}\label{newconv1}
e^{\delta_m\frac{R_k}{2}} \nr \wh{\oplus}_{a_k}(\gamma_+\eta^+_k, \gamma_-\eta^-_k)
\nr_{m+2,-\delta_{m}}\rightarrow 0
\end{equation*}
and
\begin{equation*}\label{newconv2}
e^{\delta_m\frac{R_k}{2}} \nr
 \wh{\ominus}_{a_k}(\gamma_+\eta^+_k, \gamma_-\eta^-_k)\nr_{m+2,-\delta_{m}}\rightarrow 0,
\end{equation*}
so that 
\begin{equation}\label{newconv3}
\qquad \nr \wh{\oplus}_{a_k}(\gamma_+\eta^+_k, \gamma_-\eta^-_k)
\nr_{m+2,-\delta_{m}}\leq \varepsilon_ke^{-\delta_m\frac{R_k}{2}}
\end{equation}
and
\begin{equation}\label{newconv4}
\qquad \nr \wh{\ominus}_{a_k}(\gamma_+\eta^+_k, \gamma_-\eta^-_k)\nr_{m+2,-\delta_{m}}\leq \varepsilon_ke^{-\delta_m\frac{R_k}{2}}
\end{equation}
for a  sequence $\varepsilon_k$ of positive numbers converging to $0$.

Now, in view of  the formula \eqref{CRUCIALLL}, 
the maps $\oplus_{a_k}(r^+_k, r^-_k)$ and $\ominus_{a_k}(r^+_k, r^-_k)$ solve the equations
\begin{equation}\label{again1}
\ov{\partial}_{v_k}\oplus_{a_k}(r^+_k, r^-_k)=\wh{\oplus}_{a_k}(\eta^+_k, \eta^-_k)
\end{equation}
\begin{equation}\label{again2}
\ov{\partial}_{0}\ominus_{a_k}(r^+_k, r^-_k)=\wh{\ominus}_{a_k}(\eta^+_k, \eta^-_k).
\end{equation}
Introducing the functions $\gamma_k(s)=\oplus_{a_k}(\gamma_+, \gamma_-)=\beta_{a_k}(s)\cdot \gamma_+(s)+(1-\beta_{a_k}(s))\cdot \gamma_-(s-R_k)$, we note that $\gamma_k$ is equal to $0$   on  the interval $[s_0-1, R_k-s_0+1]$ and equal to $1$ on $[s_0, R_k-s_0]$. 
Hence  $\gamma_k\oplus_{a_k}(r^+_k, r_k^-)=\oplus_{a_k}(\gamma_+r^+_k, \gamma_-r^-_k)$ and the same identity holds for the hat gluing.  Now, multiplying \eqref{again1} by $\gamma_k$, we obtain
\begin{equation*}
\begin{split}
\wh{\oplus}_{a_k}(\gamma_+\eta^+_k, \gamma_-\eta_k^- )&=\gamma_k \wh{\oplus}_{a_k}(\eta^+_k, \eta_k^-)\\
&=\gamma_k \ov{\partial}_{v_k}\oplus_{a_k}(r^+_k, r^-_k)\\
&=\ov{\partial}_{v_k}(\gamma_k\oplus_{a_k}(r^+_k, r^-_k))-(\gamma_k)'\oplus_{a_k}(r^+_k, r^-_k))\\
&=\ov{\partial}_{v_k}(\gamma_k\oplus_{a_k}(r^+_k, r^-_k))-
\oplus_{a_k}(\gamma_+'r^+_k,\gamma'_- r^-_k))\\
&=\ov{\partial}_{v_k}\oplus_{a_k}(\gamma_+r^+_k, \gamma_-r^-_k)-
\wh{\oplus}_{a_k}(\gamma_+'r^+_k,\gamma'_- r^-_k).
\end{split}
\end{equation*} 
Since $\gamma_k=0$ outside of $[s_0-1,R_k-s_0+1]$ and 
$j^\pm(v_k)=i$ on $[s_0-1,\infty )\times S^1$ and $(-\infty ,-s_0+1]\times S^1$, the above identity becomes 
\begin{equation}\label{last}
\ov{\partial}_{0}\oplus_{a_k}(\gamma_+r^+_k, \gamma_-r^-_k)=\wh{\oplus}_{a_k}(\gamma_+\eta^+_k, \gamma_-\eta_k^- )+\wh{\oplus}_{a_k}(\gamma_+'r^+_k,\gamma'_- r^-_k).
\end{equation}
The maps  $\oplus_{a_k}(\gamma_+r^+_k, \gamma_-r^-_k)$ and $ \wh{\oplus}_{a_k}(\gamma_+\eta^+_k, \gamma_-\eta^-_k):Z_{a_k}\to \R^2$ have  compact supports contained in the finite cylinders  $[s_0-1, R_k-s_0+1]\times S^1$ and hence we can view them as defined on the infinite cylinders $Z^*_{a_k}$.  By  Lemma \ref{cr-mean},  there are unique maps $\xi_k\in H^{3+m,-\delta_m}(Z^*_{a_k}, \R^2)$ having mean-values $[\xi_k]_{a_k}=0$ and satisfying 
\begin{equation}\label{poker100}
\ov{\partial}_0\xi_k=\wh{\oplus}_{a_k}(\gamma_+\eta^+_k, \gamma_-\eta^+_k).
\end{equation}
In addition, by \eqref{newconv3} and Lemma \ref{cr-mean}, 
\begin{equation}\label{newestcor1}
\nr \xi_k\nr^{\ast}_{3+m,-\delta_m}\leq C_m\varepsilon_ke^{-\delta_m\frac{R_k}{2}}
\end{equation}
where the constant $C_m$ is independent  of $k$.
Then, in view of \eqref{last}, 
\begin{equation}\label{lastminus}
\ov{\partial}_{0}[ \oplus_{a_k}(\gamma_+r^+_k, \gamma_-r^-_k)-\xi_k]=\wh{\oplus}_{a_k}(\gamma_+'r^+_k,\gamma'_- r^-_k).
\end{equation}
Abbreviating 
$$q_k=\oplus_{a_k}(\gamma_+r^+_k, \gamma_-r^-_k),$$
 we claim that $q_k-\xi_k-[q_k-\xi_k]_{a_k}$ belongs to $H^{3+(m+1),-\delta_{m+1}}(Z^*_{a_k})$. To see this, we first estimate the $H^{3+m,-\delta_{m+1}}(Z^*_{a_k})$-norms of the maps $\wh{\oplus}_{a_k}(\gamma_+'r^+_k,\gamma'_- r^-_k)$. 
The derivatives $\gamma'_+$ and $\gamma_-'$ vanish  outside of the intervals $[s_0-1,s_0]$ and $[-s_0, -s_0+1]$ so that the square of the $H^{3+m,-\delta_{m+1}}(Z^*_{a_k})$-norm of 
$\wh{\oplus}_{a_k}(\gamma_+'r^+_k,\gamma'_- r^-_k)$ is equal to a constant times the sum of the following integrals
\begin{equation*}
\int_{\Sigma}\abs{D^{\alpha}r^+_k}^2e^{-2\delta_{m+1}\abs{s-\frac{R_k}{2}}}\quad \text{and}\quad 
\int_{\Sigma}\abs{D^{\alpha}r^-_k(s-R_k,t-\vartheta_k)}^2e^{-2\delta_{m+1}\abs{s-\frac{R_k}{2}}}
\end{equation*}
where the sum is taken over all multi-indices $\alpha$ of length $\abs{\alpha}\leq 3+m$ and where $\Sigma=[s_0-1, s_0]\times S^1$.
%\begin{gather*}
%\int_{[s_0-1, s_0]\times S^1}\abs{D^{\alpha}r^+}^2e^{-2\delta_{m+1}\abs{s-\frac{R_k}{2}}}\\
%\intertext{and}
%\int_{[s_0-1, s_0]\times S^1}\abs{D^{\alpha}r^-(s-R_k,t-\vartheta_k)}^2e^{-2\delta_{m+1}\abs{s-\frac{R_k}{2}}}
%\end{gather*}
%where the sum is taken over all multi-indices $\alpha$ of length $\abs{\alpha}\leq 3+m$.  

Recalling that $s_0+2<\frac{R_k}{2}$ and  $\abs{r^+_k}_{H^{3+m, \delta_m}}\leq 1$, the  first integral can be estimated as follows,
\begin{equation*}
\begin{split}
&\int_{[s_0-1, s_0]\times S^1}\abs{D^{\alpha}r^+_k}^2e^{-2\delta_{m+1}\abs{s-\frac{R_k}{2}}}=
\int_{[s_0-1, s_0]\times S^1}\abs{D^{\alpha}r^+_k}^2e^{2\delta_ms} 
e^{2\delta_{m+1}(s-\frac{R_k}{2})-2\delta_ms}\\
&\phantom{==}\leq e^{-\delta_mR_k} e^{2(\delta_{m+1}-\delta_m)s_0} \int_{[s_0-1, s_0]\times S^1}\abs{D^{\alpha}r^+}^2e^{2\delta_ms}
\leq   e^{-\delta_{m+1}R_k} e^{2(\delta_{m+1}-\delta_m)s_0}.
\end{split}
\end{equation*}
Since the same estimate holds  for the second integral, 
\begin{equation*}
\int_{[R_k-s_0, R_k-s_0+1]\times S^1}\abs{D^{\alpha}r^-(s-R_k,t-\vartheta_k)}^2e^{-2\delta_{m+1}\abs{s-\frac{R_k}{2}}}\leq e^{-\delta_{m+1}R_k} e^{2(\delta_{m+1}-\delta_m)s_0},
\end{equation*}
we obtain
\begin{equation}\label{gammaeq1}
\nr \wh{\oplus}_{a_k}(\gamma_+'r^+_k,\gamma'_- r^-_k)\nr^\ast_{3+m, -\delta_{m+1}}\leq C_me^{-\delta_{m+1}\frac{R_k}{2}}.
\end{equation}
Since the  map $\wh{\oplus}_{a_k}(\gamma_+'r^+_k,\gamma'_- r^-_k)$ belongs to $H^{3+m,-\delta_{m+1}}(Z^*_{a_k})$, it follows from  Lemma \ref{cr-mean}  that there exists a unique map  $f_k\in H^{3+(m+1),-\delta_{m+1}}(Z^*_{a_k})$  having mean-value  $[f_k]_{a_k}=0$ and satisfying 
\begin{equation}\label{eqf1}
\ov{\partial}_0f_k=\wh{\oplus}_{a_k}(\gamma_+'r^+_k, \gamma_-'r^-_k).
\end{equation}
Moreover, by \eqref{gammaeq1}, the following estimate holds, 
\begin{equation}\label{estfk}
\nr f_k\nr^*_{3+(m+1), -\delta_{m+1}}\leq C_me^{-\delta_{m+1}\frac{R_k}{2}}.
\end{equation}
Now, observe that both maps $f_k$ and $q_k-\xi_k-[q_k-\xi_k]_{a_k}$ have their averages equal to $0$ and satisfy equations  \eqref{lastminus} and \eqref{eqf1} .
Moreover,  $q_k-\xi_k-[q_k-\xi_k]_{a_k}=q_k-\xi_k-[q_k]_{a_k}\in H^{3+m,-\delta_m}(Z^*_{a_k})\subset H^{3+m,-\delta_{m+1}}(Z^*_{a_k})$ and $f_k\in H^{3+(m+1), -\delta_{m+1}}(Z^*_{a_k})\subset H^{3+m, -\delta_{m+1}}(Z^*_{a_k})$. Since  $\wh{\oplus}_{a_k}(\gamma_+r^+_k, \gamma_-r^-_k)\in H^{3+m,-\delta_{m+1}}(Z^*_{a_k})\subset H^{2+m,-\delta_{m+1}}(Z^*_{a_k})$, it follows from Lemma \ref{cr-mean}  applied to the Cauchy-Riemann operator 
$\ov{\partial }_0:H^{3+m, -\delta_{m+1}}(Z^*_{a_k})\to H^{2+m, -\delta_{m+1}}(Z^*_{a_k})$  that  $f_k=q_k-\xi_k-[q_k-\xi_k]_{a_k}$, as claimed.   In addition,  in view of the inequality \eqref{estfk}, the  following estimate holds,
\begin{equation}\label{poker102}
\nr q_k-\xi_k-[q_k-\xi_k]_{a_k}\nr^*_{3+(m+1), -\delta_{m+1}}\leq C_me^{-\delta_{m+1}\frac{R_k}{2}}\end{equation}
where the constant $C_m$ is independent of $k$.

We will also need estimates for  $\ominus_{a_k}(\gamma_+r^+_k, \gamma_-r_k^-)$. We first observe that  $\wh{\ominus}_{a_k}(\eta^+_k, \eta_k^-)=\wh{\ominus}_{a_k}(\gamma_+\eta^+_k, \gamma_-\eta_k^-)$ and $\ominus_{a_k}(r^+_k, r_k^-)=\ominus_{a_k}(\gamma_+r^+_k, \gamma_-r_k^-)$ so that  the equation  \eqref{again2} can be written as
\begin{equation}\label{last0}
\ov{\partial}_{0}\ominus_{a_k}(\gamma_+r^+_k, \gamma_-r^-_k)=\wh{\ominus}_{a_k}(\gamma_+\eta^+_k, \gamma_-\eta_k^- ).
\end{equation}
Then, abbreviating  
$$\zeta_k:=\ominus_{a_k}(\gamma_+r^+_k, \gamma_-r^-)$$
and using $\wh{\ominus}_{a_k}(\gamma_+\eta^+_k, \gamma_-\eta^-)\in H^{2+m,\delta_m}(C_a)$, it follows from \eqref{last0} and  Lemma \ref{cr-asym} that $\zeta_k$ is a unique solution in $ H^{3+m,\delta_m}_c(C_a)$  of the equation
$$\ov{\partial}_0\zeta_k=\ominus_{a_k}(\gamma_+\eta^+_k, \gamma_-\eta^+_k).$$
Moreover, in view of \eqref{newconv4}, its $H_c^{3+m,\delta_m}$-norm is estimated as 
\begin{equation}\label{newestcor2} 
\abs{\zeta_k}_{H_c^{3+m,\delta_m}}=\left[ \abs{\zeta_{k,\infty}}^2+\nr \wh{\zeta}_k\nr^2_{3+m,\delta_m}\right]^{1/2}\leq C_m\varepsilon_ke^{-\delta_m \frac{R_k}{2}},
\end{equation}
where  $\wh{\zeta}_k=\zeta_k-(1-2\beta_{a_k}) \zeta_{k, \infty}$ and  $\zeta_{k, \infty}=\lim_{s\to \infty}\zeta_k (s, t)=\avk (r^+_k, r^-_k)$. Hence
\begin{equation*}%\label{last0}
\ov{\partial}_{0}[ \ominus_{a_k}(\gamma_+r^+_k, \gamma_-r^-_k)-\zeta_k]=0.
\end{equation*}

We denote  the restrictions  of $q_k$ and $\xi_k$ to  the finite cylinder $Z_{a_k}$ by $\ov{q}_k$ and $\ov{\xi}_k$ and observe that the restriction of $q_k-\xi_k$ to $Z_{a_k}$ is equal to $\ov{q}_k-\ov{\xi}_k$ and 
$[q_k-\ov{\xi}_k]_{a_k}=[q_k]_{a_k}$ since $[\ov{\xi}_k]_{a_k}=0$.
In view of the estimates
\eqref{newestcor1}
% \eqref{average1}
 and \eqref{poker102}, we have 
\begin{equation}\label{xihat}
\nr \ov{\xi}_k\nr_{3+m,-\delta_m}\leq C_m\varepsilon_ke^{-\delta_m\frac{R_k}{2}}
\end{equation}
%\begin{equation}\label{poker103}
%\abs{[\ov{q}_k-\ov{\xi}_k]_{a_k}}=\abs{[\ov{q}_k]_{a_k}} \leq C_me^{-\delta_m \frac{R_k}{2}}
%\end{equation}
and
\begin{equation}\label{ovqk}
\nr \ov{q}_k-\ov{\xi}_k-[\ov{q}_k-\ov{\xi}_k]_{a_k}\nr_{3+(m+1), -\delta_{m+1}}\leq C_me^{-\delta_{m+1}\frac{R_k}{2}}.
\end{equation}

Now,  using Theorem \ref{propn-1.27}   we find  a unique pair $(p^+_k, p^-_k)\in \wh{E}_m$ satisfying 
\begin{equation}\label{poker99}
\begin{aligned}
\oplus_{a_k}(p^+_k, p^-_k)&=\ov{\xi}_k\\
\ominus_{a_k}(p^+_k, p^-_k)&=\zeta_k
\end{aligned}
\end{equation}
%\begin{equation}\label{poker103}
%\abs{[\ov{q}_k-\ov{\xi}_k]_{a_k}}=\abs{[\ov{q}_k]_{a_k}} \leq C_me^{-\delta_m \frac{R_k}{2}}
%\end{equation}
and we claim that the sequence of maps $(p^+_k, p^-_k)$ converges to $(0,0)$ in $\wh{E}_m$.  In view of of the estimate for the total gluing map in Theorem \ref{boxdot-est1} it suffices to prove the convergence 
$$\abs{(\ov{\xi}_k,\zeta_k)}_{G^{a_k}_m}^2\to 0.$$
Indeed, using \eqref{newestcor2} and \eqref{xihat} and $ [\ov{\xi}_k]_{a_k}]=0$  and the estimate \begin{equation*}
\nr \zeta_{k,\infty}\nr^2_{3+m,-\delta_m}=\int_{[0,R_k]\times S^1}\abs{\zeta_{k,\infty}}^2e^{-\delta_m\abs{s-\frac{R_k}{2}} }\leq \frac{ \abs{\zeta_{k,\infty} }^2}{\delta_m},
\end{equation*}
we obtain
\begin{equation*}
\begin{split}
\abs{(\ov{\xi}_k,\zeta_k)}_{G^{a_k}_m}^2&=\abs{\zeta_{k,\infty}}^2+e^{\delta_m R_k}\left[\nr \ov{\xi}_k+\zeta_{k, \infty}\nr^2_{3+m,-\delta_m}+
\nr \wh{\zeta}_{k}\nr^2_{m+3,\delta_m}\right]\\
&\leq C_m e^{\delta_m R_k} \left[\abs{\zeta_{k,\infty} }^2+\nr \ov{\xi}_k\nr^2_{3+m,-\delta_m}+\nr \wh{\zeta}_{k}\nr^2_{3+m,-\delta_m}\right]\\
&\leq C_m \varepsilon_k^2,
\end{split}
\end{equation*}
as claimed.

From \eqref{poker99} and $\oplus_{a_k}(\gamma_+r^+_k, \gamma_-r^-_k)=\ov{q}_k$ and $\ominus_{a_k}(\gamma_+r^+_k, \gamma_-r^-_k)=\ominus_{a_k}(p_k^+, p_k^-)=\zeta_k$, we have the following equations,
\begin{equation*}
\begin{aligned}
\oplus_{a_k}(\gamma_+r^+_k-p_k^+, \gamma_-r^-_k-p_k^-)&=\ov{q}_k-\ov{\xi}_k\\
\ominus_{a_k}(\gamma_+r^+_k-p_k^+, \gamma_-r^-_k-p_k^-)&=0.
\end{aligned}
\end{equation*}
We claim that the sequence $(\gamma_+r^+_k-p_k^+, \gamma_-r^-_k-p_k^-)$ is bounded in $\wh{E}_m$. By  Theorem \ref{boxdot-est1} is suffices to show that  the $G^{a_k}_{m+1}$-norm of 
$(\ov{q}_k-\ov{\xi}_k, 0)$ of maps  on the right-hand side is bounded.
Using \eqref{poker102} and the fact that $[\ov{\xi}_k]_{a_k}=0$,  we find that 
\begin{equation}\label{poker104}
\begin{split}
&\abs{(\ov{q}_k-\ov{\xi}_k, 0)}^2_{G^{a_k}_{m+1}}\\
&\phantom{===}=\abs{[q_k-\ov{\xi}_k]_{a_k}}^2+e^{\delta_{m+1}R_k}\cdot \nr \ov{q}_k-\ov{\xi}_k-[q_k-\ov{\xi}_k]_{a_k}\nr^2_{3+(m+1), -\delta_{m+1}}\\
&\phantom{===}=\abs{[q_k]_{a_k}}^2+e^{\delta_{m+1}R_k}\cdot \nr q_k-\ov{\xi}_k-[q_k]_{a_k}\nr^2_{3+(m+1), -\delta_{m+1}}\\
&\phantom{===}\leq \abs{[q_k]_{a_k}}^2 +C^2_m
\end{split}
\end{equation}
and  it remains to estimate $\abs{[q_k]_{a_k}}$.  Recall that $[q_k]_{a_k}=\avk (r^+_k, r^-_k)$ since $q_k=\oplus_{a_k}(\gamma_+r^+_k, \gamma_-r^-_k)$. Using  the Sobolev embedding theorem on bounded domains and the bound  
$\abs{r^+_k}_{H^{3+m,\delta_m}}\leq 1$, we estimate
$$\abs{e^{\delta_m \cdot }\cdot r^+_k}_{C^0(\Sigma_k)}\leq C_m\abs{e^{\delta_m \cdot }\cdot r^+_k}_{H^{m+3}(\Sigma_k)}\leq C_m$$
where $\Sigma_k=[\frac{R_k}{2}-1, \frac{R_k}{2}+1]\times S^1$.
This  implies 
\begin{equation*}
\abs{[r^+_k]_{a_k}}\leq C_me^{-\delta_m R_k/2}
\end{equation*}
and, similarly,  $\abs{[r^-_k]_{a_k}}\leq C_me^{-\delta_m R_k/2}$. Consequently,  
\begin{equation*}\label{average1}
\abs{ [q_k]_{a_k}}=\abs{\avk (r^+_k, r^-_k)}=\frac{1}{2}\abs{[r^+_k]_{a_k}+[r^-_k]_{a_k}}\leq C_me^{-\delta_m \frac{R_k}{2}}.
\end{equation*}
This implies that $\abs{(\ov{q}_k-\ov{\xi}_k, 0)}^2_{G^{a_k}_{m+1}}$ are  bounded and, in view by Theorem \ref{boxdot-est1},  that 
\begin{equation*}
\abs{(\gamma_+r^+_k-p_k^+, \gamma_-r^-_k-p_k^+)}_{E_{m+1}}\leq C_m.
\end{equation*}
Now, using the compact embedding $\wh{E}_{m+1}\to \wh{E}_m$, we  may assume after taking a subsequence that 
\begin{equation*}
(\gamma_+r^+_k-p_k^+, \gamma_-r^-_k-p_k^+)\to (g^+, g^-)\quad \text{in $\wh{E}_m$}
\end{equation*}
for some pair $(g^+, g^-)\in \wh{E}_m$. We have already proved that $(p_k^+, p_k^-)\to (0,0)$ in $\wh{E}_m$. Hence 
\begin{equation*}
(\gamma_+r^+_k, \gamma_-r^-_k)\to (g^+, g^-)\quad \text{in $\wh{E}_m$}.
\end{equation*}
This proves the second convergence in \eqref{contwo1}.

Summing up, we have proved that sequences $(\alpha_+r^+_k, \alpha_-r^-_k)$ and 
$(\gamma_+r^+_k, \gamma_-r^-_k)$ posses subsequences  which converge to some elements $(f^+, f^-)$ and $(g^+, g^-)$, respectively, in $\wh{E}_m$. Since the sequence of asymptotic constants $(h^+_{k,\infty}, h^-_{k,\infty})$ converges to $(h^+_{k}, h^-_{k})$, it follows that 
$(h^+_k, h^-_k)$ converges  in $\wh{E}_m$ to the element $(h^+, h^-)=(h^+_{\infty}+f^++g^+, h^-_{\infty}, f^-+g^).$  We  also have that $f^\pm (0, 0)=(0, 0)$ and 
$f^\pm (\{0\}\times S^1)\subset  \{0\}\times \R$. The proof of the proposition is complete.
\end{proof}

\end{document}